\documentclass[reqno,12pt,twoside]{article}
\usepackage{amsmath,amssymb}
\usepackage{amsmath, amsfonts,amsthm,amssymb,amsbsy,upref,color,graphicx,amscd,hyperref,enumerate,wasysym,subfigure}
\usepackage[active]{srcltx}
\usepackage[utf8]{inputenc}
\usepackage{bbm,stmaryrd}
\usepackage{algorithm}
\usepackage{array}
\usepackage{algorithmicx}
\usepackage{algpseudocode}
\usepackage{graphicx}

\usepackage{enumitem}
\usepackage{multirow}
\usepackage{spreadtab}
\graphicspath{{./all_pics_pdf/}}

\theoremstyle{definition}
\newtheorem{thm}{Theorem}[section]

\newtheorem{prop}[thm]{Proposition}
\newtheorem{conjecture}[thm]{Conjecture}
\newtheorem{defin}[thm]{Definition}
\newtheorem{rem}[thm]{Remark}

\numberwithin{equation}{section}

\newcommand{\subjclass}[1]{\bigskip\noindent\emph{2010 Mathematics Subject Classification:}\enspace#1}
\newcommand{\keywords}[1]{\noindent\emph{Keywords:}\enspace#1}

\makeatletter
\DeclareFontFamily{U}{tipa}{}
\DeclareFontShape{U}{tipa}{m}{n}{<->tipa10}{}
\newcommand{\arc@char}{{\usefont{U}{tipa}{m}{n}\symbol{62}}}%

\newcommand{\arc}[1]{\mathpalette\arc@arc{#1}}

\newcommand{\arc@arc}[2]{%
  \sbox0{$\m@th#1#2$}%
  \vbox{
    \hbox{\resizebox{\wd0}{\height}{\arc@char}}
    \nointerlineskip
    \box0
  }%
}
\makeatother

\def\R  {\mathbb R}

\def\cD{\mathcal D}
\def\fP{\mathfrak P}
\def\fL{\mathfrak L}
\def\sA{\mathsf A}
\def\sB{\mathsf B}
\def\sC{\mathsf C}
\def\sD{\mathsf D}
\def\sX{\mathsf X}
\def\sY{\mathsf Y}
\def\sO{\mathsf O}
\definecolor{mg}{rgb}   {0.85,  0.,    0.85}
\definecolor{gn}{rgb}   {0,  .39,    0}
\newcommand{\Mg}{\color{black}}
\newcommand{\Gn}{\color{black}}
\newcommand{\Rd}{\color{black}}
\newcommand{\Bl}{\color{black}}

\begin{document}

\title{Minimal partitions for $p$-norms of eigenvalues}

\author{Beniamin Bogosel, 
Virginie Bonnaillie-No\"el\\
D\'epartement de Math\'ematiques et Applications\\ \'Ecole normale sup\'erieure, CNRS, PSL Research University,\\ 45 rue d'Ulm, 75005 Paris\\
beniamin.bogosel@ens.fr, bonnaillie@math.cnrs.fr}

\date{}


\maketitle
\begin{abstract}
In this article we are interested in studying partitions of the square, the disk and the equilateral triangle which minimize a \emph{p}-norm of eigenvalues of the Dirichlet-Laplace operator. The extremal case of the infinity norm, where we minimize the largest fundamental eigenvalue of each cell, is one of our main interests. We propose three numerical algorithms which approximate the optimal configurations and we obtain tight upper bounds for the energy, which are better than the ones given by theoretical results. A thorough comparison of the results obtained by the three methods is given. We also investigate the behavior of the minimal partitions with respect to \emph{p}. This allows us to see when partitions minimizing the 1-norm and the infinity-norm are different.     


\end{abstract}

\section{Introduction}

\subsection{Motivation}
In this paper we are interested in determining minimal partitions for cost functionals involving the $p$-norm of some spectral quantities ($ p\geq 1$ or $p=\infty$). 

Let $\Omega$ be a bounded and connected domain in $\R^2$ with piecewise-$\mathcal C^1$ boundary and $k$ be a positive integer $k\geq1$.
For any domain $D\subset\Omega$, $(\lambda_{j}(D))_{j \geq 1}$ denotes the eigenvalues of the Laplace operator on $D$ with Dirichlet boundary conditions, arranged in non decreasing order and repeated with multiplicity. 

We denote by $\fP_{k}(\Omega)$ the set of $k$-partitions $\cD=(D_{1},\ldots, D_{k})$ such that 
\begin{itemize}
\item $(D_{j})_{1\leq j\leq k}$ are connected, open and mutually disjoint subsets of $\Omega$, 
\item ${\rm Int}(\bigcup_{1\leq j\leq k}\overline{D_{j}})\setminus \partial\Omega = \Omega$.
\end{itemize}
For any $k$-partition $\cD\in\fP_{k}(\Omega)$, we define the \emph{$p$-energy} by
\begin{equation}\label{eq.nrjkp}
\Lambda_{k,p}(\mathcal{D})=\left(\frac{1}{k}\sum_{i=1}^{k}\lambda_{1}(D_i)^p\right)^{1/p},\qquad\forall p\geq 1.
\end{equation}
By extension, if we consider the {\Bl infinity} norm, we define the \emph{energy} of $\cD$ by
\begin{equation}\label{eq.nrjk}
\Lambda_{k,\infty}(\mathcal{D})=\max_{1\leq i\leq k}\lambda_{1}(D_i).
\end{equation}
With a little abuse of notation, we notice that 
$$\Lambda_{k,p}(\mathcal{D})=\frac{1}{k^{1/p}}\Big\| \big(\lambda_{1}(D_{1}),\ldots,\lambda_{1}(D_{k})\big)\Big\|_{p}.$$
The index $\infty$ is omitted when there is no confusion. 
The optimization problem we consider is to determine the infimum of the $p$-energy ($1\leq p\leq \infty$) among the partitions of $\fP_{k}(\Omega)$:
\begin{equation}\label{eq.Lkp}
\fL_{k,p}(\Omega) = \inf_{\cD\in\fP_{k}(\Omega)}\Lambda_{k,p}(\cD),\qquad \forall 1\leq p\leq \infty,\quad \forall k\geq1.
\end{equation}
A partition $\cD^*$ such that $\Lambda_{k,p}(\cD^*)=\fL_{k,p}(\Omega)$ is called a $p$-minimal $k$-partition of $\Omega$.

This optimization problem has been a subject of great interest in the last twenty years. Two cases are especially studied: the sum which corresponds to $p=1$ and the max, corresponding to $p=\infty$. General aspects concerning existence results for optimal partitions problems are presented in \cite{BucButHen98,BucBut05}. Existence and regularity results for optimal partitioning problems regarding non-linear eigenvalue problems, containing as a particular case the Dirichlet eigenvalues, are considered in \cite{CoTeVe03}. In \cite{CafLin07} the authors consider the minimization of the partitions minimizing the sum of the Dirichlet-Laplace eigenvalues, stating the spectral honeycomb conjecture and initiating many theoretical and numerical works on the subject. In \cite{HelHofTer09} the authors consider the partitions minimizing the maximum of the fundamental eigenvalues and they provide results concerning connections between such optimal partitions and nodal partitions, for particular values of $k$. 
More recently, the link between these two optimization problems is taken into consideration in \cite{HHO10}. In particular, a criterion is established to assert that a $\infty$-minimal $k$-partition is not a $1$-minimal $k$-partition. This criterion is given in Proposition~\ref{l2norm} and applied in Section~\ref{ss.summax}.
 
There are few cases for which optimal partitions are known explicitly for the spectral quantities we consider here. This motivates the development of numerical algorithms which can find approximations of optimal partitions and suggest candidates as optimal partitions. The case $p=1$, corresponding to the sum of the eigenvalues, was considered in \cite{BouBucOud09}, where an algorithm based on a relaxation procedure was presented. The algorithm allowed the study of partitions made of several hundreds of cells and shows that it is likely that partitions made of hexagons are a good candidate to being minimal as $k \to \infty$. The numerical minimization of the largest eigenvalue has been considered in \cite{BHV,BH11,BonLen14,BonLen16}. 
 In \cite{BHV}, we exhibit some candidates for the 3-partition of the square and the disk by using a mixed Dirichlet-Neumann approach that will be used in Section~\ref{sec.DN} in a more systematic way. 
The nodal partition of a suitable Aharonov-Bohm operator can produce rather good candidates for the minimal partitions for the max and \cite{BH11,BonLen14} focus on the computation of the spectrum in the case of the square and angular sectors.
Then \cite{BonLen16} is a first adaptation of the algorithm of \cite{BouBucOud09} for the max but without analysis of the behavior according to the parameters and the $p$-norm. This article only gives candidates for a family of tori.

There are also other works dealing with optimal partitions for eigenvalues. Among these we mention \cite{OOWgraphs} where the authors use a rearrangement algorithm to find numerical minimizers for spectral graph partitions, \cite{ZOgraph} where authors present various results concerning graph and plane partitions. In \cite{Bozorgnia} algorithms for minimizing the sum and the maximum of the eigenvalues are provided, but with few explicit examples. In \cite{cybhol} the authors present a model of chemical reaction which leads to a segregation of phases and is in connection with the minimization of the sum of the eigenvalues. The analogue problem of minimizing the sum of the eigenvalues of the Laplace-Beltrami operator on surfaces was considered numerically in \cite{ElliottRanner2014Surfaces}. The algorithms we propose in the following are combining aspects from some of the works presented above. In particular, our iterative algorithms use the numerical relaxation for eigenvalue problems presented in \cite{BouBucOud09}, for different functionals, replacing the sum by a $p$-norm or adding a penalization of the difference of the eigenvalues. In some cases we can exploit the particular structure of the result obtained using the iterative algorithms, and try to express such partitions as nodal partitions corresponding to eigenvalue problems on domains with additional Dirichlet boundary conditions. When symmetry is available we may reduce the computational domain by considering mixed Dirichlet-Neumann boundary conditions.

We start by stating the following existence result {(see \cite{BucButHen98,HelHofTer09})}.
\begin{thm} \label{thm.ex}
For any $k\geq1$ and $p\in[1,+\infty]$, there exists a regular $p$-minimal $k$-partition.
\end{thm}
Let us recall that a $k$-partition $\cD$ is called \emph{regular} if its boundary, $N(\cD)= \cup_{1\leq i\leq k}\partial D_{i}$, is locally a regular curve, except at a finite number of singular points, where a finite number of half-curves meet with equal angles. We say that $\cD$ satisfies the \emph{equal angle meeting property}.

{In the case $k=1$, {since} $\Omega$ is connected, then the $p$-minimal $1$-partition is $\Omega$ itself, for any $p$. From now, we will consider $k\geq 2$.}
\begin{rem}
Note that if we relax the condition ${\rm Int}(\bigcup_{1\leq j\leq k}\overline{D_{j}})\setminus \partial\Omega = \Omega$ and consider the optimization problem among partitions such that we have only an inclusion 
\begin{equation}\label{eq.partiincl}
{\rm Int}(\bigcup_{1\leq j\leq k}\overline{D_{j}})\setminus \partial\Omega \subset \Omega,
\end{equation} 
Theorem~\ref{thm.ex} is still available and any $p$-minimal $k$-partitions is strong (this means we have equality in \eqref{eq.partiincl}).
\end{rem}

\subsection{Main results and organisation of the paper}
 The goal of the article is to study the minimization of the $p$-norm of the eigenvalues for $p$ large and $p=\infty$. We use our algorithms in a comparative way for three basic geometries: the square, the disk and the equilateral triangle, for a number of cells $k$ between $2$ and $10$. 
In Section \ref{iterative} we present  an iterative algorithm for the optimization of the $p$-norm based on the results of \cite{BouBucOud09}.  We use a relaxed framework for the computation of the eigenvalues and we adapt the expression of the functional and the gradients provided in \cite{BouBucOud09} in order to deal with $p$-norms. We observe that the implementation produces different results when we consider the minimization problem for $p=1$ or $p=\infty$, and therefore we have a new numerical confirmation that, in general, optimal partitions change between $p=1$ and $p=\infty$. This motivates us to look closer at the case $p=\infty$ and to seek algorithms which are adapted to this case. 

In Section \ref{theory} we recall some theoretical aspects needed in order to analyze our numerical results and also to propose more efficient algorithms. Among these results we underline the equipartition property concerning the case $p=\infty$ and a $L^2$-norm criterion which can indicate whether an optimal partition for $p=\infty$ is not optimal for $p=1$.

Next, in Section \ref{section.max} we concentrate our attention on the numerical study of the $\infty$-minimal partitions. We describe a new iterative method based on a penalization of the difference of the eigenvalues and the mixed Dirichlet-Neumann approach where we restrict ourselves to nodal partitions of a mixed problem. 
Here we compare the three methods and exhibit better upper bounds for $\fL_{k,\infty}(\Omega)$ for the three geometries considered: the square, the equilateral triangle and the disk. At the end of this section, we show that almost all of the candidates to be $\infty$-minimal $k$-partition can not be optimal for the sum, in coherence with theoretical results of \cite{HHO10}.

In Section \ref{section.pnorm} we analyze the behavior of the optimal partitions for the $p$-norm with respect to $p$ by looking at the evolution of the associated energies and the partitions. In the case of the square, the disk or the equilateral triangle we notice that the energy $\fL_{k,p}(\Omega)$ seems to be strictly increasing with $p$, except some particular cases where the energy and partitions do not vary with $p$, suggesting that in these cases we have the same optimal partitions for $p=1$ and $p=\infty$. We conclude in Section \ref{section.conclusion} presenting a summary of our numerical results and formulating some relevant conjectures in the further study of minimal spectral partitions.
\section{Numerical iterative algorithm}
\label{iterative}
\subsection{Numerical method for the sum}
The problem of minimizing numerically the sum of the first eigenvalues of the Dirichlet-Laplace operator corresponding to a partition of a planar domain $\Omega$ has been studied numerically by Bourdin, Bucur and Oudet in \cite{BouBucOud09}. 
In order to simplify the computation and the representation of the partition they represented each cell of the partition as a discrete function on a fixed finite differences grid. It is possible to compute the first eigenvalue of a subset $D$ of $\Omega$ by using a relaxed formulation of the problem based on \cite{dalmaso-mosco}. If $\varphi$ is a function which approximates $\chi_D$, the characteristic function of $D$, then we consider the problem
\begin{equation} 
\left\{
\begin{array}{cl}
-\Delta u + C(1-\varphi) u = \lambda_j(C,\varphi) u& \qquad\mbox{ in } \Omega,\\
\Mg u=0&\Mg\qquad\mbox{ on }\partial\Omega,
\end{array}
\right.
\label{relax}
\end{equation}
with $C \gg 1$. In the case where $\varphi = \chi_D$ it is proved that $\lambda_1(C,\varphi) \to \lambda_1(D)$ as $C\to\infty$. Moreover, in \cite{BoVe16} the following quantitative estimation of the rate of convergence is given: if $\varphi = \chi_D$ then
\begin{equation}
 \frac{|\lambda_1(D)-\lambda_1(C,\varphi)|}{\lambda_1(D)} = O(C^{-1/6}). 
 \label{quantitative}
 \end{equation}
The same estimate remains true for higher eigenvalues. 
As a consequence of the quantitative estimation given above, it is desirable to have a penalization constant $C$ as large as possible in our computations, in order to obtain a good approximation of the eigenvalues. The discretization of the problem \eqref{relax} is straightforward if we consider a finite differences grid. We consider a square bounding box containing the domain $\Omega$. On this box we construct a $N \times N$ uniform grid and we approximate the Laplacian of $u$ using centered finite differences. This allows us to write a discrete version of problem \eqref{relax} in the following matrix form
\begin{equation}
 (A+\text{diag}(C(1-\tilde \varphi))\tilde u = \lambda_1(C,\tilde\varphi) \tilde u, 
 \label{discrete}
 \end{equation}
where the matrix $A$ is the discrete Laplacian on the finite differences grid and  $\tilde u$ a column vector. The Dirichlet boundary condition on $\partial \Omega$ is implemented in \eqref{discrete} by imposing that the density functions $\tilde \varphi$ take zero values on nodes on $\partial \Omega$. The matrices involved in the discrete form of the problem \eqref{discrete} are sparse and thus the problem can be solved efficiently in Matlab using \texttt{eigs}. We note here that the domain $\Omega$ does not need to fill the whole bounding box and that imposing that the functions $\varphi$ are zero on the nodes outside $\Omega$ automatically adds a penalization factor on these nodes. In this way we can study various geometries, like the disk and the equilateral triangle, while still working on a finite-difference grid on a square bounding box. 

\begin{rem}
Finite element formulations are also possible and we refer to \cite{BoVe16} for a brief presentation. One drawback is that if we consider finite elements then the discrete problem analogue to \eqref{discrete} is a generalized eigenvalue problem. The computational cost in this case is higher and this prevents us from being able to work with fine discretizations. 
\end{rem}

In our numerical study of optimal partitioning problems in connection to spectral quantities we use the approach described above to represent the cells and to compute the eigenvalues. We replace each set $D_j$ by a discrete density function $\tilde \varphi_j : \Omega \to [0,1]$ and use the formulation \eqref{relax} and its discrete form \eqref{discrete} to compute an approximation of $\lambda_1(D_j)$. The condition that the sets $(D_j)_{1\leq j\leq k}$ form a partition of the domain $\Omega$ can be implemented by imposing that the densities $\tilde \varphi_j$ associated to $D_j$ have sum equal to one:
\[ \sum_{j=1}^k \tilde\varphi_j = 1.\]
In order to have an efficient optimization algorithm we use a gradient based approach. For this we compute, { for any $\tilde\varphi=\tilde\varphi_{j}$, $1\leq j\leq k$,} the gradient of $\lambda_1(C,\tilde\varphi)$ with respect to each node of the grid and, as in \cite{BouBucOud09}, we get
\[ \partial_{i} \lambda_1(C,\tilde \varphi)  = -C \tilde u_{i}^2,\qquad i = 1,\ldots,N.\] 

\subsection{Adaptation for the $p$-norm}
As we see in the introduction, we are not only interested in the optimization problem for the sum (see \eqref{eq.nrjkp} with $p=1$), but also for any $p$-norm and one of our objectives is to study numerically the minimizers of the quantity 
\begin{equation}
 \max_{1\leq j\leq k} \lambda_1(D_j).
 \label{max-form}
 \end{equation}
This functional is non-smooth and therefore we cannot minimize it directly. One way to approach minimizers of \eqref{max-form} has been proposed in \cite{BonLen16} and it consists in minimizing instead the $p$-norms $\Lambda_{k,p}(\mathcal{D})$ defined in \eqref{eq.nrjkp}, for large $p$:
It is clear that as $p \to \infty$ these $p$-norms $\Lambda_{k,p}(\mathcal{D})$ converge to the largest eigenvalue among $\{\lambda_1(D_j),\ 1\leq j\leq k\}$. In order to optimize $\Lambda_{k,p}(\mathcal{D})$ we modify the expression of the gradient in the algorithm presented in \cite{BouBucOud09} by adding a factor corresponding to the derivative of the $p$-norm
\[ \partial_i \Lambda_{k,p}(\mathcal{D})
=\left( \frac{1}{k}\sum_{j=1}^k \lambda_1(C,\tilde\varphi_j)^p \right)^{1/p-1} \times \left(\frac{1}{k}\sum_{j=1}^k  \lambda_1(C,\tilde\varphi_j)^{p-1} \partial_i \lambda_1(C,\tilde \varphi_j) \right). \]

\subsection{Grid restriction procedure}
We perform the optimization starting from random admissible densities on a $60\times 60$ grid on the square bounding box. In order to have a more precise description of the contours we perform a few successive refinements by doubling the number of discretization points in both horizontal and vertical directions, until we reach a $480\times 480$ grid. More precisely, given a grid size, we apply a gradient descent algorithm using the expression of the gradient of the eigenvalue given in the previous subsection. At each iteration, after the update of the functions $\tilde \varphi_j$ we project them on the constraint condition by replacing each function $\tilde \varphi_j$ by $|\tilde \varphi_j|/(\sum_{i=1}^k |\tilde \varphi_i|)$. This projection algorithm is the same as the one suggested in \cite{BouBucOud09}. We stop when the value of the $p$-norm does not decrease when considering a step length of at least $10^{-6}$. Once we obtain a numerical solution on a given grid we use an interpolation procedure to pass to a denser grid. Then we restart the gradient descent algorithm on this new grid starting from the interpolated partition. We stop when we reach a grid of the desired size, in our case $480\times 480$. We notice that on the $480\times 480$ grid we cannot use a penalization parameter $C$ which is greater than $10^4$, since the matrix $A+\text{diag}(C(1-\tilde \varphi))$ becomes ill conditioned. Indeed, we can see that a large part of the grid is not really used in the computation of the eigenvalue, since, in most cases, roughly $N^2/k$ of the points of {\Gn a $N \times N$} grid are covered by the support of $\tilde \varphi_{j}$ (which should converge to some subdomain $D_{j}$ of a minimal $k$-partition). In order to surpass this problem and to be able to increase the parameter $C$ we propose the following modification of the algorithm used in \cite{BouBucOud09}.

The initial densities are chosen randomly and projected onto the constraint like shown in \cite{BouBucOud09}. At each iteration of the gradient method, we look for the points of the grid which satisfy $\tilde \varphi_j > 0.01$ (represented with dark blue in Figure \ref{grid-refinement}) and then we compute the smallest rectangular region of the grid which contains these points (represented with red in Figure \ref{grid-refinement}). As you can see in Figure \ref{grid-refinement} the first two situations correspond to cases where the cell function $\tilde \varphi_{j}$ is not localized. On the other hand, from the moment when the cell is concentrated on only one part of the partitioned region $\Omega$ the rectangular neighborhood is much smaller and the amount of points where we need to impose the penalization is diminished. The points where the penalization is imposed are represented with cyan in Figure \ref{grid-refinement}. Note that in order to allow the cells to interact we extend the rectangular neighborhood with at least $5$ rows/columns (if contained in $\Omega$). In order to keep the advantage of working on a fixed computation grid, we set the cell's discrete values and gradient equal to zero on the points outside the local rectangular grid. This is natural, since cells which are far away do not have great impact on the dynamic of the current cell. Note that this procedure does not restrict the movement of the cells since these rectangular neighborhoods are dynamically computed at each iteration. Since the number of points on which we impose the penalization is significantly decreased the discrete problem remains well posed even for larger values of $C$ of order $10^7$. Figure~\ref{grid-refinement} represents the evolution of the set $\{\tilde \varphi_{7}>0.01\}$ and so of the local grid after 1, 10, 25, 45 and 85 iterations of the gradient method when we implement the algorithm with $k=10$ and $p=1$. Here, we have not yet done any {\Gn refinement} of the grid.
\begin{figure}[h!]
\centering 
\includegraphics[width = 0.19 \textwidth]{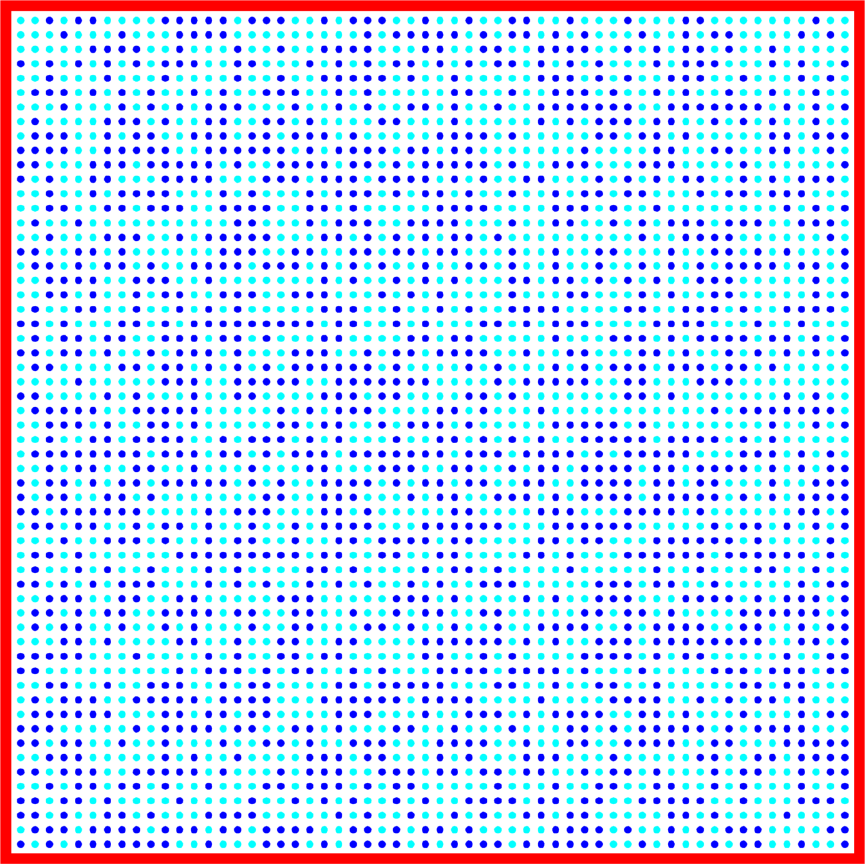}
\includegraphics[width = 0.19 \textwidth]{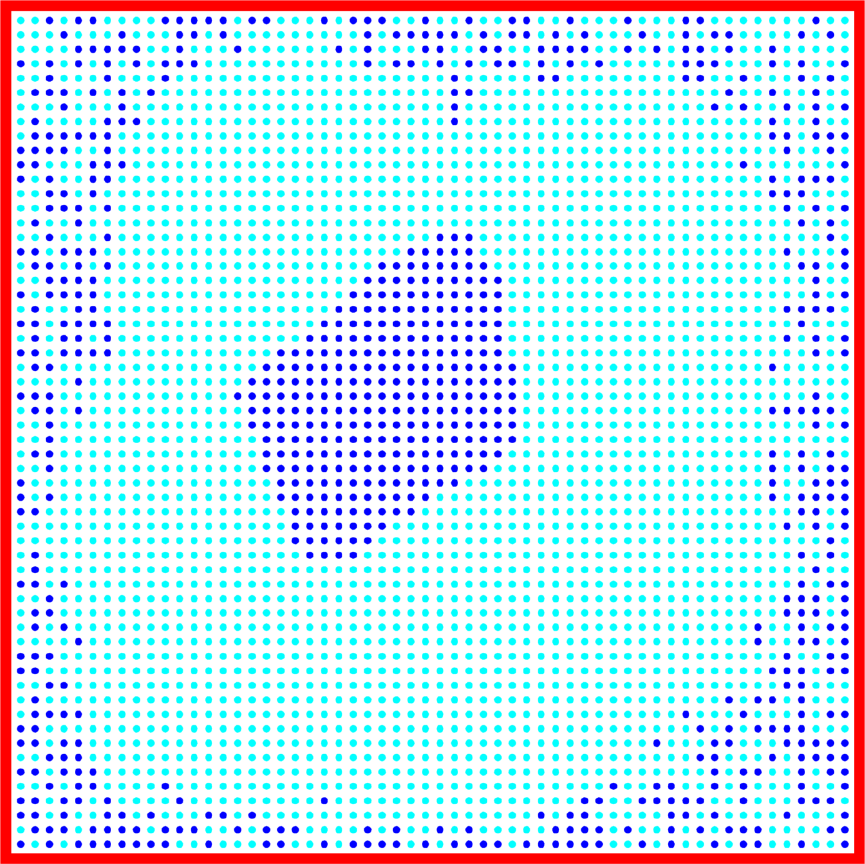}
\includegraphics[width = 0.19 \textwidth]{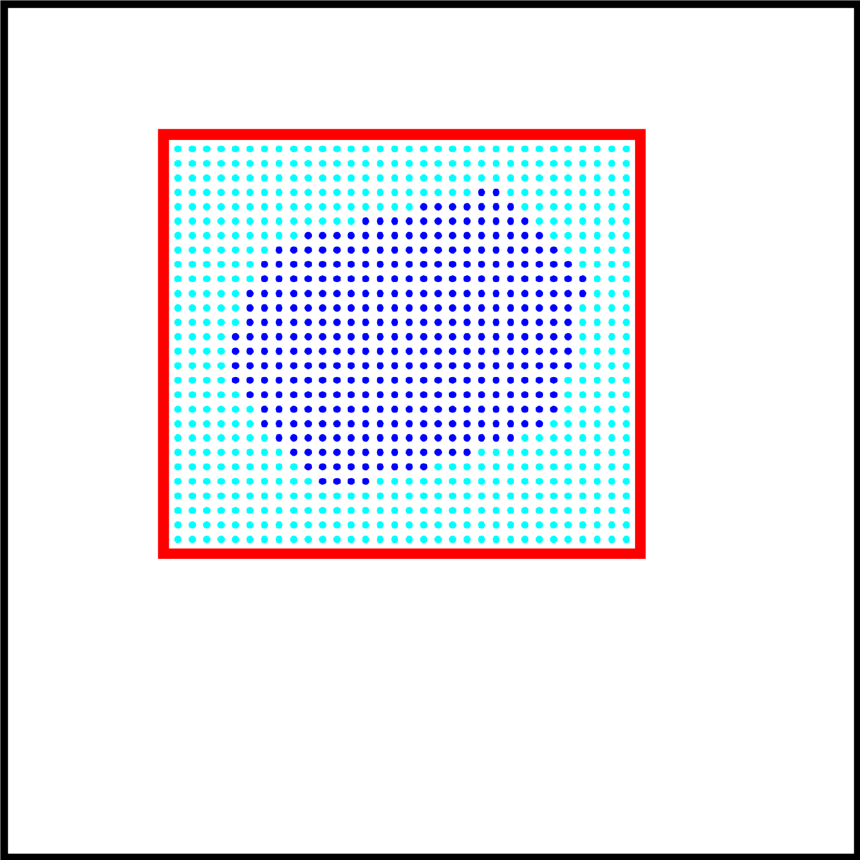}
\includegraphics[width = 0.19 \textwidth]{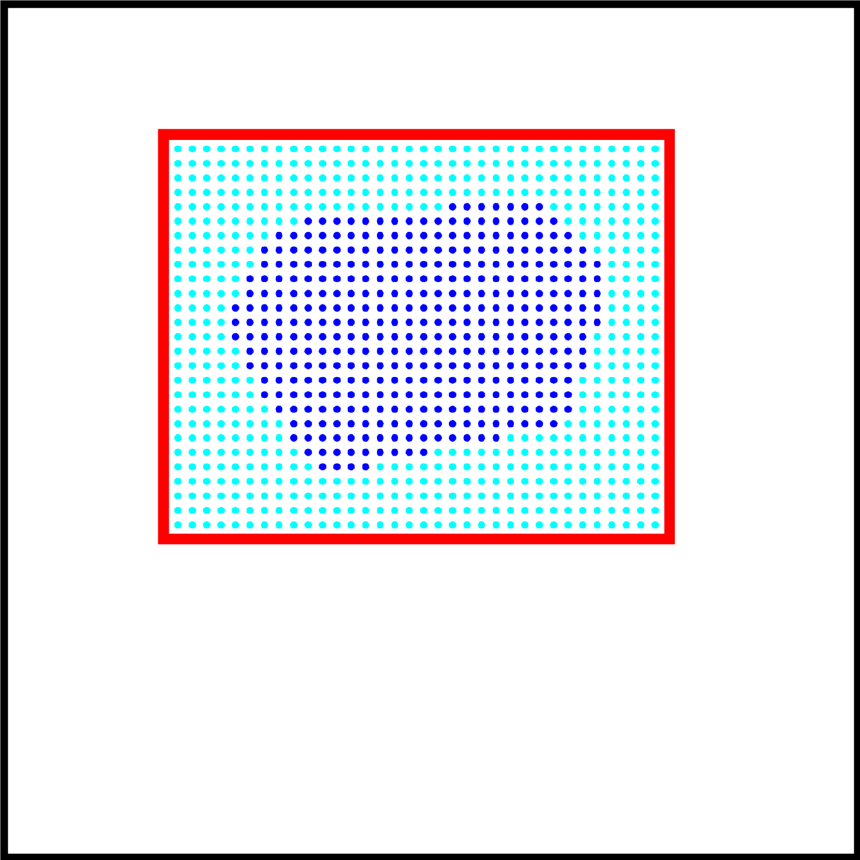}
\includegraphics[width = 0.19 \textwidth]{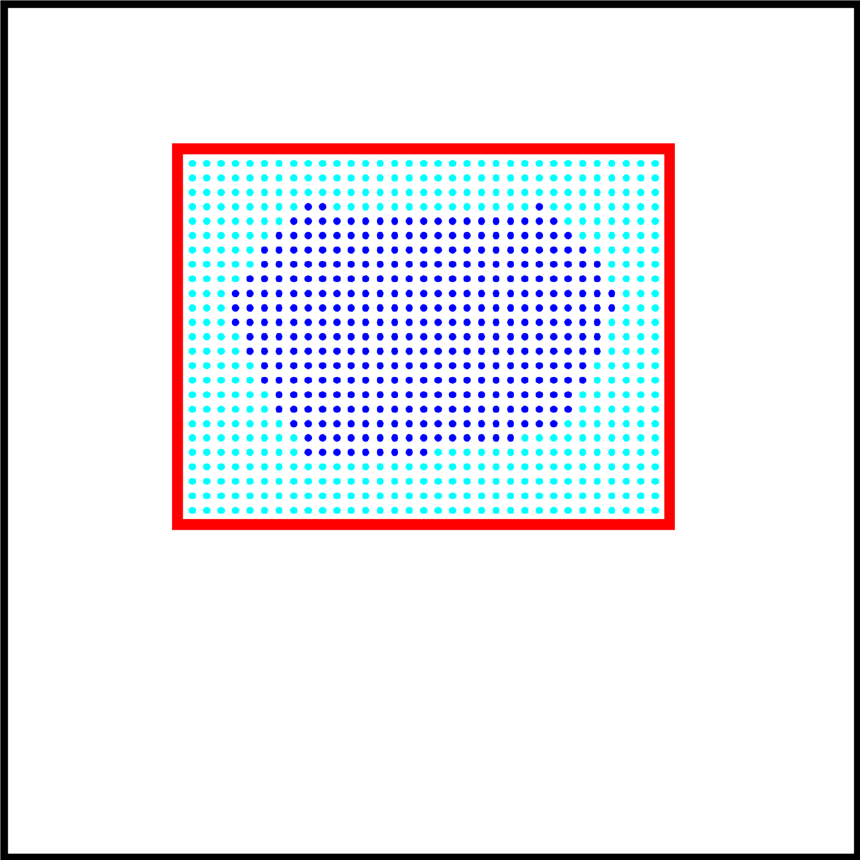}
\caption{Evolution of the local grid for a cell for iterations $1,10,25,45,85$. This computation corresponds to $k=10$ and $p=1$.}
\label{grid-refinement}
\end{figure}

The optimization procedure described above uses a relaxed formulation. Let us now describe how this allows to construct a partition $\cD^{k,p}$ of $\Omega$ whose energy will be computed with a finite element method. 
\begin{itemize}
\item For each $i \in \{1,\ldots,k\}$ we look for the grid points where $\tilde \varphi_i  \geq \tilde \varphi_j$ for every $j \neq i$.
\item We use Matlab's \texttt{contour} function to find the contour associated to these points.
\end{itemize}
This approach, as opposed to looking directly at some level sets of $\tilde \varphi$, has the advantage that the contours we obtain form a strong partition $\cD^{k,p}$ of the domain $\Omega$. 
Then we compute the first Dirichlet-Laplacian eigenvalue on each subdomain of the partition by using a finite element method : either each cell is then triangulated using the free software Triangle \cite{triangle} and its Dirichlet-Laplacian eigenvalues are computed using the finite elements library {\sc M\'elina} \cite{melina}, or we use \textsc{FreeFEM}++ \cite{freefem}. In cases where both {\sc M\'elina} and \textsc{FreeFEM}++ are used we recover the same results.

\subsection{Remarks on the accuracy of the numerical methods} 
The use of the relaxed formulation \eqref{relax} is well adapted when working with partitions, but it leads to a certain loss of precision. A study of the precision of the method compared to the precision of more precise spectral methods is performed in \cite[Section 5]{BoVe16}. The previous study looks at a few examples and takes into account both the finite difference discretization parameter $N$ and the penalization parameter $C$. First, let's note that the quantitative error result \eqref{quantitative} shows a slow convergence as $C$ increases. This was also observed in the numerical computations made in \cite{BoVe16}. Moreover, relative errors observed in the simulations in \cite{BoVe16} range between $10^{-3}$ and $10^{-2}$ for $C$ up to $10^9$ and $N \leq 500$. In our case we cannot expect to have a better accuracy, since when considering multiple cells in our computational domain, this would correspond to an eventual lower resolution when we restrict to a grid around each of the cells. Therefore, having an error around $1\%$ is to be expected. 

We underline that the errors in the computation of the eigenvalues come mainly from the use of the relaxed formulation in our iterative algorithms. When computing the eigenvalues of shapes defined after extracting the contours of the partitions, third and fourth order Lagrange finite element methods {\Gn($\mathbb P_{3}$ or $\mathbb P_{4}$)}  are used (in \textsc{M\'elina} and \textsc{FreeFem++}), which are quite precise. When using the Dirichlet-Neumann method in Section \ref{section.max} precise high order finite element methods are used, which assure the high accuracy of the results.

\subsection{Numerical results\label{ssec.num}}
We denote by $\cD^{k,p}$ the partition obtained by the iterative numerical method. 
We study three particular geometries of $\Omega$:  a square $\square$ of sidelength $1$, a disk $\Circle$ of radius $1$ and an equilateral triangle $\triangle$ of sidelength $1$. We perform computations up to $p=50$. Computations for higher $p$ lead to instabilities in our numerical algorithms due to large powers which appear in the computation of the $p$-norm and its derivative. Moreover, for $p \in (40,50)$ the optimal energy of partitions varies very little, of the order of $0.01\%$. We notice that the partitions obtained numerically for $p=50$ are good candidates to approximate the $\infty$-minimal $k$-partitions, since when performing the analysis of the evolution of the optimal energies and eigenvalues with respect to $p$ the maximal eigenvalues is greatly reduced between $p=1$ and $p=50$. Further analysis presented in the next sections reinforce this argument.

Let us first consider the case of the disk. When $k=2,3,4, 5$, the algorithm gives the same partition for the two optimization problems (the sum $p=1$ and the max $p=\infty$). These partitions, given in Figure~\ref{fig.disk}, are composed of $k$ similar angular sectors of opening $2\pi/k$ and then, the first eigenvalues on each cell are equal. 
Some comments about the relation to the notion of equipartition will be addressed in the next section.
It is conjectured that the ``Mercedes partition'' is minimal for the max, but this result is not yet proved (see \cite{HelHof10, BoHe13}). These simulations reinforce this conjecture.

\begin{figure}[h!]
\centering 
\begin{tabular}{m{3cm}m{3cm}m{3cm}m{3cm}m{3cm}}
\multicolumn{1}{c}{$k=2$} & \multicolumn{1}{c}{$k=3$} & \multicolumn{1}{c}{$k=4$} & \multicolumn{1}{c}{$k=5$}\\
\includegraphics[width = 0.2\textwidth]{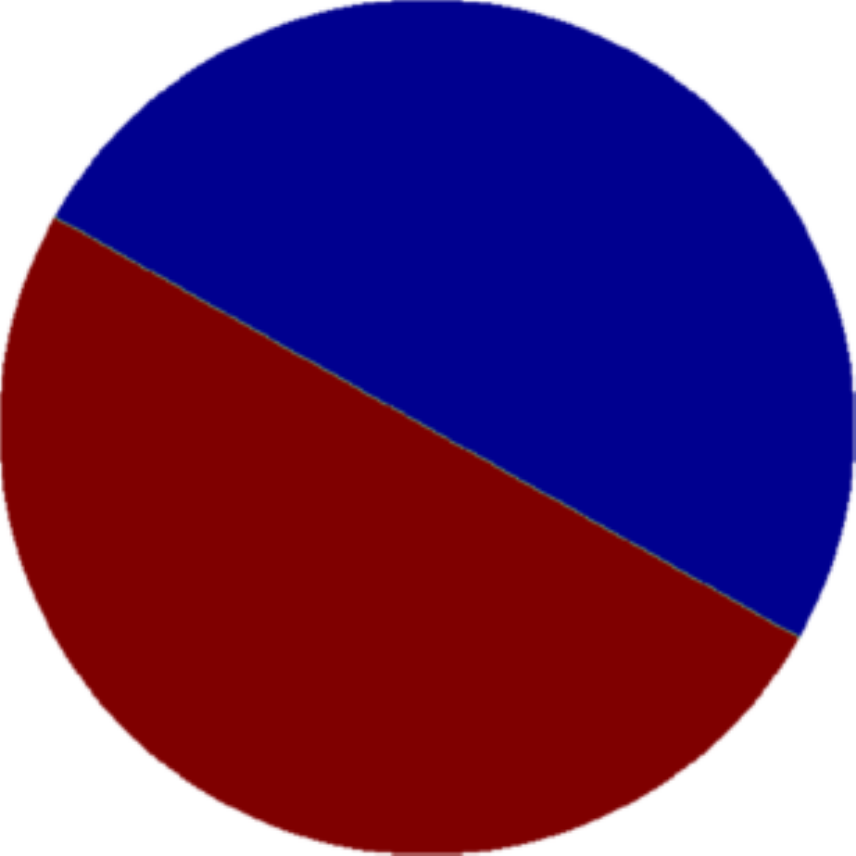} &
 \includegraphics[width = 0.2\textwidth]{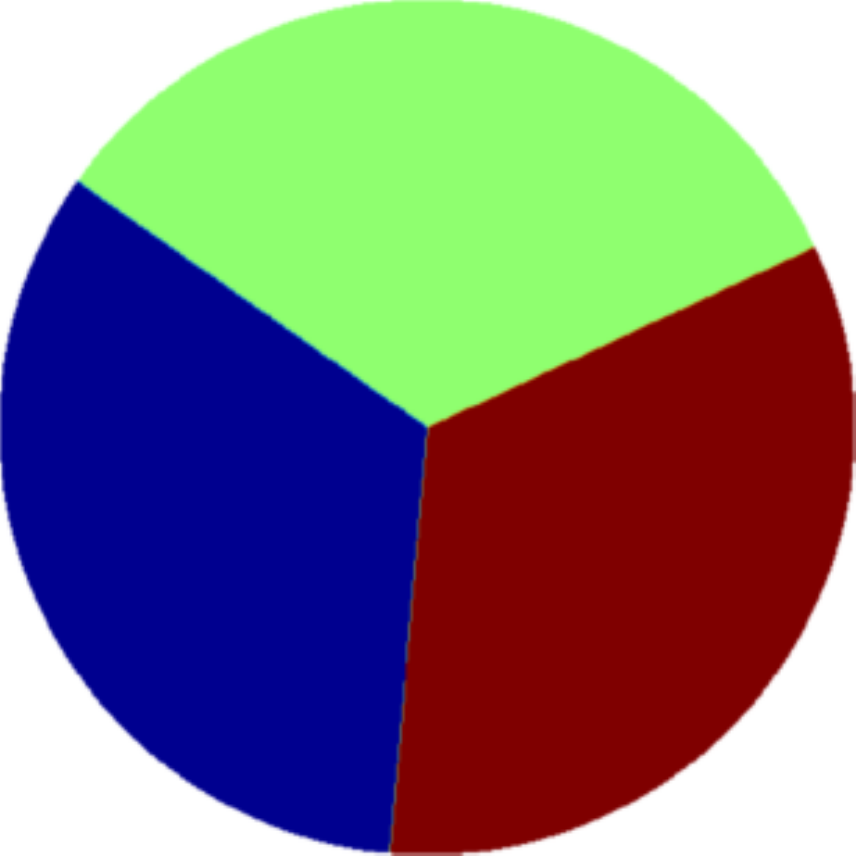} &
 \includegraphics[width = 0.2\textwidth]{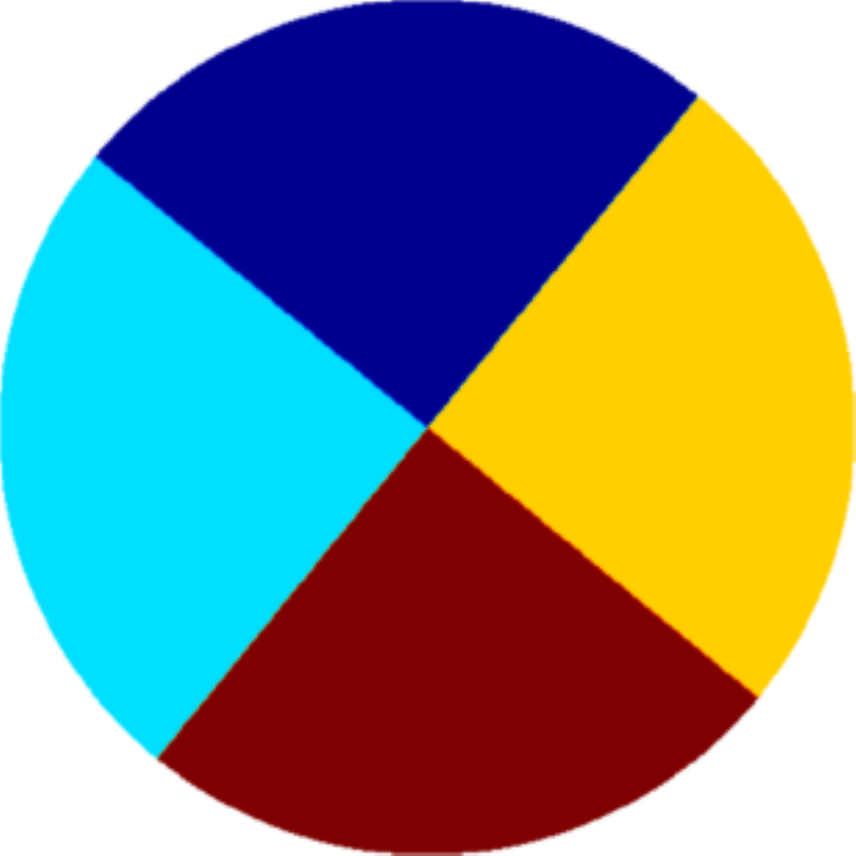} &
 \includegraphics[width = 0.2\textwidth]{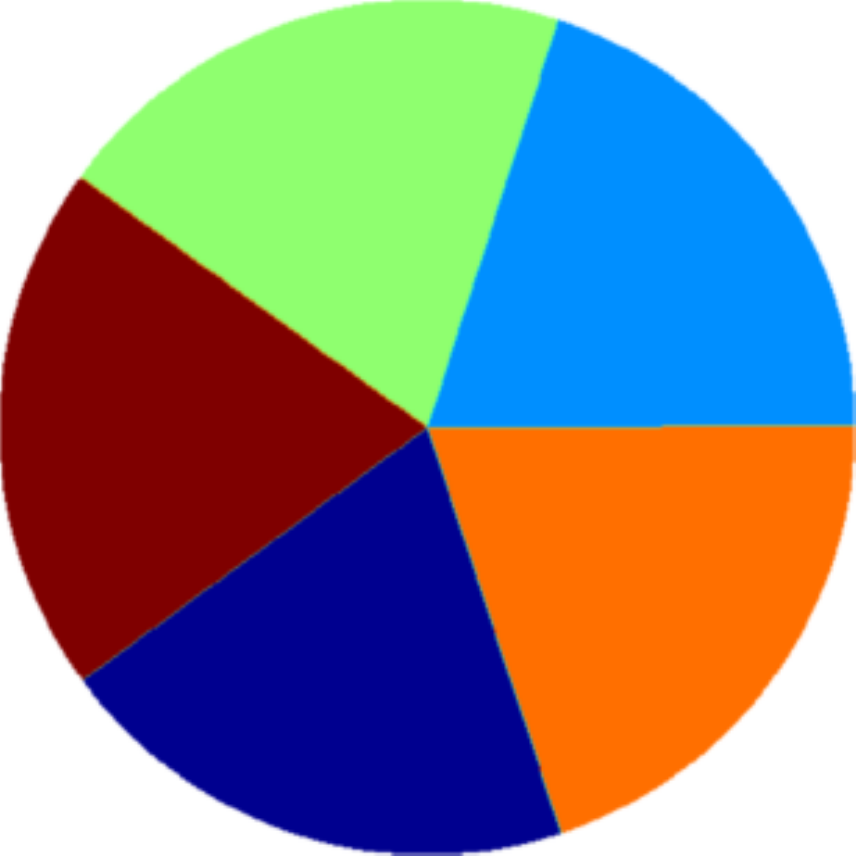}
\end{tabular}
\caption{Candidates for $p$-minimal $k$-partitions of the disk for $p=1$ and $50$.}
\label{fig.disk}
\end{figure}

We illustrate in Figure \ref{fig.equilateral} the results obtained for $p=1,50$ and $k \in\{2,3,4, 5\}$ in the case of the equilateral triangle. Note that except for $k=4$, partitions do not change much their structure. The case $k=4$ for the equilateral triangle is one of the few cases where the topology of the partition changes significantly with $p$, approaching the partition into $4$ equal triangles as $p$ is increasing. 
\begin{figure}[h!]
\centering 
\begin{tabular}{m{1.5cm}m{3cm}m{3cm}m{3cm}m{3cm}m{3cm}}
& \multicolumn{1}{c}{$k=2$} & \multicolumn{1}{c}{$k=3$} & \multicolumn{1}{c}{$k=4$} & \multicolumn{1}{c}{$k=5$}\\
{ $p=1$ }& \includegraphics[width = 0.19\textwidth]{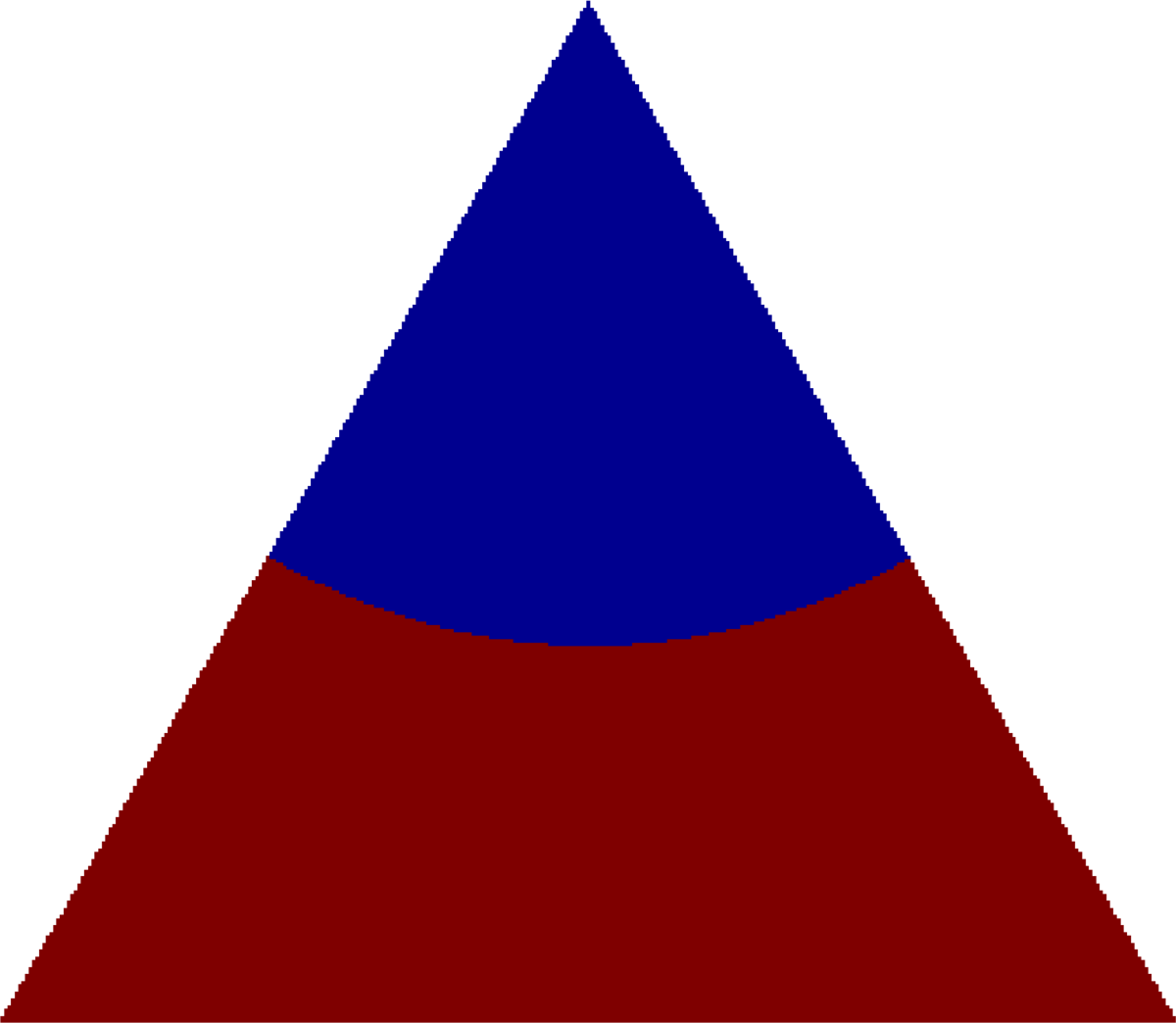} &
 \includegraphics[width = 0.19\textwidth]{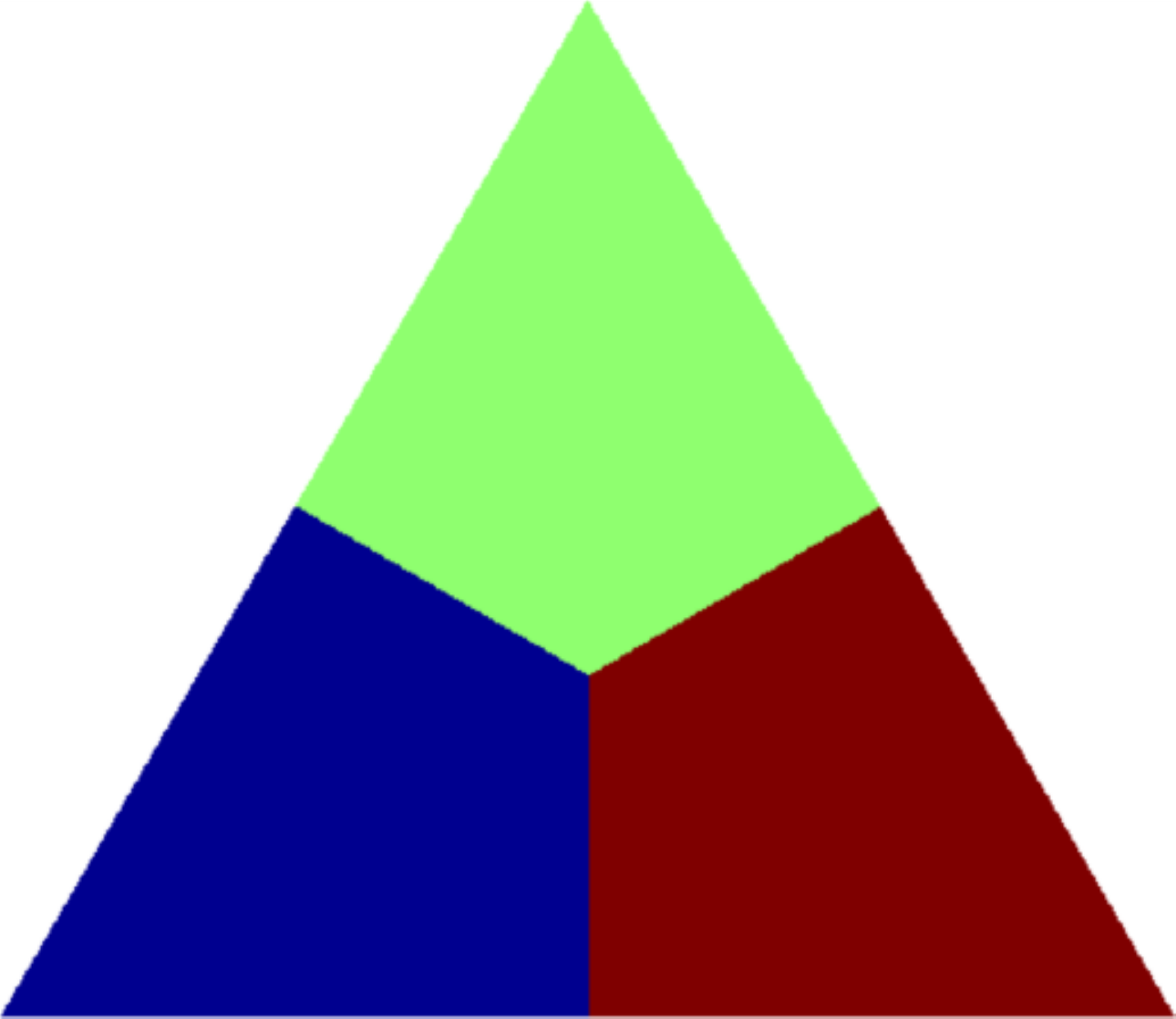} &
 \includegraphics[width = 0.19\textwidth]{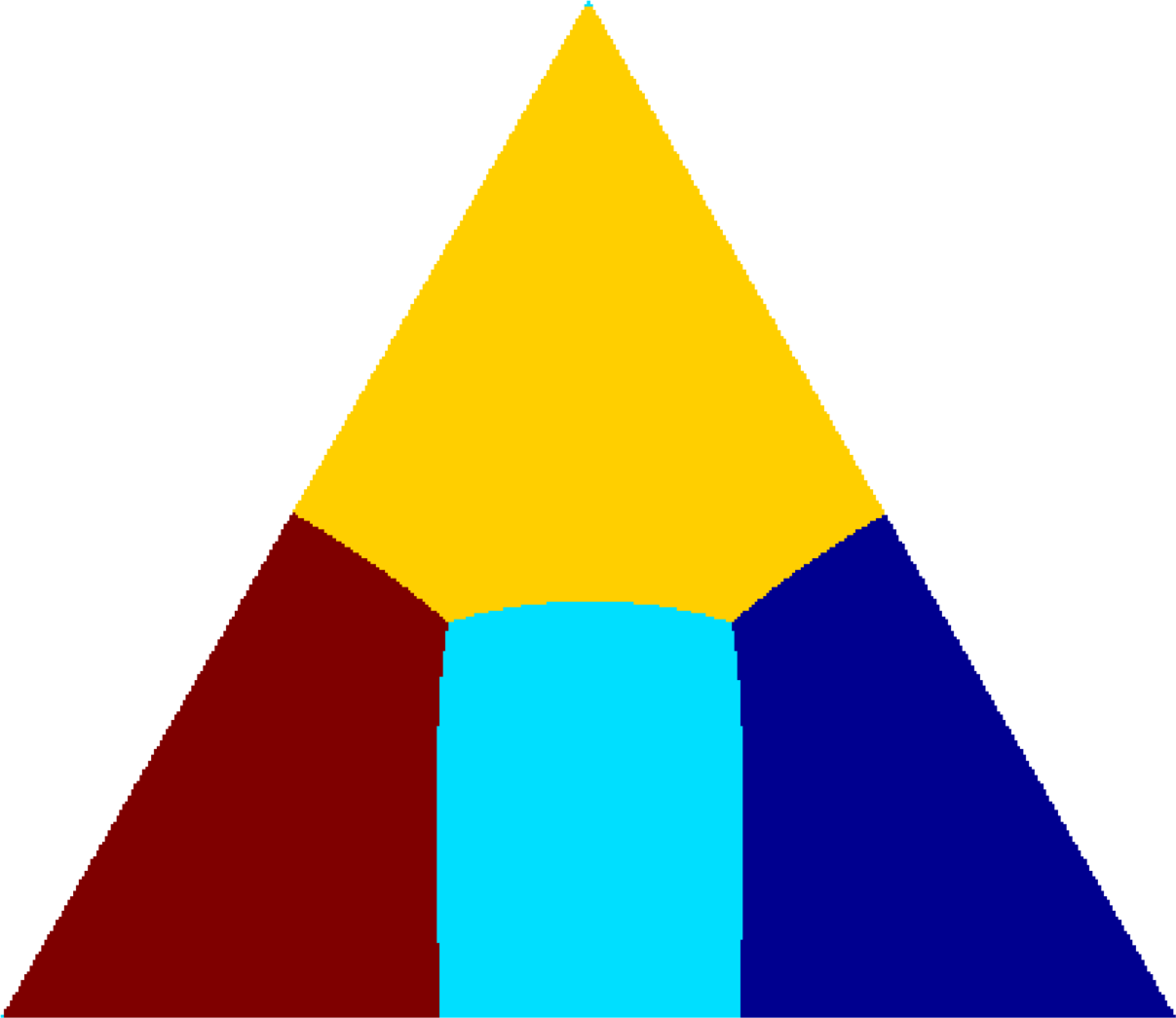} &
 \includegraphics[width = 0.19\textwidth]{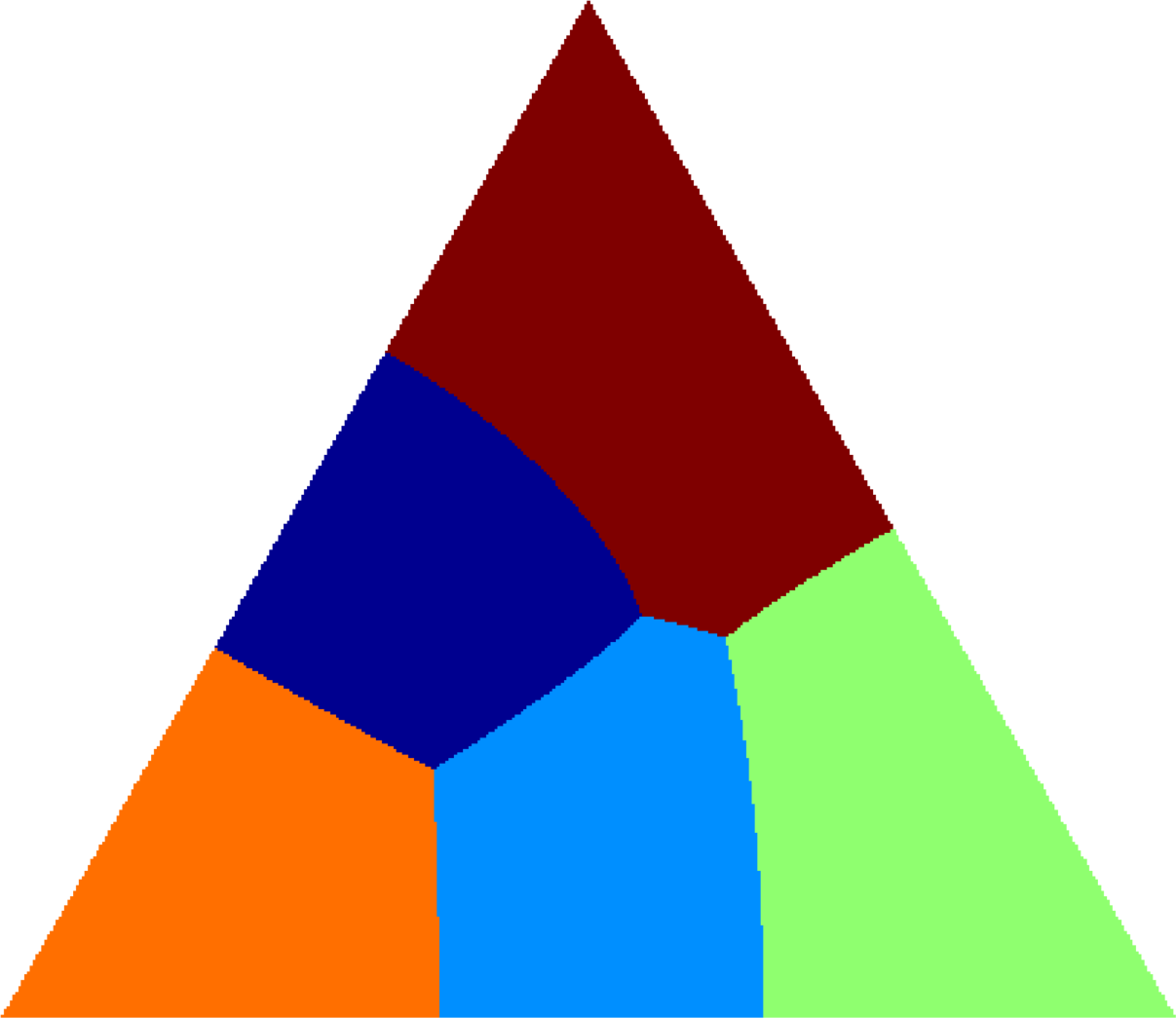}
\\
$p=50$ &  \includegraphics[width = 0.19\textwidth]{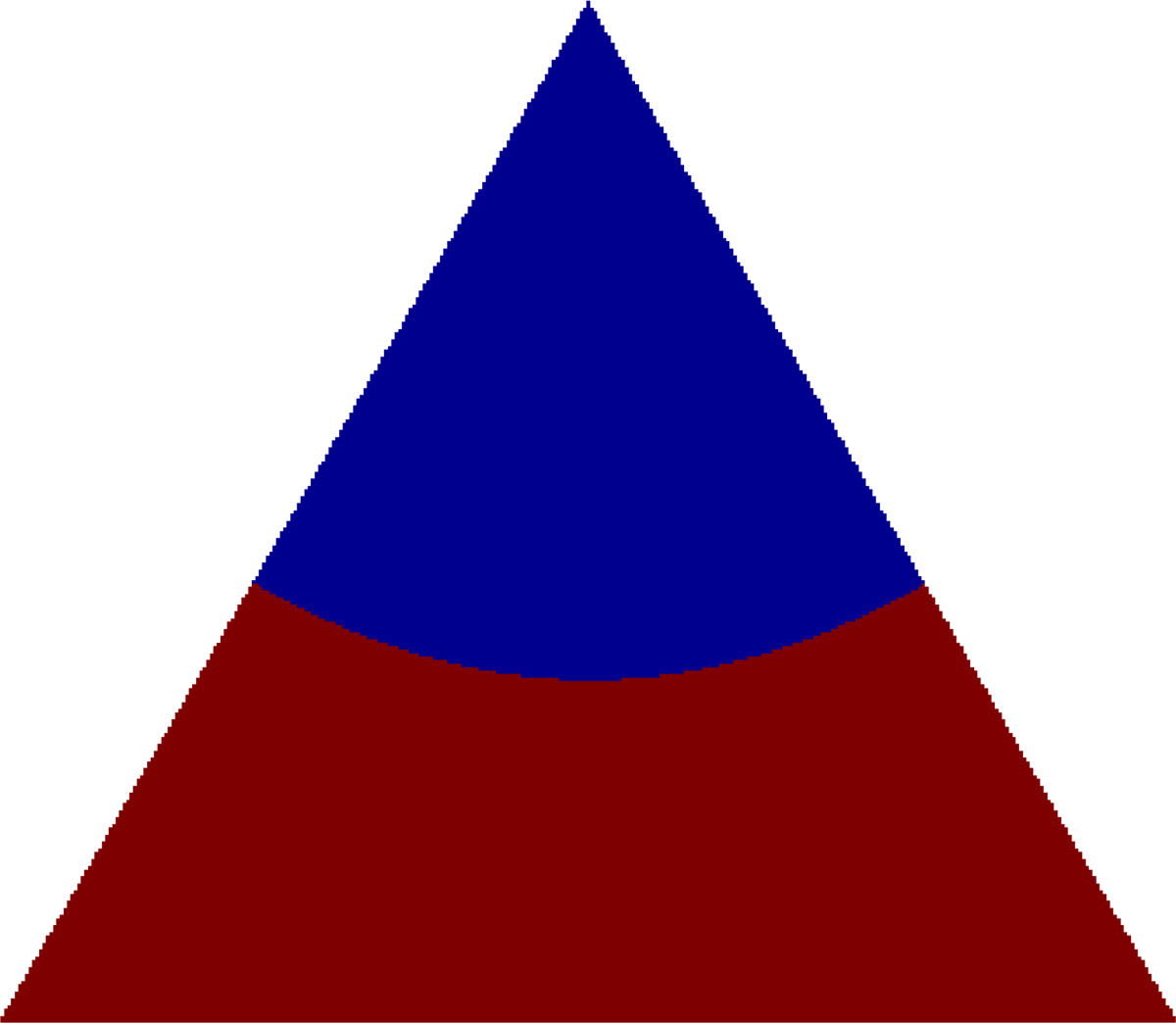} &
 \includegraphics[width = 0.19\textwidth]{sequi3} &
 \includegraphics[width = 0.19\textwidth]{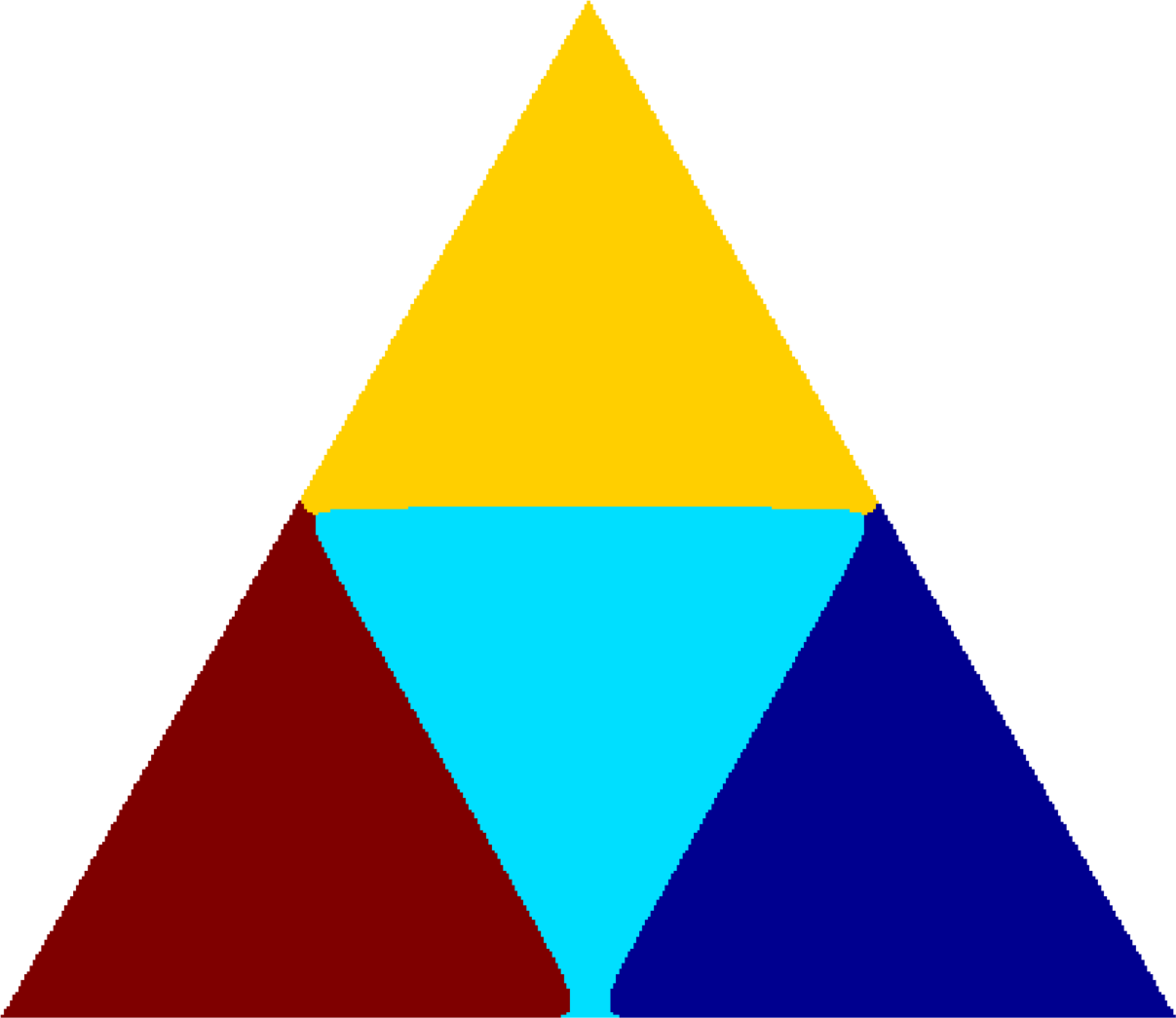} &
 \includegraphics[width = 0.19\textwidth]{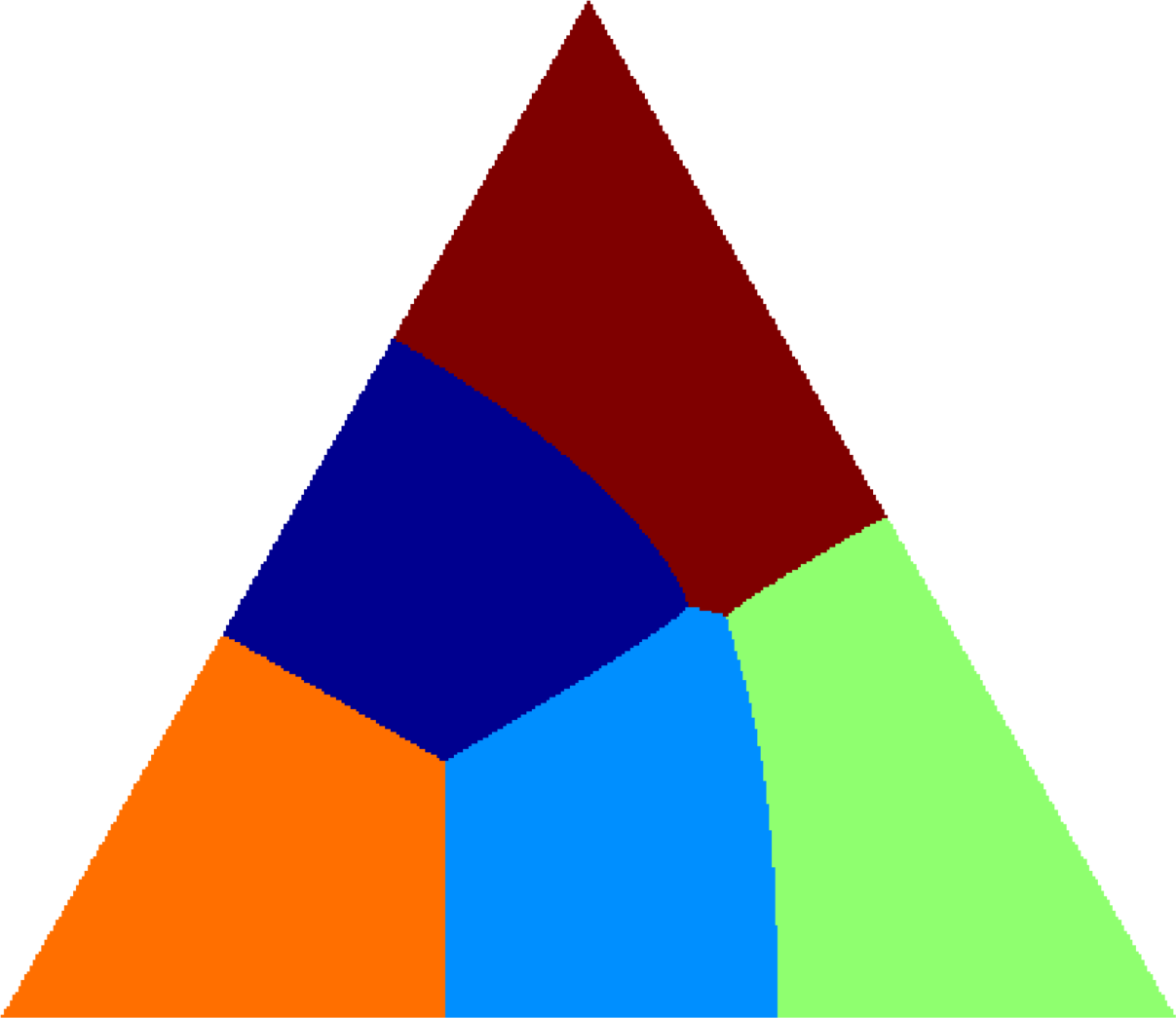}
\end{tabular}
\caption{Candidates for $p$-minimal $k$-partitions of the equilateral triangle when $p=1$ and $50$.}
\label{fig.equilateral}
\end{figure}
In Table~\ref{table-equi}, we analyze the energies of the numerical $p$-minimal $k$-partitions for $p=1,50$. For each partition, we give the energy $\Lambda_{k,p}(\cD^{k,p})$ (which corresponds to the energy for which $\cD^{k,p}$ should be optimal) and the largest first eigenvalue on the cells of $\cD^{k,p}$, that is to say $\Lambda_{k,\infty}(\cD^{k,p})$. 
We can observe that the minimizer for $p=1$ has a larger maximal eigenvalue than the one obtained for $p=50$. This indicates that partitions $\cD^{k,1}$ which minimize $\Lambda_{k,1}$ are not necessarily good candidates for minimizing $\Lambda_{k,\infty}$ and that the candidates $\cD^{k,50}$ give better upper bound for $\fL_{k,\infty}(\Omega)$ than the candidates $\cD^{k,1}$. Indeed, we observe that
 $$ \Lambda_{k,\infty}(\cD^{k,50}) \leq \Lambda_{k,\infty}(\cD^{k,1}),\qquad 2\leq k\leq 5.$$
Furthermore, by definition of $\fL_{k,\infty}(\triangle)$, we have $\fL_{k,\infty}(\triangle) \leq \Lambda_{k,\infty}(\cD^{k,50})$ for any $k$.
In the case $p=50$, the energies $\Lambda_{k,50}$ and $\Lambda_{k,\infty}$ are rather close, which leads one to believe that the numerical $p$-minimal $k$-partition with $p=50$ is a rather good candidate to minimize the maximum of the first eigenvalues $\Lambda_{k,\infty}$.

\begin{table}[h!]
\centering
\begin{tabular}{|c|c|c|c|c|c|c|}
\hline 
&\multicolumn{2}{c|}{$\cD^{k,1}$} &\multicolumn{2}{c|}{$\cD^{k,50}$}\\
\hline
$k$&$\Lambda_{k,1}$&$\Lambda_{k,\infty}$ &$\Lambda_{k,50}$&$\Lambda_{k,\infty}$ \\
\hline
$2$ & $106.62$ & $136.11$  & $123.25$ & $123.38$ \\ \hline
$3$ & $143.05$ & $143.07$  & $143.06$ & $143.07$ \\ \hline
$4$ & $206.15$ & $229.44$  & $209.86$ & $211.71$ \\ \hline
$5$ & $249.62$ & $273.69$  & $251.06$ & $252.68$ \\ \hline
\end{tabular}\\[5pt]
\caption{ Energies of $\cD^{k,p}$ for the equilateral triangle when $p=1$ and $p=50$.}
\label{table-equi}
\end{table}

The situation when $p = 1$ appears to be very different from when $p =\infty$. 
Thus we recall in the following section some theoretical results regarding properties of the partitions minimizing $\Lambda_{k,\infty}$ as well as criteria allowing to decide whether a partition optimal for the max are not optimal for the sum.

\section{Theoretical results}
\label{theory}
In this section, let us recall some theoretical results about the $p$-minimal $k$-partitions.  With these theoretical results we can comment  on the implementation done in the previous section. This is also useful to propose some new adaption of the algorithm in the next section. 
\subsection{Monotonicity}
First of all, let us recall a monotonicity result. 
\begin{thm}
Let $k\geq 1$ and $1 \leq p \leq q <\infty$. We have monotonicity
\begin{itemize}
\item with respect to the domain
$$\Omega\subset\tilde\Omega\quad \Rightarrow\quad \fL_{k,p}(\tilde\Omega)\leq\fL_{k,p}(\Omega);$$
\item with respect to the number $k$ of domains of the partition
$$\fL_{k,p}(\Omega) < \fL_{k+1,p}(\Omega);$$
\item with respect to the $p$-norm 
\begin{equation}\label{eq.monop}
\frac{1}{k^{1/p}}\mathfrak L_{k,\infty}(\Omega) \leq \mathfrak L_{k,p}(\Omega) \leq \mathfrak L_{k,q}(\Omega) \leq \mathfrak L_{k,\infty}(\Omega),\qquad \forall 1\leq p\leq q< \infty.
\end{equation}
\end{itemize}
\end{thm}
The proof of the third point is based on the monotonicity for the $p$-norm. Indeed, 
for any partition $\cD\in\fP_{k}(\Omega)$ and for any $1\leq p\leq q< \infty$, we have
\begin{equation}
\frac 1{k^{1/p}} \Lambda_{k,\infty}(\cD) \leq \Lambda_{k,p}(\cD) \leq 
\Lambda_{k,p}(\cD) \leq \Lambda_{k,\infty}(\cD).
\end{equation}
We notice that the results of Table~\ref{table-equi} are coherent with \eqref{eq.monop} since $\Lambda_{k,p}(\cD^{k,p})$ should be close to $\fL_{k,p}(\triangle)$ and we observe that $\Lambda_{k,1}(\cD^{k,1}) \leq\Lambda_{k,50}(\cD^{k,50})$.
The monotonicity $p\mapsto \fL_{k,p}(\Omega)$ does not imply the monotonicity $p\mapsto \Lambda_{k,p}(\cD^{k,p})$. Nevertheless, if $\cD^{k,p}$ is very close to the $p$-minimal $k$-partition, we can observe numerically such a monotonicity.

\subsection{Equipartition}
We say that $\cD=(D_{1},\ldots, D_{k})$ is an \emph{equipartition} if the first eigenvalue on each subdomain $\lambda_{1}(D_{j})$ are equal. The equipartitions play an important role in these optimization problems. 
Indeed, as soon as the $p$-minimal $k$-partition is an equipartition, it is minimal for any larger $q$. Furthermore any $\infty$-minimal $k$-partition is an equipartition (see \cite[Chap. 10]{BookEigen}):
\begin{prop}\label{prop.equip}\ 
\begin{itemize}
\item If ${{\mathcal D}^*=(D_{i})_{1\leq i\leq k}}$ is a {$\infty$-minimal $k$-partition},
then ${\mathcal D}^*$ is an equipartition:
$${\lambda_{1}(D_i)=\fL_{k,\infty}(\Omega)\,,\qquad \mbox{ for any }\quad1\leq i\leq k}.$$
\item Let $p\geq 1$ and ${\mathcal D}^*$ a $p$-minimal $k$-partition.
If ${\mathcal D}^*$ is an equipartition, then 
$$\mathfrak{L}_{k,q}(\Omega)=\mathfrak{L}_{k,p}(\Omega),\qquad\mbox{ for any }\quad q\geq p.$$
\end{itemize}
\label{properties-min}
\end{prop}
Consequently, it is natural to set
\begin{equation}\label{eq.p}
p_{\infty}(\Omega,k) = \inf\{p\geq1, \mathfrak{L}_{k,p}(\Omega)=\mathfrak{L}_{k,\infty}(\Omega)\}.
\end{equation}
Let us apply this result in the case of the disk, see Figure~\ref{fig.disk}. If we can prove that the $p$-minimal $k$-partition for the norm $p=1$ and $2\leq k\leq 5$, is the equipartition with $k$ angular sectors, then according to Proposition~\ref{prop.equip}, this partition is minimal for any {$p\geq1$ and $p_{\infty}(\Circle,k)=1$}. 
In the case of the equilateral triangle, Table~\ref{table-equi} makes us think that the $p$-minimal $k$-partition is not an equipartition when $k=2,4,5$ and thus  $p_{\infty}(\triangle,k)\geq 50$ in that case.

\subsection{Nodal partition}
When dealing with optimal partitioning problems for functionals depending on spectral quantities it is quite natural to consider nodal partitions. These partitions give, at least, some upper bounds of the optimal energies. 
Let us recall the definition of a nodal partition.
\begin{defin}
Let  $u$ be an eigenfunction of the Dirichlet-Laplacian on $\Omega$. 
The nodal sets of $u$ are the components of 
$$ \Omega\setminus N(u)\qquad\mbox{ with }\qquad N(u) = \overline{\{ x\in \Omega|\, u(x)=0\}}.$$
The partition composed by the nodal sets is called nodal partition.
\end{defin}
Nevertheless, to be useful, it is important to have {\Gn information} about the number of components of the nodal partitions. 
According Courant's theorem, any eigenfunction $u$ associated with $\lambda_{k}(\Omega)$ has at most $k$ nodal domains. 
An eigenfunction is said {\it Courant sharp} if it has exactly $k$ nodal domains. 
The following result, proved by Helffer-Hoffmann--Ostenhof-Terracini \cite{HelHofTer09} gives some bounds using the eigenvalues of the Dirichlet-Laplacian on the whole domain $\Omega$ and {\Gn gives explicitly} the cases when we can determine a $\infty$-minimal $k$-partition. 
\begin{thm}\label{thm.HHOT}
For $k\geq1$, $L_k(\Omega)$ denotes the smallest eigenvalue (if any) for which there exists an eigenfunction with $k$ nodal domains. We set $L_k(\Omega) =+\infty$ if there is no eigenfunction with $k$ nodal domains. Then we have
\begin{equation}\label{ineq.HHOT}
\lambda_{k}(\Omega)\leq\fL_{k,\infty}(\Omega) \leq L_{k}(\Omega).
\end{equation}
If $\mathfrak L_{k,\infty}(\Omega)=L_k(\Omega)$ or $\mathfrak L_{k,\infty}(\Omega)=\lambda_k(\Omega)$, then 
${\lambda_k(\Omega)=\mathfrak L_{k,\infty}(\Omega)=L_k(\Omega)}$ and then any Courant sharp eigenfunction associated with $\lambda_{k}(\Omega)$ produces a $\infty$-minimal $k$-partition. 
\end{thm}
Consequently, if there exists a Courant sharp eigenfunction associated with the $k$-th eigenvalue, then the $\infty$-minimal $k$-partition is nodal. Otherwise the $\infty$-minimal $k$-partition is not nodal. 
Note that we always have $\lambda_{2}(\Omega) = L_{2}(\Omega)$ (since the second eigenfunctions has exactly two nodal domains), then any $\infty$-minimal $2$-partition is nodal and 
\begin{equation}\label{eq.k2}
\fL_{2,\infty}(\Omega) = \lambda_{2}(\Omega).
\end{equation}
As soon as $k\geq3$, it is not so easy and it is then important to determine for which $k$ we have equality $\lambda_{k}(\Omega)=L_{k}(\Omega)$. Pleijel \cite{Plei56} established that it is impossible for $k$ large: 
\begin{thm}
There exists $k_{0}$ such that $\lambda_{k}(\Omega)<L_{k}(\Omega)$ for $k\geq k_{0}$.
\end{thm}
Therefore, a $\infty$-minimal $k$-partition is never nodal when $k>k_{0}$. 
This result proves the existence of such $k_{0}$ but is not quantitative. Recently, B\'erard-Helffer \cite{BH16} and van den Berg-Gittins \cite{vdBG16} exhibit an explicit bound for $k_{0}$. 

In some specific geometries, we can determine exactly for which eigenvalue $\lambda_{k}(\Omega)$, there exists an associated Courant sharp eigenfunction. For such $k$, we thus exhibit a $\infty$-minimal $k$-partition whose energy is $\lambda_{k}(\Omega)$.
The following property gives such result for the disk \cite[Proposition 9.2]{HelHofTer09}, the square \cite{MR3445517}, and the equilateral triangle \cite{BH2016LMP} (see also references therein).
\begin{prop}\label{prop.k24}
If $\Omega$ is a square $\square$, a disk $\Circle$ or an equilateral triangle $\triangle$, then 
$${\lambda_{k}(\Omega)=\mathfrak L_{k,\infty}(\Omega)= L_k(\Omega) \qquad\mbox{if and only if}\qquad k=1,2,4}.$$
Thus the $\infty$-minimal $k$-partition is nodal if and only if $k=1,2,4$. 
\label{courant_sharp}
\end{prop}
Figure~\ref{fig.PartNod} gives examples of $\infty$-minimal $k$-partitions. {\Mg Note that when $\Omega=\square,\Circle, \triangle$, since $\lambda_{2}(\Omega)$ is double, the $\infty$-minimal $2$-partition is not unique whereas for $k=4$ we do have uniqueness (modulo rotation for the disk). The eigenspace associated with $\lambda_{2}(\Omega)$ produces a family of $\infty$-minimal $2$-partitions which is invariant by rotation in the case of the disk.} We note that for $\Omega=\Circle, \triangle$ and $k=2,4$, we recover the $k$-partitions obtained numerically in Figures~\ref{fig.disk} and \ref{fig.equilateral}.

\begin{figure}[h!]
\begin{center}
\subfigure[$k=2$\label{fig4.tri2}]{
\includegraphics[height=2cm]{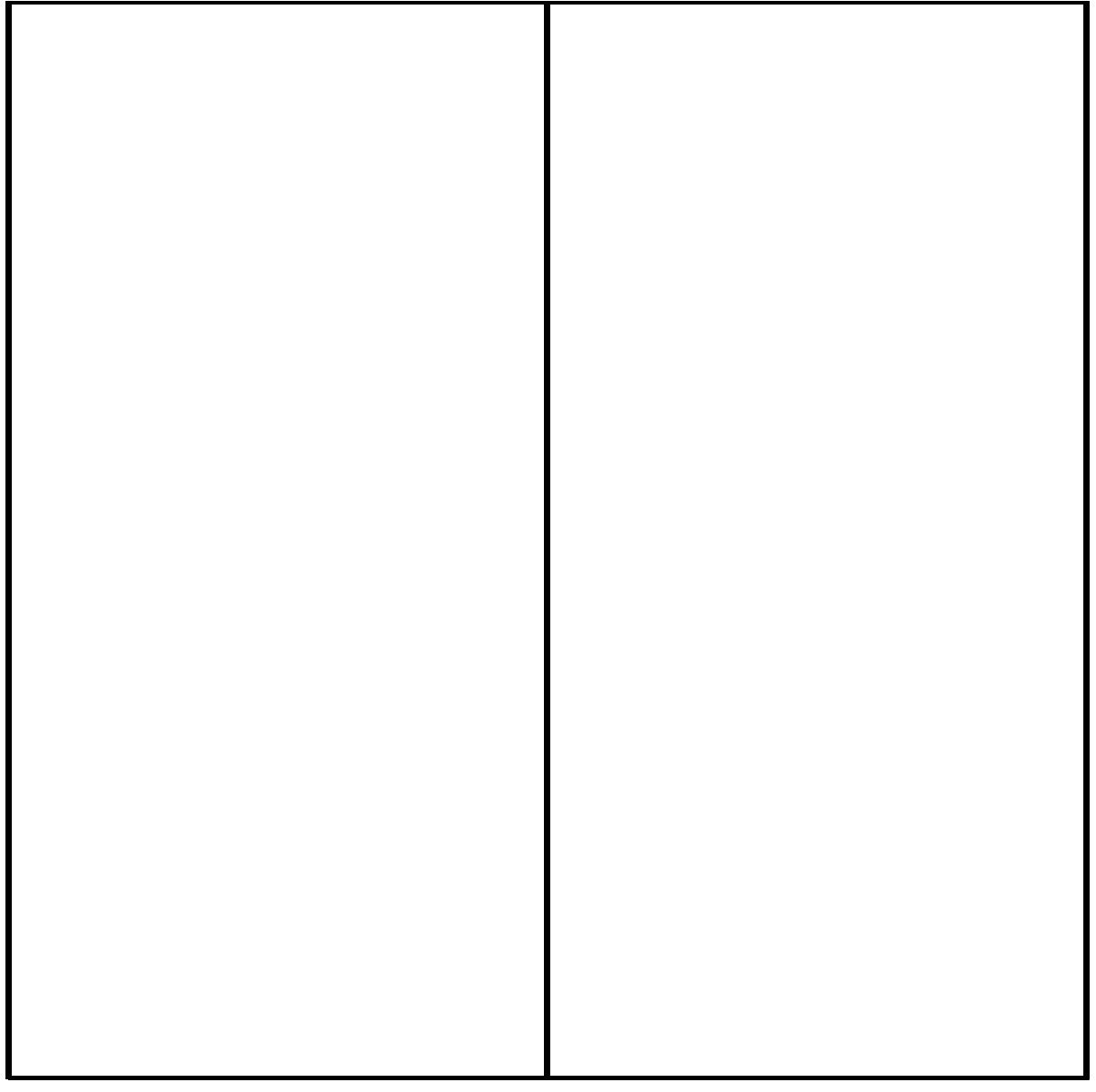}\
\includegraphics[height=2cm]{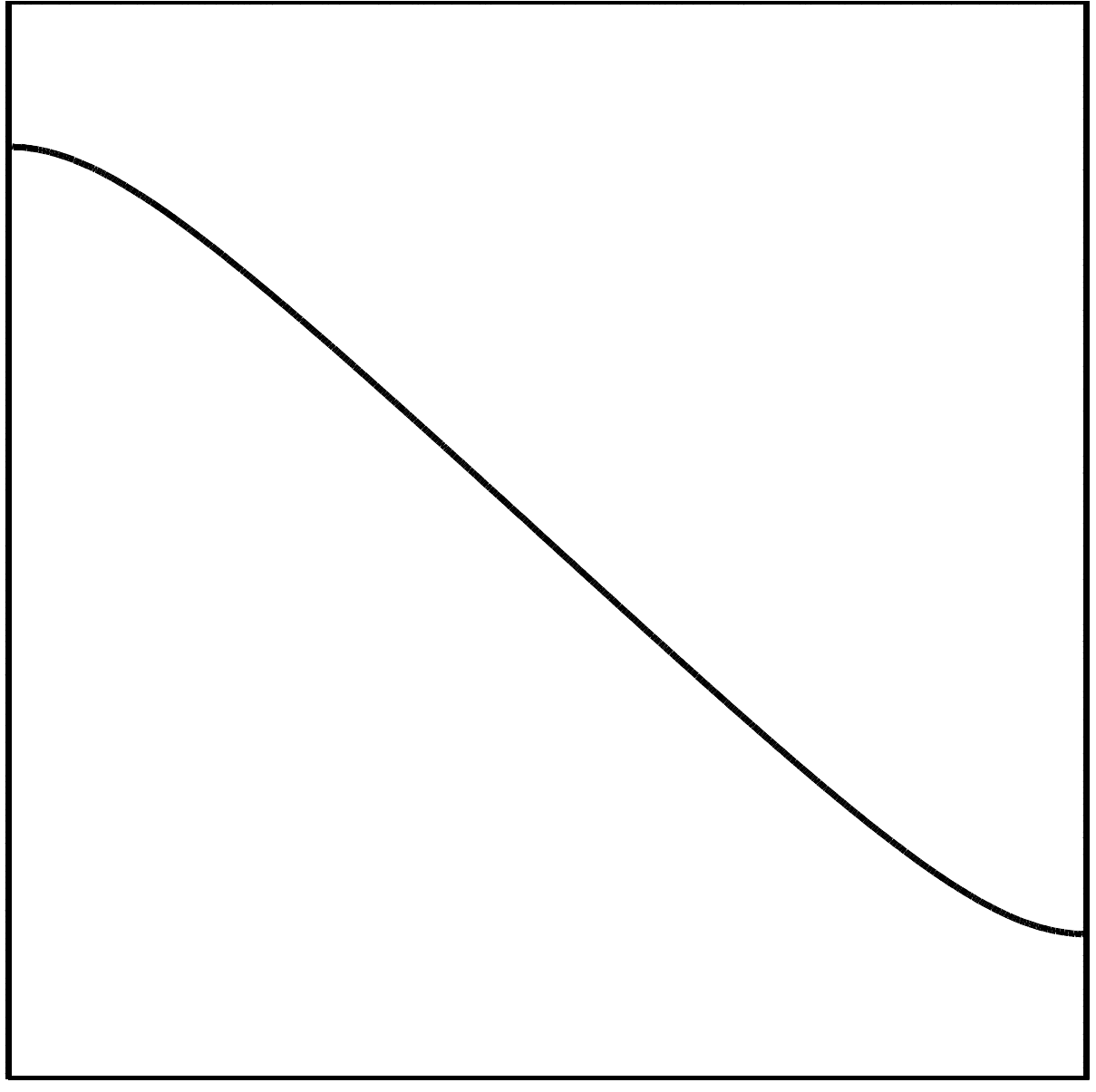}\quad
\includegraphics[height=2cm]{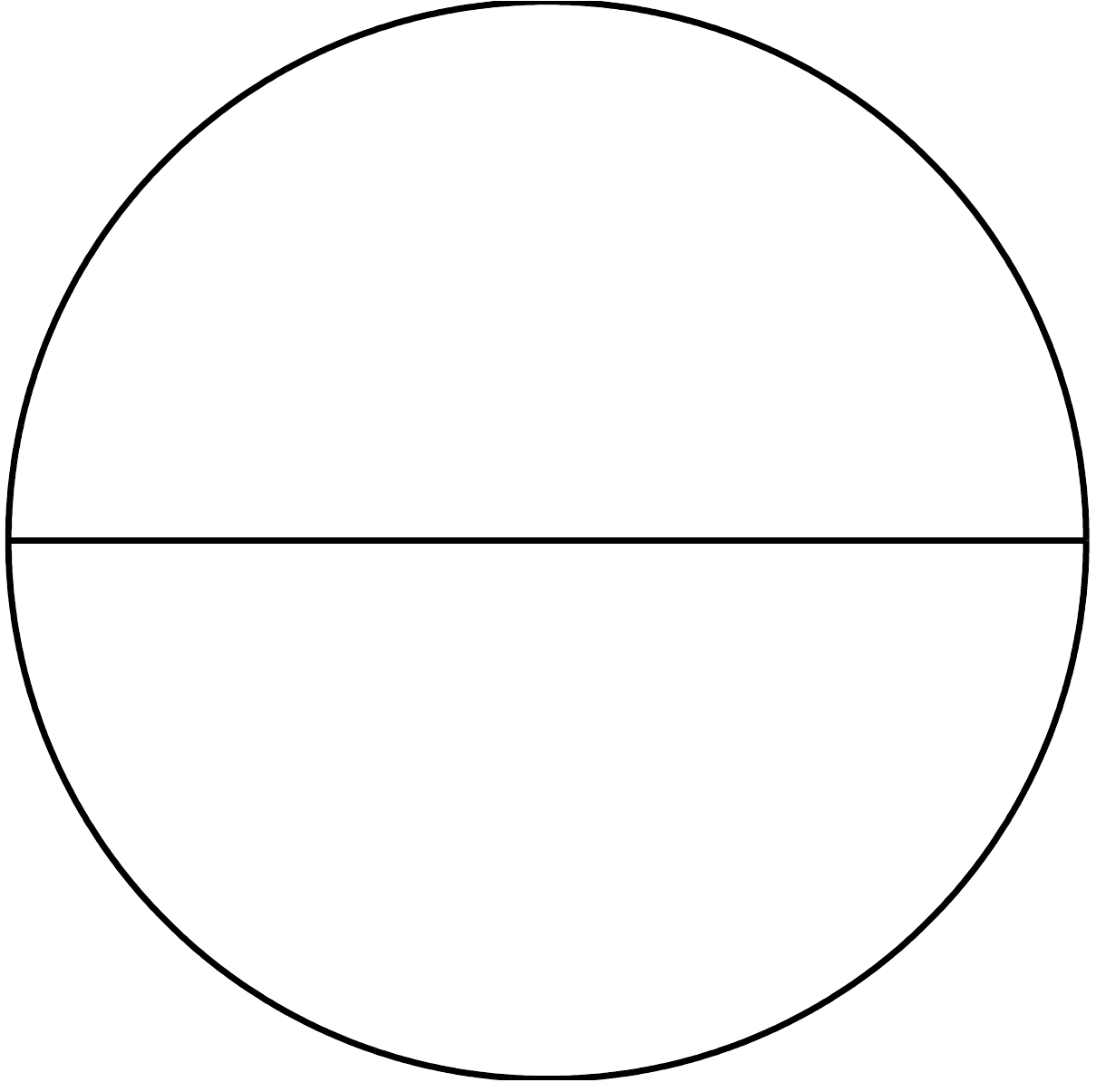}\
\includegraphics[height=2cm]{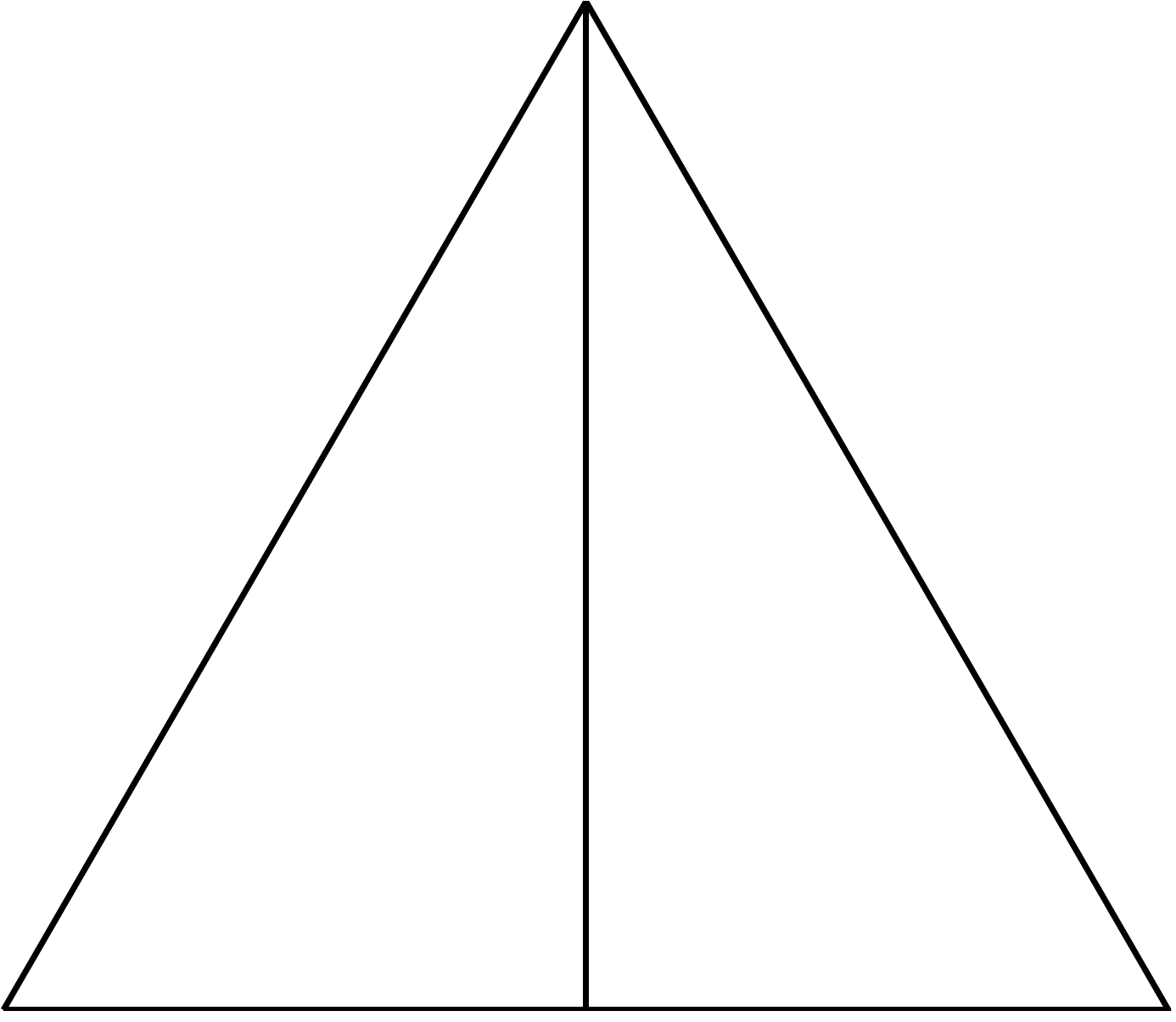}\
\includegraphics[height=2cm]{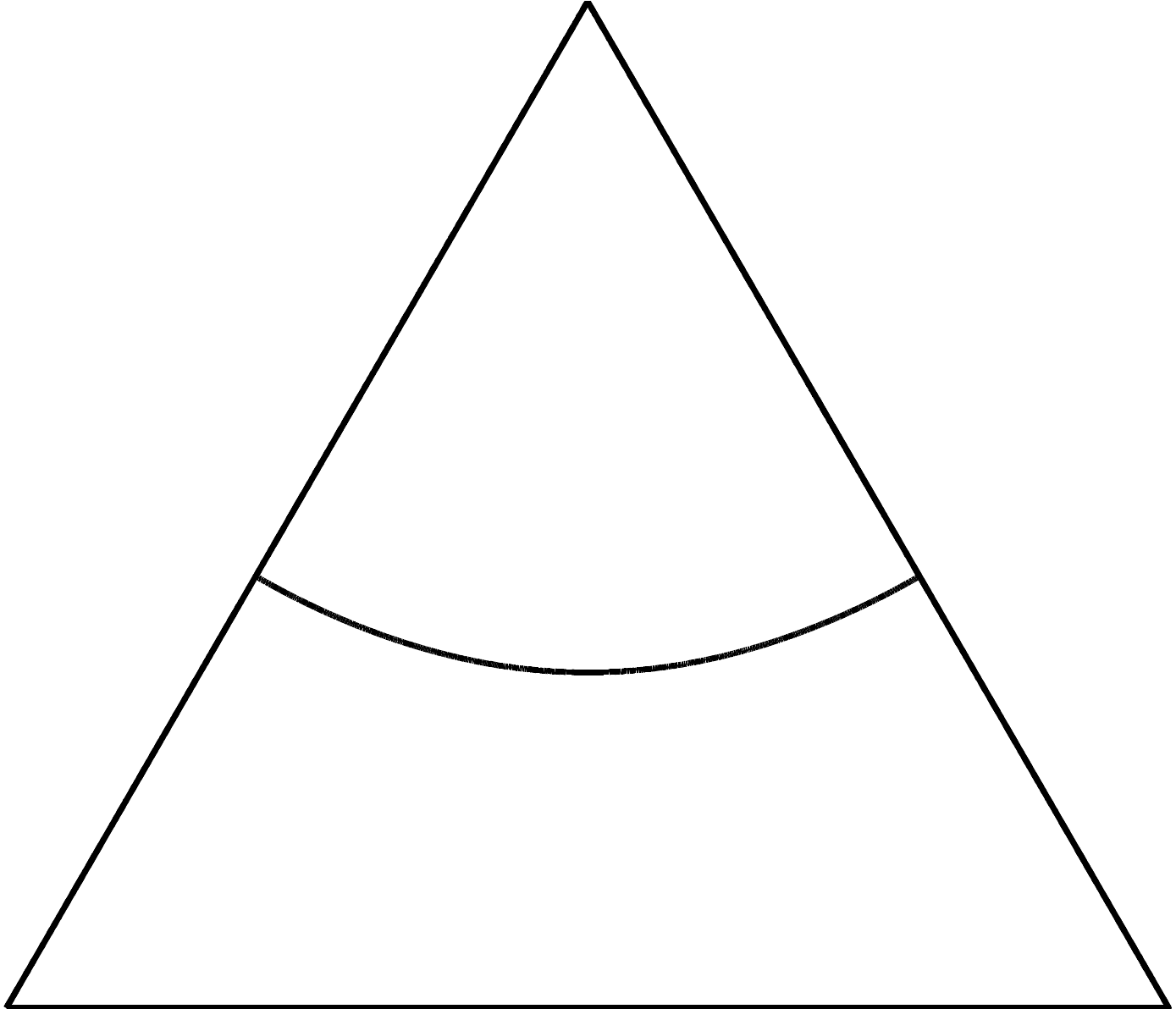}
}
\subfigure[$k=4$\label{fig4.tri4}]{
\includegraphics[height=2cm]{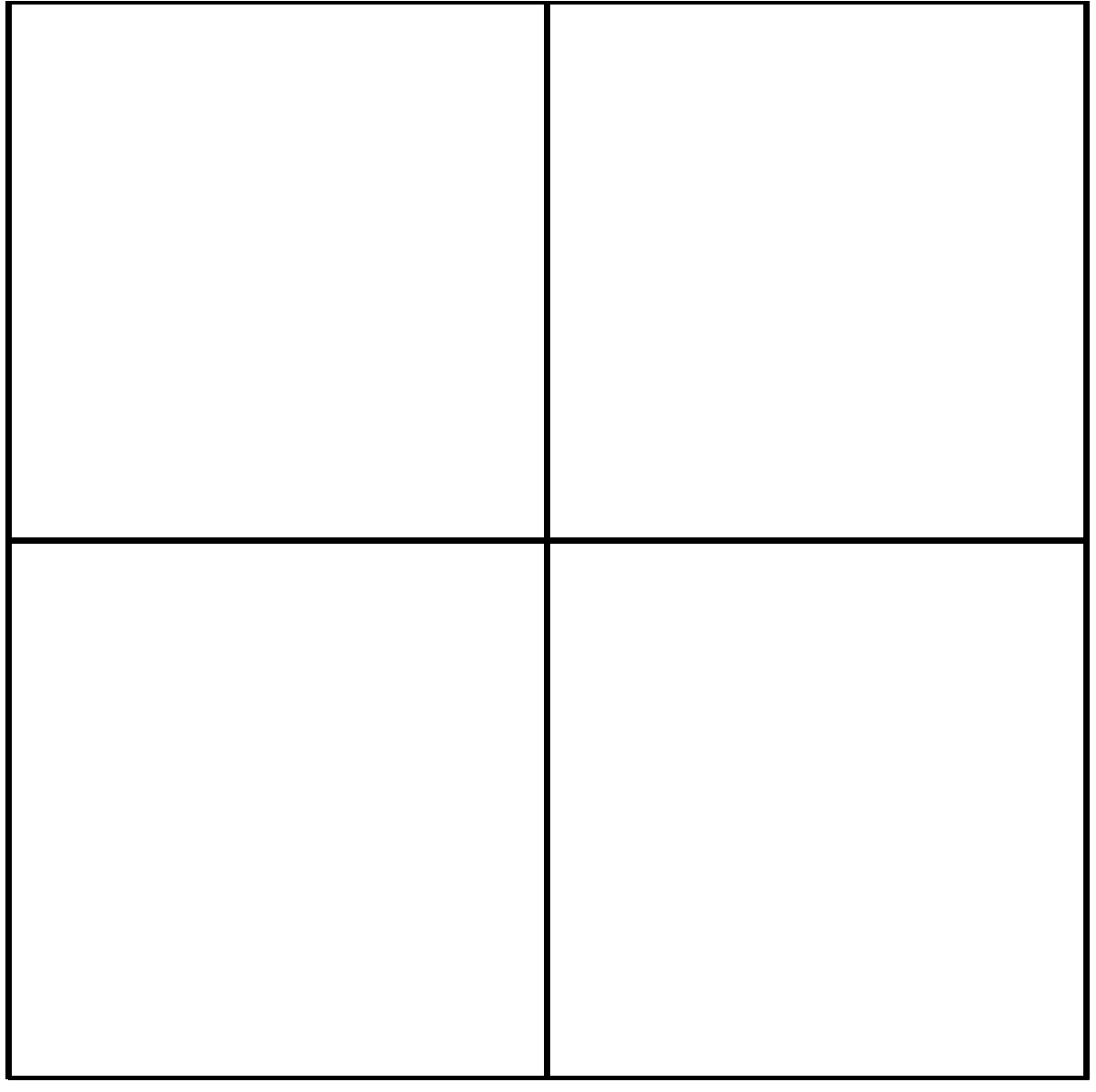}\quad
\includegraphics[height=2cm]{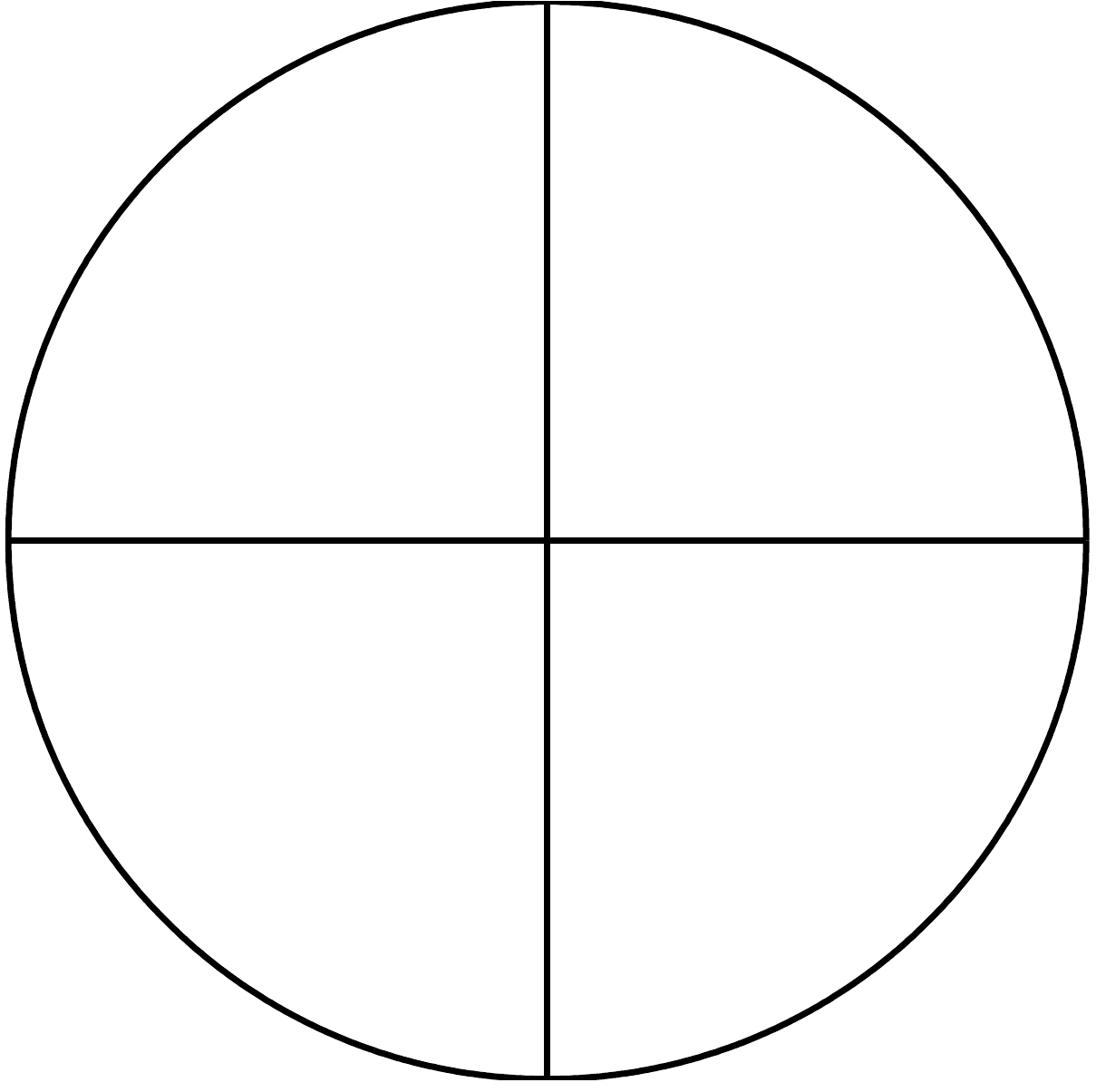}\quad
\includegraphics[height=2cm]{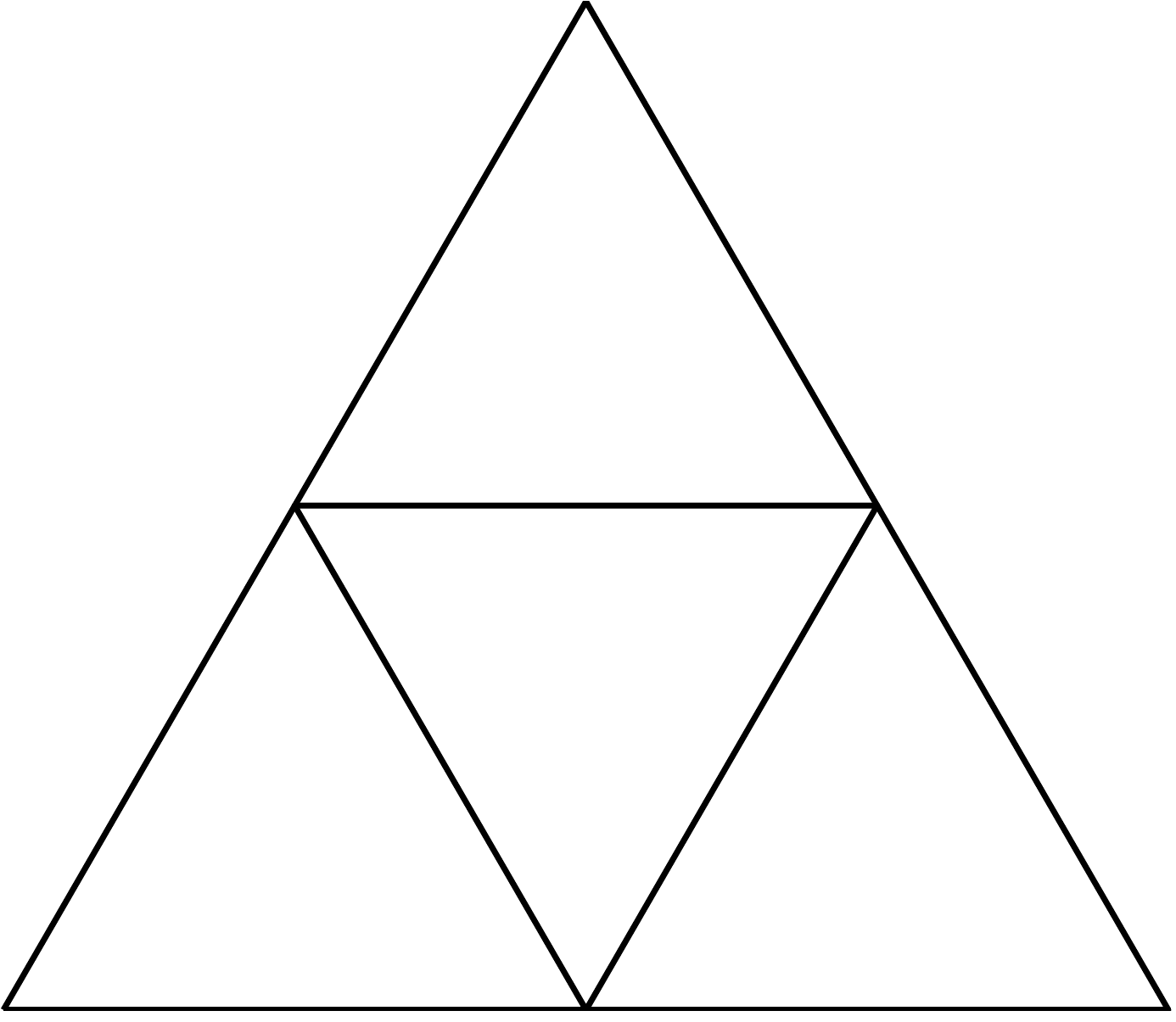}
}
\end{center}
\caption{Nodal $\infty$-minimal $k$-partitions, $k=2,4$.\label{fig.PartNod}}
\end{figure}

\subsection{Bounds with spectral quantities }
\subsubsection{Lower bounds}
The lower bounds \eqref{ineq.HHOT} can be generalized when considering the $p$-norm instead of the $\infty$-norm and we have (see \cite{HelHofTer09} for $p=\infty$ and \cite{HelHof10} for the general case)
\begin{equation}\label{ineq.HHOTnormp}
\left(\frac{1}{k}\sum_{i=1}^{k}\lambda_{i}(\Omega)^p\right)^{1/p}
\leq \mathfrak L_{k,p}(\Omega) \leq L_{k}(\Omega).
\end{equation}
When $\Omega$ is a square, a disk or an equilateral triangle, the eigenvalues are explicit and thus they produce explicit lower and upper bounds. Computing the number of nodal domains of some eigenfunctions give us an upper bound for $L_{k}(\Omega)$ (see Table~\ref{table.mu(vp)}). Note that when eigenvalues are double, we may have eigenfunctions with different numbers of nodal domains. We mention all possible values, since the goal is to find upper bounds $L_k$ for each $k$.
\begin{table}[h!]
\begin{center}\begin{tabular}{c | cl}
$\Omega$ & $\lambda_{m,n}(\Omega)$ & $m,n$\\
\hline
& \\
$\square$ & $\pi^2(m^2+n^2)$& $m,n\geq 1$\\[5pt]
$\triangle$ & $\frac{16}9 \pi^2(m^2+mn+n^2)$& $m,n\geq 1$\\[5pt]
$\Circle$ & $j_{m,n}^2$ & $m\geq0$, $n\geq 1$ {(multiplicity 2 for $m\geq1$)}\\[5pt]
& \multicolumn{2}{c}{{\small\Bl where $j_{m,n}$ is the $n$-th positive zero of the Bessel function of the first kind $J_{m}$.}}
\end{tabular}
\end{center}
\caption{Eigenvalues for $\Omega=\square,\ \triangle,\ \Circle$. \label{tab.vp}}
\end{table}
 
\begin{table}[h!]
\begin{center}\begin{tabular}{|c||c|c||c|c||c|c|}
\hline
& \multicolumn{2}{c||}{Square} & \multicolumn{2}{c||}{Disk}& \multicolumn{2}{c|}{Equilateral triangle}\\
\hline
 $k$ & $\lambda_{k}(\square)$ & $\mu(u_{k})$ & $ \lambda_{k}(\Circle)$ & $\mu(u_{k})$ & $\lambda_{k}(\triangle) $ & $\mu(u_{k})$ \\
 \hline
1   &   19.739 & 1 &  5.7831 & 1 &  52.638 & 1 \\
2   &   49.348 & 2 & 14.6819 & 2 & 122.822 & 2 \\
3   &   49.348 & 2 & 14.6819 & 2 & 122.822 & 2 \\
4   &   78.957 & 4 & 26.3746 & 4 & 210.552 & 4 \\
5   &   98.696 & 3 & 26.3746 & 4 & 228.098 & 4 \\
6   &   98.696 & 3 & 30.4713 & 2 & 228.098 & 3 \\
7   &  128.305 & 4 & 40.7065 & 6 & 333.373 & 4 \\
8   &  128.305 & 4 & 40.7065 & 6 & 333.373 & 4 \\
9   &  167.783 & 4 & 49.2184 & 4 & 368.465 & 4 \\
10  &  167.783 & 4 & 49.2184 & 4 & 368.465 & 4 \\
\hline
\end{tabular}\end{center}
\caption{Lowest eigenvalues $\lambda_{k}(\Omega)$ and number of nodal sets for associated eigenfunctions $u_{j}$ of the Dirichlet-Laplacian on $\Omega=\square$, $\Circle$ and $\triangle$.\label{table.mu(vp)}}
\end{table}

\subsubsection{Upper bounds}
Let us mention that in the case of the disk, we can easily construct a $k$-partition of $\Circle$ by considering the partition with $k$ angular sectors of opening $2\pi/k$. If we {denote} by $\Sigma_{{2\pi}/k}$ an angular sector of opening $2\pi/k$, then we have the upper bound
\begin{equation} \label{eq.sect}
\mathfrak L_{k,p}(\Circle) \leq \lambda_{1}(\Sigma_{{2\pi}/k}).
\end{equation}
Recall that the {eigenvalues} of a sector $\Sigma_{\alpha}$ of opening $\alpha$ are given by (see \cite{BonLen14}) : 
$$ \lambda_{m,n}(\alpha)=j_{m\frac\pi\alpha,n}^2,$$
where $j_{m\frac\pi\alpha,n}$ is the $n$-th positive zero of the Bessel function of the first kind $J_{m\frac\pi\alpha}$. In particular, we have
$$\lambda_{1}(\Sigma_{{2\pi}/k}) = j^2_{\frac{k}2,1}.$$\
Let us remark that if $k$ is odd, the $k$-partition with $k$ angular sectors is not nodal and \eqref{eq.sect} gives a new upper bound which can be better than \eqref{ineq.HHOT} or \eqref{ineq.HHOTnormp}. If $k$ is even, we have $L_{k}(\Circle)\leq \lambda_{1}(\Sigma_{{2\pi}/k})$.

In the case of the square, we will use the following upper bound which is weaker but more explicit than \eqref{ineq.HHOTnormp}:
$$ \mathfrak L_{k,p}(\square) \leq \inf_{m,n\geq1}\{\lambda_{m,n}(\square)| mn=k\}
\leq \lambda_{k,1}(\square),$$
with $\lambda_{m,n}(\square)$ defined in Table~\ref{tab.vp}.

\subsection{Candidates for the sum and the max}
We have seen in Section~\ref{ssec.num} that the candidates to be minimal for the sum and the max seem to be the same in the case of the disk when $k=2,3,4,5$ (see Figure~\ref{fig.disk}) whereas they are different for the equilateral triangle  when $k=2, 4,5$ (see Figure~\ref{fig.equilateral}). 
Then it could be interesting to have some criteria to discriminate if a $\infty$-minimal $k$-partition can be minimal for the sum ($p=1$). 
A necessary condition is given in \cite{HHO10}:
\begin{prop}
Let ${\mathcal D}=(D_1,D_{2})$ be a $\infty$-minimal $2$-partition and $\varphi_{2}$ be a second eigenfunction  of the Dirichlet-Laplacian on $\Omega$ having $ D_1$ and $ D_2$ as nodal domains. 

Suppose that 
$\displaystyle\int_{D_1} |\varphi_{2}|^2 \neq \int_{D_2}|\varphi_{2}|^2$, then 
$\mathfrak L_{2,1}(\Omega) < \mathfrak L_{2,\infty} (\Omega).$
\end{prop}
Since any $\infty$-minimal $2$-partition is nodal, we can use the previous criterion by considering neighbors in a $k$-partition. 
We say that two sets $ D_i,D_j$ of the partition $\mathcal D$ are neighbors and write $ D_i\sim D_j$, if
$D_{ij}={\rm{ Int\,}}(\overline {D_i\cup D_j})\setminus \partial \Omega $
is connected.
\begin{prop}\label{l2norm}
Let ${\mathcal D}=(D_i)_{1\leq i\leq k}$ be a  {\Mg $\infty$-minimal} $k$-partition and $D_i\sim D_j$ be a pair of neighbors. We denote  
$D_{ij}={\rm Int}\overline {D_{i}\cup D_{i}}.$
There exists a second eigenfunction $\varphi_{ij}$ of the Dirichlet-Laplacian on $D_{ij}$ having $ D_i$ and $ D_j$ as nodal domains.

If $\displaystyle\int_{D_i} |\varphi_{ij}|^2 \neq \int_{D_j}|\varphi_{ij}|^2$, then $\mathfrak L_{k,1}(\Omega) < {\Mg \mathfrak L_{k,\infty} (\Omega)=}\lambda_{2} (D_{ij})$.
\end{prop}

\section{Candidates for the infinity norm}
\label{section.max}
\subsection{Penalization method}
We note that the results obtained in Section \ref{ssec.num} using the $p$-norm approach do not consist of exact equipartitions. We recall that this is a necessary condition for a partition to be a solution of the min-max problem \eqref{eq.Lkp} with $p=\infty$ (see Proposition~\ref{prop.equip}). We use the following idea in order to force the eigenvalues to be closer.
If we are able to minimize the sum of eigenvalues 
\[ \lambda_1(D_1)+\ldots+\lambda_1(D_k),\]
under the constraint $\lambda_1(D_1) = \ldots = \lambda_1(D_k)$, we are in fact minimising the maximal eigenvalue. We can, thus, for every parameter $\varepsilon>0$ consider the smooth functionals
\[ F_\varepsilon((D_i)) = \frac{1}{k}\sum_{i=1}^k \lambda_1(D_i)+\frac{1}{\varepsilon} \sum_{1\leq i< j \leq k} (\lambda_1(D_i)-\lambda_1(D_j))^2,\]
{\it i.e.} the average of the eigenvalues plus a term penalizing pairs of non-equal eigenvalues.
We define the functional
\[ F((D_i)) = \begin{cases} \max\{ \lambda_1(D_i), 1\leq i\leq k\}  & \text{ if } (D_i) \text{ is an equipartition}, \\
+\infty & \text{ otherwise}.\end{cases} \]
 We note that functional $F_\varepsilon$ may not have minimizers in the class of domains, since it is not decreasing with respect to inclusions of sets. However, the functional $F$ admits a minimizer consisting of open, connected sets, and therefore each of these sets has at most $k$ holes. It is, therefore, not restrictive, in our case to consider the functionals $F_\varepsilon$ only for families of domains with at most $k$ holes. We denote by $\mathcal S_k$ the family of partitions of $\Omega$ consisting of domains with at most $k$ holes. In view of Sverak's theorem \cite{Sverak} the eigenvalues of the Dirichlet-Laplace operator are stable under Hausdorff convergence in the class $\mathcal S_k$.
Then we have the following result.

\begin{prop}
The functionals $F_\varepsilon$ $\Gamma$-converge to $F$ {\Bl for the topology induced by the Hausdorff distance on $\mathcal S_k$. More precisely, for $(D_i^\varepsilon),(D_i) \in \mathcal{S}_k$ we have:}
\begin{itemize}
\item for every $(D_i^\varepsilon) \to (D_i)$ as $\varepsilon\to 0$, $\liminf_{\varepsilon \to 0} F_\varepsilon((D_i^\varepsilon)) \geq F((D_i)) ,$
\item for every  $(D_i)$, we can find $(D_i^\varepsilon) \to (D_i)$ such that $\limsup_{\varepsilon \to 0} F_\varepsilon((D_i^\varepsilon)) \leq F((D_i)).$
\end{itemize}
Consequently any limit point of a sequence of minimizers of $F_\varepsilon$ is a minimizer for $F$.
\label{pen-method}
\end{prop}

\begin{proof} Let $(D_i^\varepsilon) \in \mathcal{S}_k$ 
be a sequence of partitions of $\Omega$ with at most $k$ holes, which converges to $(D_i) \in \mathcal{S}_k$ in the Hausdorff metric. Since the Dirichlet-Laplace eigenvalues are stable under the Hausdorff convergence we directly obtain
\[ \liminf_{\varepsilon \to 0}F_\varepsilon((D_i^\varepsilon)) \geq  F((D_i)).\]
The above inequality is obvious if $(D_i)$ is not an equipartition, since then we have \[\liminf_{\varepsilon \to 0}F_\varepsilon((D_i^\varepsilon))=+\infty.\] On the other hand, if $\Mg(D_i)$ is an equipartition we clearly see that the inequality is true since
\[ \liminf_{\varepsilon \to 0} F_\varepsilon((D_i^\varepsilon)) \geq \frac 1k\sum_{i=1}^k \lambda_1(D_i)) = \max_{1\leq i\leq k} \lambda_1(D_i). \]
The $\Gamma-\limsup$ part is straightforward by choosing a constant sequence. 
\end{proof}

We use the result of Proposition \ref{pen-method} to construct a numerical algorithm which approaches the min-max problem \eqref{eq.Lkp} with $p=\infty$. We minimize the functional $F_\varepsilon$ for $\varepsilon \in \{10,1,0.1,0.01\}$ and each time we start from the result of the previous optimization. {\Gn We justify the choice of the parameter $\varepsilon$ as follows. We do not start directly with a small value of $\varepsilon$, since the penalization part of the functional would dominate and we would reach a local minimum where the eigenvalues are almost equal. Therefore, we start with a reasonably high value of $\varepsilon$ which makes the two parts of the functional $F_\varepsilon$ similar in magnitude. Then we progressively decrease $\varepsilon$ in order to diminish the differences between the eigenvalues. We stop at $0.01$, since the difference between the optimal value at $\varepsilon=0.01$ and $\varepsilon=0.001$ is less than $0.001$. Moreover, for smaller $\varepsilon$ the penalization part would dominate and the value of the maximal eigenvalue no longer decreases.}  In the minimization of $F_\varepsilon$ we use the same discrete framework presented in Section \ref{ssec.num} as well as the penalized eigenvalue problem \eqref{discrete}.  
In Table \ref{pnorm_pen} we present the minimal and maximal eigenvalues obtained when minimizing $\Lambda_{k,50}$ and when using the penalization method described in this section, that is to say, {\Gn we compare}
\[\min\{\lambda_{1}(D_{j}),1\leq j\leq k\}\quad \mbox{ and }\max\{\lambda_{1}(D_{j}),1\leq j\leq k\},
\]
 where $(D_{j})$ is either the numerical $p$-minimal $k$-partition $\cD^{k,p}$ for $p=50$ or the partition obtained with the penalization method.
 We also added the relative differences between maximal and minimal eigenvalues. Comparing these differences we note that the penalization method gives partitions which are closer to being an equipartition.  We also observe that the maximal value among the first eigenvalues is lower for the penalization method. Thus, in the cases considered here, this method gives us better candidates. The partitions obtained with the penalization method are presented in Figure \ref{pen-figs}.
 
\begin{table}[h!]
\centering 
\begin{tabular}{|c|c|c|c|c|c|c|c|}
\hline
\multirow{2}{*}{$\Omega$} & \multirow{2}{*}{$k$} & \multicolumn{3}{c|}{ $\fL_{k,50}(\Omega)$} & \multicolumn{3}{c|}{penalization}\\ \cline{3-8}
                                &         & min & max & diff.(\%) & min & max & diff.(\%) \\ \hline 
\multirow{7}{*}{$\triangle$}  
     & 4 & $208.92$ & $211.71$ & $1.32$ & $209.15$ & $211.04$ & $0.89$ \\ \cline{2-8} 
     & 5 & $249.17$ & $252.67$ & $1.38$ & $251.27$ & $252.17$ & $0.36$ \\ \cline{2-8}
     & 6 & $275.37$ & $276.16$ & $0.28$ & $275.34$ & $276.22$ & $0.31$ \\ \cline{2-8}
     & 7 & $338.04$ & $348.24$ & $2.92$ & $343.51$ & $345.91$ & $0.69$ \\ \cline{2-8}
     & 8 & $388.47$ & $391.06$ & $0.66$ & $388.46$ & $389.53$ & $0.27$ \\ \cline{2-8}
     & 9 & $422.80$ & $431.92$ & $2.11$ & $425.34$ & $428.74$ & $0.79$ \\ \cline{2-8}
     & 10& $445.50$ & $456.66$ & $2.44$ & $450.74$ & $453.25$ & $0.55$ \\ \hline
\multirow{6}{*}{$\square$}  
     & 5 & $103.75$ & $105.82$ & $1.95$ & $104.24$ & $104.60$ & $0.34$ \\ \cline{2-8}
     & 6 & $125.79$ & $128.11$ & $1.81$ & $126.36$ & $128.14$ & $1.38$ \\ \cline{2-8}
     & 7 & $144.49$ & $147.44$ & $2.00$ & $145.81$ & $146.90$ & $0.74$ \\ \cline{2-8}
     & 8 & $160.48$ & $161.64$ & $0.71$ & $160.76$ & $161.28$ & $0.32$ \\ \cline{2-8}
     & 9 & $176.64$ & $179.21$ & $1.49$ & $177.13$ & $178.08$ & $0.53$\\ \cline{2-8}
     & 10& $200.00$ & $206.85$ & $3.31$ & $202.78$ & $204.54$ & $0.86$ \\ \hline     
\end{tabular}\\[5pt]
\caption{Minimal and maximal eigenvalues of the candidates obtained by the $p$-norm and the penalization methods.}
\label{pnorm_pen}
\end{table}

\begin{figure}[h!t]
\centering 
\includegraphics[width = 0.1\textwidth]{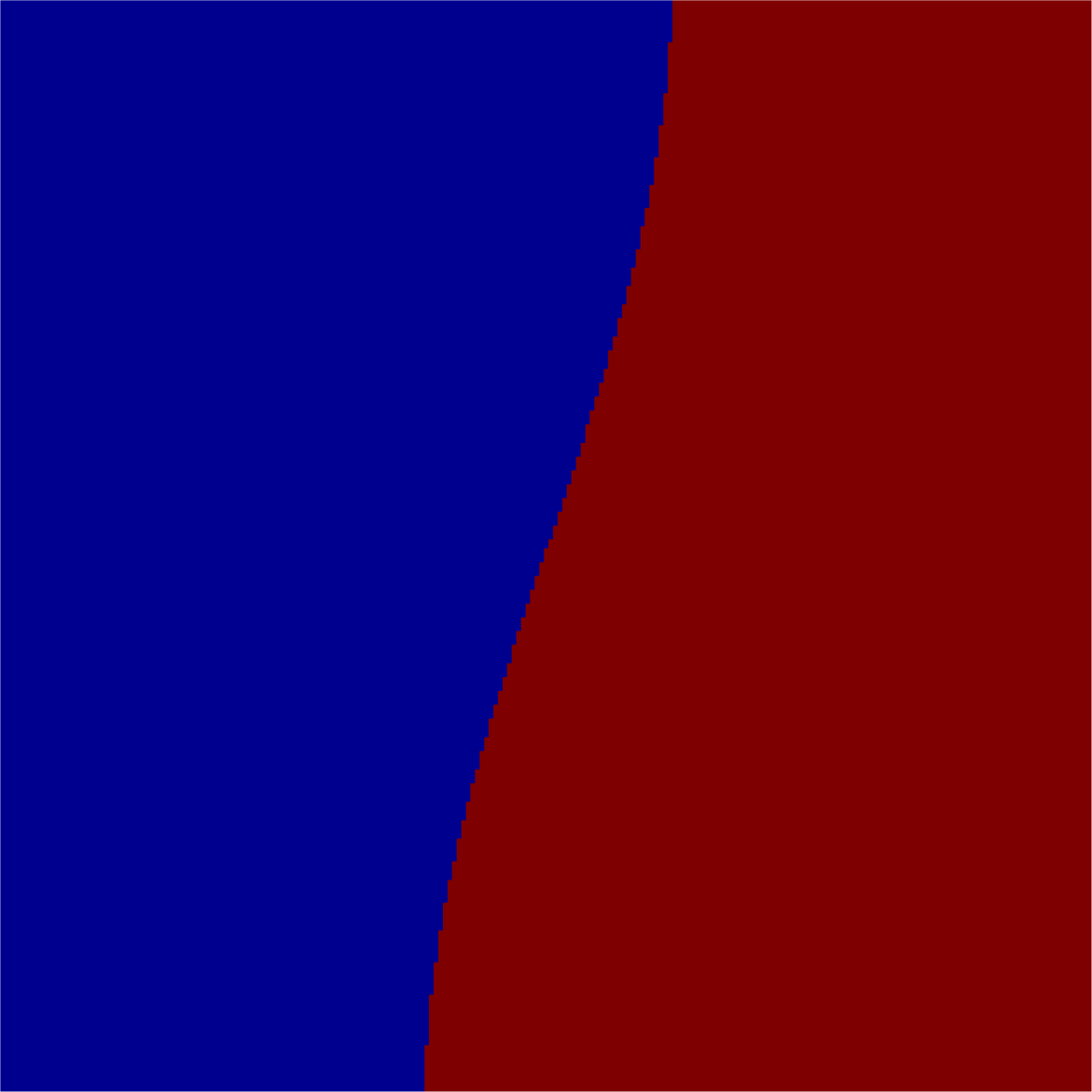}~
\includegraphics[width = 0.1\textwidth]{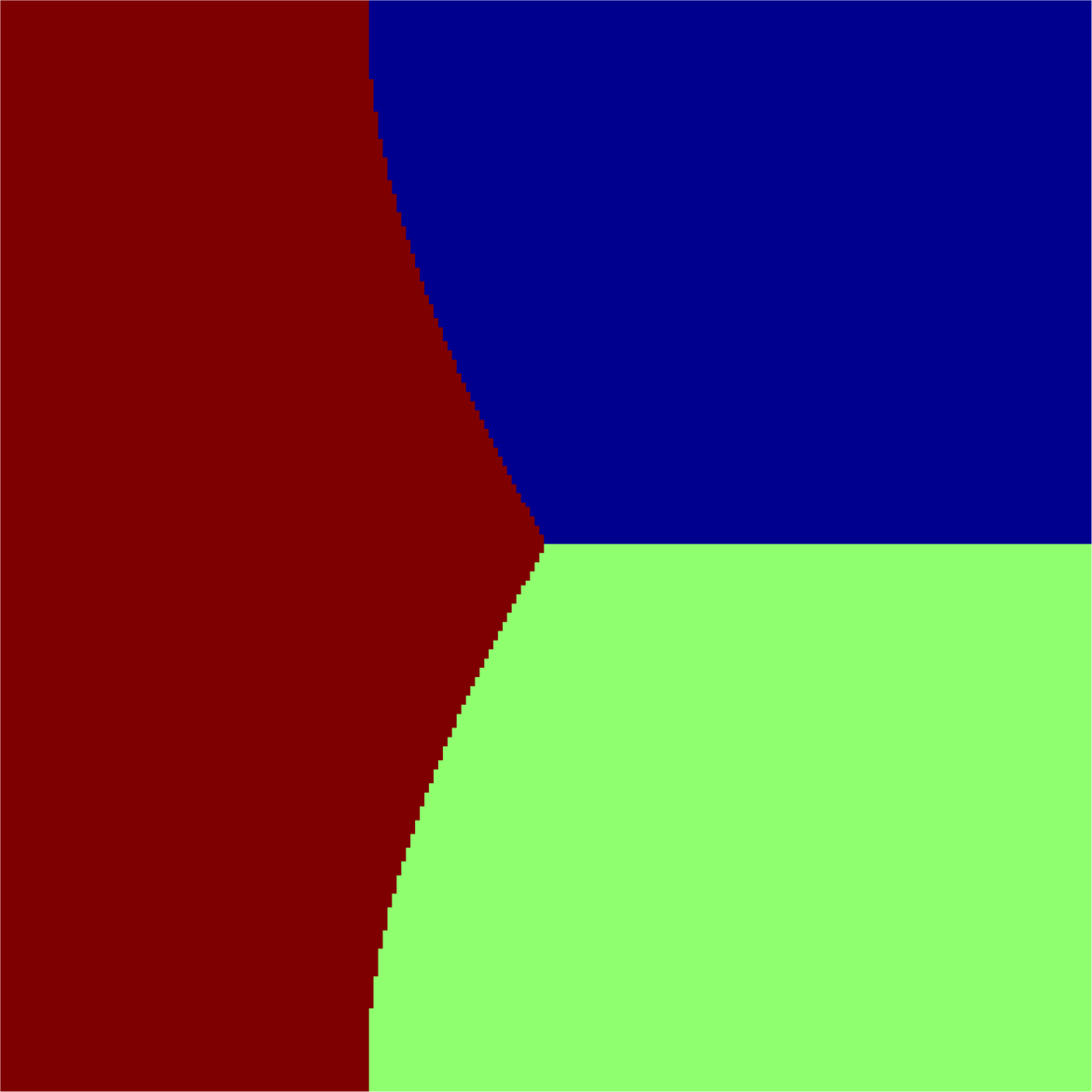}~
\includegraphics[width = 0.1\textwidth]{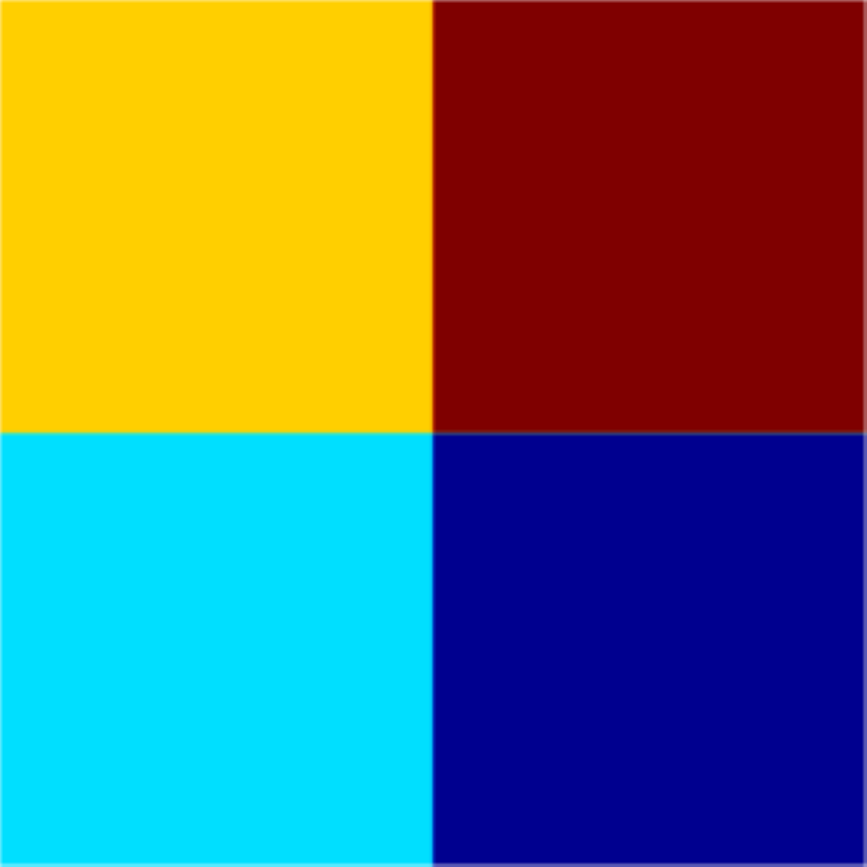}~
\includegraphics[width = 0.1\textwidth]{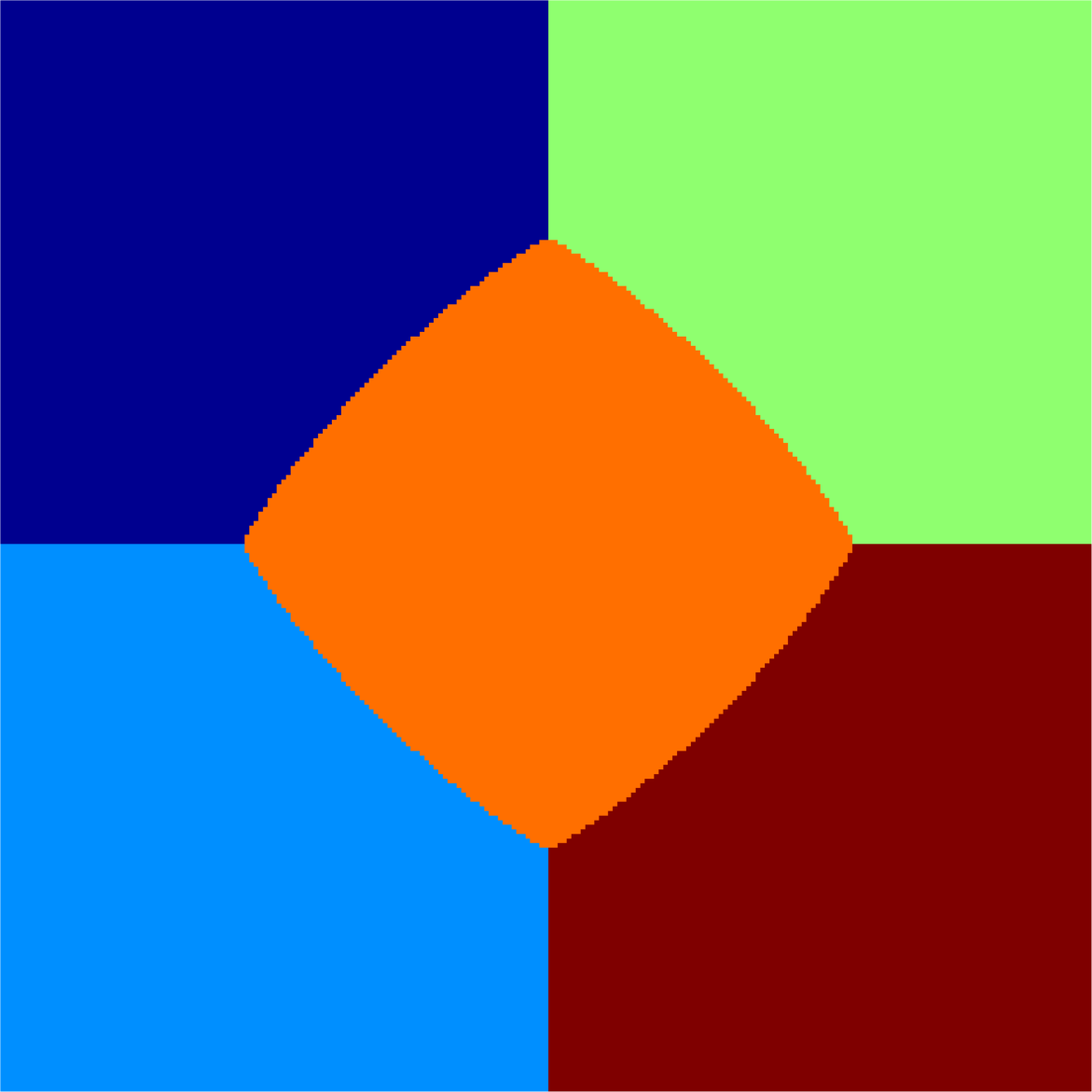}~
\includegraphics[width = 0.1\textwidth]{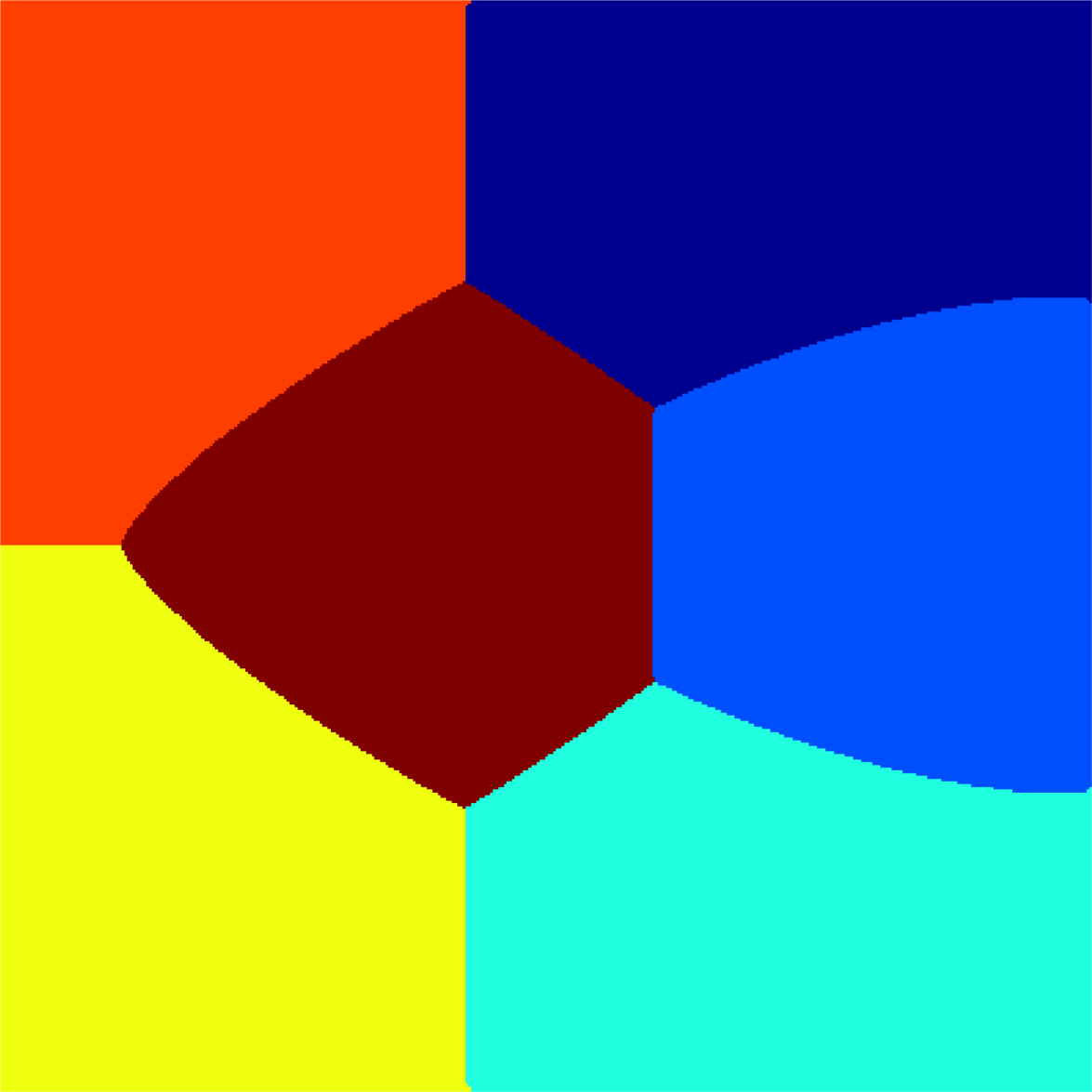}~
\includegraphics[width = 0.1\textwidth]{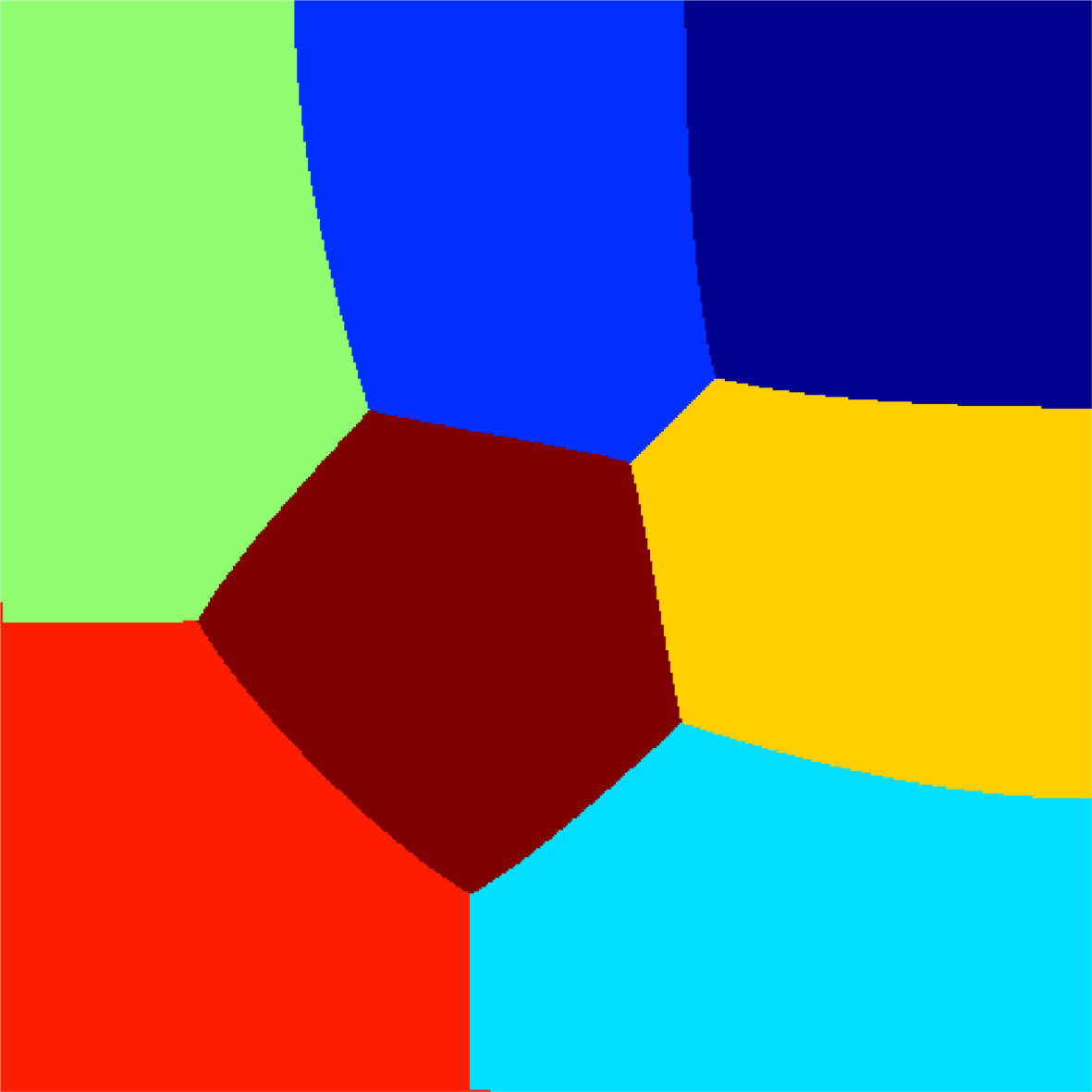}~
\includegraphics[width = 0.1\textwidth]{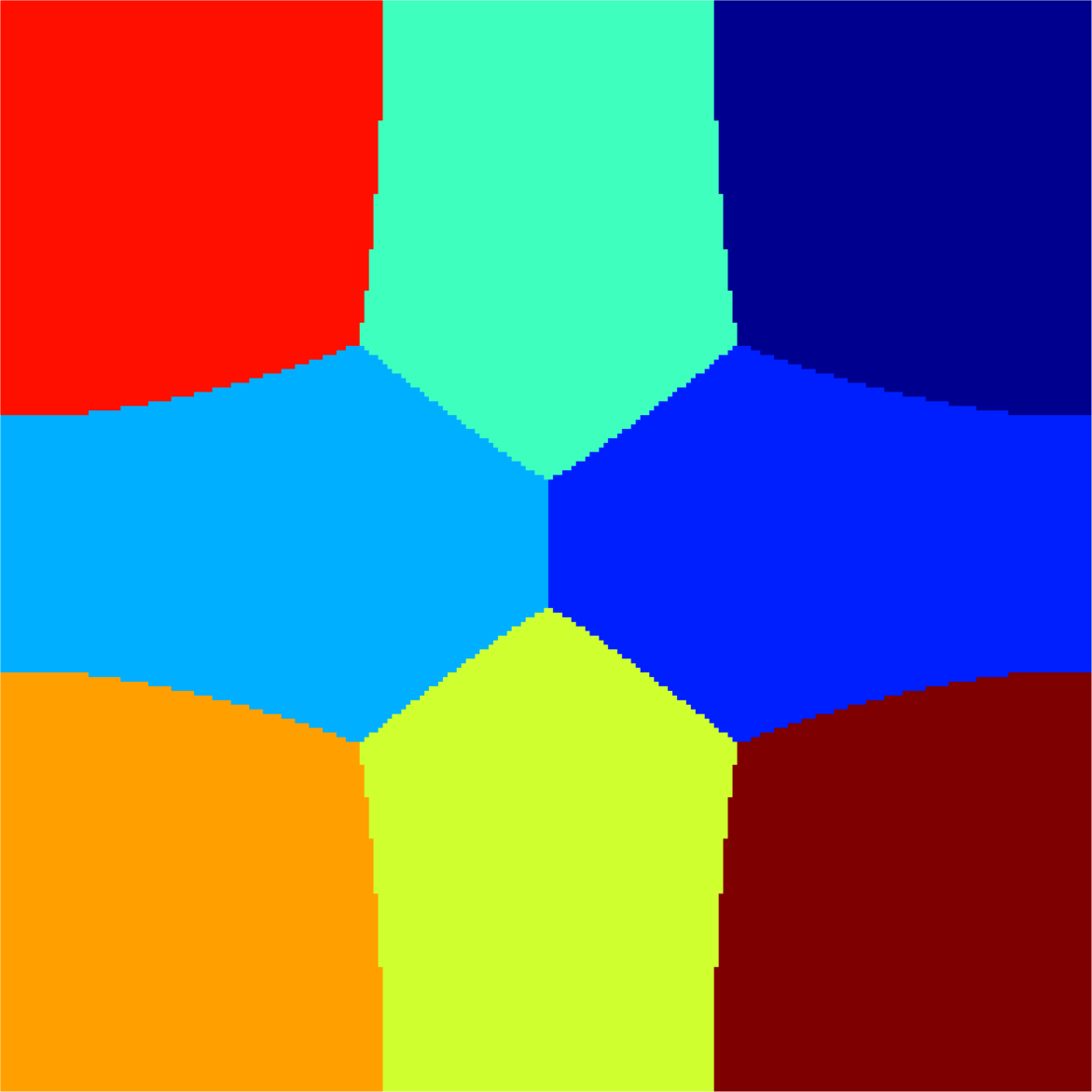}~
\includegraphics[width = 0.1\textwidth]{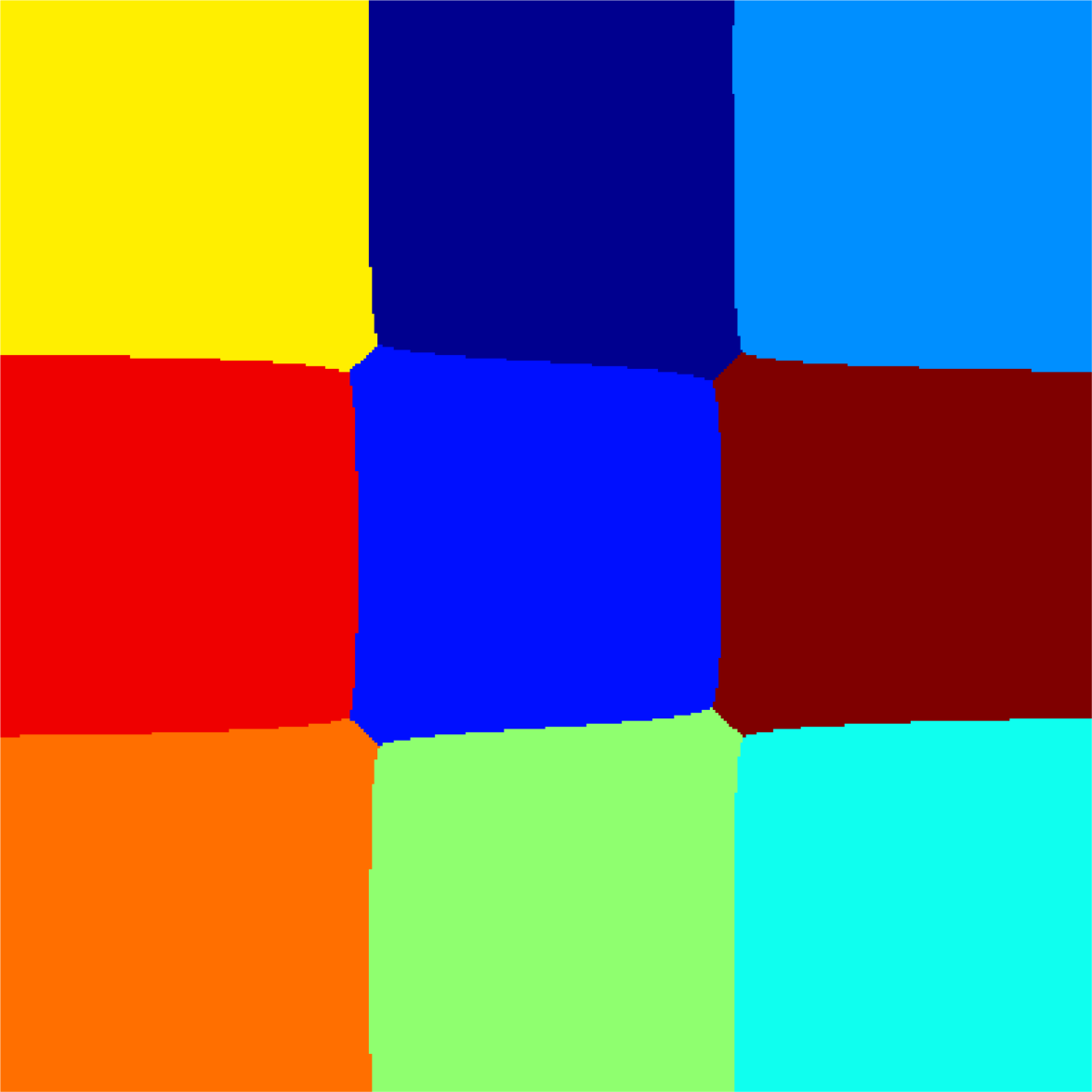}~
\includegraphics[width = 0.1\textwidth]{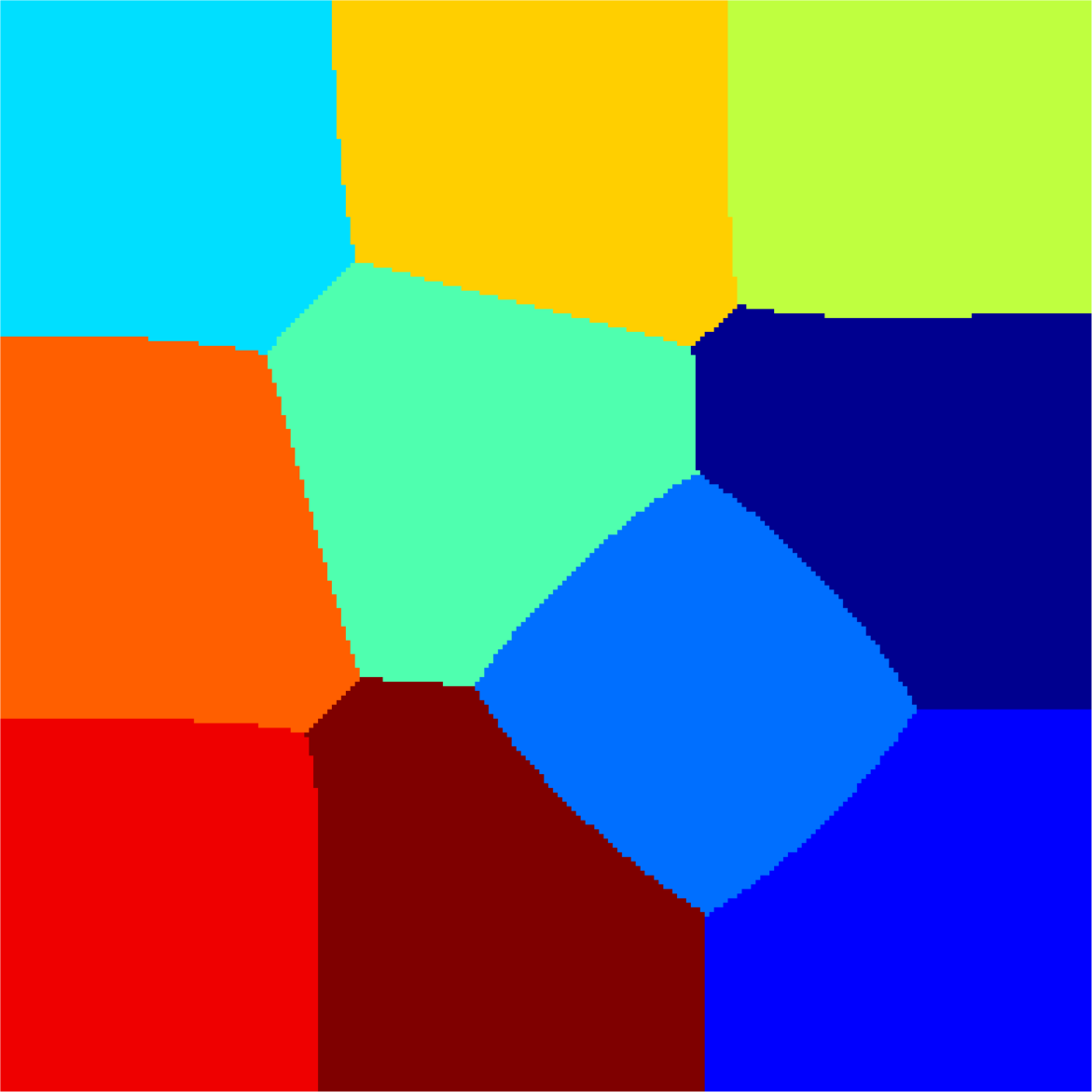}
\vspace{0.1cm}

\includegraphics[width = 0.1\textwidth]{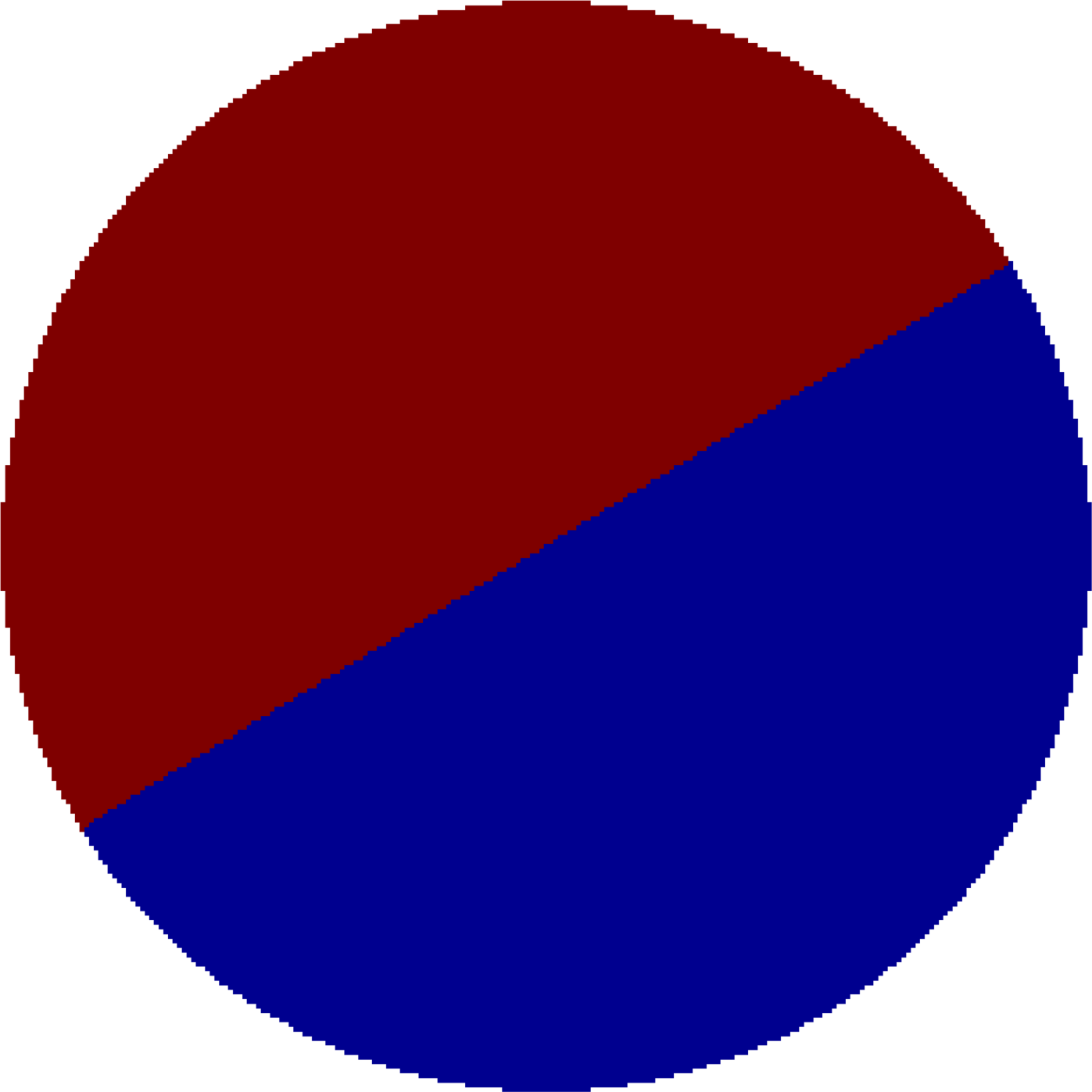}~
\includegraphics[width = 0.1\textwidth]{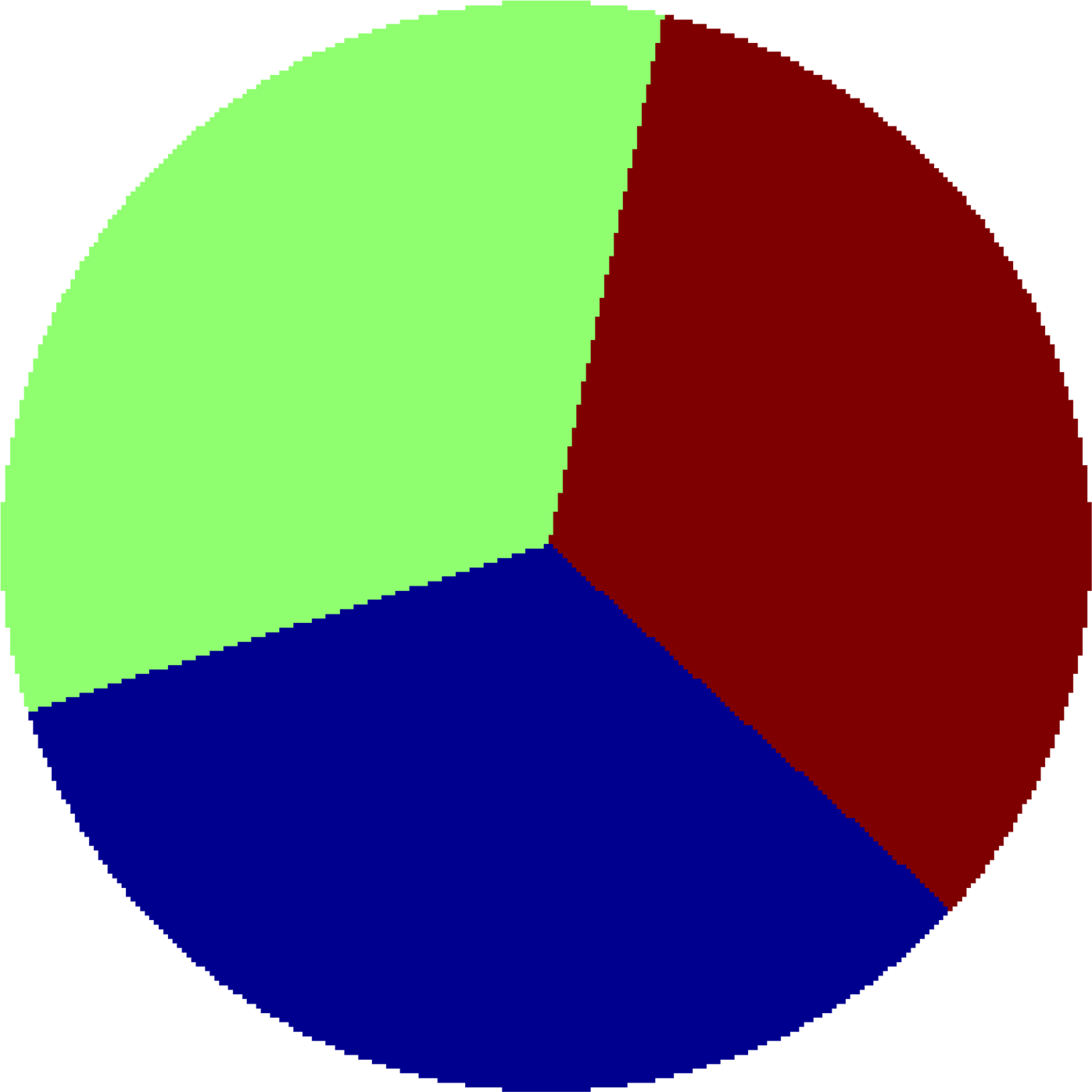}~
\includegraphics[width = 0.1\textwidth]{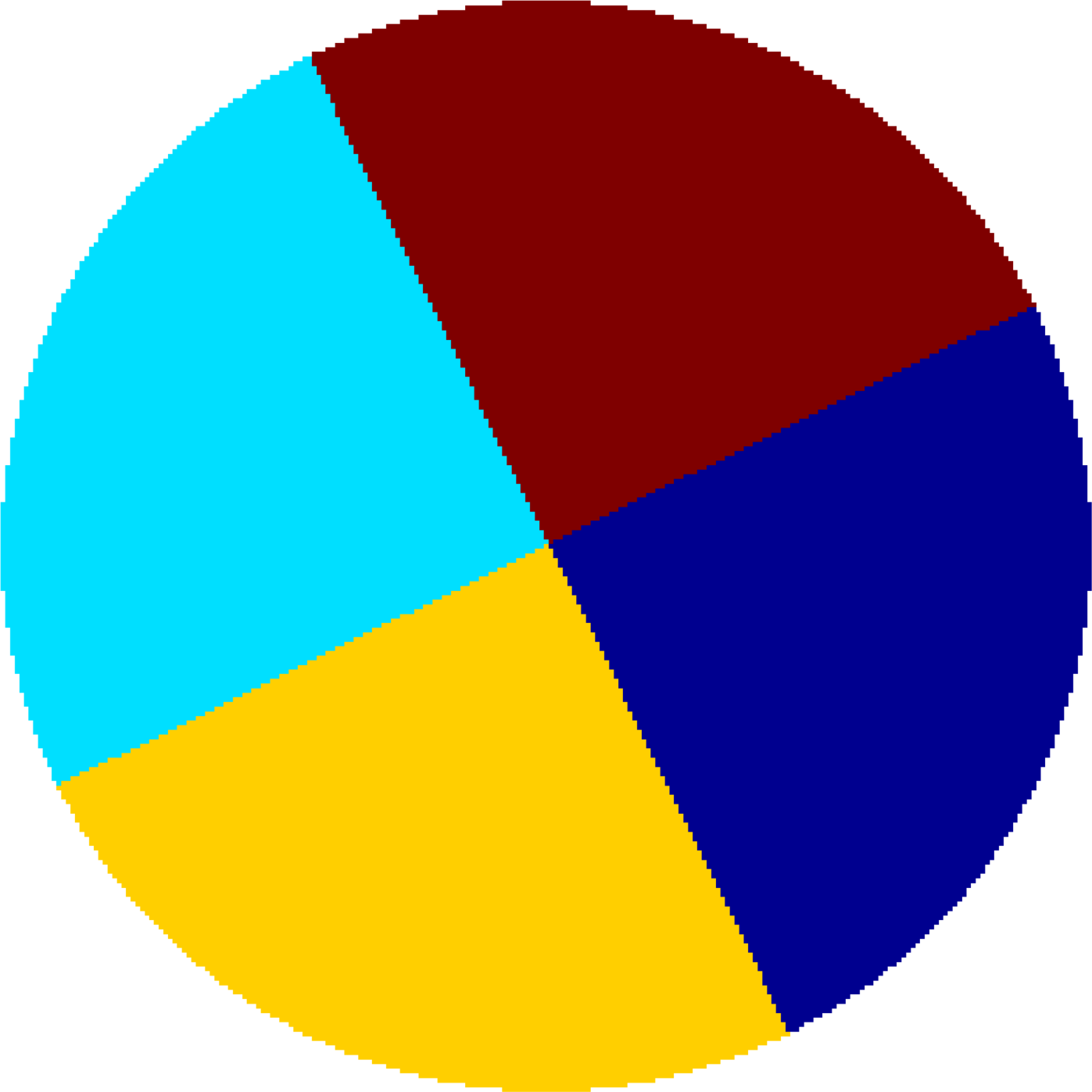}~
\includegraphics[width = 0.1\textwidth]{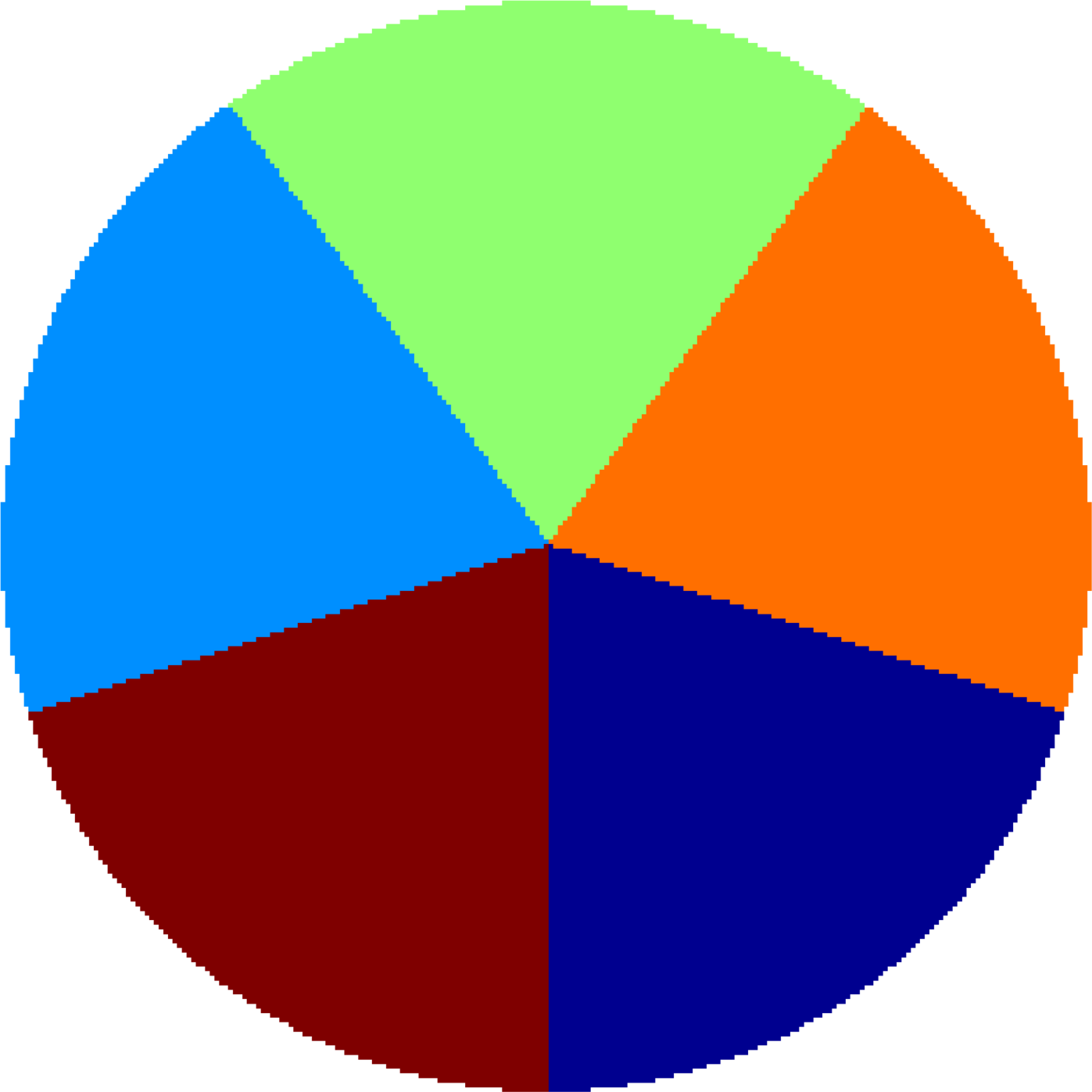}~
\includegraphics[width = 0.1\textwidth]{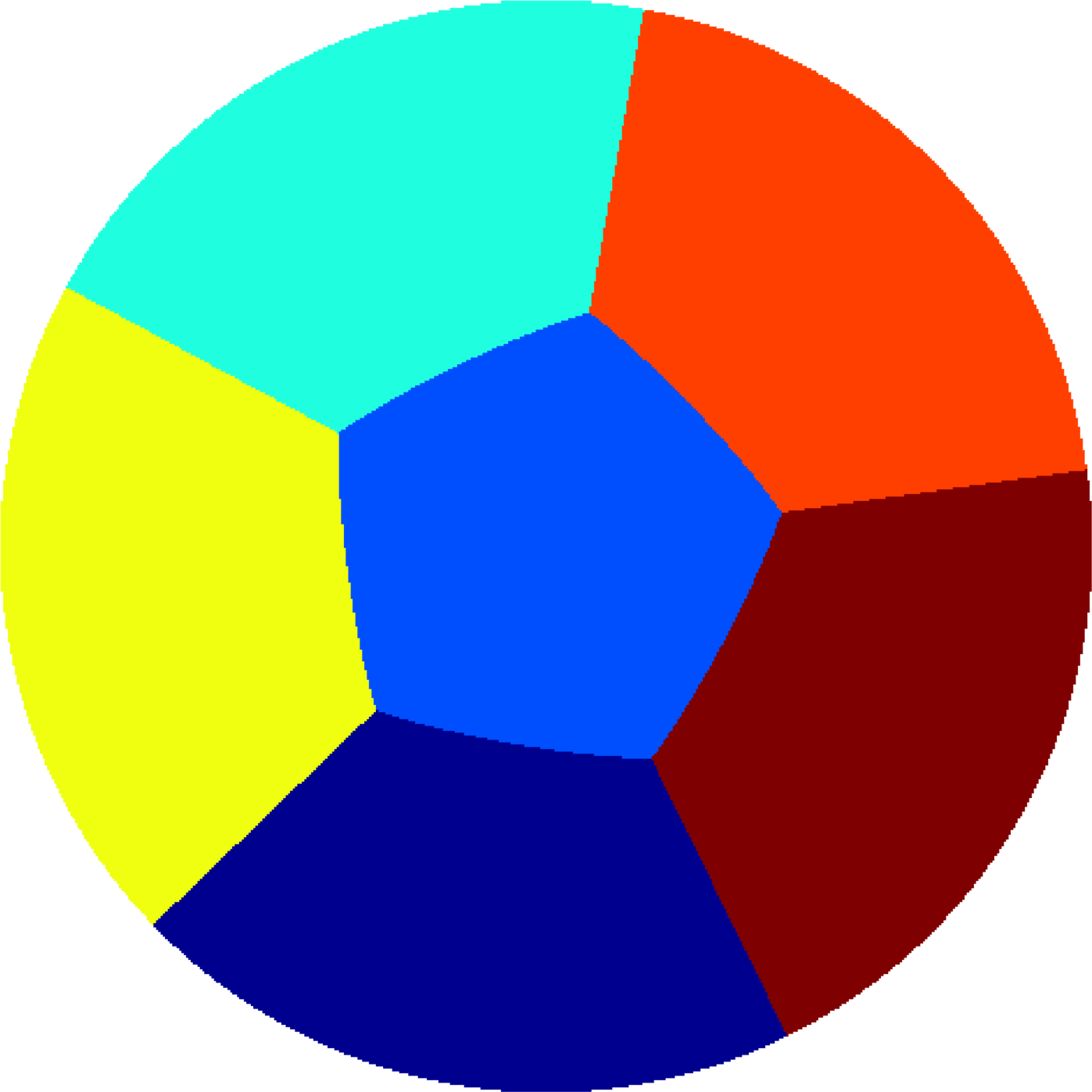}~
\includegraphics[width = 0.1\textwidth]{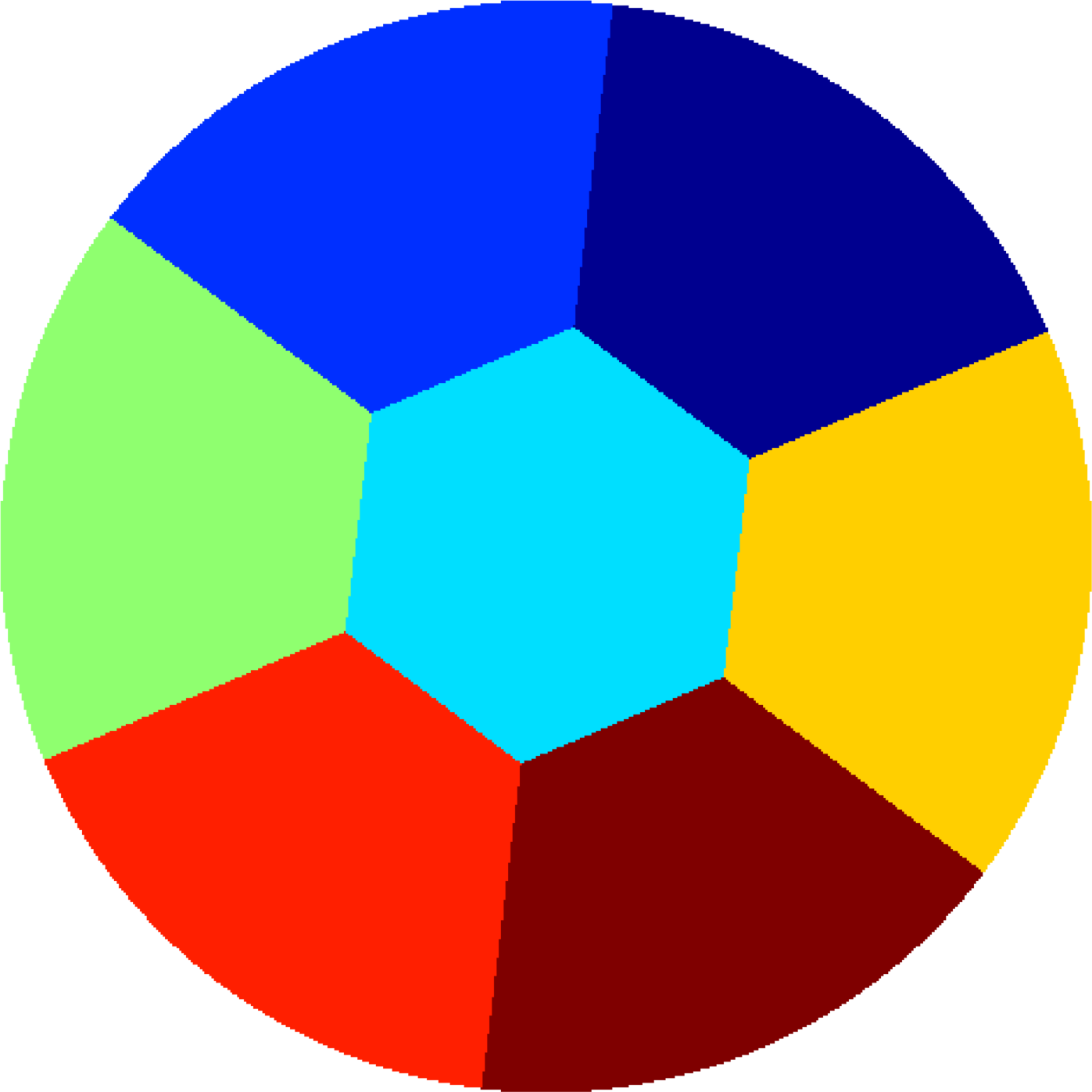}~
\includegraphics[width = 0.1\textwidth]{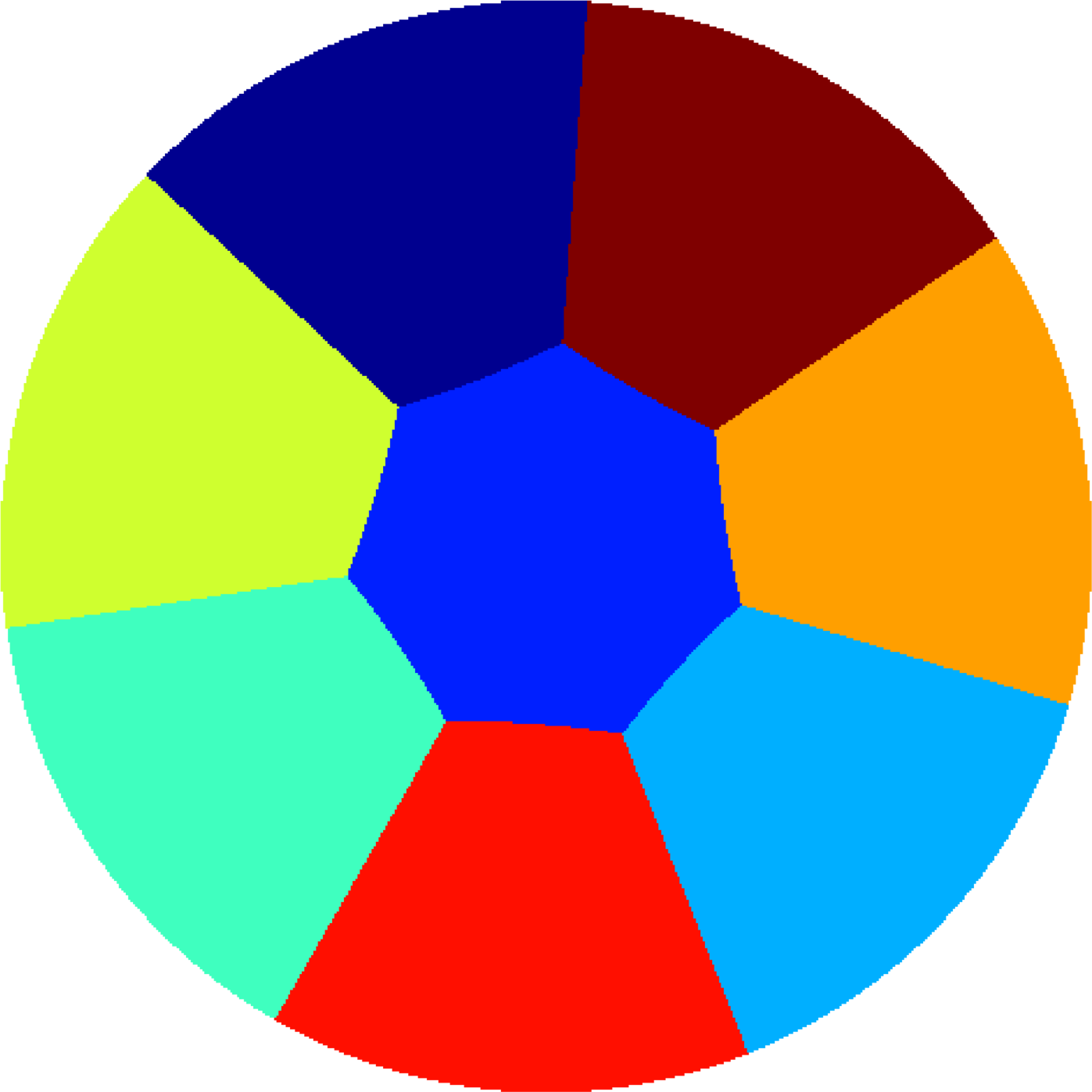}~
\includegraphics[width = 0.1\textwidth]{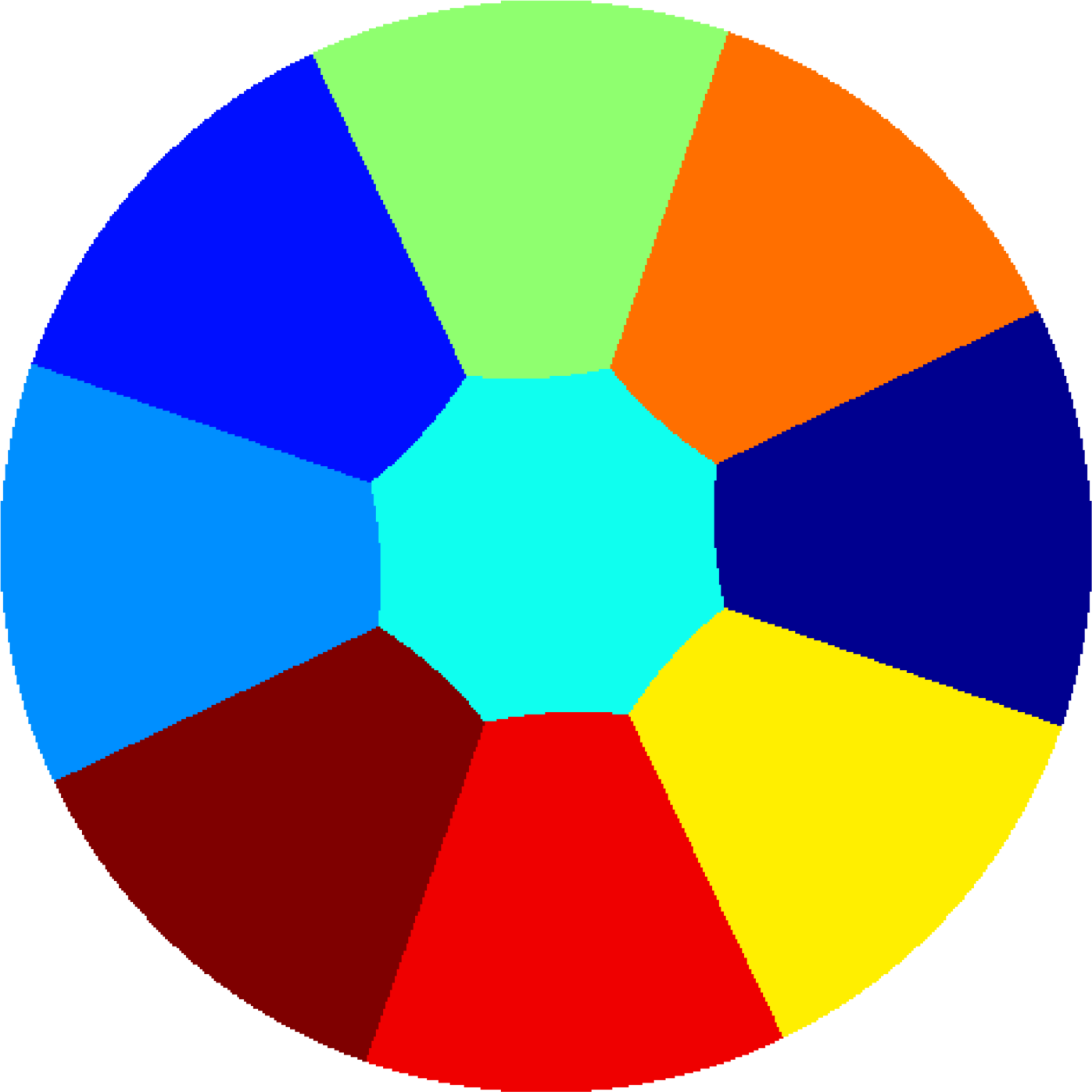}~
\includegraphics[width = 0.1\textwidth]{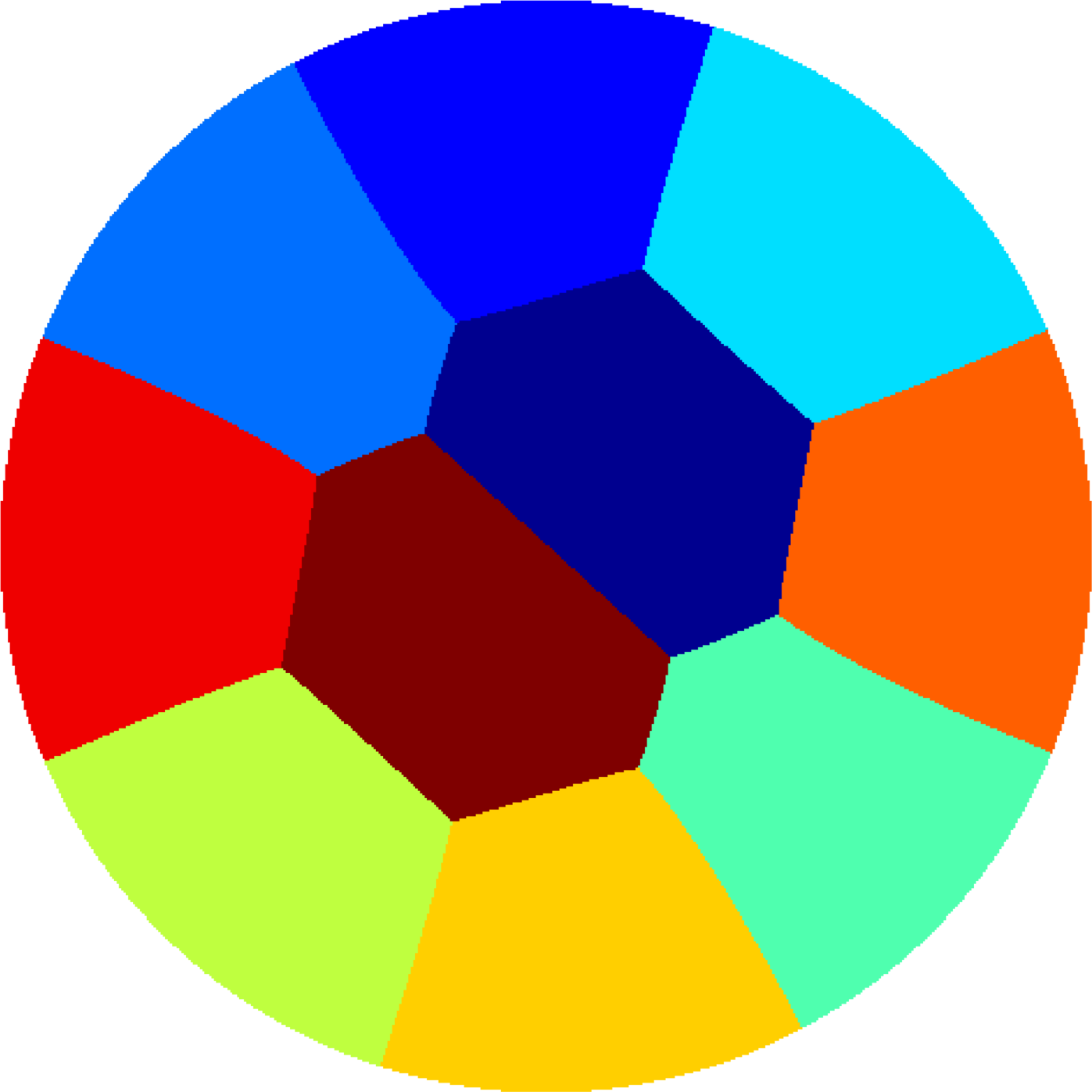}
\vspace{0.1cm}

\includegraphics[width = 0.1\textwidth]{sequi2}~
\includegraphics[width = 0.1\textwidth]{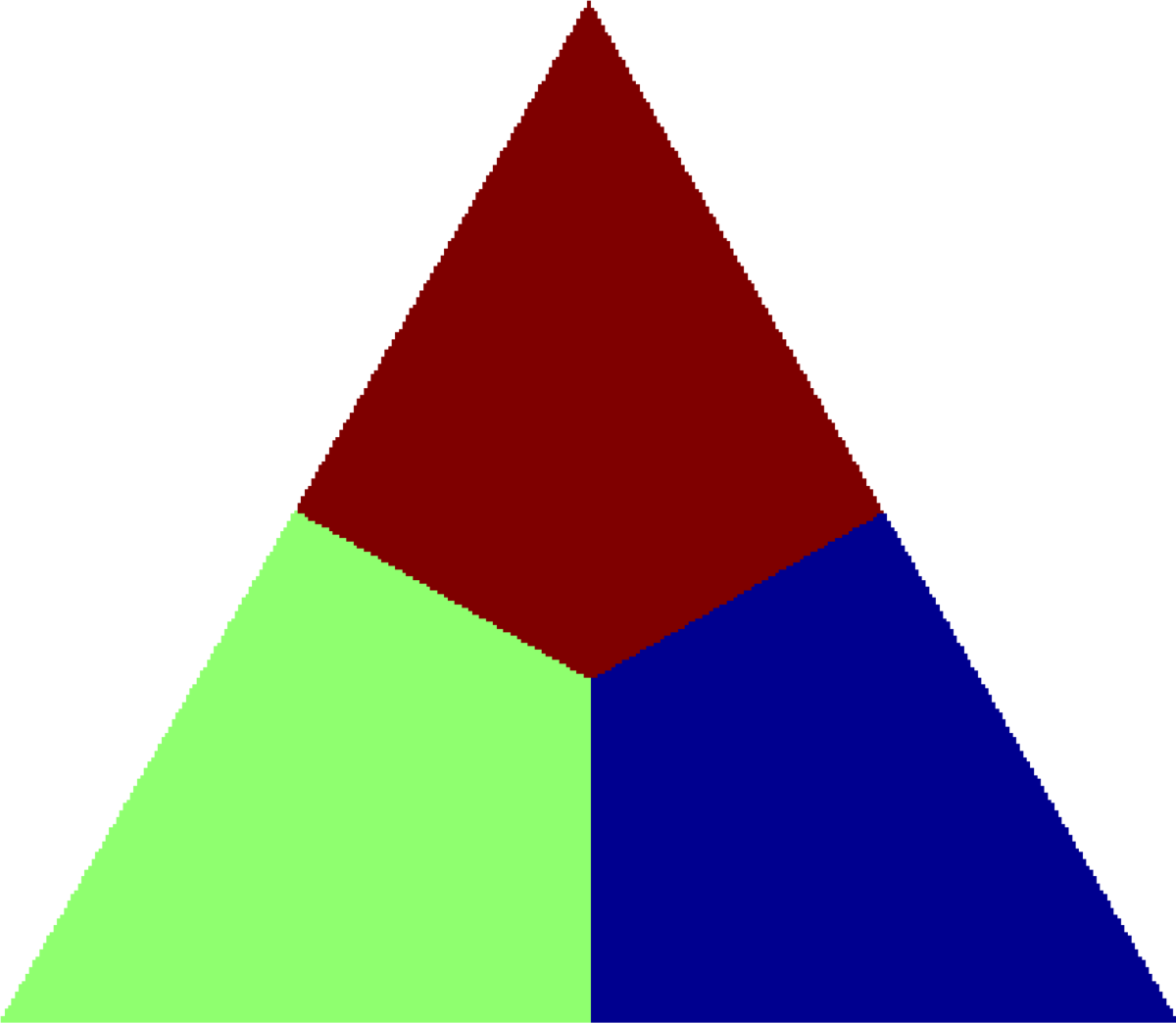}~
\includegraphics[width = 0.1\textwidth]{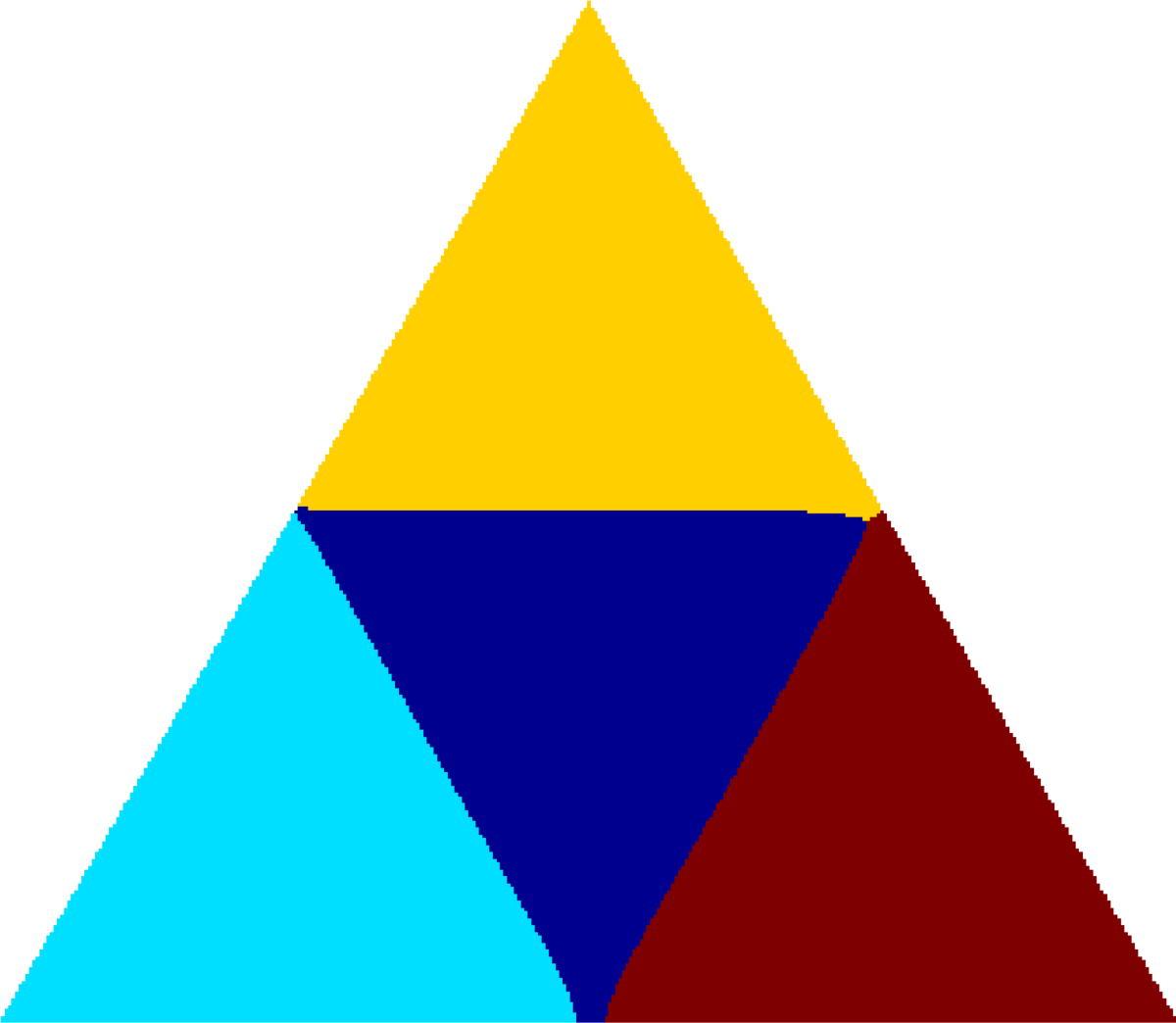}~
\includegraphics[width = 0.1\textwidth]{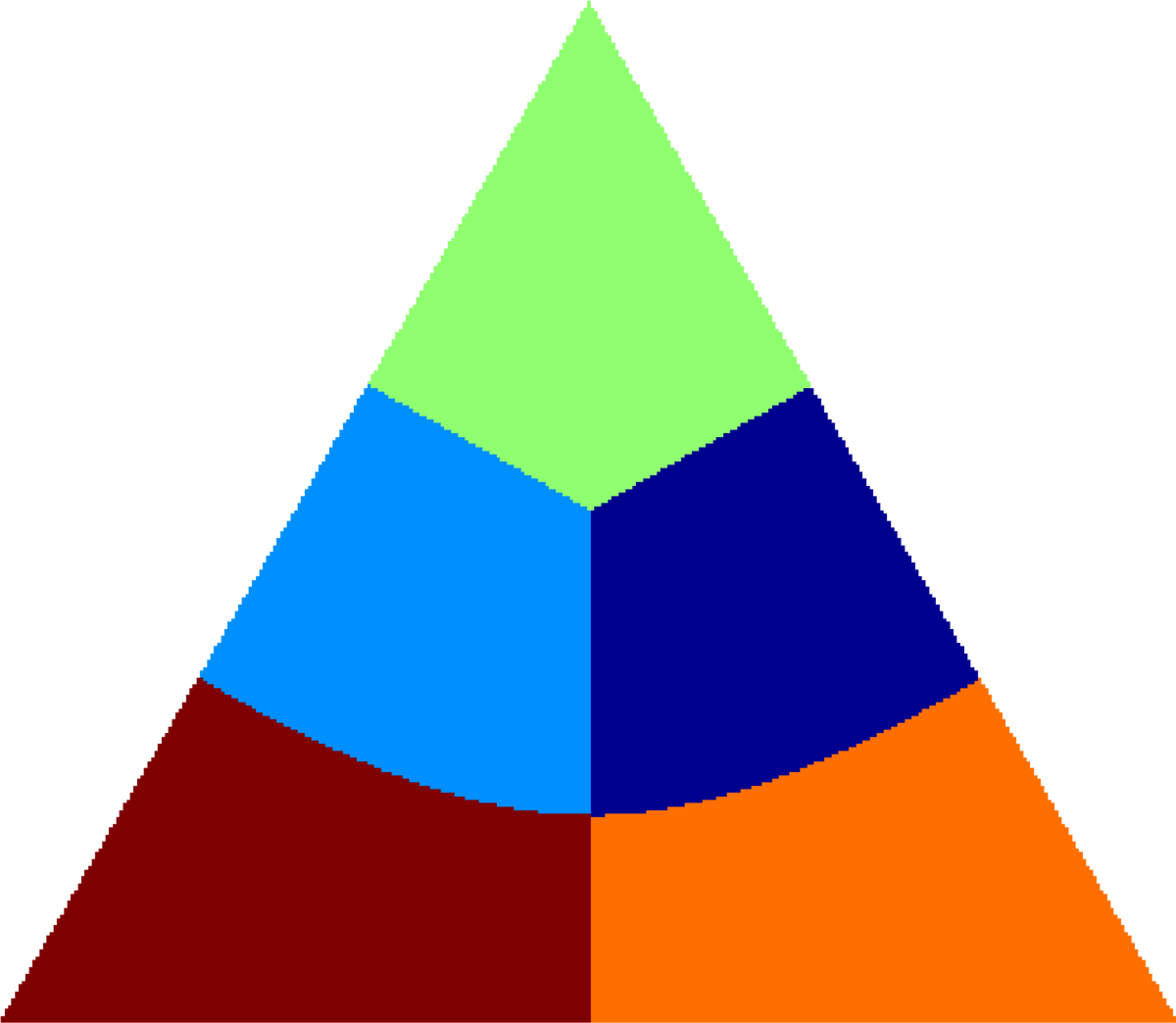}~
\includegraphics[width = 0.1\textwidth]{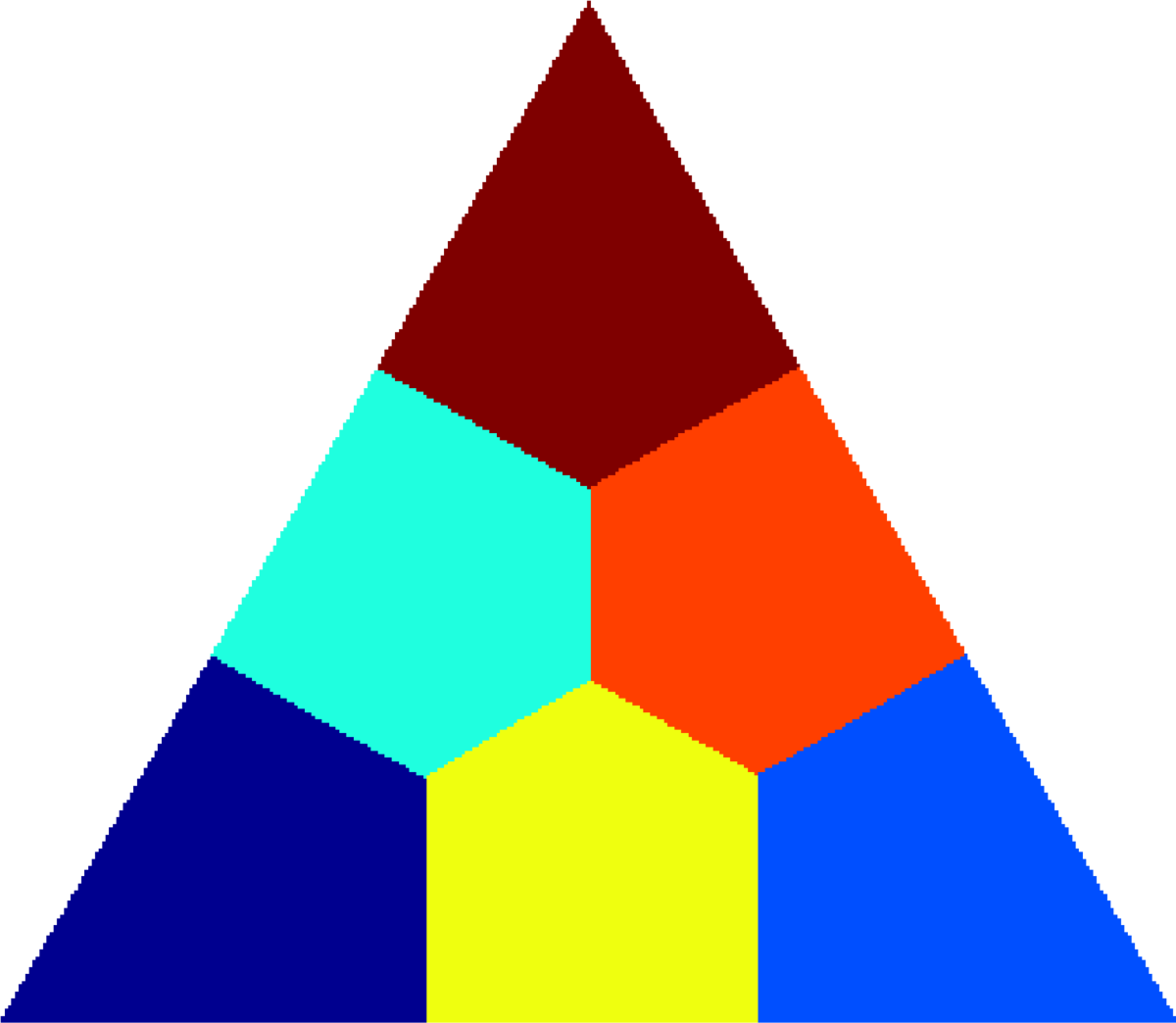}~
\includegraphics[width = 0.1\textwidth]{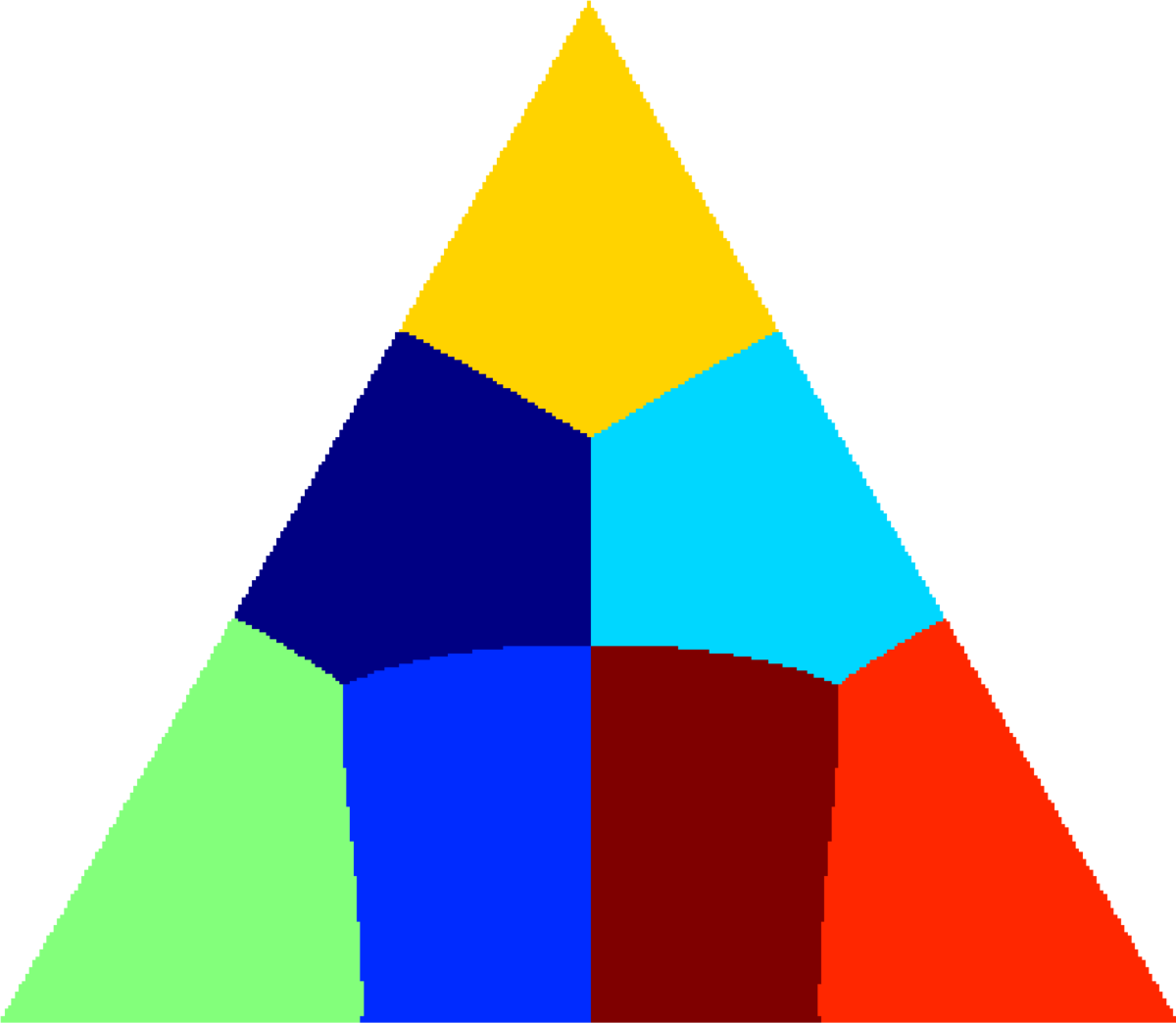}~
\includegraphics[width = 0.1\textwidth]{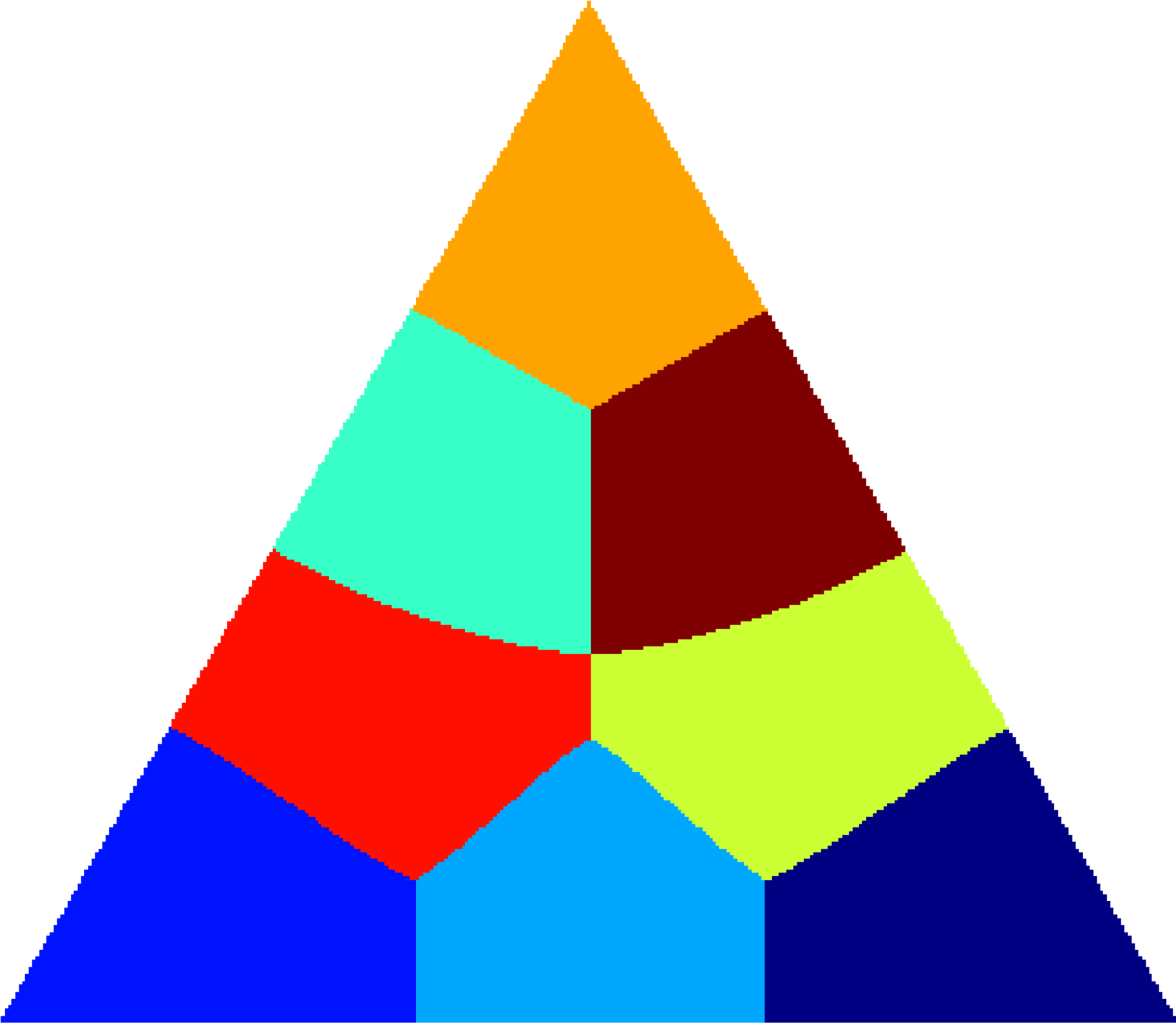}~
\includegraphics[width = 0.1\textwidth]{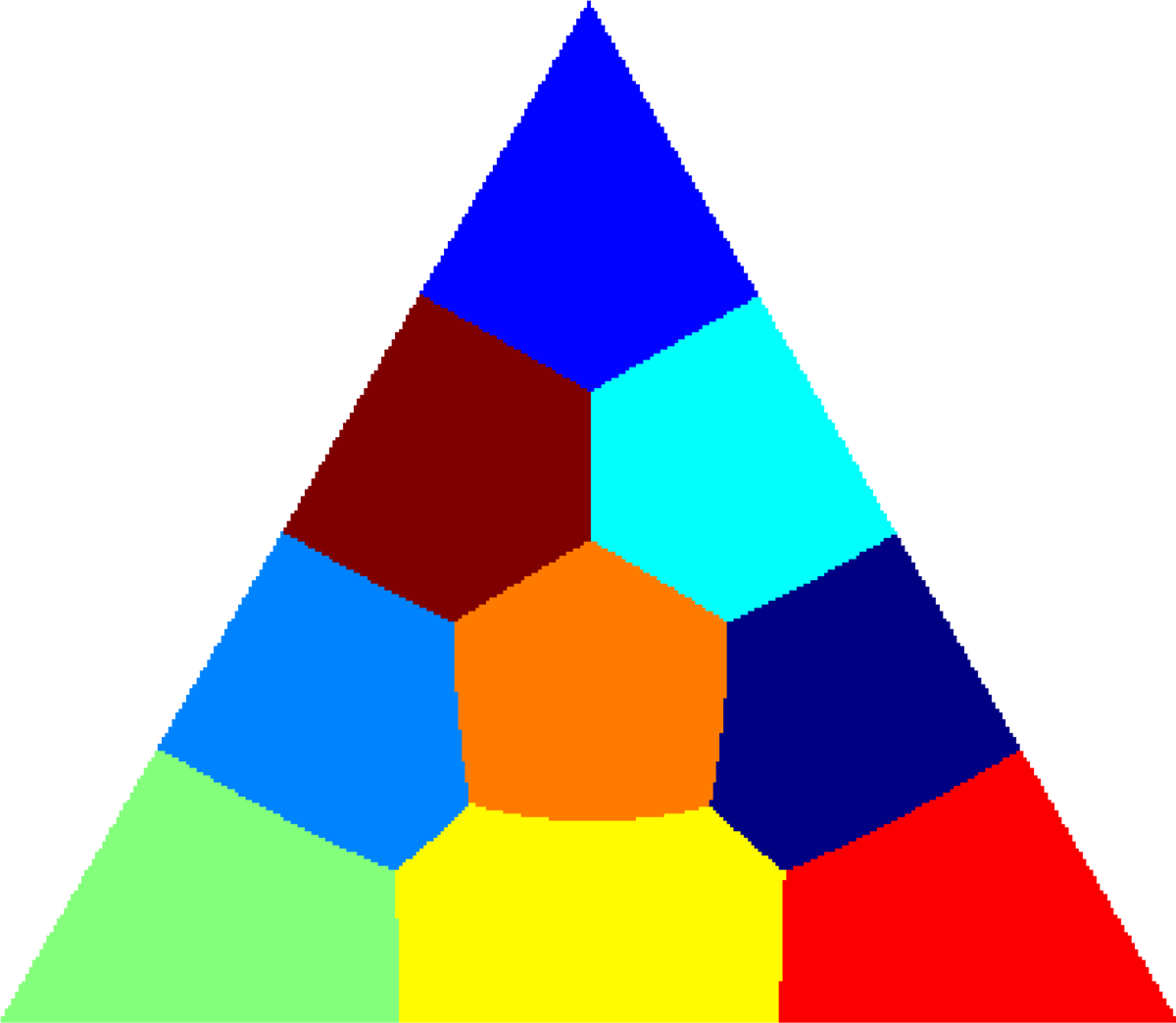}~
\includegraphics[width = 0.1\textwidth]{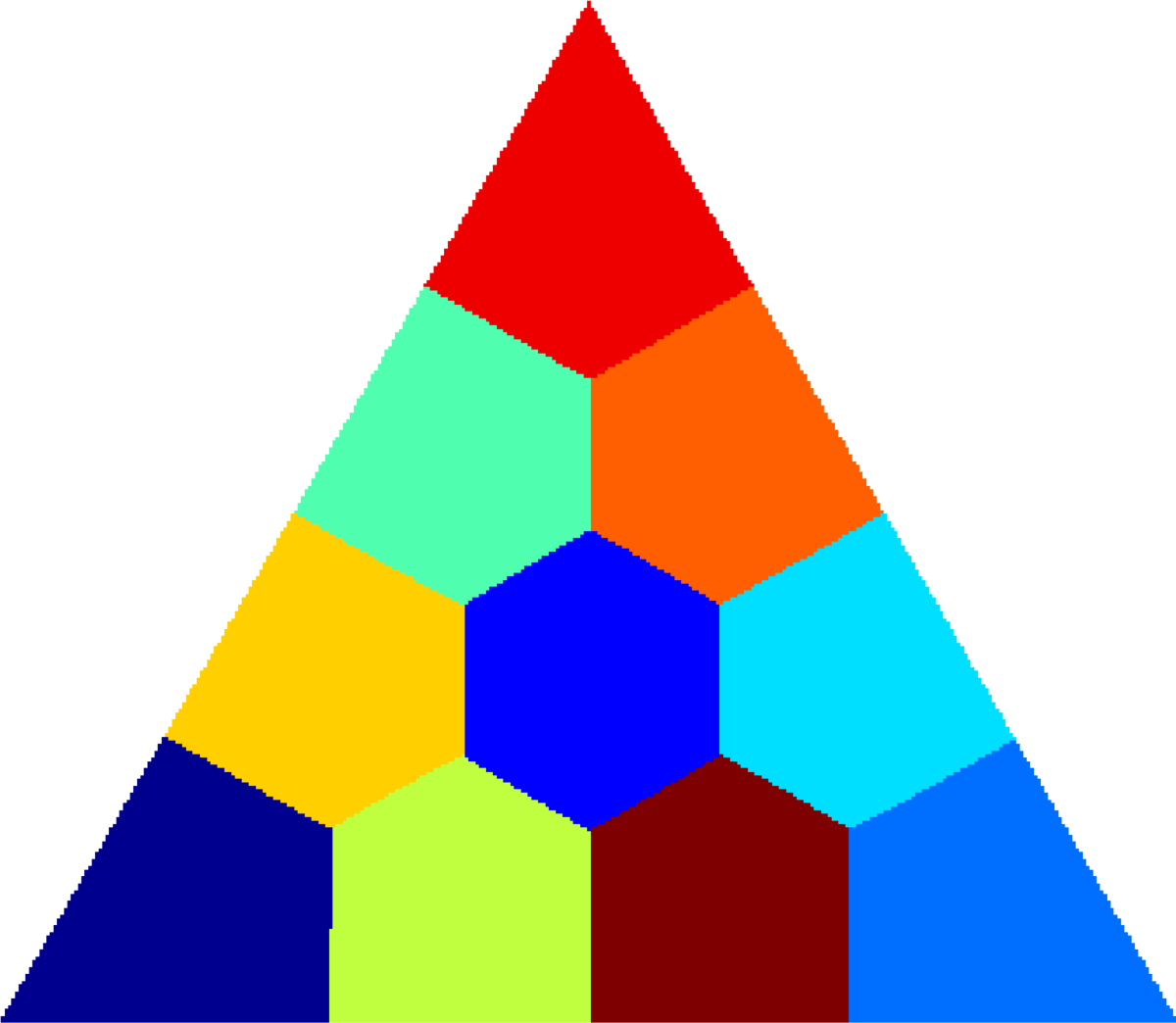}
\caption{Partitions obtained with the penalization method.}
\label{pen-figs}
\end{figure}

\begin{table}[h!]
\centering 
\begin{tabular}{|c|c|c|c|c|c|c|c|c|c|}
\hline 
 & \multicolumn{3}{c|}{Disk} & \multicolumn{3}{c|}{Square} & \multicolumn{3}{c|}{Equilateral triangle}\\ \hline
$k$ & $p=50$ & pen. & explicit & $p=50$ & pen. & explicit & $p=50$ & pen. & explicit \\ \hline  
$2$  & $14.68$ & $14.68$ & $14.68$   & $49.348$  & $49.348$ & $49.348$ & $123.38$ & $122.96$ & $122.82$          \\ \hline 
$4$  & $26.42$ & $26.42$ & $26.37$ & $78.957$ & $78.957$ & $78.957$ & $211.71$ & $211.04$ & $210.55$ \\ \hline 
\end{tabular}\\[5pt]
\caption{Comparison of the two methods for $\Omega=\Circle, \square, \triangle$ in explicit cases.}
\label{known_results}
\end{table}
Since in the cases $k=2,4$ we know the explicit optimizers we summarize in Table \ref{known_results} the results obtained with our numerical approaches in these two cases. We observe that the penalization method produces better candidates. 

Synthesized results are presented in Table~\ref{synthese-max} where we also present the values obtained with the mixed Dirichlet-Neumann method presented in the next section.

\subsection{Dirichlet-Neumann approach\label{sec.DN}}
The penalization method proposed in the previous section gives improved results in some situations as compared to the $p$-norm method. Still, the results we obtain are close, but not precisely an equipartition, as the theoretical results state in Proposition \ref{properties-min} presented in the previous section. In the following we propose a method which can further improve some of our results by working directly with equipartitions. 

A natural way of obtaining equipartitions is to use nodal partitions corresponding to some eigenvalue problems. Note that for $k=2$ any minimal $2$-partition for $\fL_{2,\infty}(\Omega)$ is a nodal partition for the second eigenvalue of the Dirichlet-Laplacian on $\Omega$ (see Theorem~\ref{thm.HHOT} and \eqref{eq.k2}). According to Proposition~\ref{prop.k24}, when $\Omega=\square,\ \bigcirc,\ \triangle$,  no $\infty$-minimal $k$-partition is nodal except for $k=1,2,4$. This is also observed numerically because the partitions we exhibit  for $k\notin \{2,4\}$ have { at least one critical point with odd degree}. Since nodal partitions are bipartite, the degree of every singular point must be even. Therefore we cannot expect to be able to express our partitions as nodal partitions corresponding to an eigenvalue problem on the domain $\Omega$ with Dirichlet boundary conditions on $\partial \Omega$. It is possible, nevertheless, to represent equipartitions as nodal partitions on $\Omega$, by adding some additional Dirichlet conditions on curves inside $\Omega$. Since we are interested in finding equipartitions with minimal energy, we wish to be able to easily parametrize these curves on which we impose the additional Dirichlet boundary conditions, in order to optimize their position. 

In the cases presented in this section, we consider adding additional Dirichlet conditions on segments in $\Omega$. Therefore, we look at the results obtained with the iterative methods in order to see in which cases some boundaries are segments. Moreover, every singular point of odd degree should be contained in one such segments, so that the remaining partition is nodal. If the optimal partition has certain symmetries, some of the cells may share some of the symmetry properties. Therefore we may reduce our computations to a subset of $\Omega$ by considering mixed Dirichlet-Neumann problems. 

The idea is to search for minimal partitions with the aid of nodal sets of a certain mixed Dirichlet-Neumann problem. 
This approach has already been used in \cite{BHV} for the study of the $3$-partitions of the square and the disk. In the following we identify other situations where the method applies. In those cases the partition obtained with the Dirichlet-Neumann method is an exact equipartition and it allows us to decrease even more the value of $\Lambda_{k,\infty}$ (see Table~\ref{synthese-max}).

Let us take the case of the $3$-partition in the equilateral triangle as an example. The notations are presented in Figure \ref{equi3-DN}. 
Figure~\ref{fig.candtrik3} gives the partition obtained by one of the iterative methods. 
We represent below the partition with the symmetry axis $\sA\sD$ and the triple point $\sD_r$. {\Gn It is not difficult to see that this partition can be regarded as a nodal partition if we consider an additional Dirichlet boundary condition on the segment $[\sD\sD_r]$. Due to the symmetry this is equivalent to a mixed Dirichlet-Neumann problem on the triangle $\sA\sB\sD$ with Dirichlet condition on the segment $[\sD\sD_r]$ and Neumann condition on the segment $[\sA\sD_r]$.}. 
Thus, the working configuration is the triangle $\sA\sB\sD$ with Dirichlet boundary conditions on $[\sD\sD_r], [\sD\sB]$ and $[\sA\sB]$ and a Neumann boundary condition on $[\sA\sD_r]$. We take the point $\sD_r$ variable on $[\sA\sD]$ and we look for the position of $\sD_r$ for which the nodal line touches the segment $[\sD\sD_r]$ and for which the value of the second eigenvalue is minimal. {\Mg Necessarily, the nodal line ends at $\sD_{r}$.} Figures \ref{fig.Mixtri} give examples of nodal partitions according to the position of the mixed Dirichlet-Neumann point. In the following we make the convention that red lines signify Dirichlet boundary conditions and blue dotted lines represent Neumann boundary conditions.

\begin{figure}[h!]
\begin{center}
\subfigure[$\cD^{3,50}$\label{fig.candtrik3}]{\quad\includegraphics[height = .25\linewidth]{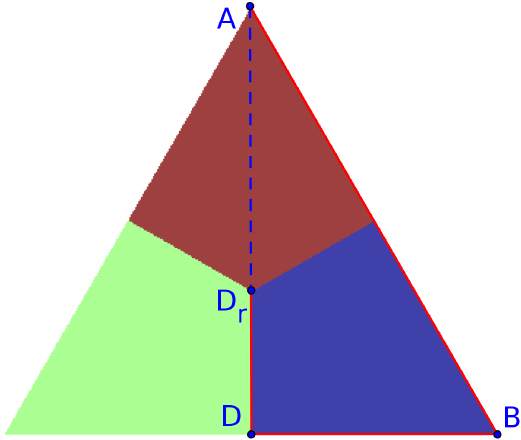}\quad}
\subfigure[Nodal lines according to the position of the mixed point\label{fig.Mixtri}]{\quad\includegraphics[height = .25\linewidth]{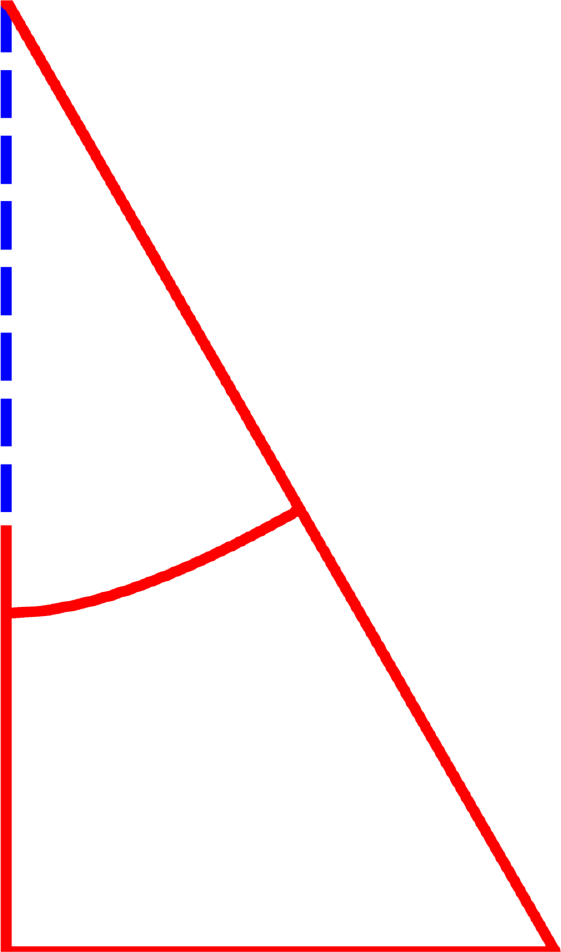}
\quad\includegraphics[height = .25\linewidth]{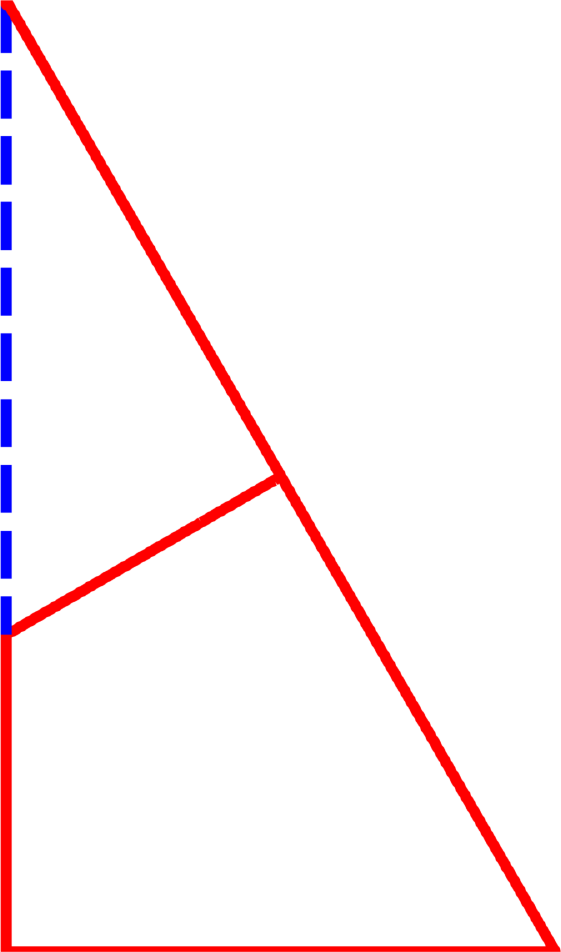}
\quad\includegraphics[height = .25\linewidth]{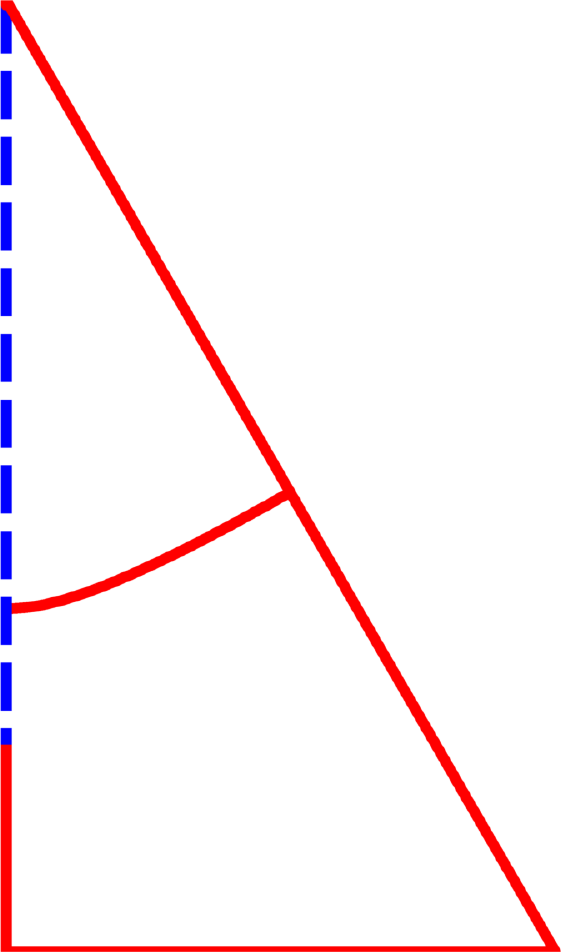}\quad}
\end{center}
\caption{Dirichlet-Neumann approach for $3$-partitions of the equilateral triangle.}
\label{equi3-DN}
\end{figure}

\subsubsection*{\bf The square.} 
We start with the case of $3$-partitions and we recall the results obtained in \cite{BHV}. The iterative algorithm gives a partition with an axis of symmetry parallel to the sides. Therefore we choose to impose a mixed condition on this axis, working on only half the square. Figure \ref{square3-DN} illustrates the choice of the mixed problem and the results. We obtain numerically that the triple point is at the center and that the value of the second Dirichlet-Neumann eigenvalue on the half-domain is $66.5812$. As it was noted in \cite{BHV}, choosing a mixed condition on the diagonal instead gives another partition with the same energy. Moreover, in \cite{BonHelHof09}, it is shown that we have a continuous family of partitions with the same energy.  
\begin{figure}[h!]
\begin{center}
\subfigure[Two mixed Dirichlet-Neumann configurations]{\quad
\includegraphics[width = 0.2\textwidth]{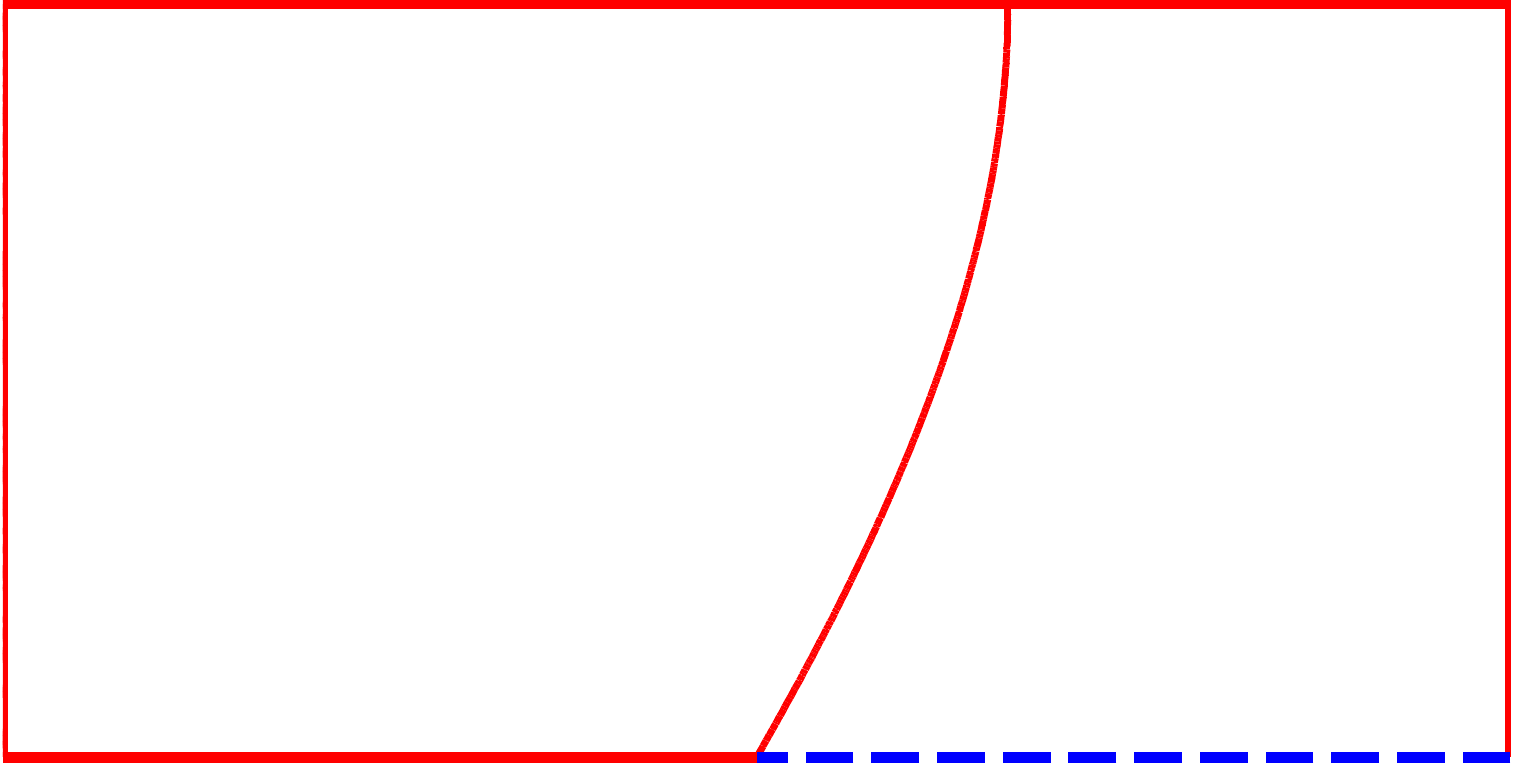}\ 
\includegraphics[width = 0.2\textwidth]{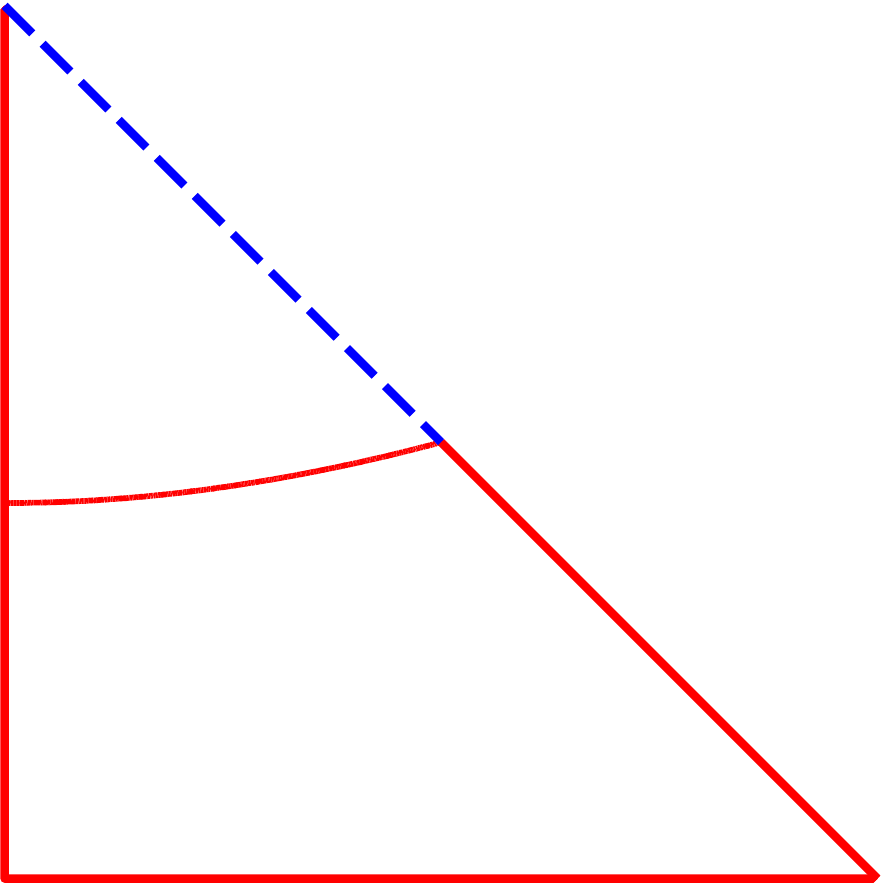}\quad}
\subfigure[Candidates for the $\infty$-minimal $3$-partition]{\quad
\includegraphics[width = 0.2\textwidth]{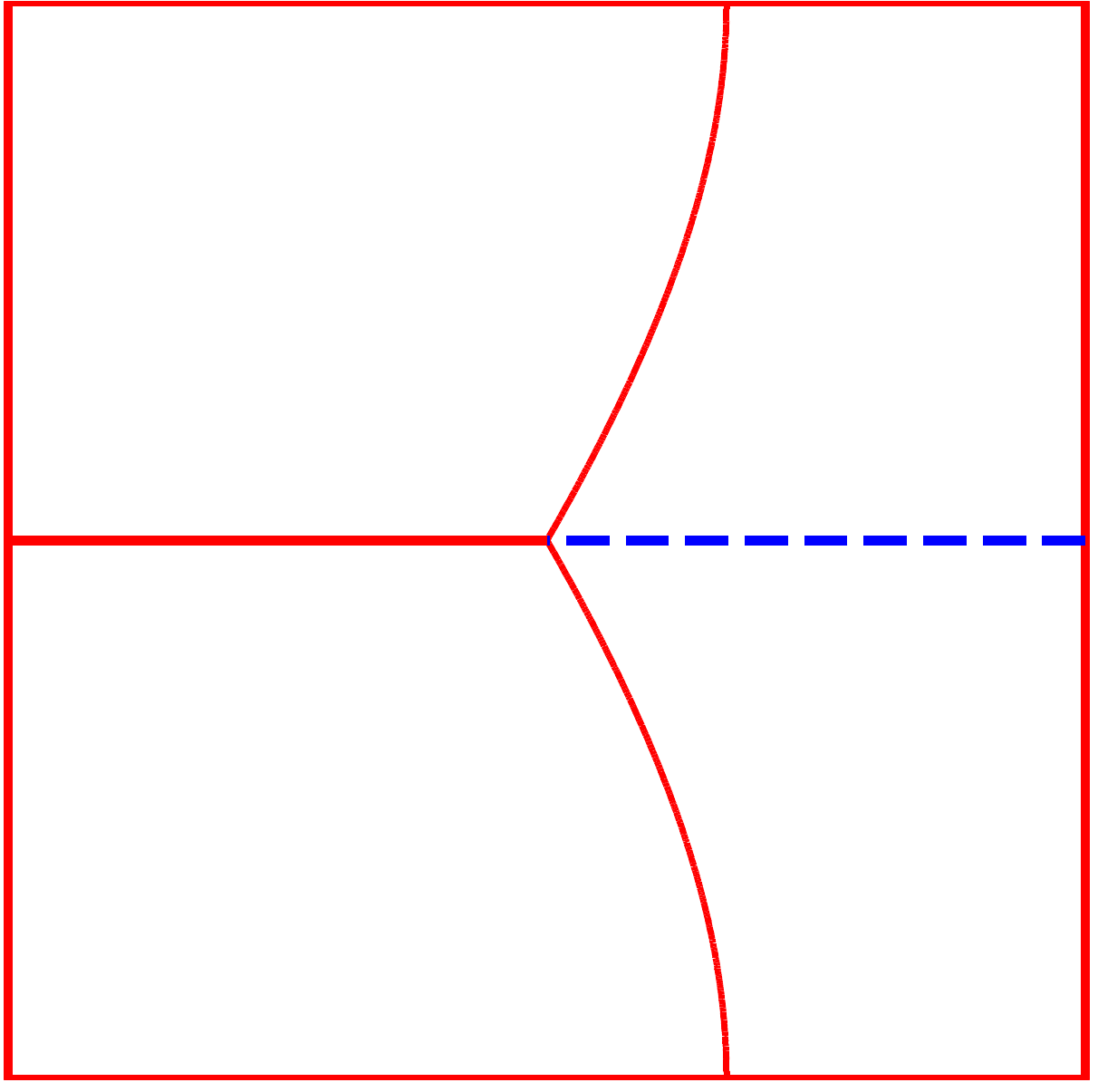}\
\includegraphics[width = 0.2\textwidth]{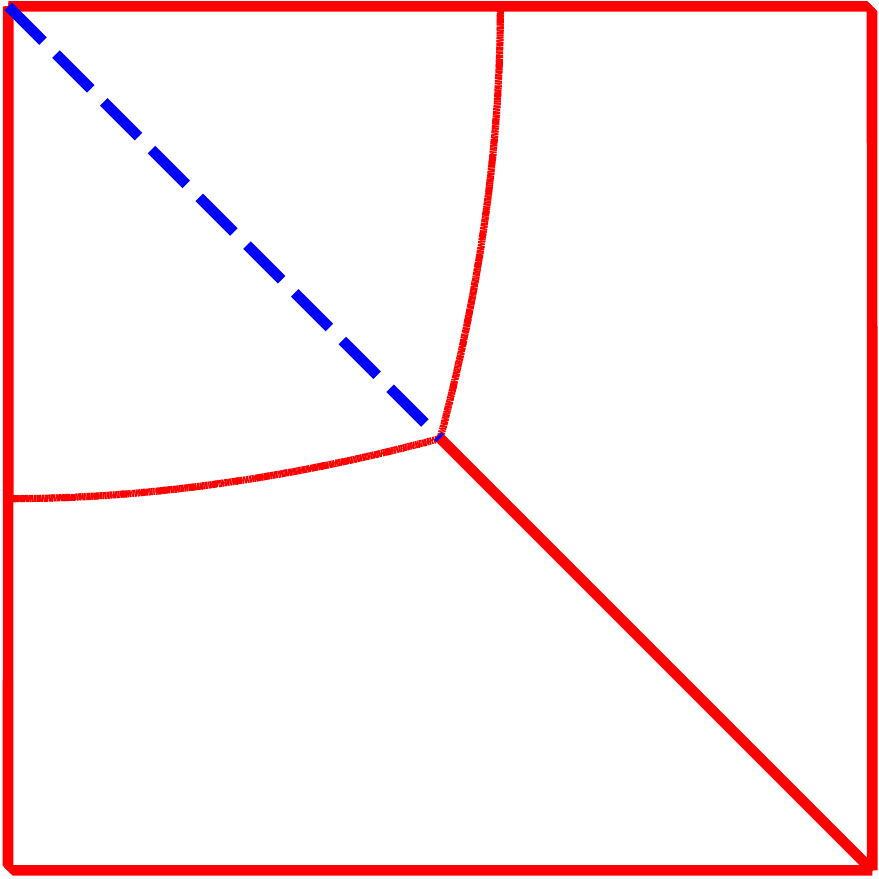}\quad}
\end{center}
\caption{Dirichlet-Neumann approach for $3$-partitions of the square.}
\label{square3-DN}
\end{figure}

In the case of the $5$-partition of the square we note that the partition obtained by the iterative algorithm seems to have the same axes of symmetry as the square. Due to the symmetry of the partition one can consider a mixed Dirichlet-Neumann problem on an eighth of the square as seen in Figure \ref{square5-DNa}. The second Dirichlet-Neumann eigenfunction of this configuration has nodal domains which extend by symmetry to a $5$-partition of the square. The  second eigenvalue of this mixed configuration is equal to the first Dirichlet eigenvalue on each cell of the $5$-partition built after symmetrization (see Figure~\ref{square5-DNb}). This second Dirichlet-Neumann eigenvalue, equal to $104.294$, gives a upper bound for $\fL_{5,\infty}(\square)$ which is lower than the ones obtained with the iterative methods.  

\begin{figure}[h!]
\centering 
\subfigure[Mixed problem\label{square5-DNa}]{\includegraphics[width= 0.2\textwidth]{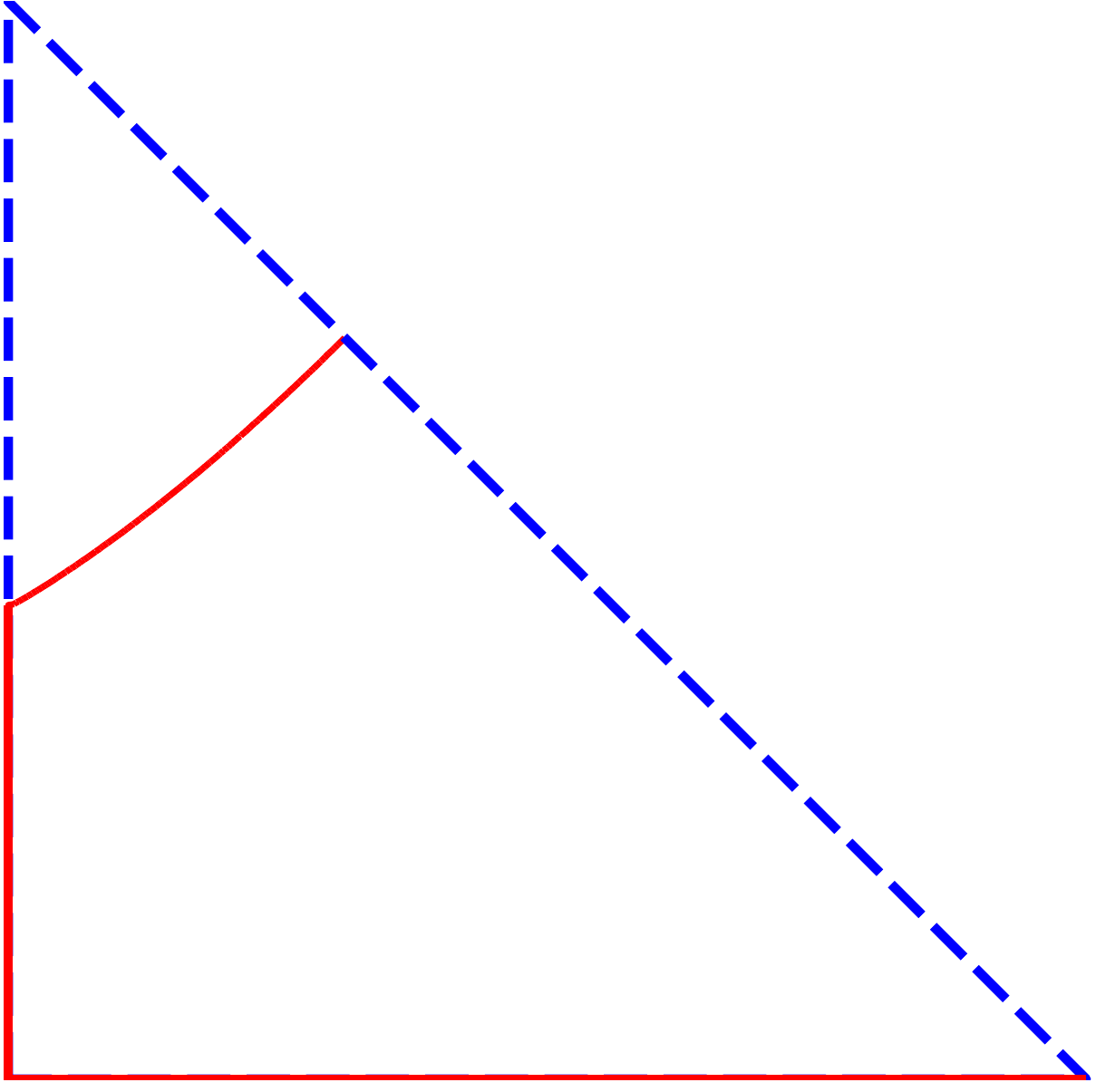}}
\qquad
\subfigure[Symmetrized $5$-partition\label{square5-DNb}]{\qquad\includegraphics[width= 0.2\textwidth]{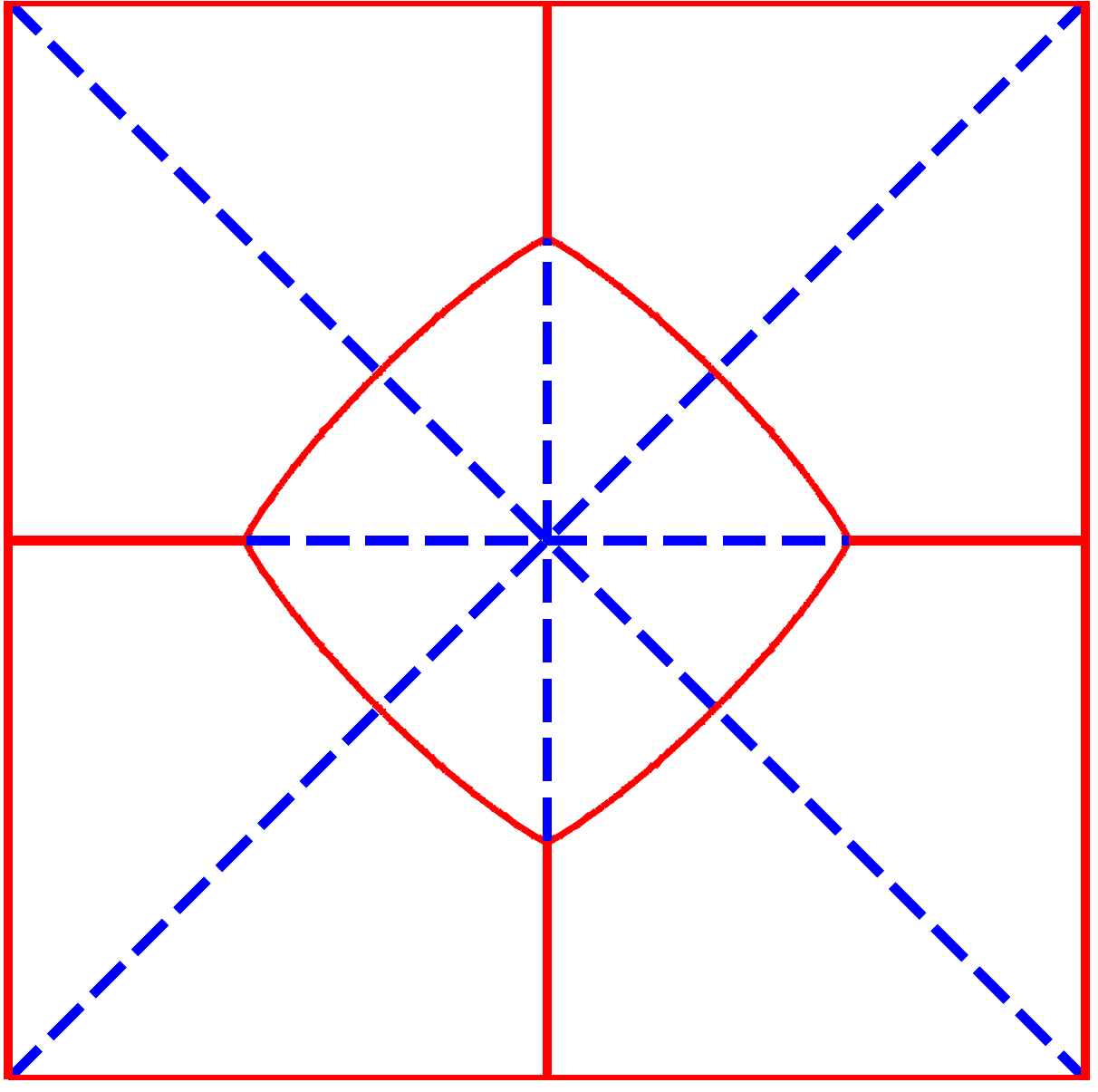}\qquad}
\caption{Dirichlet-Neumann approach for $5$-partitions of the square.} 
\label{square5-DN}
\end{figure}

The { partition} obtained with the iterative methods for $k=7$ admits two axes of symmetry and some parts of the boundaries of the partitions seem to be segments. We can thus formulate a mixed problem on the quarter of the square, denoted $\sA\sB\sC\sD$. Consider $\sX_t \in [\sA\sD]$ and $\sX_s$ inside the square such that $\angle \sX_s\sX_t\sD = 2\pi/3$. We solve the mixed problem with Dirichlet conditions on $[\sB\sC],[\sC\sD],[\sD\sX_t],[\sX_t\sX_s]$ and Neumann conditions on $[\sA\sB],[\sA\sX_t]$. A graphical representation of the configuration is given in Figure \ref{square7-DN}. We vary points $\sX_t \in [\sA\sD]$ and $\sX_s$ noting that the nodal line of the third eigenfunction of this mixed problem must touch $\sX_s$. The Dirichlet-Neumann eigenvalue is equal to $146.32$ which gives an upper bound for $\fL_{7,\infty}(\square)$ which is lower than the ones obtained with the iterative methods.

\begin{figure}[h!]
\centering 
\subfigure[Mixed problem\label{squareDN7}]{\includegraphics[height= 0.2\textwidth]{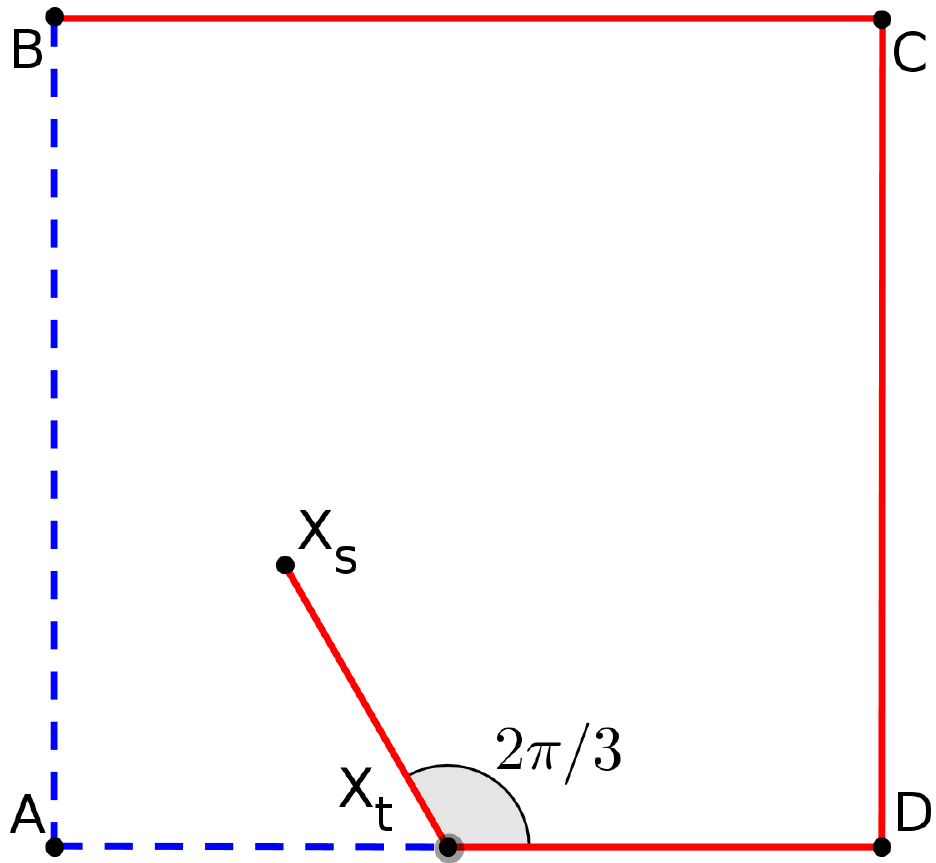}}
\qquad
\subfigure[Optimal nodal partition\label{square7-DNb}]{\qquad\includegraphics[width= 0.2\textwidth]{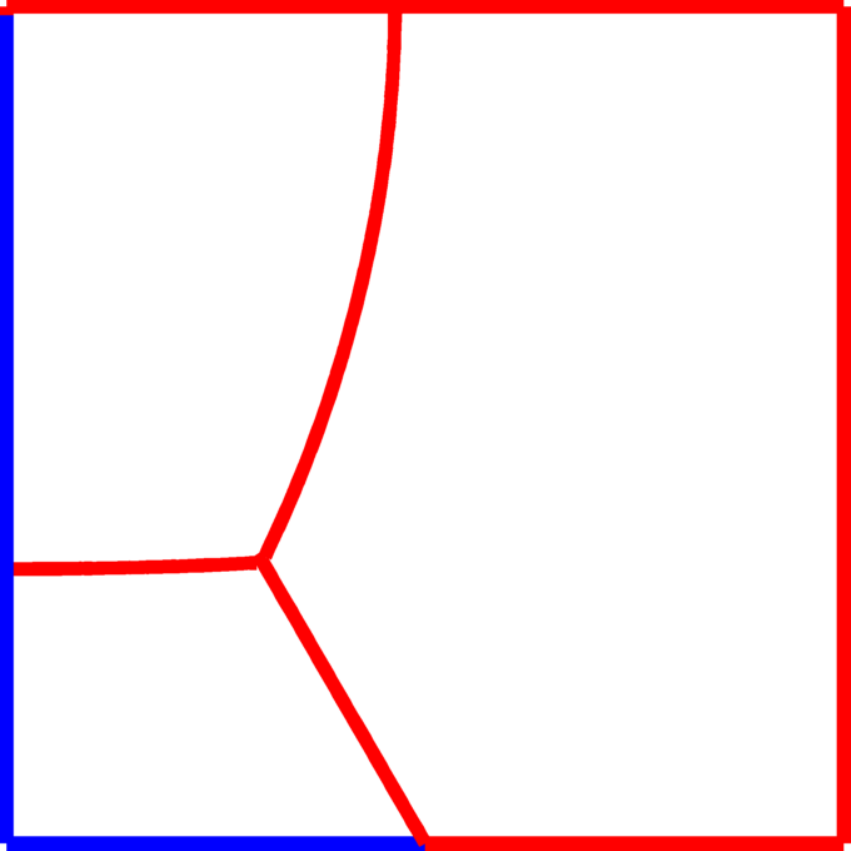}\qquad}
\quad
\subfigure[Symmetrized partition\label{square7-DNc}]{\qquad\includegraphics[width= 0.2\textwidth]{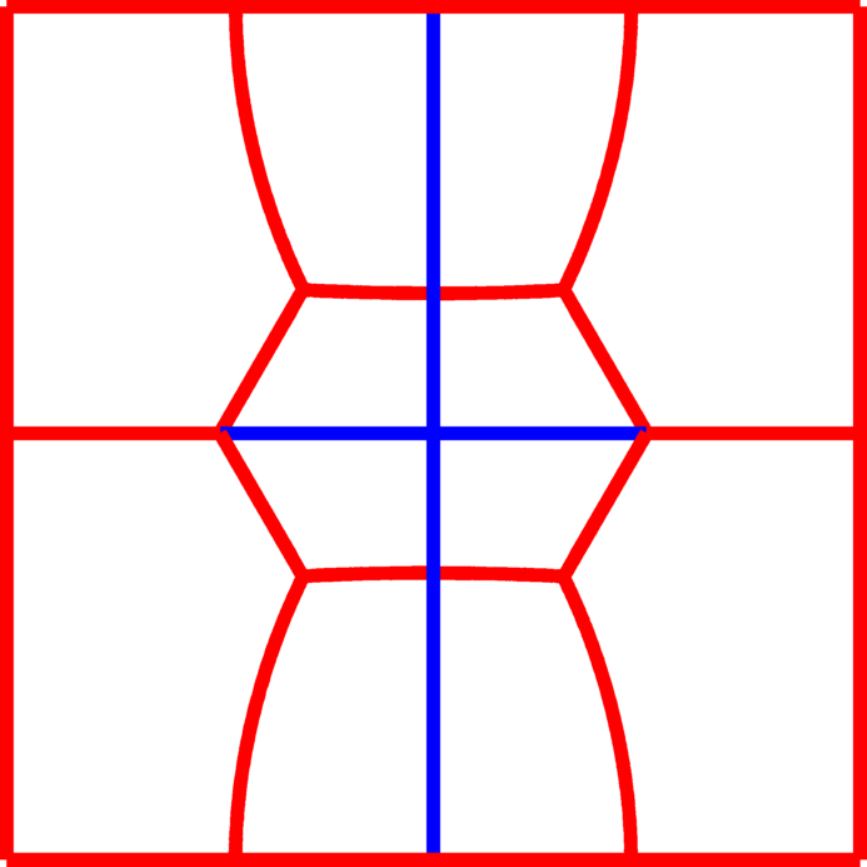}\qquad}
\caption{Dirichlet-Neumann approach for $7$-partitions of the square.} 
\label{square7-DN}
\end{figure}

In the case $k=8$ we also observe a {\Gn candidate} with two axes of symmetry and with some boundaries which seem to be segments. We formulate a mixed problem on a quarter of the square whose eigenfunction, after symmetrization has the desired structure. If we denote by $\sA\sB\sC\sD$ the quarter of the square, like in Figure \ref{square8-DN} we consider a variable point $\sX_t \in [\sA\sD]$ and another variable point $\sX_s$ inside $\sA\sB\sC\sD$ such that $\angle \sD\sX_t\sX_s = \pi/3$. We consider the mixed problem with Dirichlet boundary conditions on $[\sB\sC],[\sC\sD],[\sA\sX_t],[\sX_t\sX_s]$ and Neumann conditions on $[\sA\sB]$ and $[\sD\sX_t]$. We then vary the position of the points $\sX_s,\sX_t$ with the above properties and we compute the third eigenfunction of the mixed problem. The minimal value of the corresponding eigenfunction is attained when the nodal line corresponding to the third eigenfunction passes through $\sX_s$. In this case the Dirichlet-Neumann eigenvalue is equal to $160.87$ which gives a  better upper bound for $\fL_{8,\infty}(\square)$ than previsously.

\begin{figure}[h!]
\centering 
\subfigure[Mixed problem\label{square8-DNa}]{\includegraphics[height= 0.2\textwidth]{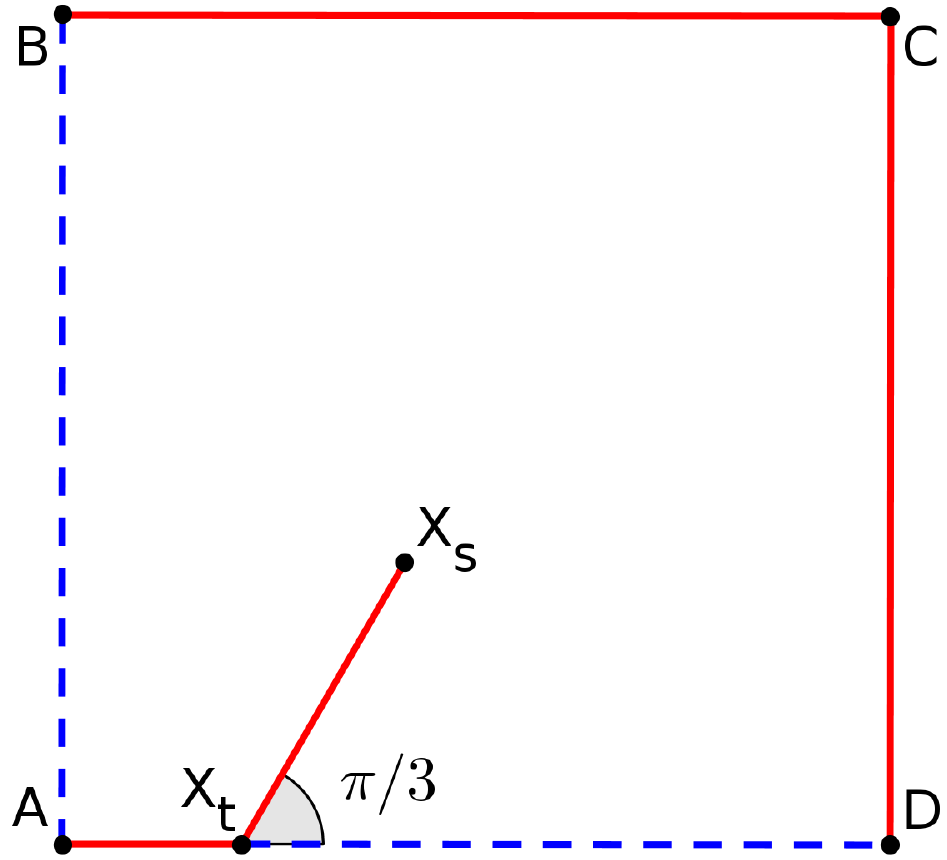}}
\qquad
\subfigure[Optimal nodal partition\label{square8-DNb}]{\qquad\includegraphics[width= 0.2\textwidth]{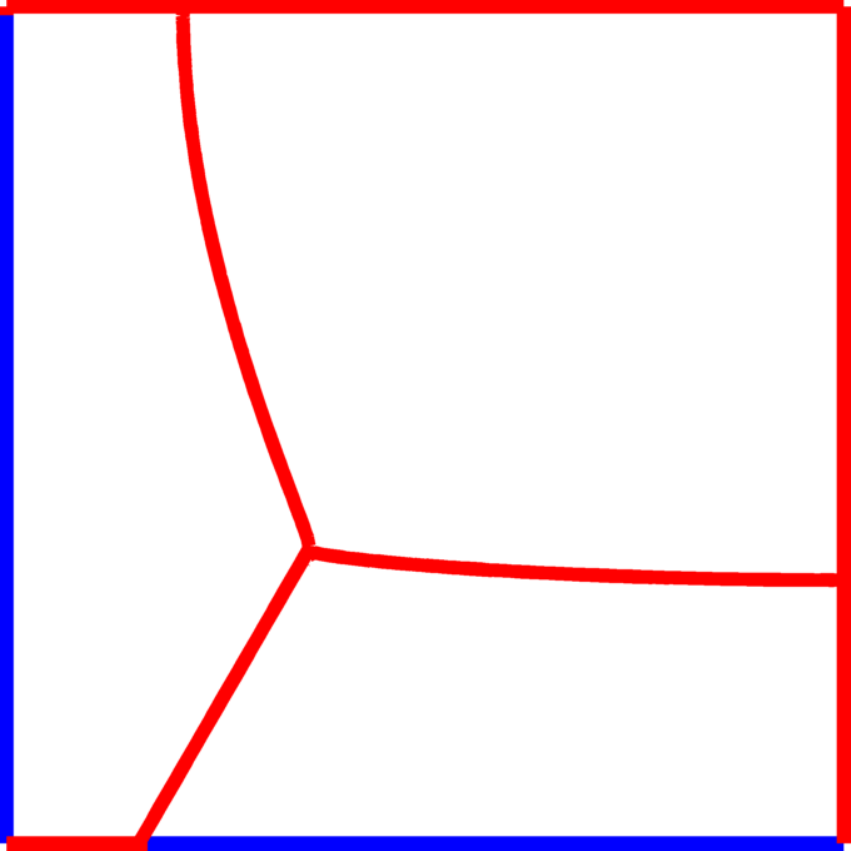}\qquad}
\quad
\subfigure[Symmetrized partition\label{square8-DNc}]{\qquad\includegraphics[width= 0.2\textwidth]{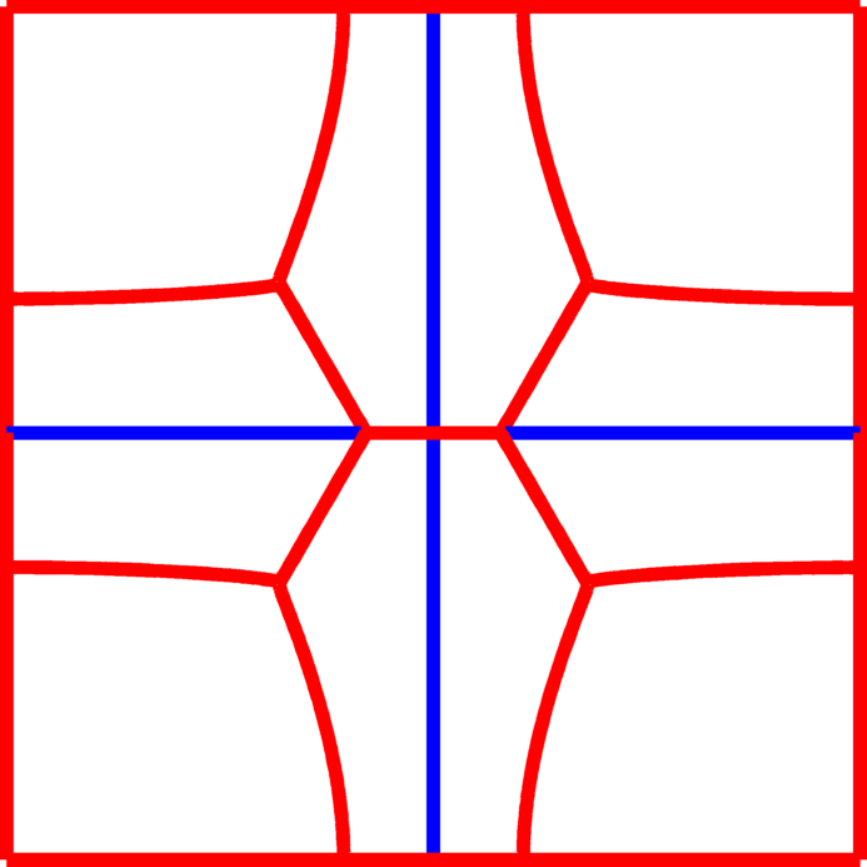}\qquad}
\caption{Dirichlet-Neumann approach for $8$-partitions of the square.} 
\label{square8-DN}
\end{figure}

\subsubsection*{\bf The Disk.} 
We know that the $\infty$-minimal $k$-partition consists in $k$ equal sectors when $k=2,4$. Numerically, it seems to be the same for $k\in\{3,5\}$ and some works tried to prove it when $k=3$ (see \cite{HelHof10, BoHe13}). 

For larger $k$ ($\Gn k \in \llbracket 6,9\rrbracket$), we observe that numerical partitions obtained with the iterative method consist of a structure which is invariant by a rotation of $2\pi/(k-1)$. This motivates us to use the Dirichlet-Neumann approach for the cases $\Gn k\in \llbracket 6,9\rrbracket$. Indeed, one can see that the invariance by a rotation of angle $2\pi/(k-1)$ allows us to represent exterior cells of the configurations as subsets of a sector of angle $2\pi/(k-1)$. This brings us to consider a mixed Dirichlet-Neumann problem on such sectors. If we consider the center of the disk at the origin, and we denote the sector by $\sO\arc{\sA\sB}$ then for $r \in (0,1)$ we consider the points $\sA_r \in [\sO\sA]$ and $\sB_r \in [\sO\sB]$ with $\sA_r\sO=\sB_r\sO = r$. We consider Neumann boundary conditions on $[\sO\sA_r]$, $[\sO\sB_r]$ and Dirichlet condition on $[\sA_r\sA],[\sB_r\sB]$ and the arc $\sA\sB$. Figure \ref{conf_sect} illustrates this mixed Dirichlet-Neumann configuration. 
\begin{figure}[h!]
\centering 
\includegraphics[width=0.3\textwidth]{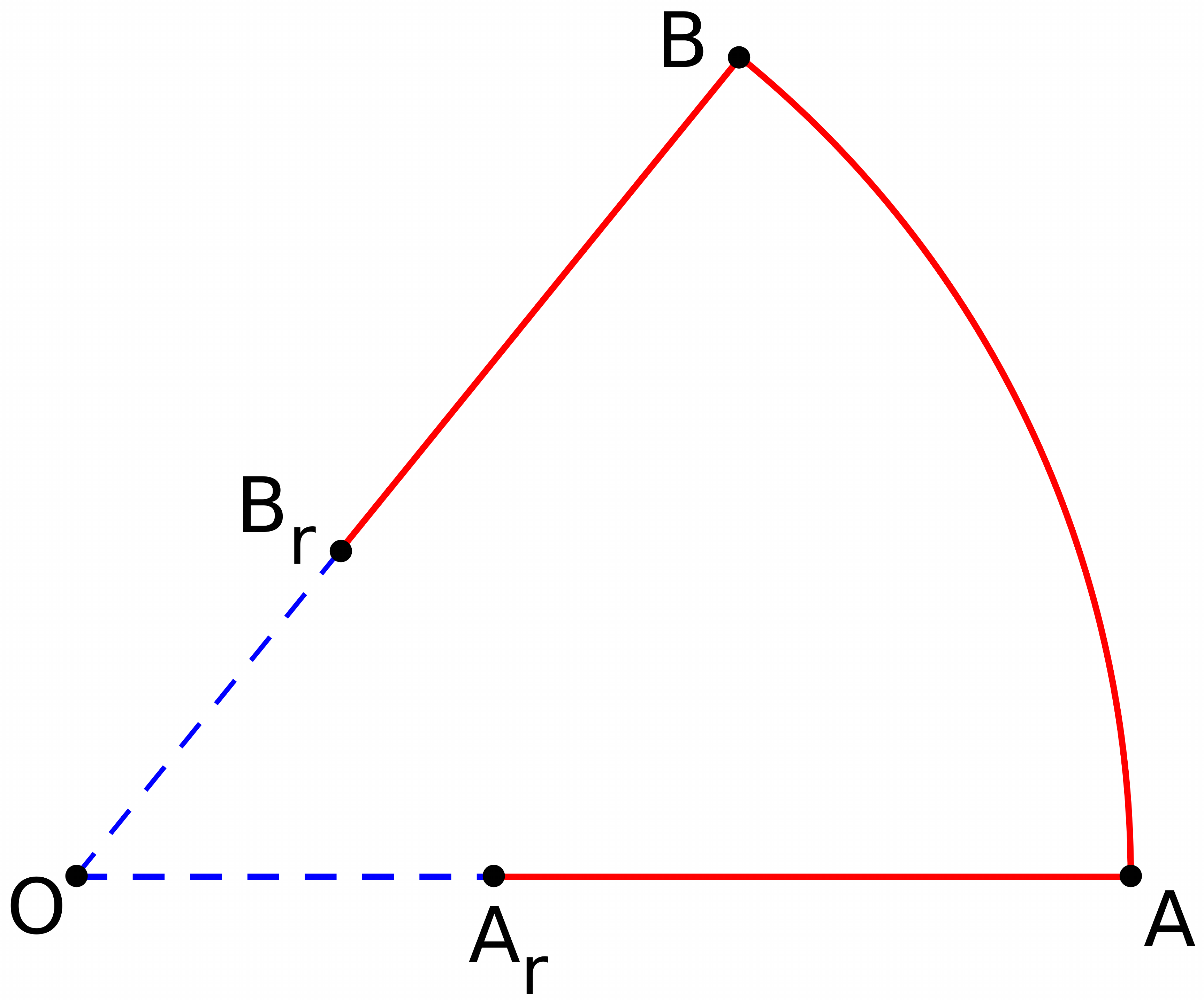}
\caption{The setup for the mixed problem on sectors.}
\label{conf_sect}
\end{figure}
Next we vary $r$ in $(0,1)$ and we record the position where the nodal line associated to the second eigenfunction of the Laplace operator with these mixed boundary conditions touches the segments $[\sA_r\sA],[\sB_r\sB]$. This is necessary in order to have  a $k$-partition after symmetrization. On the other hand we want the largest possible $r$ in order to obtain minimal eigenvalues in the symmetrized partition of the disk (since the eigenvalue of the mixed problem is decreasing when $r$ is increasing). Thus, for each $\Gn k \in \llbracket 6,9\rrbracket$ we consider the above mixed problem in the sector of angle $2\pi/(k-1)$ and we search in each case for the optimal value of $r$. The second eigenvalue of the mixed problem equals the first eigenvalue of each domain of the partition obtained by the symmetrization of this eigenvalue to the whole disk. The values obtained are recorded in Table \ref{synthese-max} and the partitions are given in Figure~\ref{fig.SectDN}. We note that for $k=10$ the same approach in a sector of angle $2\pi/9$ gives a candidate which has a larger energy than the partition obtained with the iterative algorithm {(but the topology of the 2 partitions are very different since there are two domains at the center in the configuration obtained by the iterative method)}.
\begin{figure}[h!]
\centering
\includegraphics[width=0.19\textwidth]{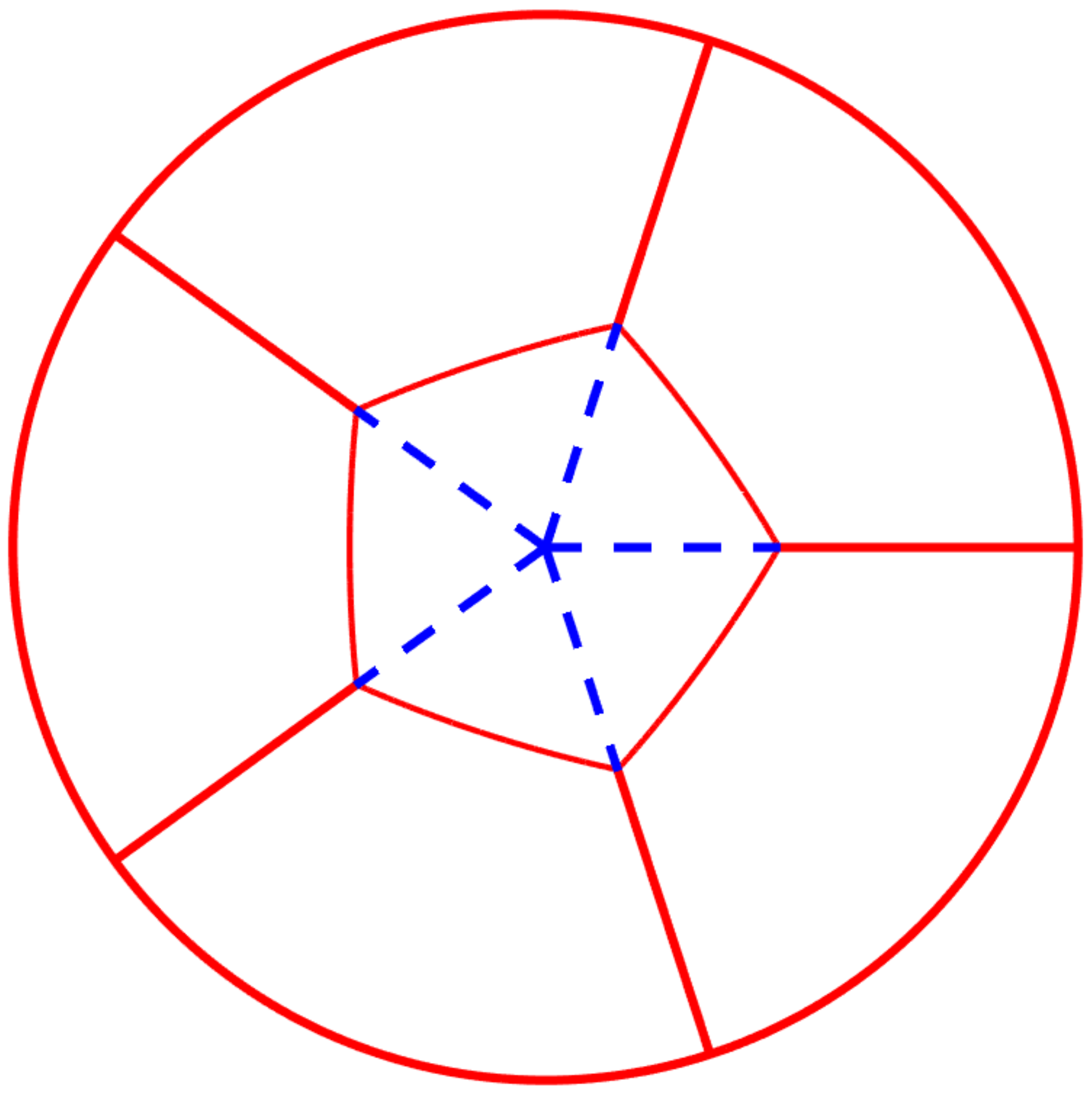}~
\includegraphics[width=0.19\textwidth]{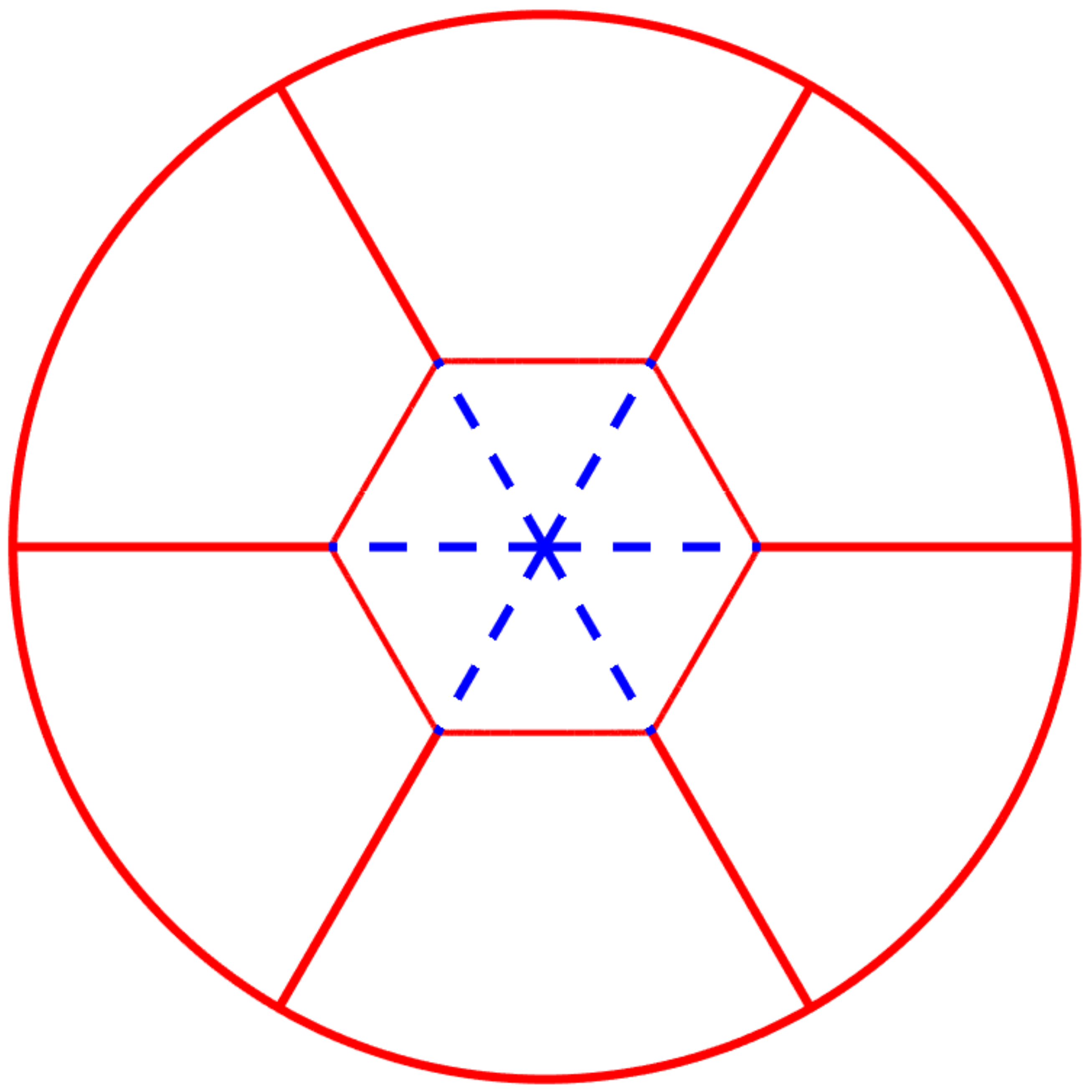}~
\includegraphics[width=0.19\textwidth]{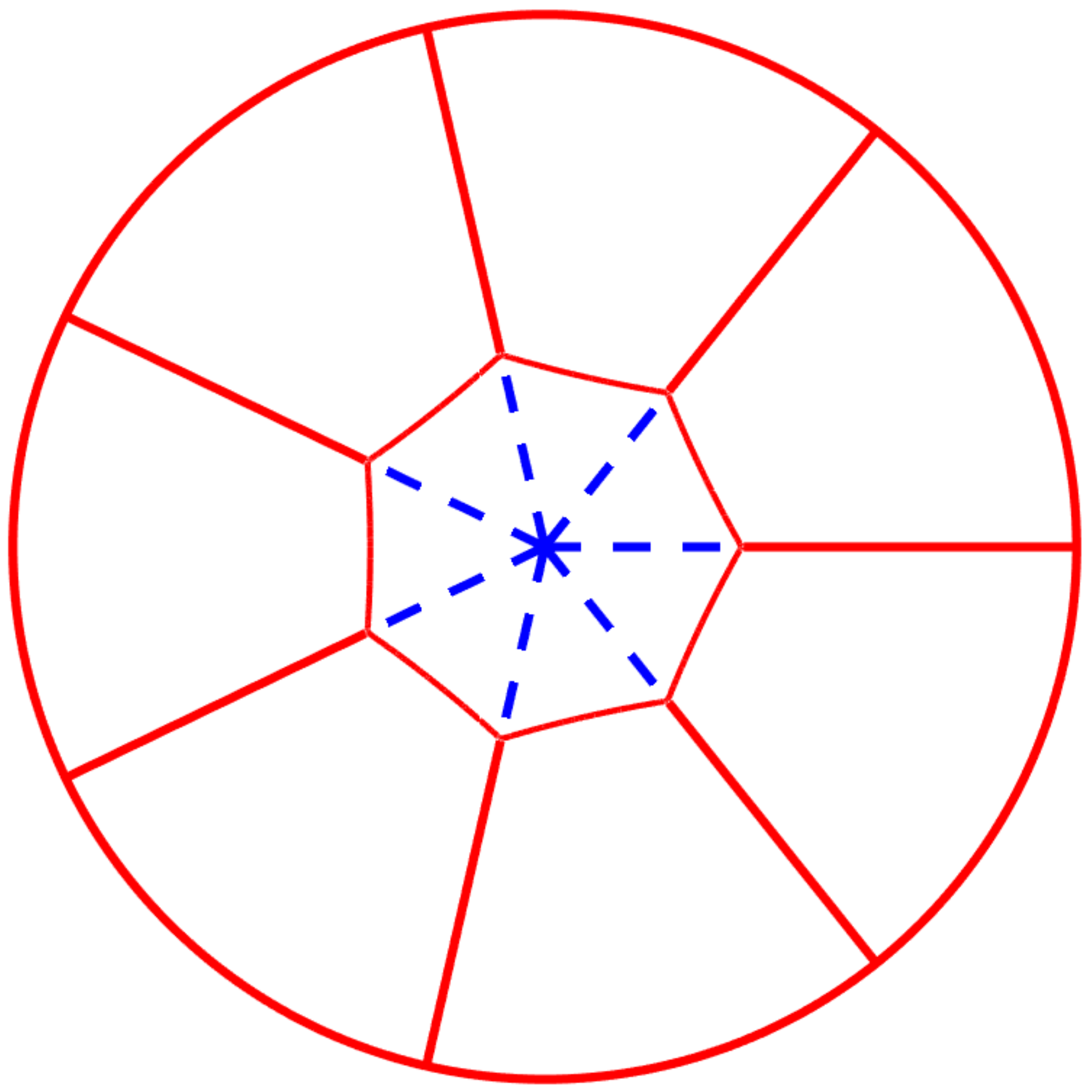}~
\includegraphics[width=0.19\textwidth]{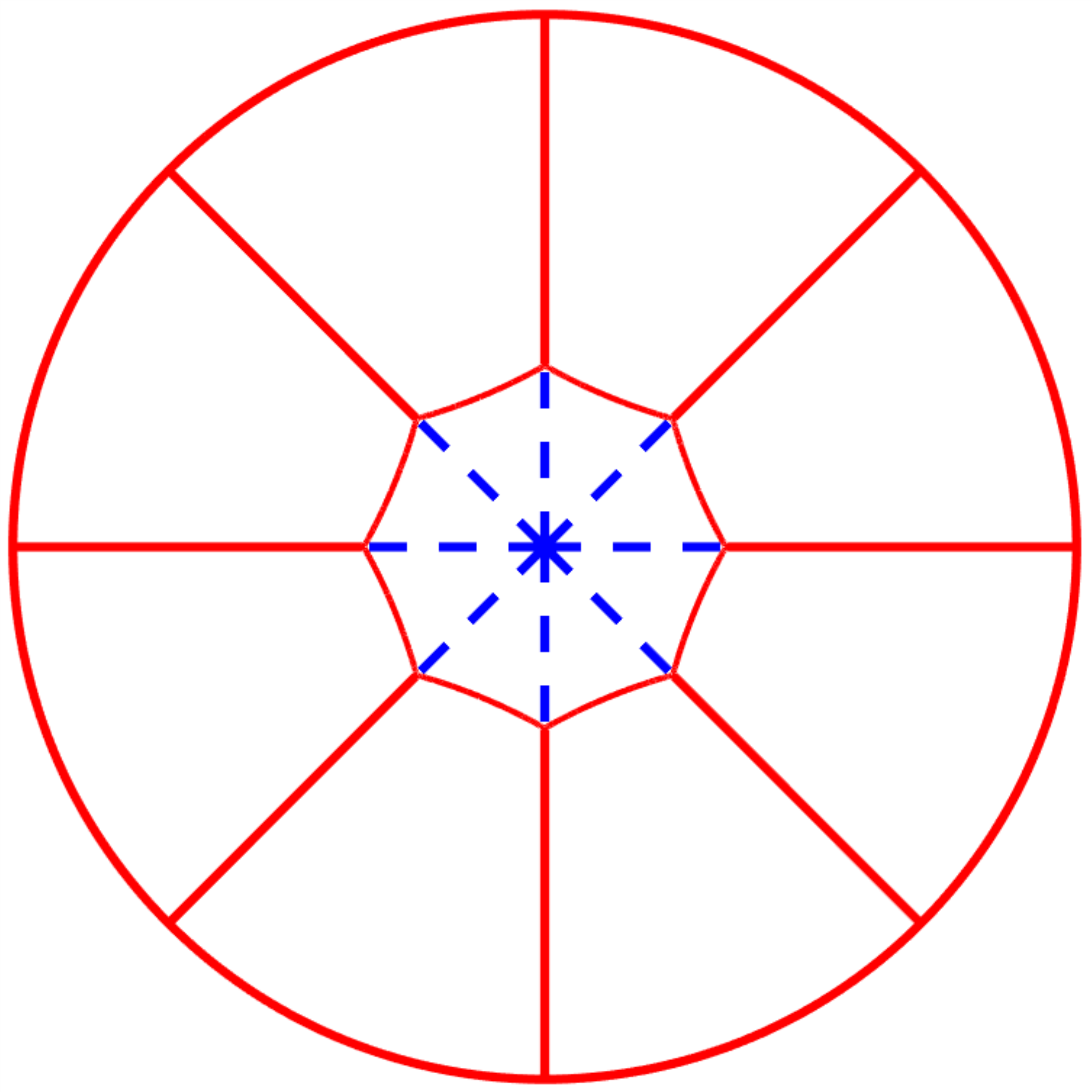}~
\includegraphics[width=0.19\textwidth]{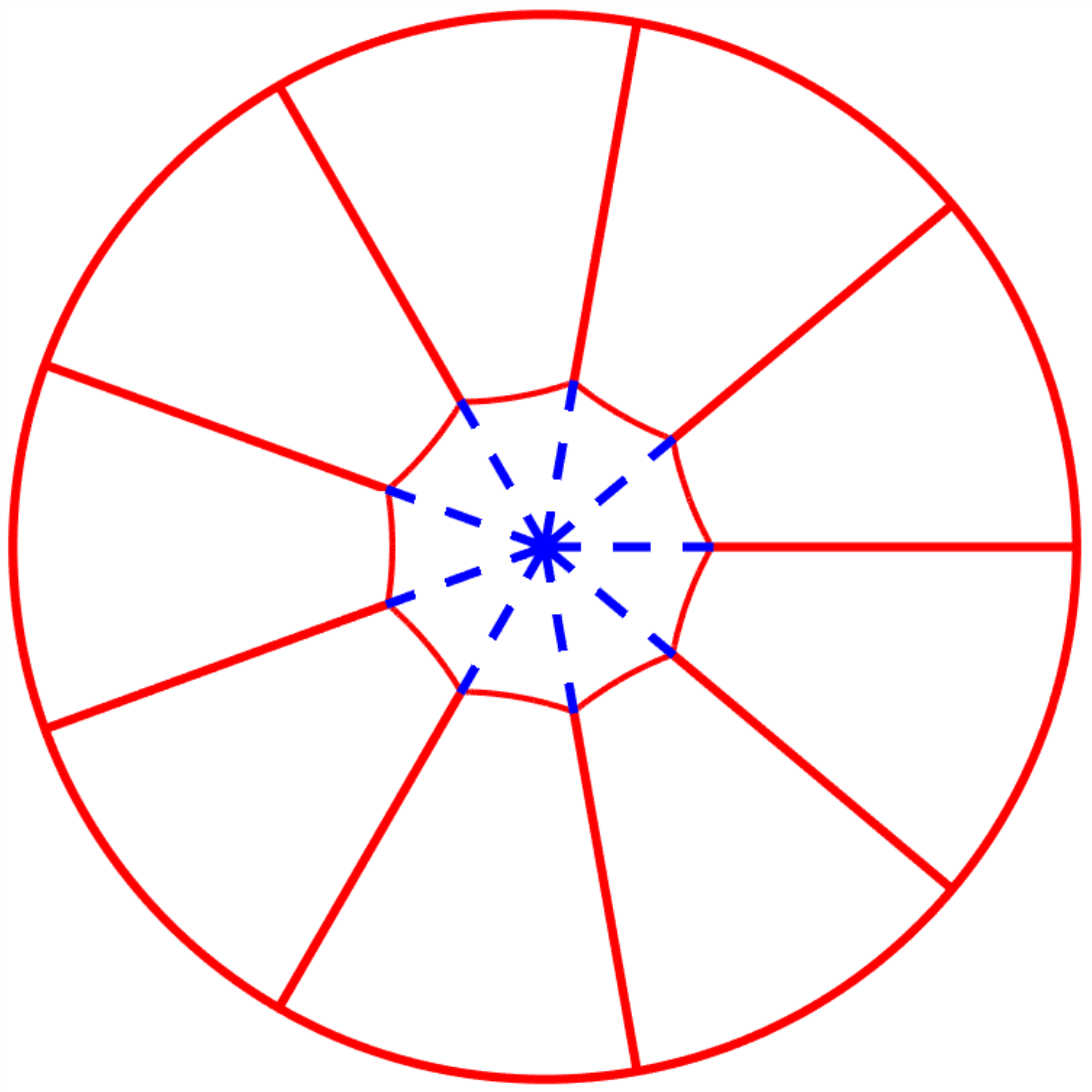}~
\caption{Candidates for the $\infty$-minimal $k$-partitions on the disk with the Dirichlet-Neumann approach, $6\leq k\leq 10$.}
\label{fig.SectDN}
\end{figure}

\subsubsection*{\bf Equilateral triangle.} In this case we also have some configurations where we can apply the Dirichlet-Neumann method. In the cases $k=3, 6, 10$ the partitions obtained by the iterative algorithm have the three axes of symmetry of the equilateral triangle. This allows us to reduce the problem to the study of mixed problems on a half or a sixth of the equilateral triangle. We also observe a possible application of the method to the case $k=5$ where we may consider Dirichlet boundary condition on part of the height of the triangle. The case $k=8$ also lets us use a mixed problem with Dirichlet boundary condition on part of the height and a vertical mobile segment.
  
We start with $k=3$ where the optimal candidate seems to be made of three congruent quadrilaterals with a common vertex at the centroid and each one having a pair of sides orthogonal to the sides of the triangle. Note that a brief idea of the method was described in Figure \ref{equi3-DN}. We consider a mixed Dirichlet Neumann problem on half of the equilateral triangle. Let $\sA\sB\sD$ be half of the equilateral triangle, where $\sA\sD$ is one of the heights of the triangle (see Figure \ref{equiDN3-details}). We consider a mobile point $\sD_r$ on the segment $[\sA\sD]$ and we compute the second eigenvalue of the Dirichlet Laplace operator with Dirichlet boundary conditions on segments $[\sD_r\sD],[\sD\sB],[\sA\sB]$ and Neumann conditions on $[\sA\sD_r]$. The choice of the Dirichlet boundary condition on $[\sD_r\sD]$ was motivated by the structure of the result in the iterative algorithm. We may ask what happens when we interchange the boundary condition on the height $[\sA\sD]$, {\it i.e.} considering Dirichlet boundary condition on $[\sA\sD_r]$ and Neumann boundary condition on $[\sD_r\sD]$. This is discussed at the end of this section in Remark \ref{other_DN}. Next we vary the position of $\sD_r$ on $[\sA\sD]$ so that the nodal line of the second eigenvalue of the mixed problem touches $[\sD\sD_r]$ exactly at $\sD_r$. As expected the position where we obtain this configuration is for $\sD\sD_r = \sA\sD/3$ which means that the triple point of the symmetrized partition is the centroid of the equilateral triangle.
\begin{figure}[h!]
\centering 
\subfigure[Mixed problem\label{equiDN3-detailsa}]{\qquad\includegraphics[height = 0.25\textwidth]{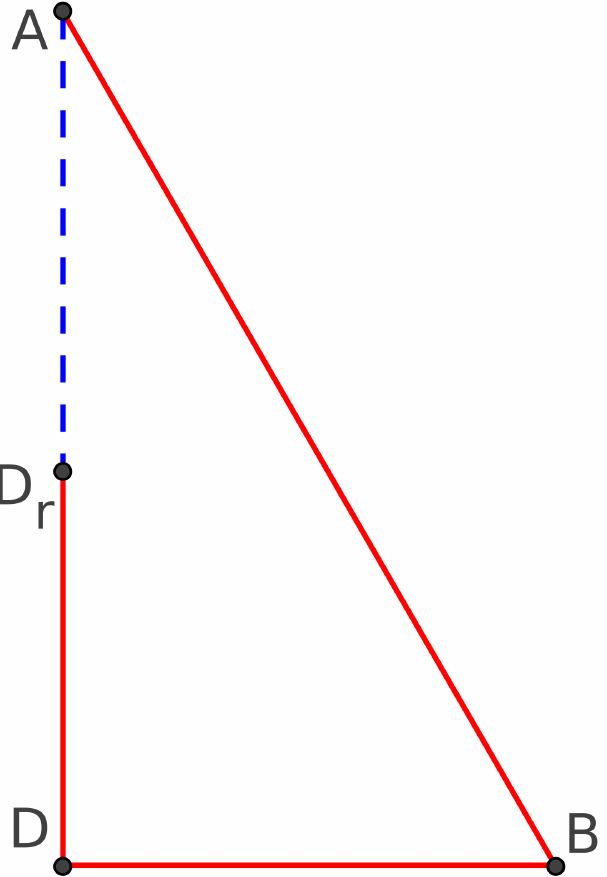}\qquad}
\subfigure[Optimal nodal partition]{\qquad\qquad\includegraphics[height =0.25\textwidth]{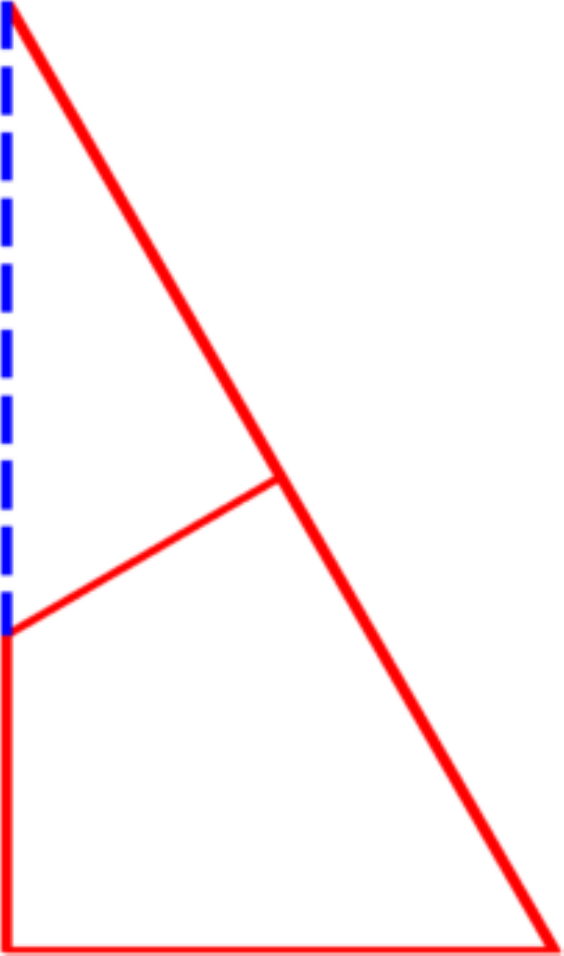}\qquad\qquad}
\subfigure[Symmetrized partition]{\qquad\includegraphics[height =0.25\textwidth]{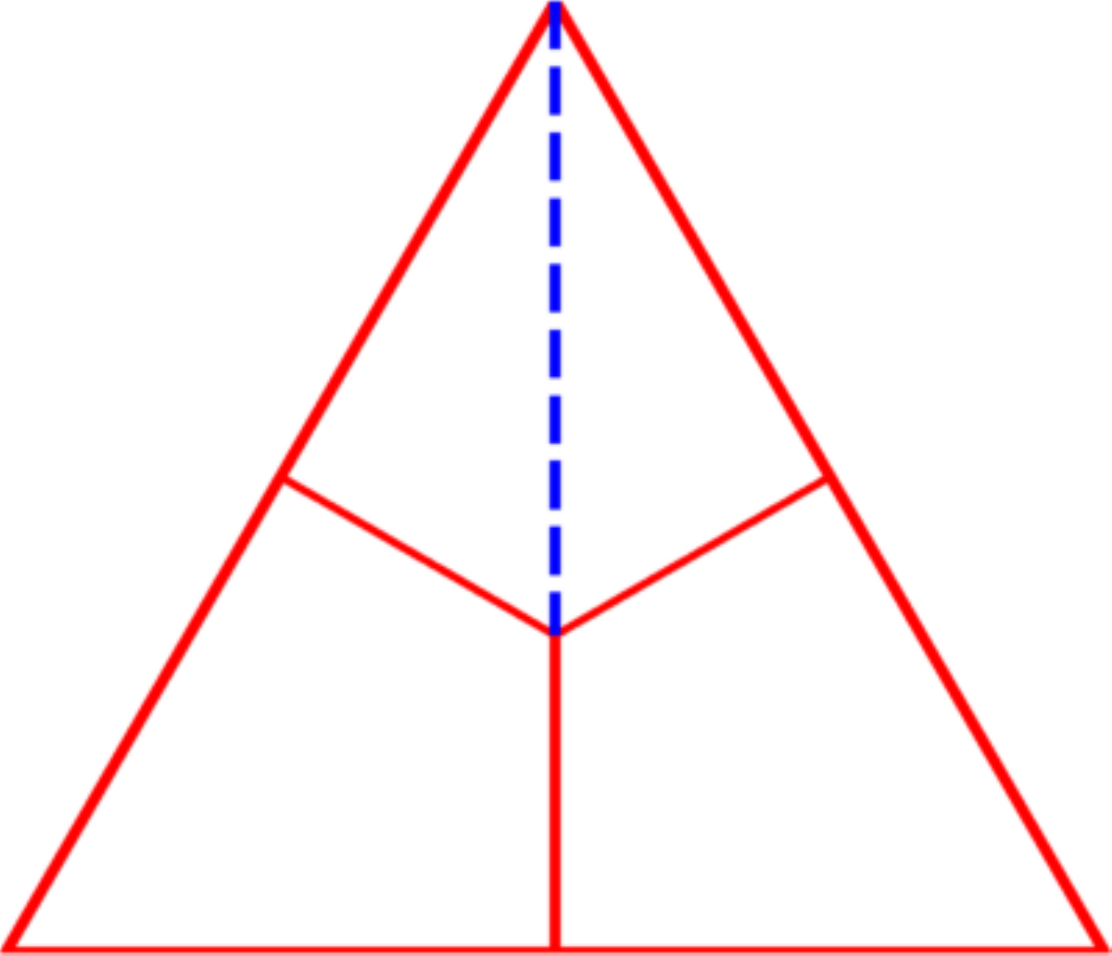}}
\caption{Dirichlet-Neumann approach for $3$-partitions of the equilateral triangle.}
\label{equiDN3-details}
\end{figure}

The case $k=5$ can be treated in the same framework, but instead of looking at the second eigenfunction of the mixed problem we study the third one. The result is presented in Figure \ref{equiDN5-details}. In optimal configuration, the triple point is such that $\sD\sD_{r} = \sA\sD/2$. 
\begin{figure}[h!]
\centering 
\subfigure[Mixed problem]{\qquad\includegraphics[height = 0.25\textwidth]{equiDN_3geo}\qquad}
\subfigure[Optimal nodal partition]{\qquad\qquad\includegraphics[height =0.25\textwidth]{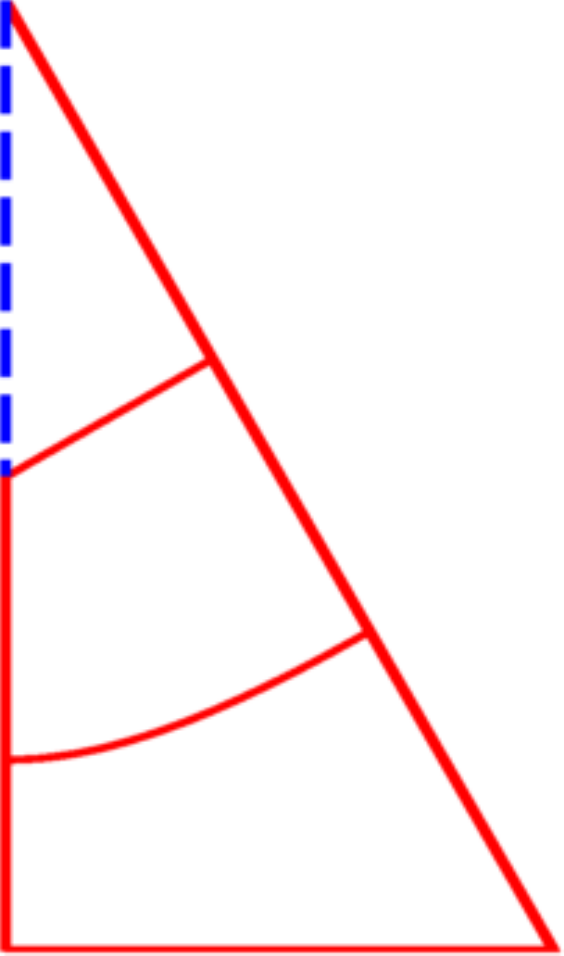}\qquad\qquad}
\subfigure[Symmetrized partition]{\qquad\includegraphics[height =0.25\textwidth]{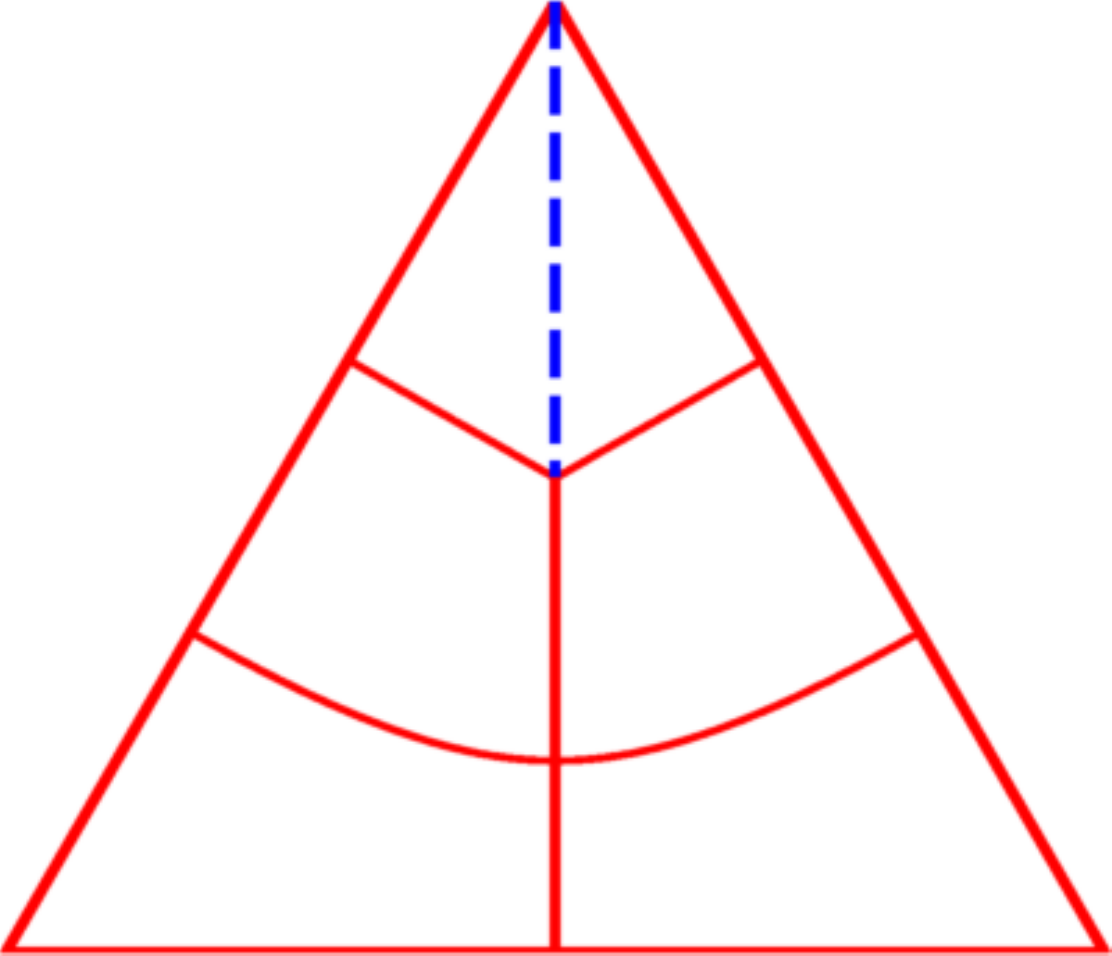}}
\caption{Dirichlet-Neumann approach for $5$-partitions of the equilateral triangle.}
\label{equiDN5-details}
\end{figure}

We continue with the case $k=8$ where we can also use a Dirichlet Neumann approach on half of the equilateral triangle. Here we observe that in addition to the axis of symmetry, one of the common boundaries between the cells also seems to be a vertical segment. We use this fact to define a mixed eigenvalue problem with four parameters on half of the equilateral triangle. Like in Figure \ref{equiDN8-details} we consider four variable points defined as follows. We consider the triangle $\sA\sB\sD$ where $\sA\sD$ is a height of the equilateral triangle. On the side $\sA\sD$ we consider two variable points $\sX_s,\sX_t$. On the segment $[\sX_s\sX_t]$ we put a Dirichlet boundary condition and on the segments $[\sA\sX_s],[\sD\sX_t]$ we have Neumann boundary conditions. We consider another variable point $\sY_r \in [\sB\sD]$ and we construct $\sY_q$ such that $\sY_q\sY_r \perp \sB\sD$ with the length of $[\sY_r\sY_q]$ as a variable. On the segment $[\sY_q\sY_r]$ we put a Dirichlet boundary condition. Of course, the remaining segments $[\sA\sB],[\sB\sD]$ also have a Dirichlet boundary conditions. We vary the position of these four points so that the fifth eigenfunction of the mixed problem has nodal lines which touch the Dirichlet parts at their extremities. The choice of the fifth eigenvalue is motivated by the fact that we need a nodal $5$-partition so that the symmetrized partition would have $8$ cells. The optimal configuration is shown in Figure \ref{equiDN8-details}. 

\begin{figure}[h!]
\centering 
\subfigure[Mixed problem]{\qquad\includegraphics[height = 0.25\textwidth]{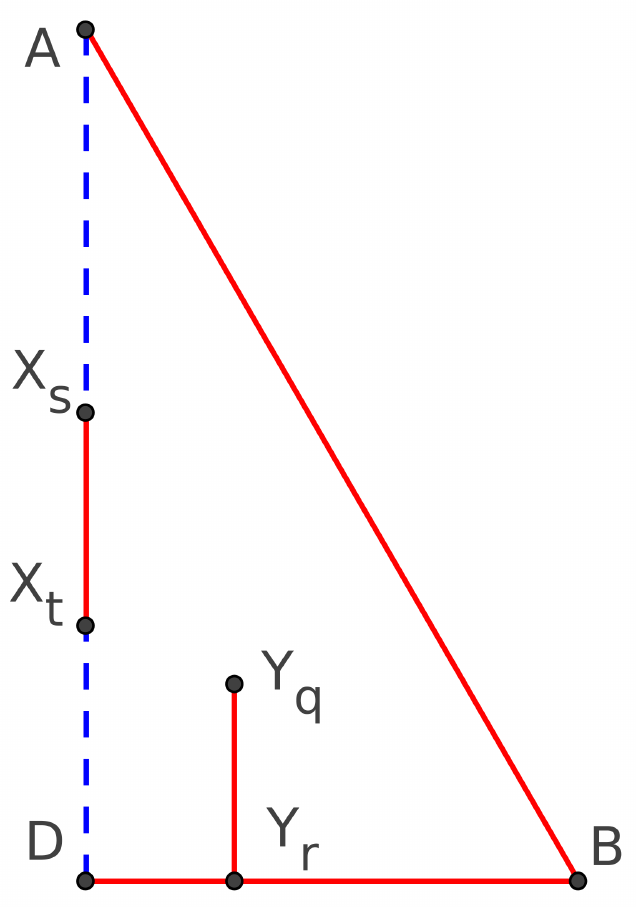}\qquad}
\subfigure[Optimal nodal partition]{\qquad\qquad\includegraphics[height =0.25\textwidth]{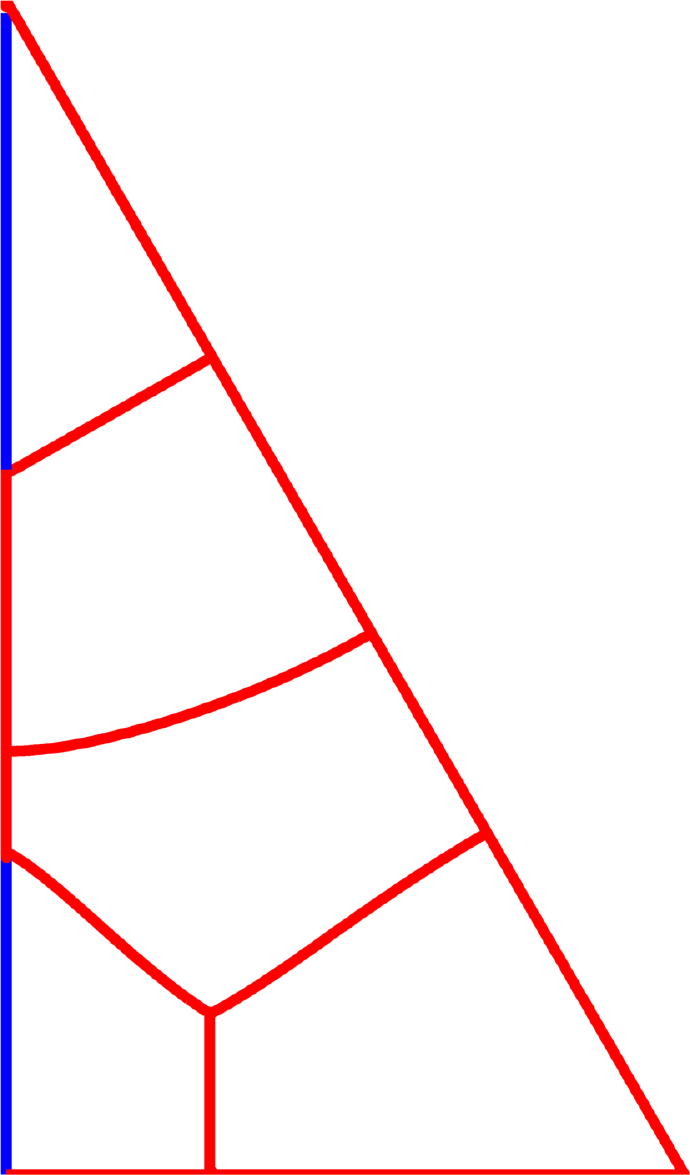}\qquad\qquad}
\subfigure[Symmetrized partition]{\qquad\includegraphics[height =0.25\textwidth]{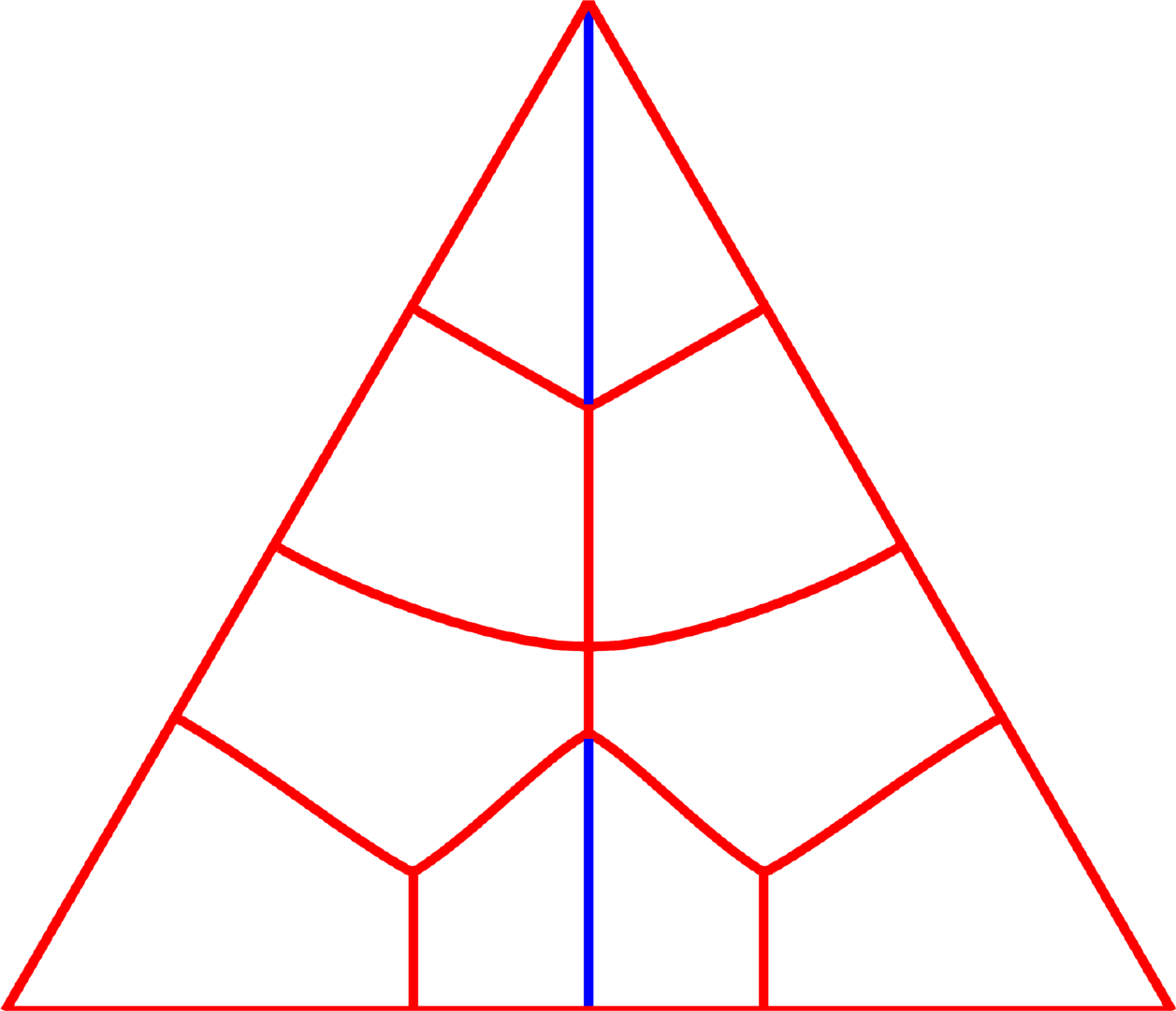}}
\caption{Dirichlet-Neumann approach for $8$-partitions of the equilateral triangle} 
\label{equiDN8-details}
\end{figure}

In the case $k=6$ the optimal partition obtained with the iterative algorithm has three axes of symmetry. Using this we can reduce the problem to the study of a mixed problem on one sixth of the equilateral triangle, {\it i.e.} a subtriangle defined by a vertex, the feet of a height and the centroid of the triangle. As in Figure \ref{equiDN6-detailsa} we consider the triangle defined by a vertex $\sA$, the feet of an altitude $\sD$ and the centroid $\sC$. On the side $\sA\sC$ we consider a mobile point $\sX_r = r\sA+(1-r)\sC$ for $r \in (0,1)$. We note that the candidate obtained with the iterative algorithm seems to correspond to a mixed problem on the triangle $\sA\sC\sD$ with Dirichlet boundary conditions on segments $[\sA\sD]$ and $[\sC\sX_r]$ and Neumann boundary conditions on $[\sC\sD]$ and $[\sA\sX_r]$. We search for the position of $\sX_r$ such that the nodal line of the second eigenfunction touches the segment $[\sC\sX_r]$ precisely at $\sX_r$ (see Figure \ref{equiDN6-detailsb}). The optimal nodal configuration and the partition obtained by performing symmetrizations is represented in Figure \ref{equiDN6-detailsc}. 

\begin{figure}[h!]
\centering 
\subfigure[Mixed problem\label{equiDN6-detailsa}]{\includegraphics[height = 0.15\textwidth]{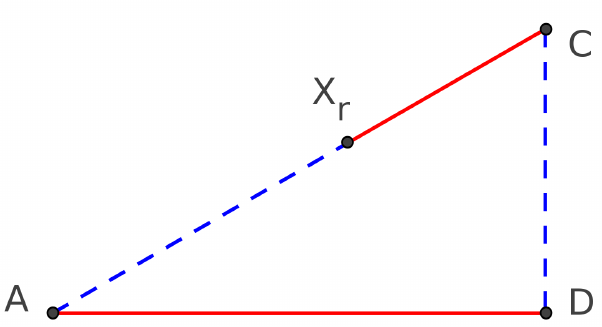}}\qquad
\subfigure[Optimal nodal partition\label{equiDN6-detailsb}]{\qquad\includegraphics[height =0.15\textwidth]{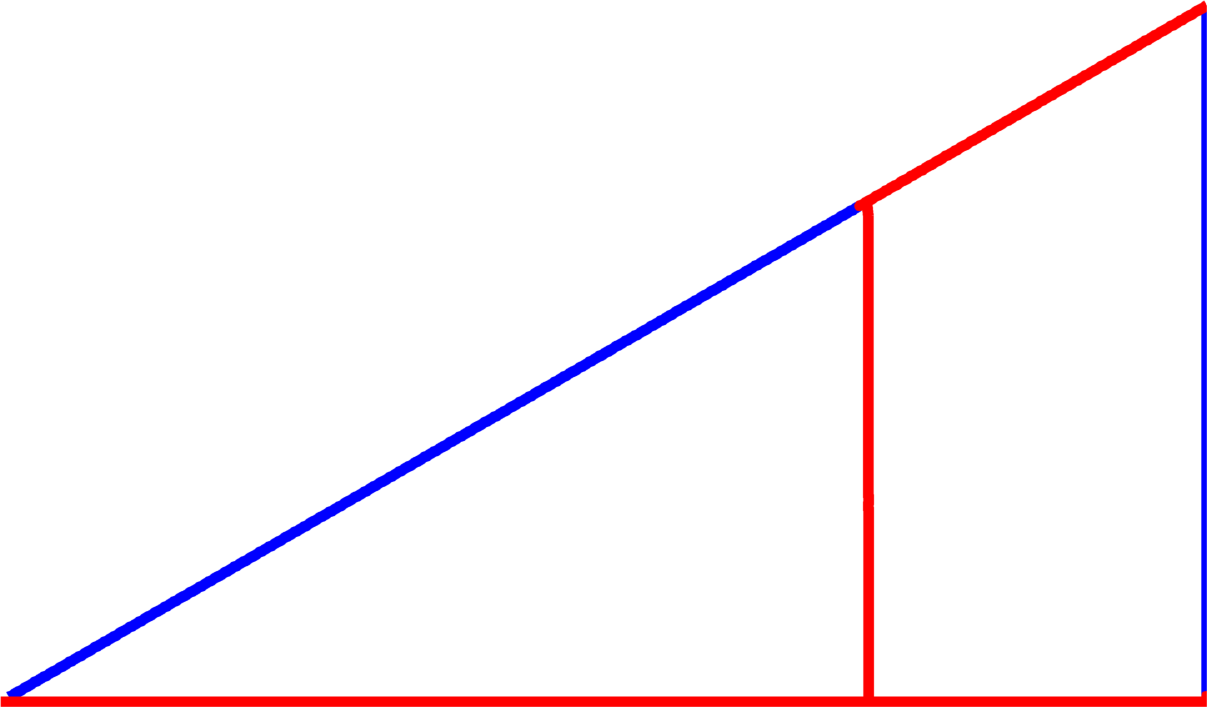}\qquad}\qquad
\subfigure[Symmetrized partition\label{equiDN6-detailsc}]{\includegraphics[height =0.22\textwidth]{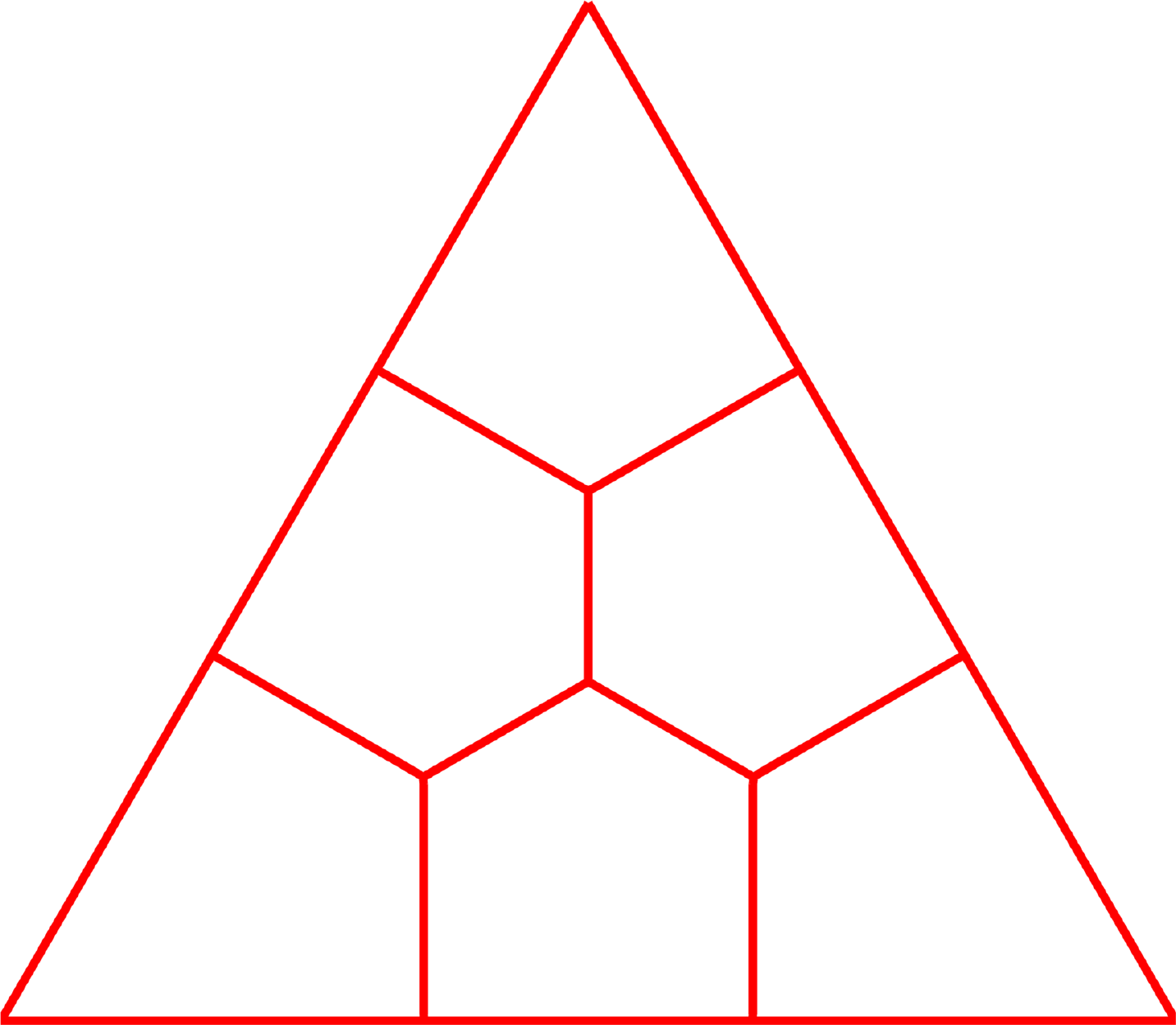}}
\caption{Dirichlet-Neumann approach for $6$-partitions of the equilateral triangle.}
\label{equiDN6-details}
\end{figure}

In the case $k=10$ we observe that the partition has again three axes of symmetry and we may try to represent it as a mixed Dirichlet-Neumann problem on a sixth of the equilateral triangle. Consider the same starting triangle $\sA\sC\sD$ like for $k=6$. Pick the variable points $X_r = (1-r)A+rD$ on $[AD]$ and $X_s = sC+(1-s)D$. Construct $Y_r \in [AC]$ such that $\sX_r\sY_r\perp \sA\sD$ and $\sY_s \in [\sA\sC]$ such that $\widehat{\sX_s\sC\sY_s} = \pi/3$ (to satisfy the equal angle property). If we pick the origin at $\sA$ and $\sD$ of coordinates $(0.5,0)$ then we obtain the following coordinates for all the above defined points: $\sC(0.5,\sqrt{3}/6)$, $\sX_r(r,0)$, $\sY_r(r,r\sqrt{3}/3)$, $\sX_s(0.5,s\sqrt{3}/6)$, $\sY_s(0.25+s\sqrt{3}/2,\sqrt{3}/12+s/2 )$. As in Figure \ref{equiDN10-detailsa} we take a Dirichlet boundary condition on segments $[\sA\sD],[\sD\sX_s],[\sY_s\sY_r]$ and Neumann boundary condition on segments $[\sA\sY_r],[\sC\sY_s],[\sC\sX_s]$. Since the numerical candidate in this case seems to have cells with polygonal borders we search the positions of $\sX_r$ and $\sX_s$ such that the nodal lines of the third eigenfunctions of the eigenvalue problem with mixed boundary conditions are exactly the segments $[\sX_r\sY_r]$ and $[\sX_s\sY_s]$. The result is shown in Figure \ref{equiDN10-details} together with the symmetrized partition.

\begin{figure}[h!]
\centering
\subfigure[Optimal nodal partition\label{equiDN10-detailsa}]{\includegraphics[height = 0.15\textwidth]{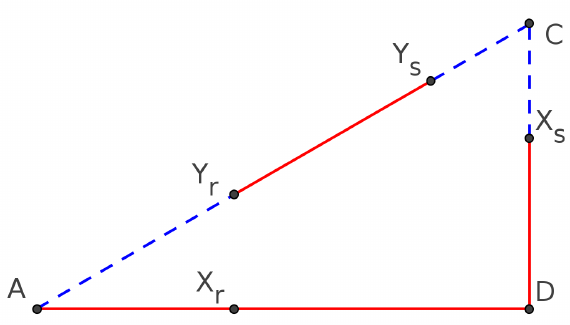}}\qquad
\subfigure[Optimal nodal partition\label{equiDN10-detailsb}]{\qquad\includegraphics[height =0.15\textwidth]{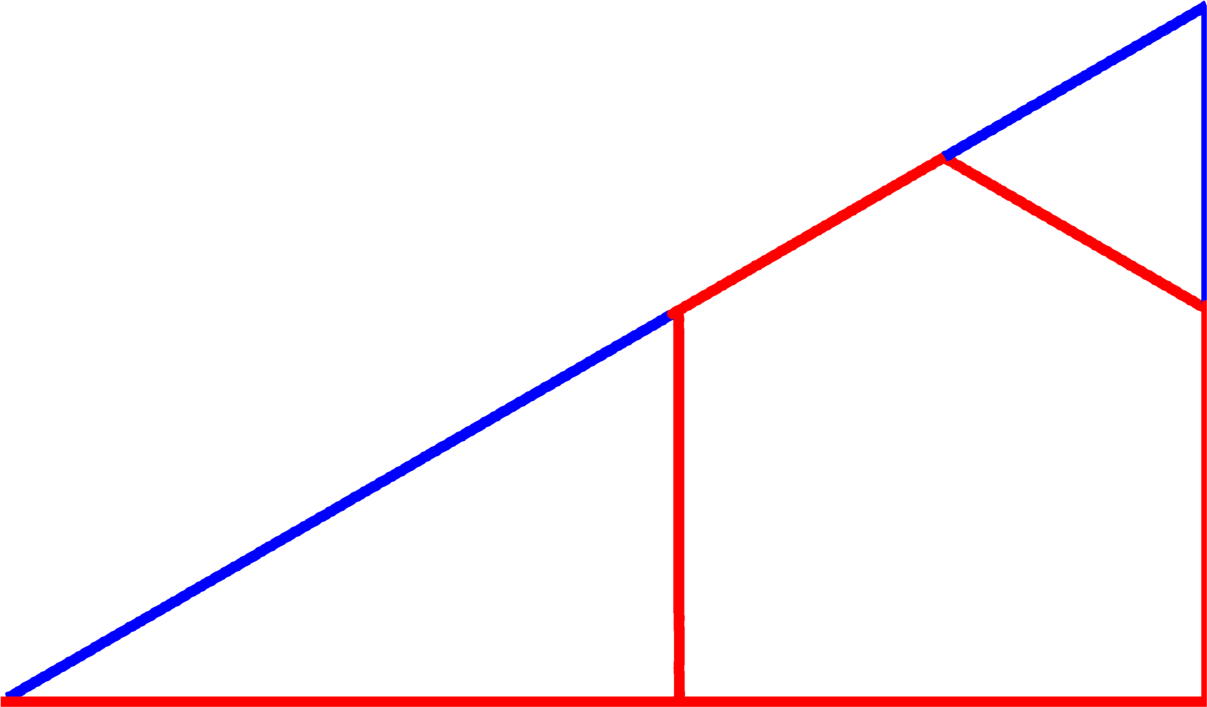}\qquad}\qquad
\subfigure[Symmetrized partition\label{equiDN10-detailsc}]{\includegraphics[height =0.22\textwidth]{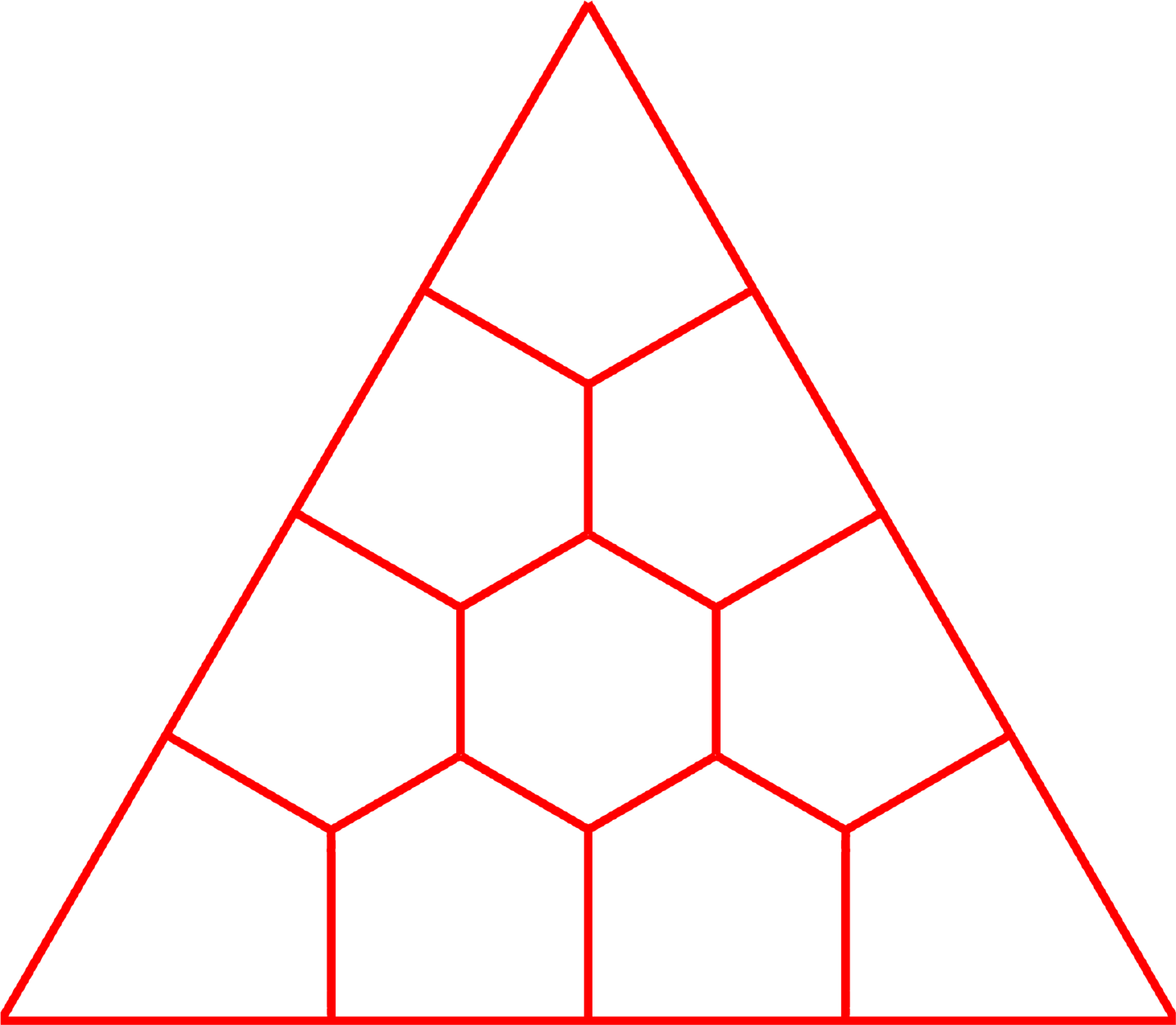}}
\caption{Dirichlet-Neumann approach for $10$-partitions of the equilateral triangle.}
\label{equiDN10-details}
\end{figure}

\begin{rem} In some cases we have chosen the Dirichlet and Neumann parts of the mixed problem based on the results given by the iterative method. We may ask what happens if we permute the two conditions. 

For the case $k=3$ on the equilateral triangle, if we consider Dirichlet boundary condition on segment $[AD_r]$ and Neumann boundary condition on $[\sD\sD_r]$ (see Figure \ref{equiDN3-detailsa} for the notations) then the optimal configuration is again when $DD_r = AD/3$, but the eigenvalues of the cells on the symmetrized domain are strictly higher than the one obtained before.

\begin{figure}[h!t]
\begin{center}
\begin{tabular}{ccc}
\includegraphics[height=3cm]{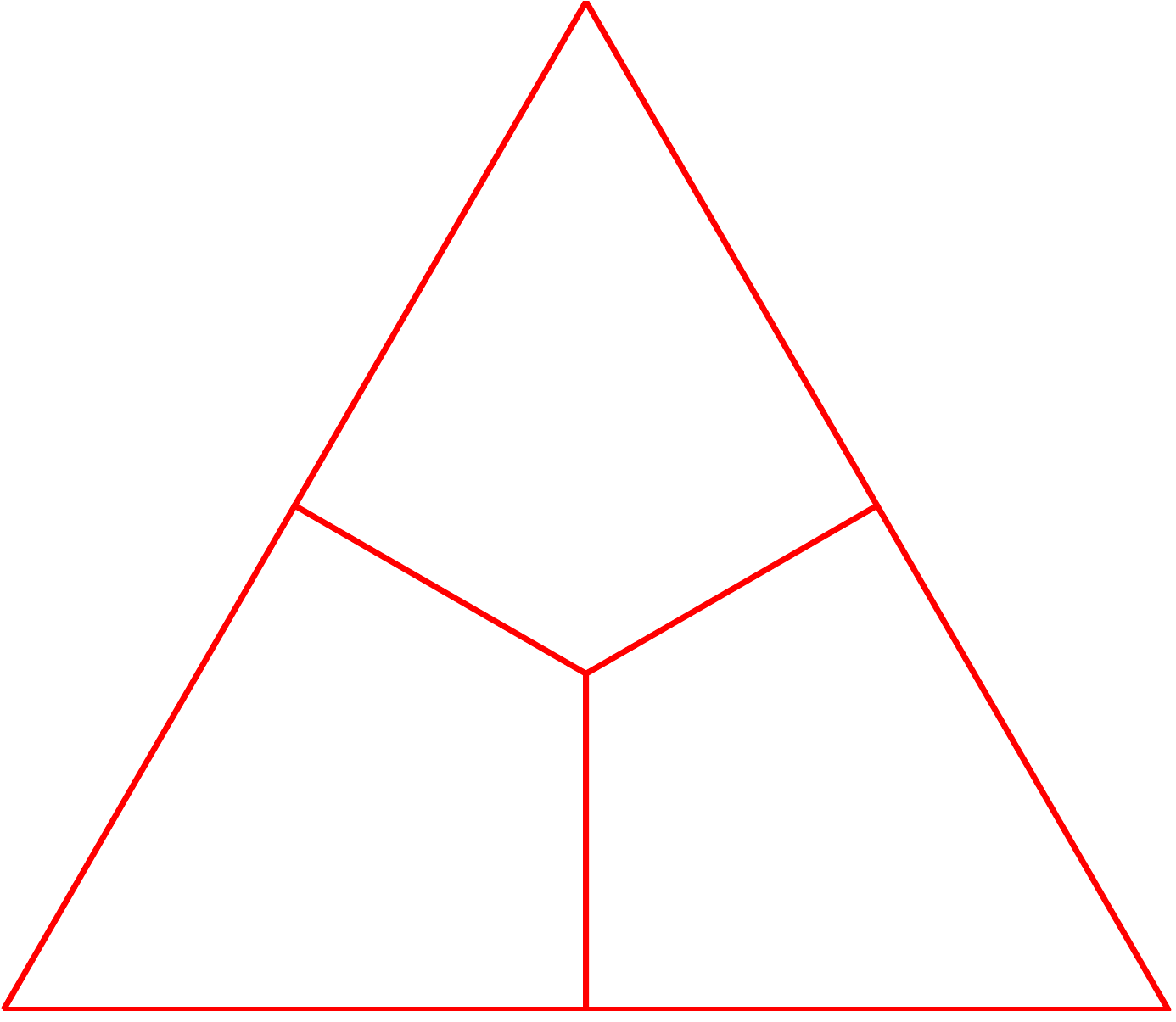} && \includegraphics[height=3cm]{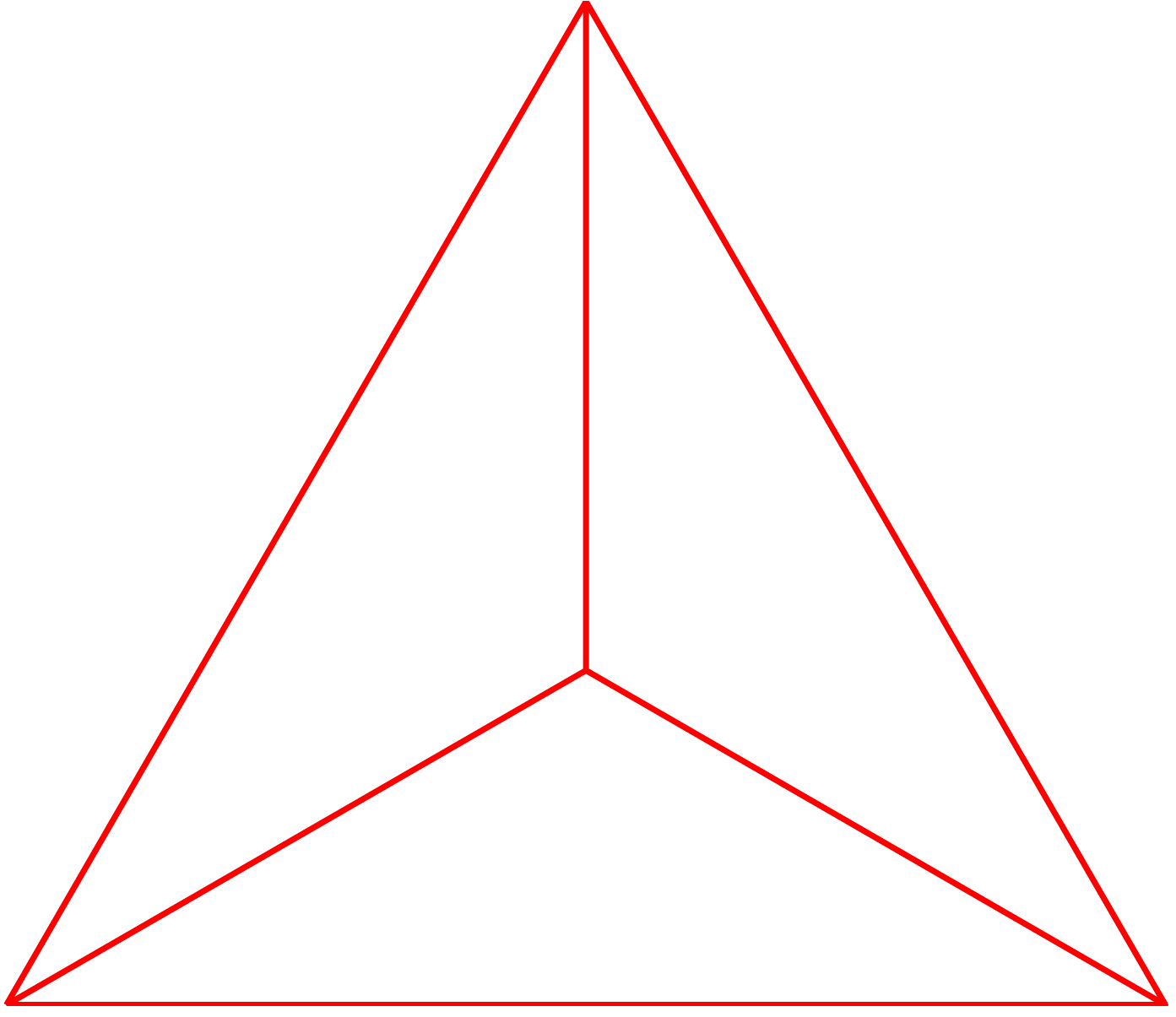}\\
  $\Lambda_{3,\infty}= 142.89$     
  &&  $\Lambda_{3,\infty}=  215.13$ 
\end{tabular}
\caption{Dirichlet-Neumann approach for $3$-partitions of the equilateral triangle.}
\end{center}
\end{figure}

For the case $k=5$ on the square we have seen that the partition seems to have all the symmetry axes of the square. As suggested by the result of the iterative method we considered a Dirichlet-Neumann condition corresponding to an axis of symmetry parallel to the sides of the square. As shown in \cite[Figure 19]{BH11} choosing a mixed boundary condition on the diagonal gives a partition with a strictly higher maximal eigenvalue. 
\label{other_DN}
\end{rem}

\begin{rem}
We note the similarity of the partitions of the equilateral triangle for $k=3$ and $k=5$ to some eigenvalues of the Aharonov-Bohm operator on a sector considered in \cite{BonLen14,MR3270167}. Thus we were able to check that the partition for $k=3$ corresponds to the third eigenvalue of the Aharonov-Bohm operator on the equilateral triangle with a singularity at its centroid. In the same way, the partition for $k=5$ corresponds to the sixth eigenvalue of the Aharonov-Bohm operator on the equilateral triangle with a singularity at the midpoint of one of the heights.
\label{arhanov-bohn}
\end{rem}

\subsection{Summary of the numerical results}

We have seen three numerical approaches for the study of the minimizers of $\fL_{k,\infty}$: the use of $p$-norms of eigenvalues with $p$ large, the penalization method and the Dirichlet-Neumann method. We make below a brief analysis and a comparison of the results given by these methods. 

First we note that the $p$-norms method and the penalization method work in all cases. In most of the cases, the penalization method does exactly what it was build for: penalize the difference between the eigenvalues while minimizing their sum. Thus there is no great surprise to see that it manages to give better upper bounds for $\fL_{k,\infty}(\Omega)$, {\Gn in most cases}. As we can see in Table \ref{pnorm_pen} the penalization method produces results where the gap between the minimal and maximal eigenvalues of cells is smaller. Inspiring from the results of the iterative methods, we can improve them by restricting the research to some particular partitions where we fixe some parts of the boundaries of the subdomains: this is the Dirichlet-Neumann approach. Once the structure is fixed, we express the partition as a nodal set of a mixed problem. In this paper, we apply this method only with fixed straight lines and symmetry. On the other hand, when the Dirichlet-Neumann can be applied, it produces equipartitions and thus gives the best upper bounds for $\fL_{k,\infty}(\Omega)$. Table~\ref{synthese-max} summarize the lowest energy $\Lambda_{k,\infty}(\cD)$ obtained according to the three methods (iterative method for $p=50$, penalization and Dirichlet-Neumann approach), and thus we deduce some upper bounds for $\fL_{k,p}(\Omega)$.

\begin{table}[h!]
\centering 
\begin{tabular}{|c|c|c|c|c|c|c|c|c|c|}
\hline 
 & \multicolumn{3}{c|}{Disk} & \multicolumn{3}{c|}{Square} & \multicolumn{3}{c|}{Equilateral triangle}\\ \hline
$k$ & $p=50$ & pen. & D-N & $p=50$ & pen. & D-N & $p=50$ & pen. & D-N \\ \hline  
$3$  & $20.25$ & $20.24$ & $20.19$   & $66.69$  & $66.612$ & $66.581$ & $143.06$ & $142.88$ & $142.88$ \\ \hline 
$5$  & $33.31$ & $33.31$ & $33.21$   & $105.82$ & $104.60$ & $104.29$ & $252.67$ & $252.17$ & $251.99$ \\ \hline 
$6$  & $39.40$ & $39.17$ & $39.02$   & $128.11$ & $127.11$ & -        & $276.16$ & $276.22$ & $275.97$ \\ \hline 
$7$  & $44.26$ & $44.25$ & $44.03$   & $147.44$ & $146.88$ & { $146.32$} & $348.24$ & $345.91$ & -        \\ \hline 
$8$  & $50.46$ & $50.64$ & $50.46$   & $161.64$ & $161.28$ & { $160.87$}        & $391.06$ & $389.53$ & $389.31$ \\ \hline 
$9$  & $58.28$ & $58.30$ & $58.25$   & $179.21$ & $178.08$ & -        & $431.92$ & $428.75$ & -        \\ \hline 
$10$ & $64.54$ & $64.27$ & $67.19$   & $206.85$ & $204.54$ & -        & $456.66$ & $453.25$ & $451.93$ \\ \hline 
\end{tabular}\\[5pt]
\caption{Lowest energies $\Lambda_{k,\infty}$ for the three methods, $\Omega=\Circle,\ \square,\ \triangle$.}
\label{synthese-max}
\end{table}

\subsubsection*{\bf The Disk.} 
We notice that for $\Gn k \in\llbracket 2,5\rrbracket$ the optimal partitions correspond to sectors of angle $2\pi/k$ and we use the upper bound \eqref{eq.sect} to fill the third column for the disk in Table~\ref{synthese-max}.

For $\Gn k \in \llbracket 6,9\rrbracket$ the best candidates are given by the Dirichlet-Neumann approach on sectors of opening $2\pi/(k-1)$;  the other two methods give close but larger results. 

When $k=10$, we can also apply a Dirichlet-Neumann approach on an angular sector of opening $2\pi/9$ but the upper bound is worse than with the iterative methods. Indeed, we observe that the optimal $10$-partition seems to have two subdomains at the center (see Figure~\ref{pen-figs}). 

\subsubsection*{\bf The Square.} 
For $k \in \{3,5,7,8\}$, it is possible to use the Dirichlet-Neumann method and this gives the lowest upper bound for $\fL_{k,\infty}(\square)$. As shown in \cite[Figure 8]{BH11} for $k=3$ we have a continuous family of solutions, each with the same maximal eigenvalue.

In most other cases, the penalization method gives the best upper bounds for $\fL_{k,\infty}(\square)$. Let us note that for $k=9$ we are not able to obtain partitions which have a lower maximal eigenvalue than the partition into $9$ equal squares, for which all cells have eigenvalue $177.65$. On the other hand, since the partition into $9$ squares is nodal and the $9$-th eigenfunction on the square is not Courant sharp,  this is not a $\infty$-minimal $9$-partition (see Theorem~\ref{thm.HHOT}). In our computations, with the $p$-norm and penalization approaches we find partitions whose energy $\Lambda_{k,\infty}$ equals $179.21$ and $178.08$ respectively. 
Since our computations using iterative methods were based on a relaxed formulation for the eigenvalues, the limited numerical precision of the method does not enable us to reach better results whereas we know that the minimal $9$-partition has an energy less than $177.65$.  

\subsubsection*{\bf The Equilateral triangle.} 
The equilateral triangle gives us lots of occasions where a Dirichlet-Neumann method can be used. For $k \in \{3,5,6,8,10\}$ this method gives us the best known upper bound for $\fL_{k,\infty}(\triangle)$. For the cases $k \in \{7,9\}$ the penalization method gives lowest upper bounds. 
{ When $k \in \{3, 6, 10\}$ numerical simulations produce partitions {whose} subdomains are particular polygons with straight lines and it seems this behavior appears for some specific values of $k$.}

\begin{rem}\label{rem.nbtri} (Remark about partitions corresponding to triangular numbers)
We note that in cases where $k$ is a triangular number, i.e. $k = n(n+1)/2$ with $n \geq 2$, the $p$-minimal $k$-partition of the equilateral triangle seems to be the same for any $p$ and to be made of three types of polygonal cells: $3$ quadrilaterals at  corners which are each a third of an equilateral triangle, $3(n-2)$ pentagons with two right angles and three angles measuring $2\pi/3$ and a family of regular hexagons. In Figure \ref{triangular} we represent some of the results obtained numerically with the iterative method for $k \in \{15,21,28,36\}$.

\begin{figure}[h!]
\centering
\includegraphics[width =0.2\textwidth]{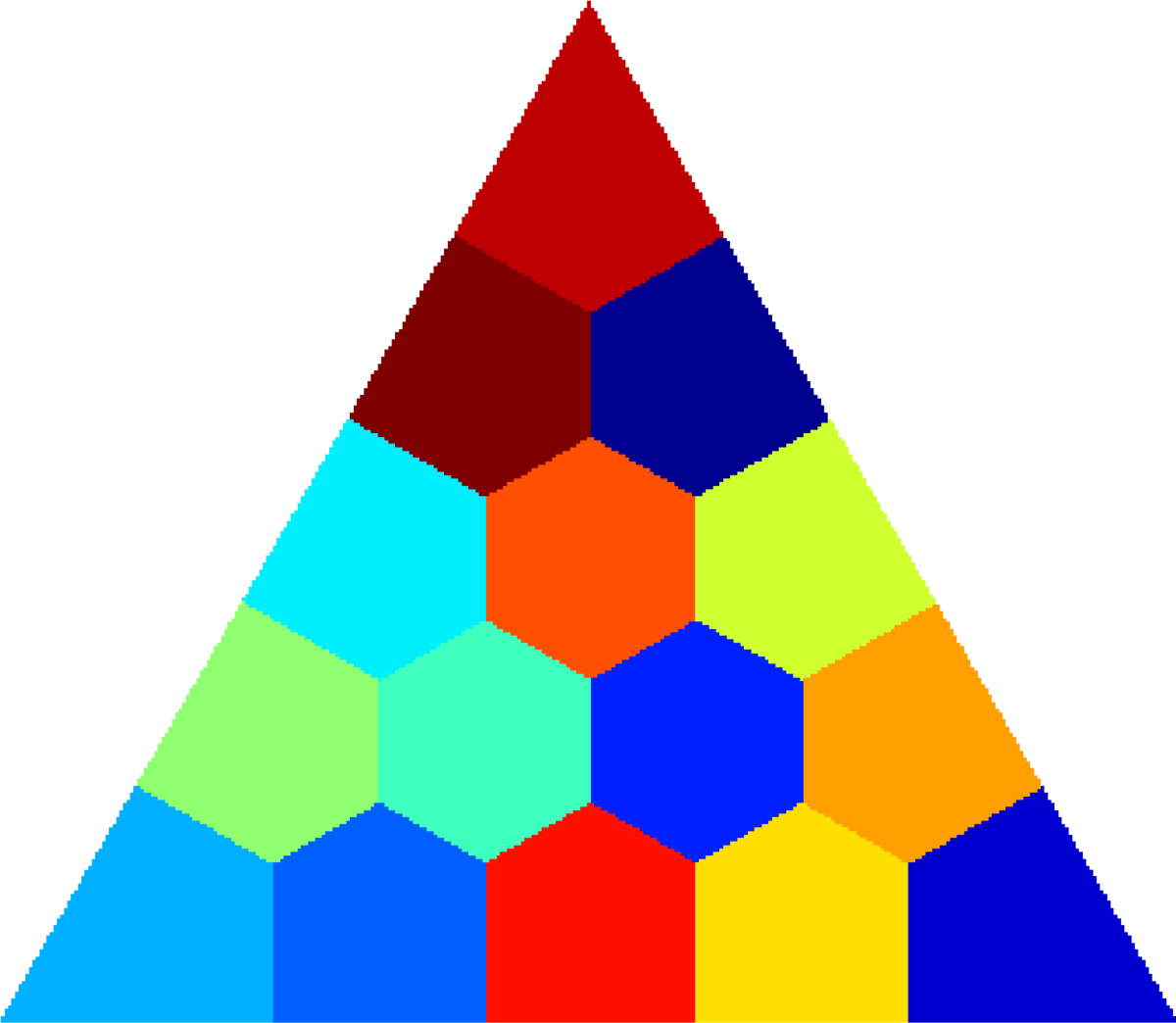}~
\includegraphics[width =0.2\textwidth]{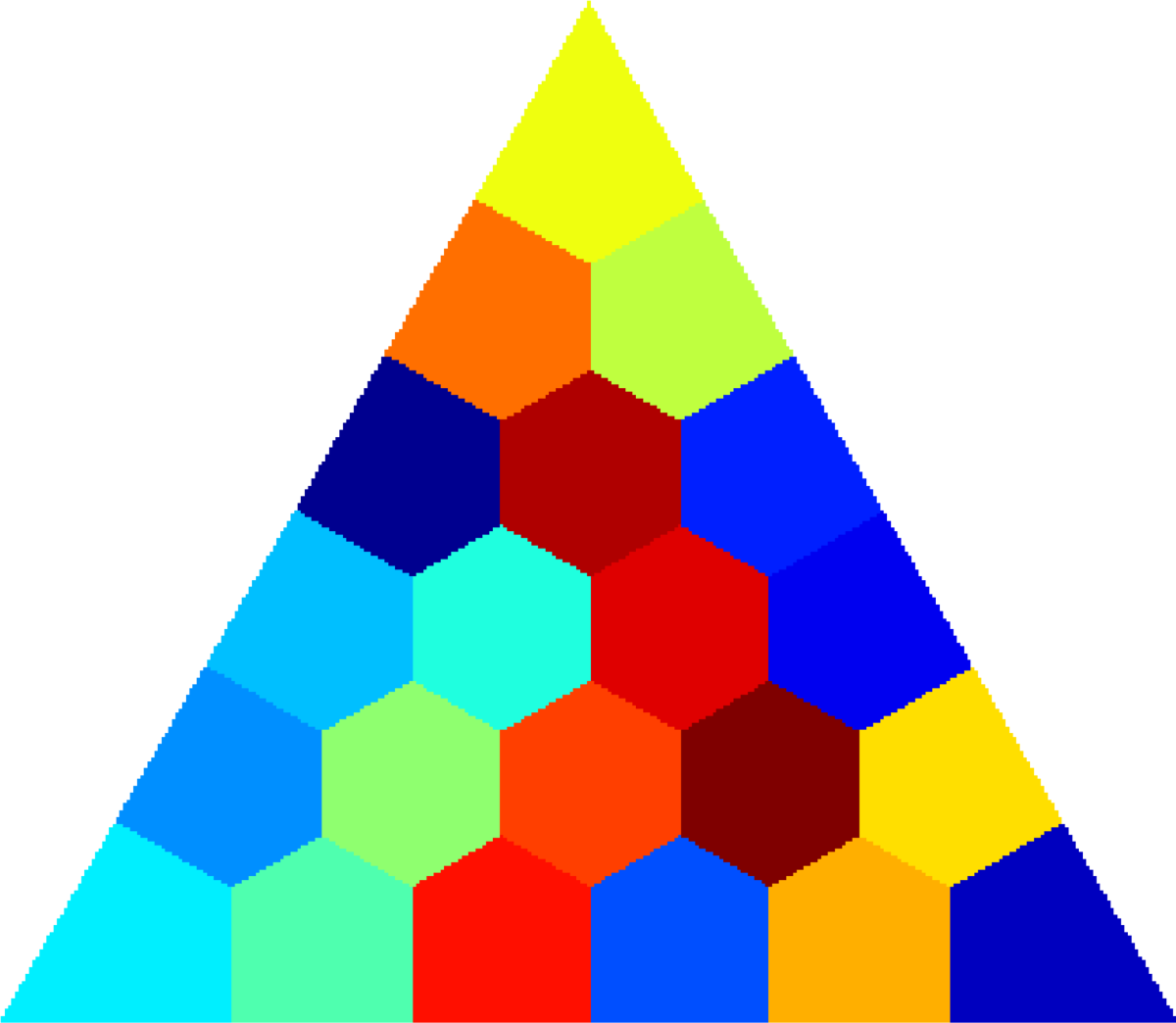}~
\includegraphics[width =0.2\textwidth]{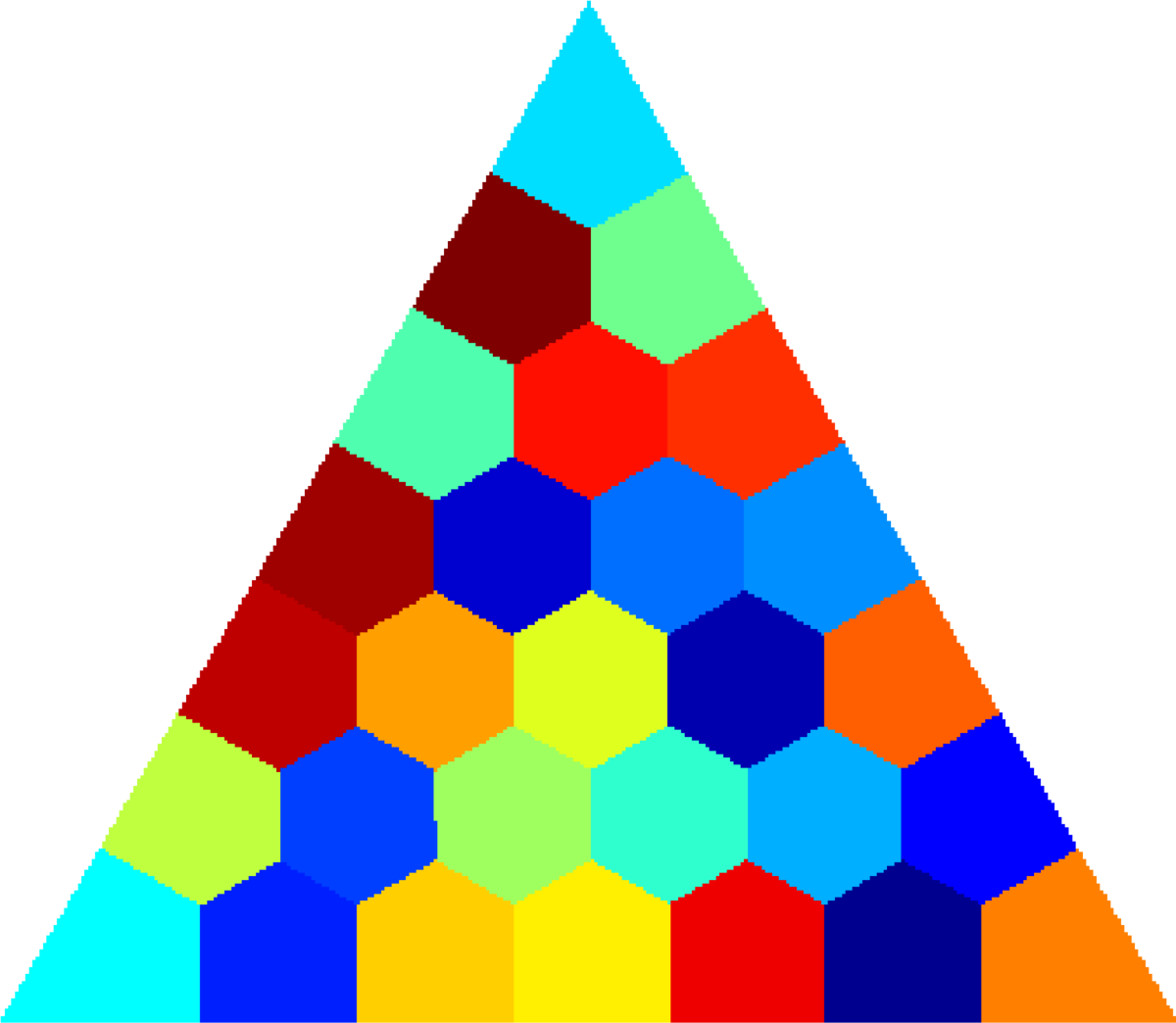}~
\includegraphics[width =0.2\textwidth]{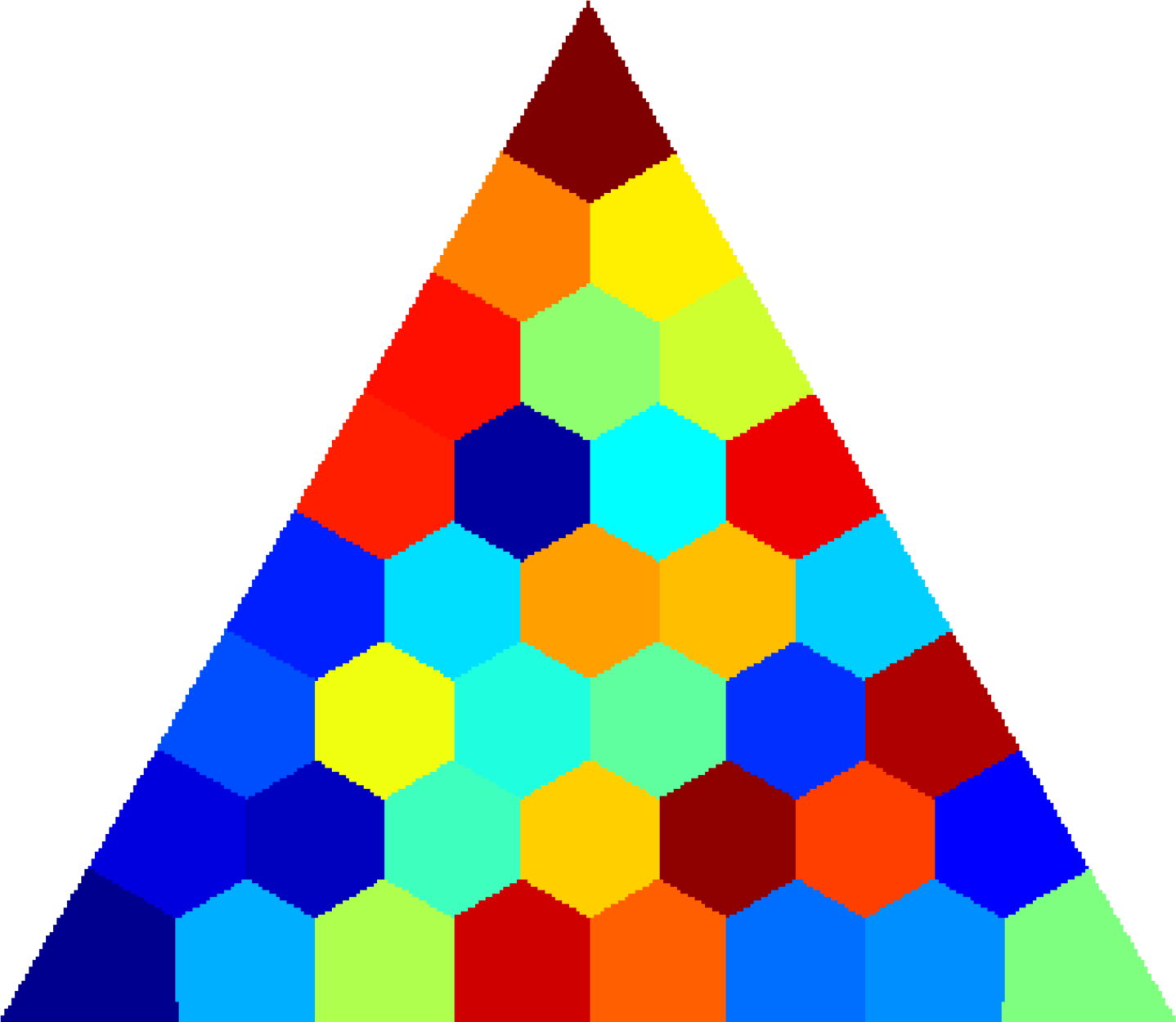}
\caption{Numerical candidates for $k \in \{15,21,28,36\}$.}
\label{triangular}
\end{figure}

\end{rem}

\subsection{Candidates for the max {\it vs.} the sum\label{ss.summax}}  
Given a candidate for minimizing $\Lambda_{k,\infty}$ we may wonder if this partition can also minimize the sum of the eigenvalues $\Lambda_{k,1}$. Such a discussion has already been made in \cite{HHO10} for $k=2$ and it is concluded that, in general, we have $\mathfrak{L}_{2,1}(\Omega) < \mathfrak{L}_{2,\infty}(\Omega)$. A criterion which allows us to make a decision in some cases was given in Proposition \ref{l2norm}. Since the optimal partition for $\Lambda_{k,\infty}$ is an equipartition, given two neighbors $D_i$ and $D_j$ from this partition, we have that $(D_i,D_j)$ forms a nodal partition for ${\rm Int\,}\overline{D_i\cup D_j}$. Our interest is to see whether the eigenfunction associated to this nodal partition has the same $L^2$ norm on the two domains $D_i,D_j$. In the case when the $L^2$ norms are different we conclude that, supposing the initial partition was optimal for the max, the corresponding partition is not optimal for the sum. Since the criterion can only be applied to equipartitions we quickly examine the candidates obtained with the Dirichlet-Neumann approach. 

Let's first remark that Proposition \ref{l2norm} does not allow us to say anything about the cases where the optimal partition for the max is made out of congruent elements. In this case the $L^2$ norms on the subdomains will evidently be the same. This is the case for the $k\in\{2,4\}$ on the square, $k\in\{3,4\}$ on the equilateral triangle and $k \in \{2,3,4,5\}$ on the disk.

For the other situations, let us apply Proposition \ref{l2norm} to the results obtained with the Dirichlet-Neumann method. For the cases $k\in\{ 5,7,8,9\}$ for the disk, $k \in \{3,5,7,8\}$ for the square and $k \in \{5,8\}$ for the equilateral triangle we always find two adjacent domains $D_i,D_j$ in the partition for which the second eigenfunction on $D_i\cup D_j$ has different $L^2$ norms on $D_i,D_j$: {\Gn the observed gap for normalized eigenfunction is larger than 0.03, which is significantly larger than the case we present below}. We can conclude that if the above configurations are optimal for the max then they are not optimal for the sum.

Let us analyze below in more detail the situation where Proposition \ref{l2norm} does not allow us to conclude {\Gn that the minimal partitions for $p=\infty$ and $p=1$ are different}. For the equilateral triangle with $k=6$ or $k=10$, the gap when we apply the $L^2$ norm criterion is less than $10^{-4}$. In these situations, the subdomains of the numerical $\infty$-minimal $k$-partition seem to be polygons with straight lines: quadrilateral, pentagon and regular hexagon (for $k=10$). If we consider only partitions whose subdomains are like this, we will now compare the best partitions for the sum or the max. Let us discuss a little more these two situations $k=6$ and $k=10$ below.
\begin{itemize}[leftmargin=*]
\item $k=6$: the partition is represented in Figure \ref{equiDN6-details}. We perform a one parameter study with respect to $r \in [0,1]$ just as in the case of the Dirichlet-Neumann approach where we compute numerically the eigenvalues on the two types of polygonal cells present in the partition (quadrilateral and pentagon). Numerically we find that the partitions minimizing the max and the sum are almost the same, in the sense that the difference between the values of $r$ which minimize the sum and the max is smaller than $10^{-4}$. Thus, the partitions minimizing the sum and the max are either the same or are too close to be distinguished numerically.
\item $k=10$: the partition is represented in Figure \ref{equiDN10-details}. We can see that we have three types of domains: a regular hexagon in the center, six pentagons and three quadrilaterals. As in the Dirichlet-Neumann mixed approach, we note that we can characterize the partition using two parameters $t,s \in [0,1]$. Next we search for the parameters which optimize the maximal eigenvalue and the sum. To obtain an equipartition (for the max), we need to consider non symmetric {\Gn pentagons}. As for $k=6$, it seems  that the optimal partitions are the same for the sum and the max (or very close). The difference between corresponding parameters is again smaller than $10^{-4}$.
\end{itemize}
This suggests that 
$$p_{\infty}(\triangle,k)=1, \qquad\mbox{ for }k=6,\ 10.$$

Next is the case of the disk for $k=7$. Here we also have $L^2$ norms which are close (the gap is around 0.03)  and we analyze this case more carefully in the following sense. Note that the central domain seems to be a regular hexagon $\hexagon$ and the exterior domains $D_{i}$ ($i=1,\ldots, 6$) are subsets of angular sector of opening $\pi/3$. We optimize the sum and the max by varying the size of the interior hexagon. Using two different finite element methods, {\sc M\'elina} and {\sc FreeFem}++, we obtain that the sum is minimized when the side of the hexagon is equal to $0.401$ and the maximum is minimized for a side equal to $0.403$. These computations let us think that the optimal partitions for the sum and the max might be different in this case. In the following Table \ref{tab.trik10}, we give the parameters for which the sum $\Lambda_{7,1}$ and the maximal eigenvalue $\Lambda_{7,\infty}$ are minimized, as well as the corresponding eigenvalues. We observe that $\Lambda_{7,1}(\cD^{7,1})<\Lambda_{7,\infty}(\cD^{7,\infty})$ in coherence with \eqref{eq.monop} and that the gap between the eigenvalues of the minimizer for the sum is significant enough to say that this partition is not an equipartition. Consequently, if the minimal $7$-partition of the disk for the max has the previous structure (a regular hexagon at the center and straight lines to join the boundary), it seems that this partition is not minimal for the sum. 

\begin{table}[h!t]
\centering 
\begin{tabular}{|c|c|c|c|c|c|}
\hline
$r$& $\lambda_1(\hexagon)$& $\lambda_1(D_1)$& $\Lambda_{7,\infty}$& $\Lambda_{7,1}$\\
\hline
0.401&  44.498& 43.949&  44.498&  44.028  \\
\hline
0.403&  44.030&  44.030&  44.030 & 44.030 \\ 
\hline
\end{tabular}\\[5pt]
\caption{Upper bounds for $\fL_{7,p}(\Circle)$ for $p=1$ and $p=\infty$.\label{tab.trik10}}
\end{table}

In the following section a more detailed analysis is devoted to showing the difference between the partitions minimizing the sum and the ones minimizing the maximal eigenvalue by looking at the evolutions of the partitions with respect to $p$ when minimizing the $p$-norm of eigenvalues.

\section{Numerical results for the $p$-norm}
\label{section.pnorm}
\subsection{Overview}
Our main interest when studying numerically the optimizers of the $p$-norm of the eigenvalues was the approximation of the $\fL_{k,\infty}$ problem. As we have seen before, the numerical $p$-minimal $k$-partition for $p=50$ is not far from being an equipartition. In this section we make some remarks, based on the numerical simulations, concerning the behavior of the $p$-minimal $k$-partitions with respect to $p$. We are interested in observing the evolution of the configuration of the partitions as $p$ varies from $1$ to $50$. In most cases the configuration is stable, but there are, however, some cases where the partitions change as $p$ grows and converge to different topological configurations when $p$ goes to $\infty$. Viewing the evolution of the maximal eigenvalue and the $p$-norm as $p$ grows can also confirm the conclusions of the previous section concerning the fact that some partitions which optimize the maximal eigenvalue may not optimize the sum. In the following, we will consider the three geometries $\Omega=\triangle, \square, \Circle$ and some values of $k$. 

For each of these parameters two figures will highlight the evolution of the $p$-minimal $k$-partitions $\cD^{k,p}$ obtained by the iterative method. The first one concerns the evolution of the energies: we represent $p\mapsto \Lambda_{k,p}(\cD^{k,p})$ in blue, $p\mapsto \Lambda_{k,\infty}(\cD^{k,p})$ in red and eventually the upper bound $L_{k}(\Omega)$ or $\Lambda^{DN}_{k}(\Omega)$ obtained by the Dirichlet-Neumann approach in magenta (see Figures~\ref{fig.eq2nrj}, \ref{fig.eq4nrj}, \ref{fig.sq3nrj}, \ref{fig.sq5nrj}, \ref{fig.sq9nrj}, \ref{evol.disk6-9}, \ref{disk.evol}). In these graphes, we observe that the curve $p\mapsto \Lambda_{k,p}(\cD^{k,p})$ (in blue) is increasing, in coherence with \eqref{eq.monop}. The decay of the curves $p\mapsto \Lambda_{k,\infty}(\cD^{k,p})$ (in red) show that as $p$ is increasing, we get a better and better upper-bound for $\fL_{k,\infty}(\Omega)$. Theses two curves converge to the same value, which is the upper bound obtained by the Dirichlet-Neumann approach when it can be applied.
We also illustrate the evolution of the boundary of $\cD^{k,p}$ according to $p$ with $p=1$ in blue and $p=50$ in red (see Figures~\ref{fig.eq2part}, \ref{fig.eq4part}, \ref{evol.equi.all}, \ref{fig.sq3part}, \ref{fig.sq5part}, \ref{evol.square6-10}, \ref{evol.square9}, \ref{compar_sq7b}-\ref{compar_sq7c}, \ref{evol.disk6-9}, \ref{disk.evol}).

\subsection{The equilateral triangle.} 
The equilateral triangle is a first example where the optimal partitions for $p=1$ and $p=\infty$ do not coincide, as seen in the previous section (see also \cite{HHO10} for $k=2$). Figure \ref{evol.equi2} represents the evolution of the energies and the optimal partitions as $p$ increases. 
We observe that even if the partitions do not change much, the maximal eigenvalue is significantly decreased as $p$ increases. 
\begin{figure}[h!t]
\centering
\subfigure[$\Bl\Lambda_{2,p}(\cD^{2,p})$, $\Rd\Lambda_{2,\infty}(\cD^{2,p})$ and $\Mg\Lambda^{DN}_{2}(\triangle)$ vs. $p$\label{fig.eq2nrj}]{\quad\qquad\includegraphics[width = 0.3\textwidth]{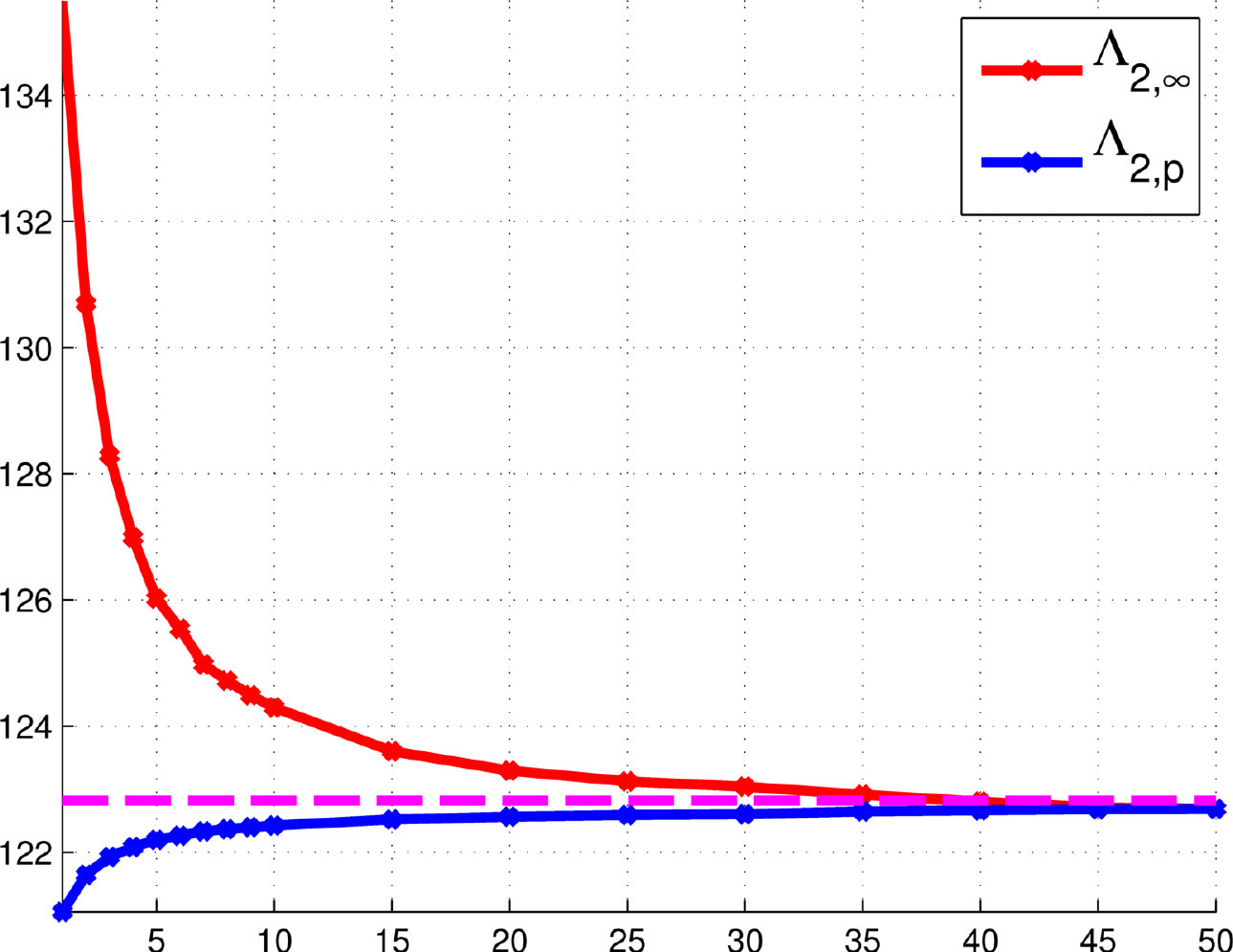}\qquad}\qquad\quad
\subfigure[$\cD^{2,p}$ vs. $p$\label{fig.eq2part}]{\includegraphics[width = 0.25\textwidth]{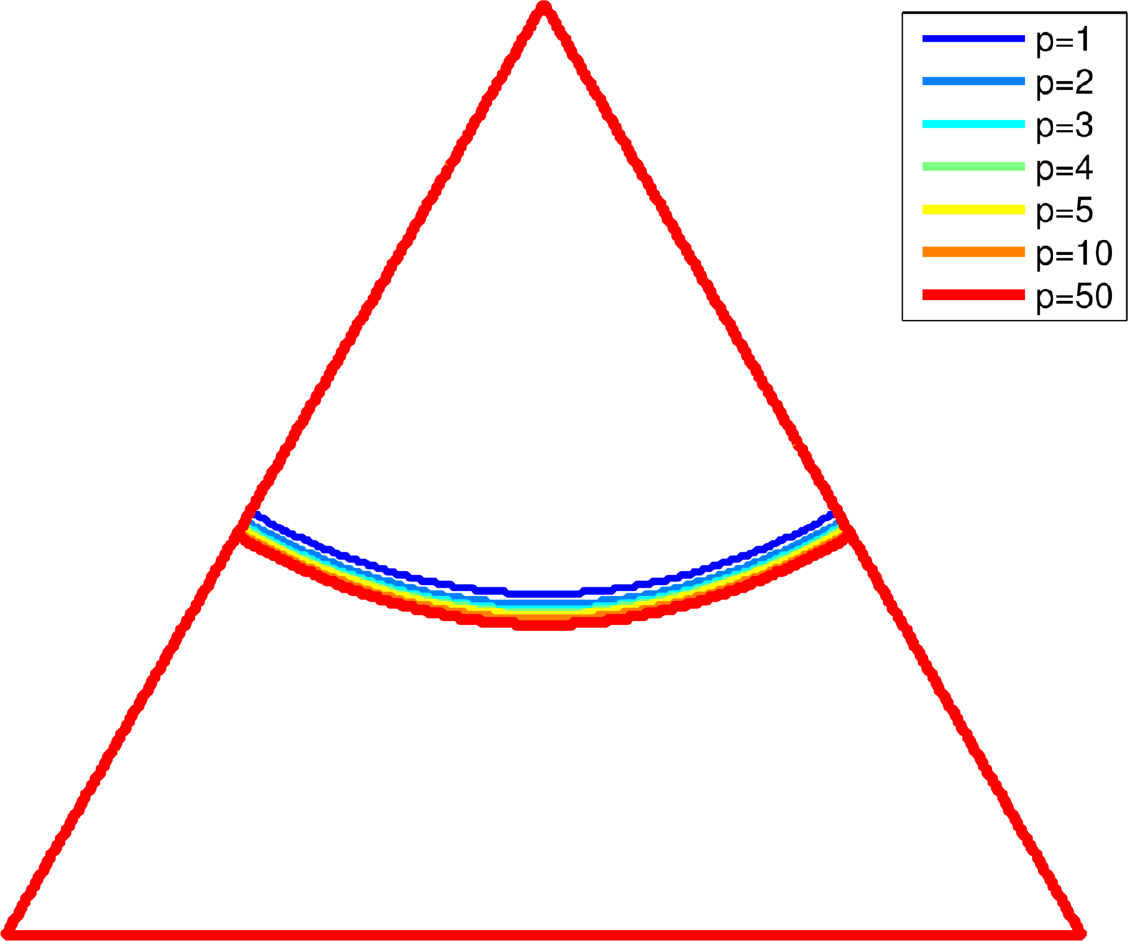}}
\caption{$p$-minimal $2$-partitions of the equilateral triangle {\it vs.} $p$.}
\label{evol.equi2}
\end{figure}

For $k=3$ we obtain an equipartition starting from $p=1$ and thus the partition does not change with $p$ and the energies are constant with respect to $p$. This suggests that the $p$-minimal $3$-partition is given by Figure~\ref{fig.candtrik3} and $p_{\infty}(\triangle, k)=1$.

For $k=4$, since the $4$-th eigenvalue of the equilateral triangle is Courant sharp, we know that the minimal $4$-partition for $p=\infty$ is the partition with 4 similar equilateral triangles  (see Figure~\ref{fig4.tri4}). The evolution of the partitions according to $p$ is given in Figure~\ref{evol.equi4}, where $L_{4}(\triangle)=\lambda_{4}(\triangle)=\fL_{4,\infty}(\triangle)$ is plotted in magenta. We observe the convergence of $\Lambda_{k,p}(\cD^{k,p})$  as well as the decay of the largest first eigenvalue to $\lambda_{4}(\triangle)$. The partition $\cD^{k,p}$ changes in a significant way with $p$. Indeed, it seems that the minimal $4$-partition for the sum has 4 singular points on the boundary and two inside. The points are moving with $p$ to collapse when $p=\infty$ where we have only 3 singular points on the boundary. Furthermore, the minimal $4$-partition for the max has more symmetry than those for $p<\infty$. 
\begin{figure}[h!t]
\centering 
\subfigure[$\Bl\Lambda_{4,p}(\cD^{2,p})$, $\Rd\Lambda_{4,\infty}(\cD^{4,p})$ and $\Mg\lambda_{4}(\triangle)$ vs. $p$\label{fig.eq4nrj}]{\qquad\includegraphics[width = 0.3\textwidth]{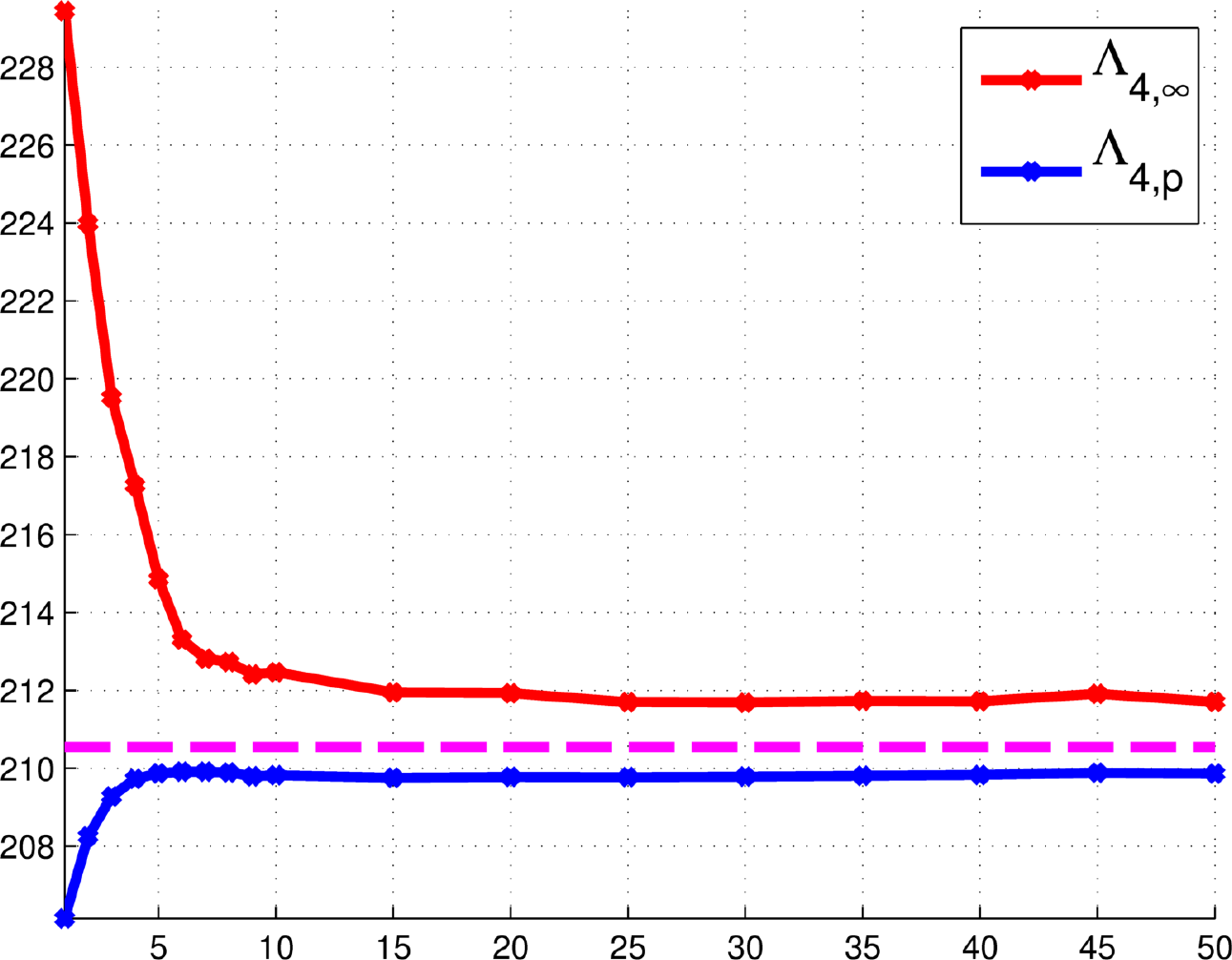}\qquad}
\subfigure[$\cD^{4,p}$ vs. $p$\label{fig.eq4part}]{\qquad\includegraphics[width = 0.25\textwidth]{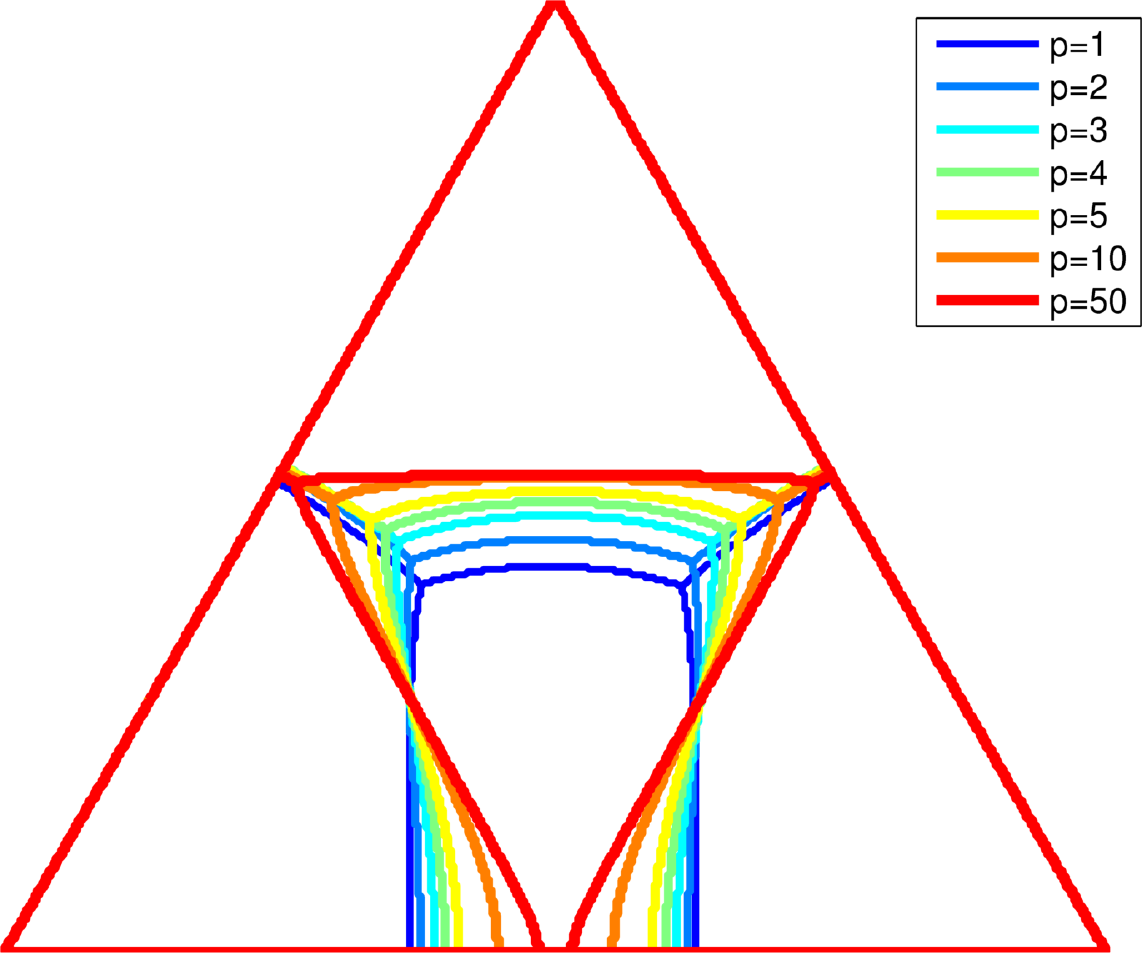}\qquad}
\caption{$p$-minimal $4$-partitions of the equilateral triangle {\it vs.} $p$.}
\label{evol.equi4}
\end{figure}

We represent in Figure \ref{evol.equi.all} the evolution of the partitions for these values of $k$. For $k \in \{5,7,8,9\}$ we observe similar behaviors for the maximal eigenvalue and the $p$-norm as the ones already shown for $k \in \{2,4\}$.
\begin{figure}[h!t]
\centering 
\includegraphics[width = 0.2 \textwidth]{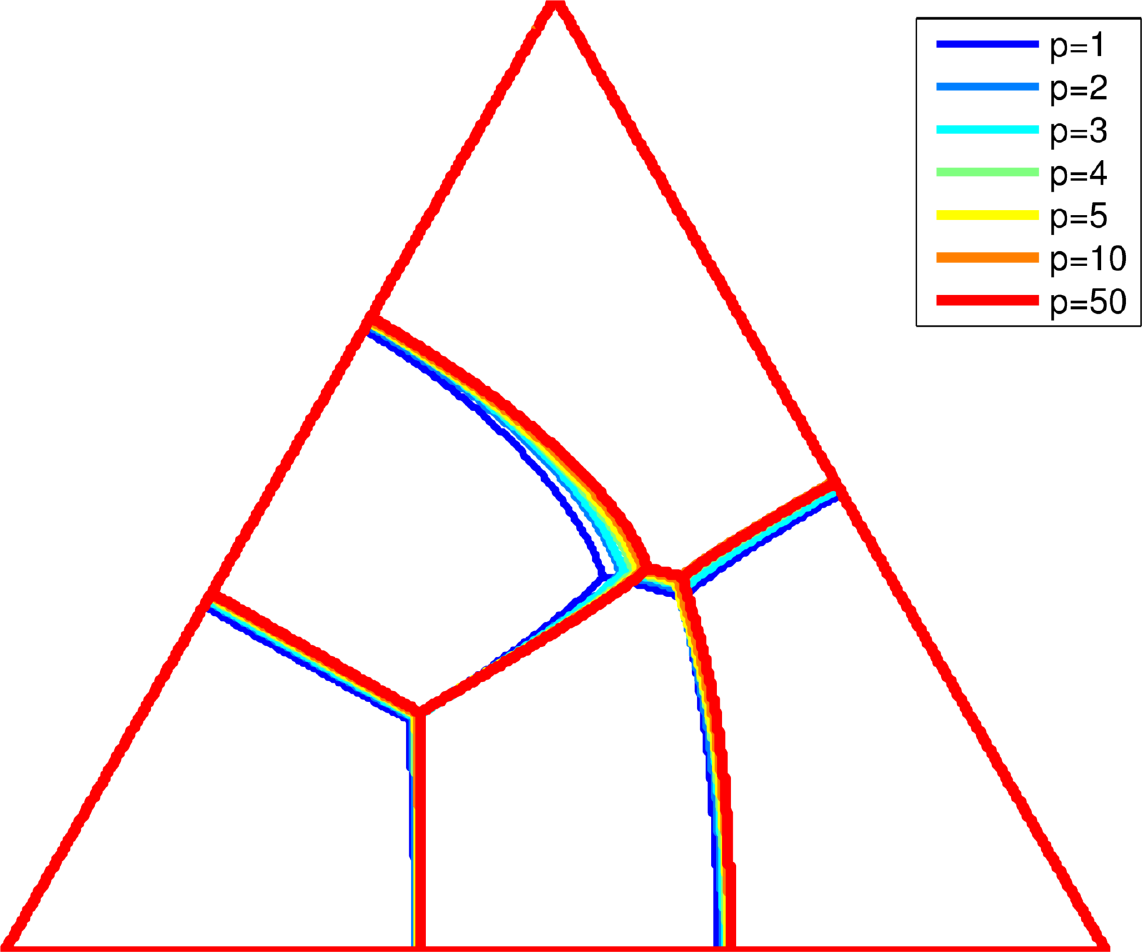}\hfill
\includegraphics[width = 0.2 \textwidth]{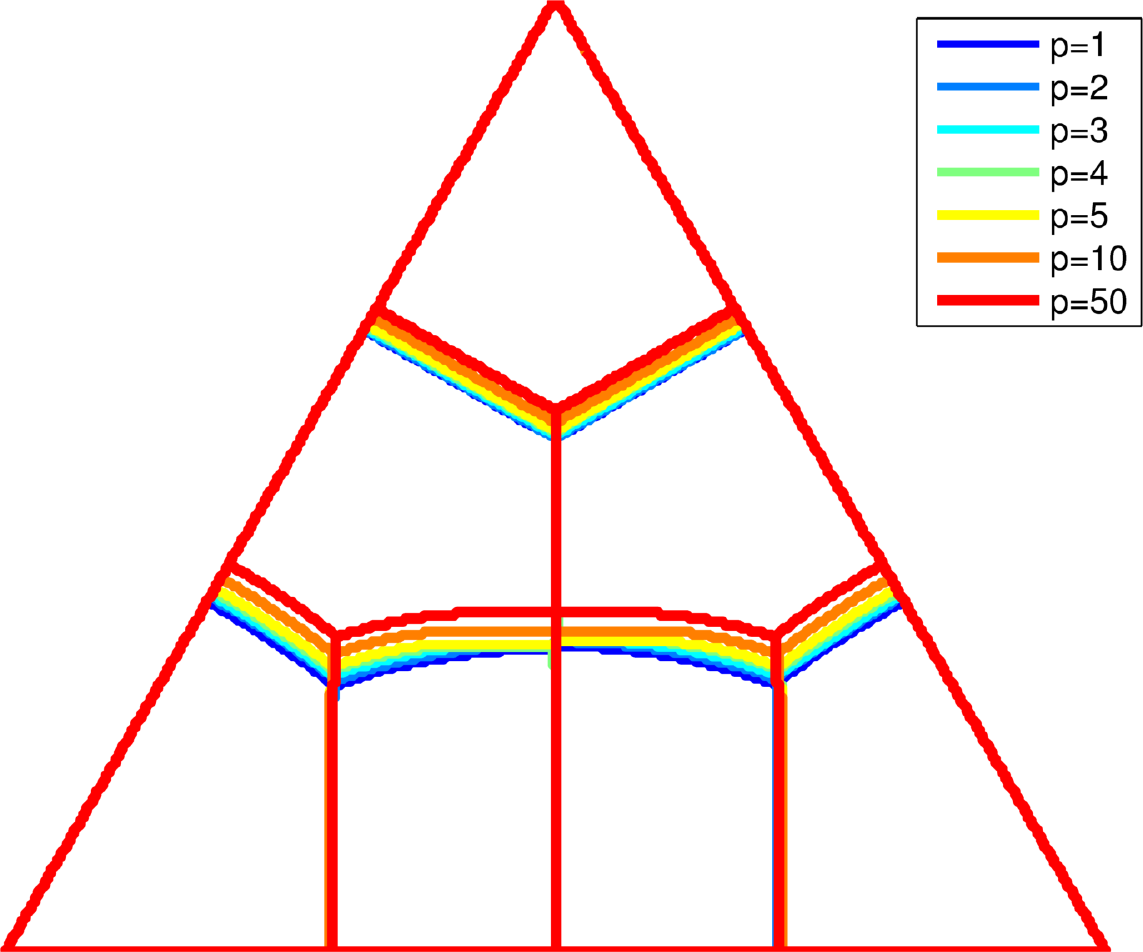}\hfill
\includegraphics[width = 0.2 \textwidth]{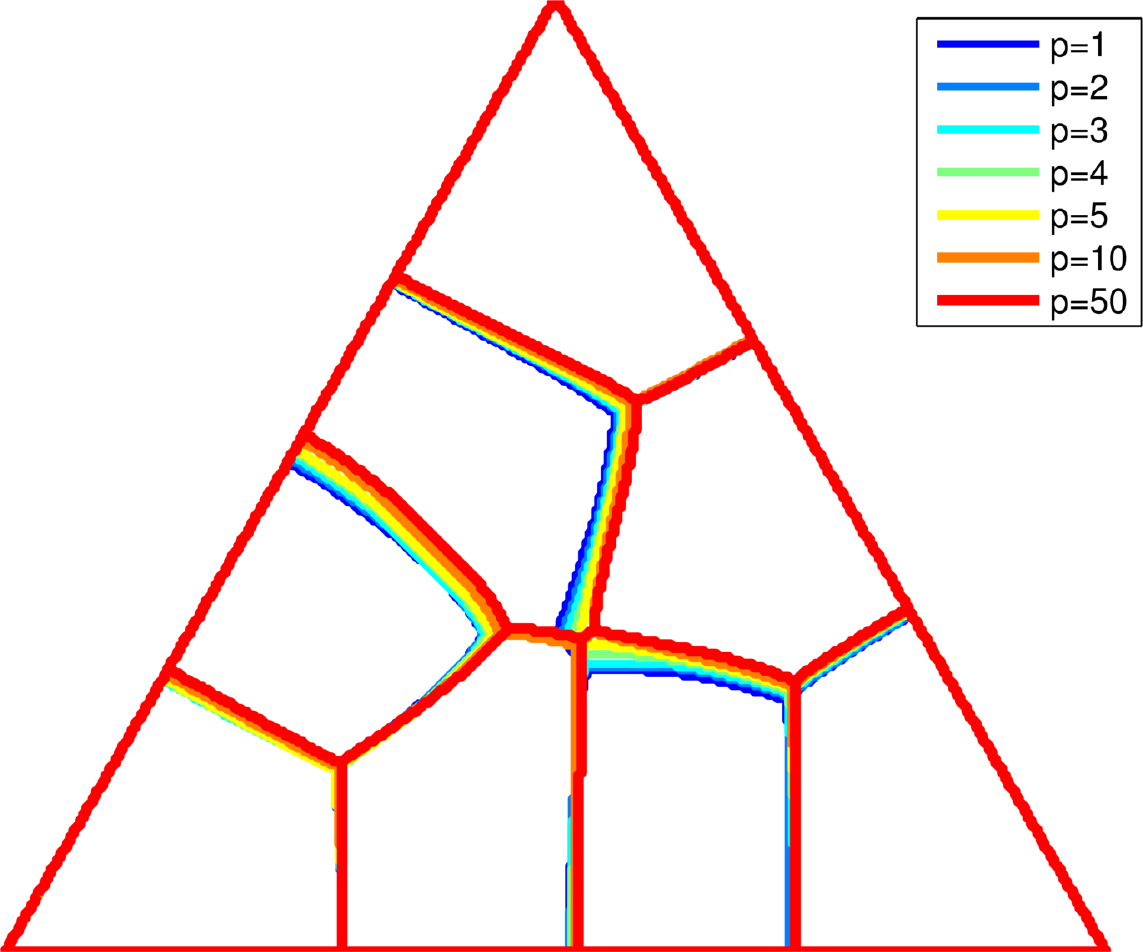}\hfill
\includegraphics[width = 0.2 \textwidth]{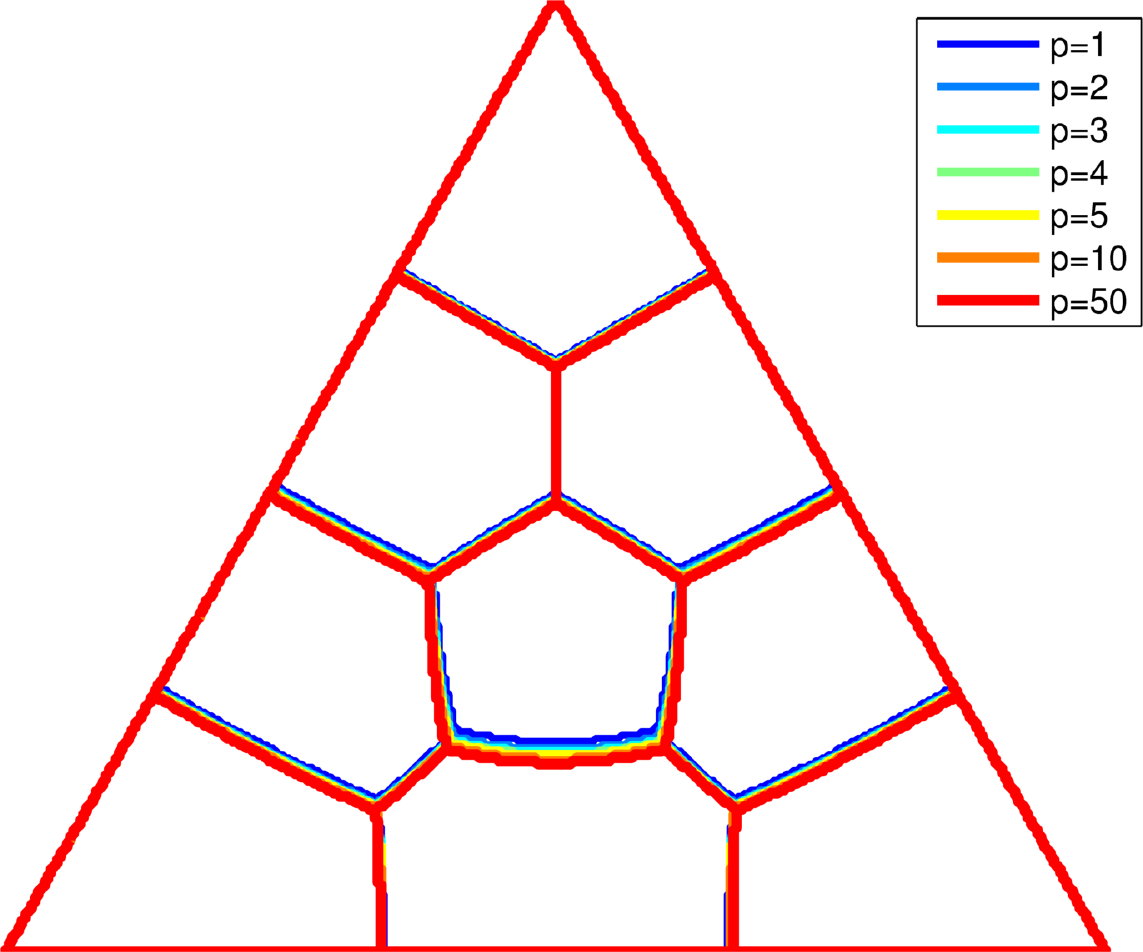}
\caption{$p$-minimal $k$-partitions of the equilateral triangle {\it vs.} $p$, for $k\in\{5,7,8,9\}$.}
\label{evol.equi.all}
\end{figure}

The remaining cases $k\in \{6,10\}$ are in the class of triangular numbers and as observed before (see Figure \ref{triangular}) in these cases it seems that the cells of the optimal partitions are polygonal domains. As seen in the previous section, the $L^2$ norm criterion does not allow us to say that the candidates found for minimizing $\mathfrak{L}_{k,\infty}$ are not minimizers for $\mathfrak{L}_{k,1}$ in these cases. The study of the evolution of the $p$-norms does not allow us to conclude that these partitions are different. In fact, the partitions are not observed to move at all and the energies do not vary much {(\Gn the relative variation $(\Lambda_{k,50}-\Lambda_{k,1})/\Lambda_{k,1}$ is less than 0.2\% when $k\in\{6,10\}$)}. This reinforces the observations at the end of the previous section where we have seen that the partitions minimizing the sum or the maximal eigenvalue are either the same or too close to decide. 
\begin{figure}[h!t]
\centering 
\includegraphics[width = 0.2\textwidth]{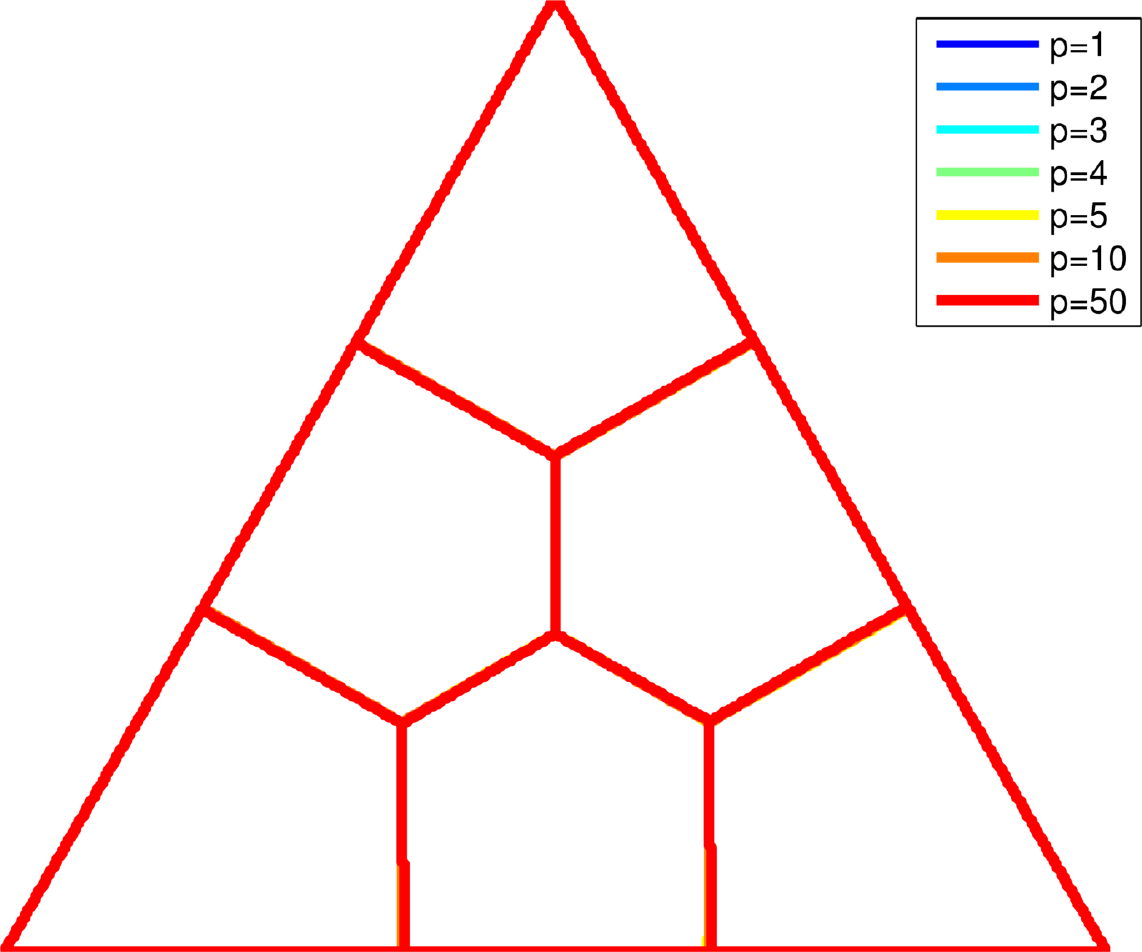}\qquad
\includegraphics[width = 0.2\textwidth]{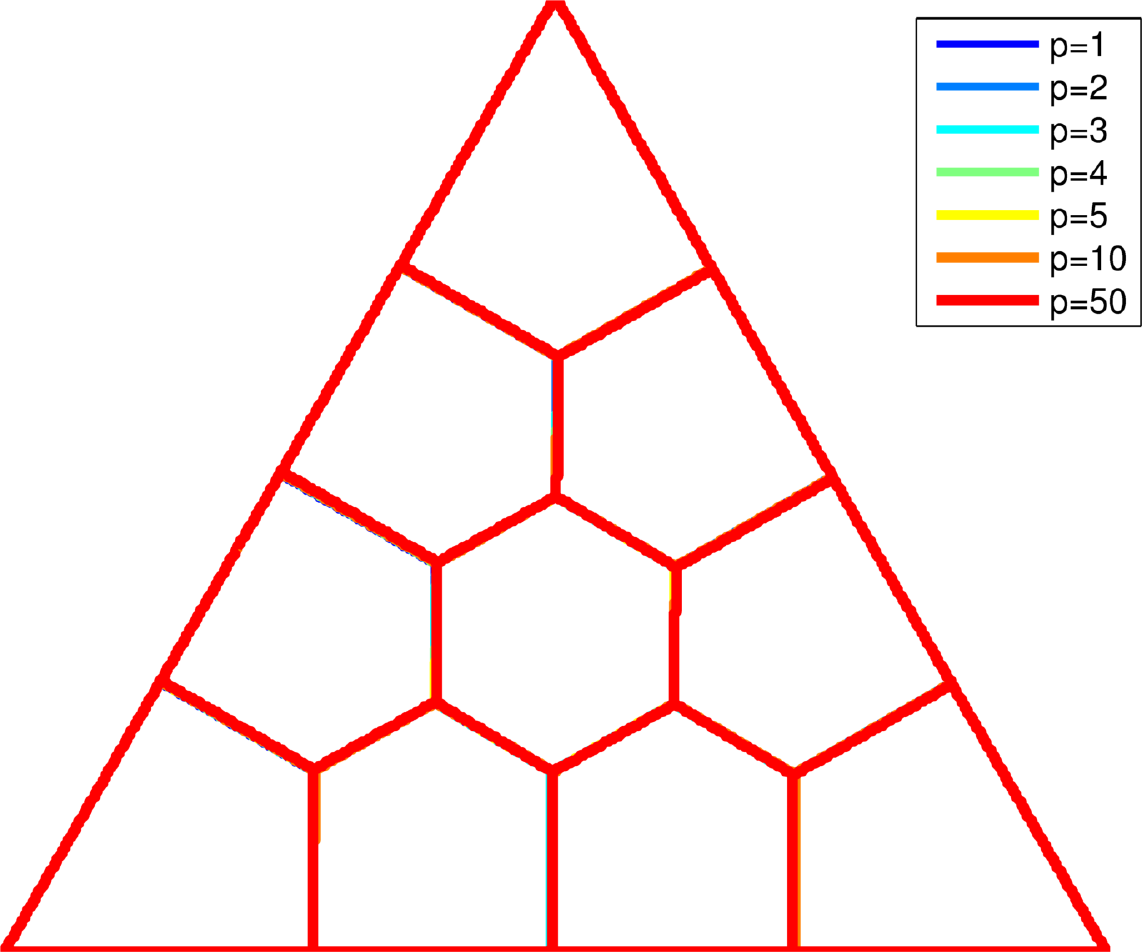}
\caption{$p$-minimal $k$-partitions of the equilateral triangle {\it vs.} $p$, for $k\in\{6,\ 10\}$.}
\label{evol.equi.6-10}
\end{figure}

\subsection{The square.} In cases $k\in\{2,4\}$ we obtain equipartitions starting from $p=1$, which makes the energies and partitions stationary (see Figure \ref{fig.PartNod} where we represent the nodal partition associated with the second and fourth eigenfunctions). 

For $k=3$ we have seen in the previous section that there seem to be different $p$-minimal $3$-partitions for $p=1$ and $p=\infty$. This can also be seen by looking at the evolution of the partitions and of the $p$-norms in Figure \ref{evol.square3}. We clearly see how the triple point approaches the center of the square, represented by a black dot in Figure~\ref{evol.square3}.
\begin{figure}[h!t]
\centering 
\subfigure[$\Bl\Lambda_{3,p}(\cD^{3,p})$, $\Rd\Lambda_{3,\infty}(\cD^{3,p})$ and $\Mg\Lambda^{DN}_{3}(\square)$ vs. $p$\label{fig.sq3nrj}]{\qquad\qquad\includegraphics[width=0.25\textwidth]{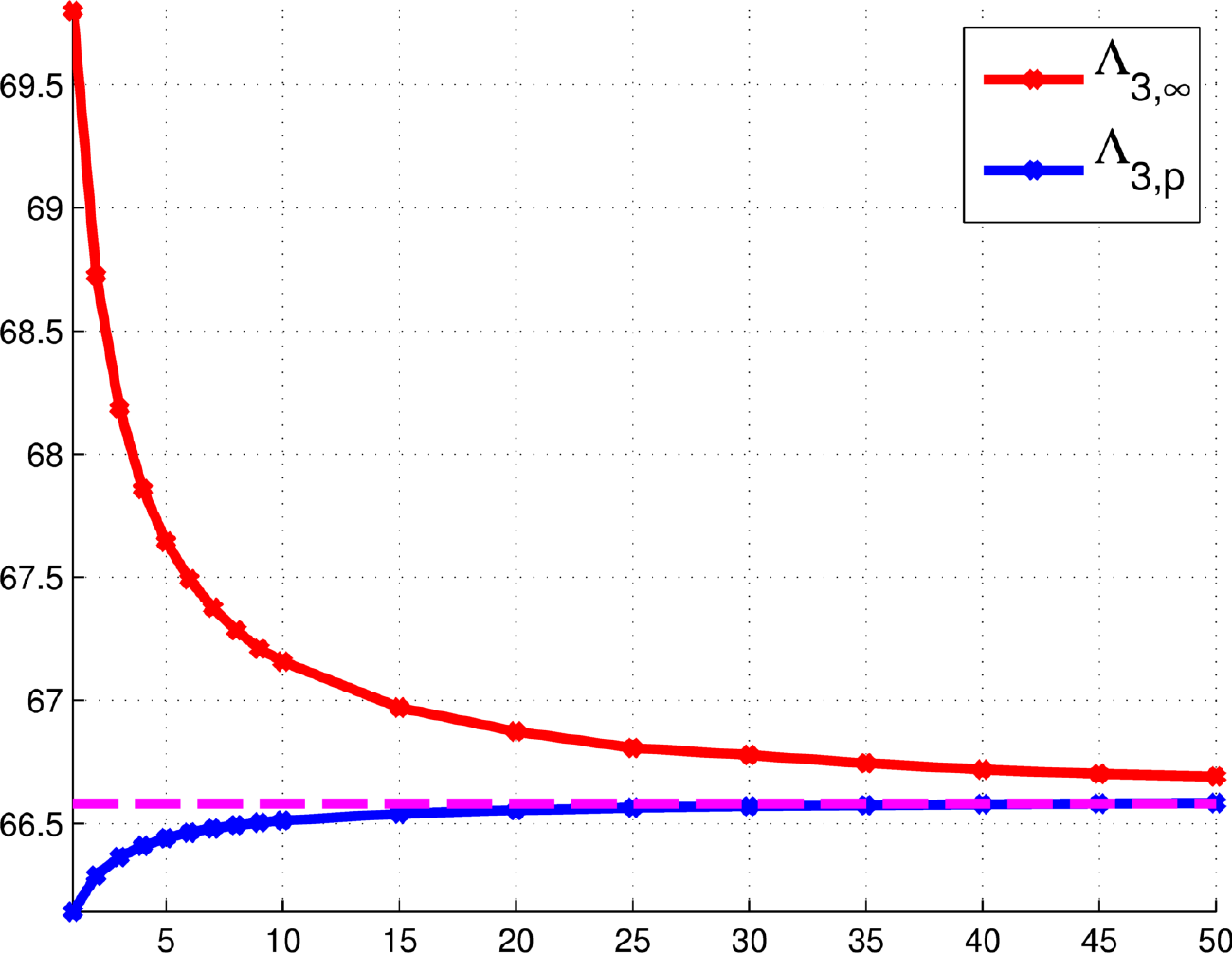}\qquad\quad}
\subfigure[$\cD^{3,p}$ vs. $p$\label{fig.sq3part}]{\qquad\includegraphics[width=0.25\textwidth]{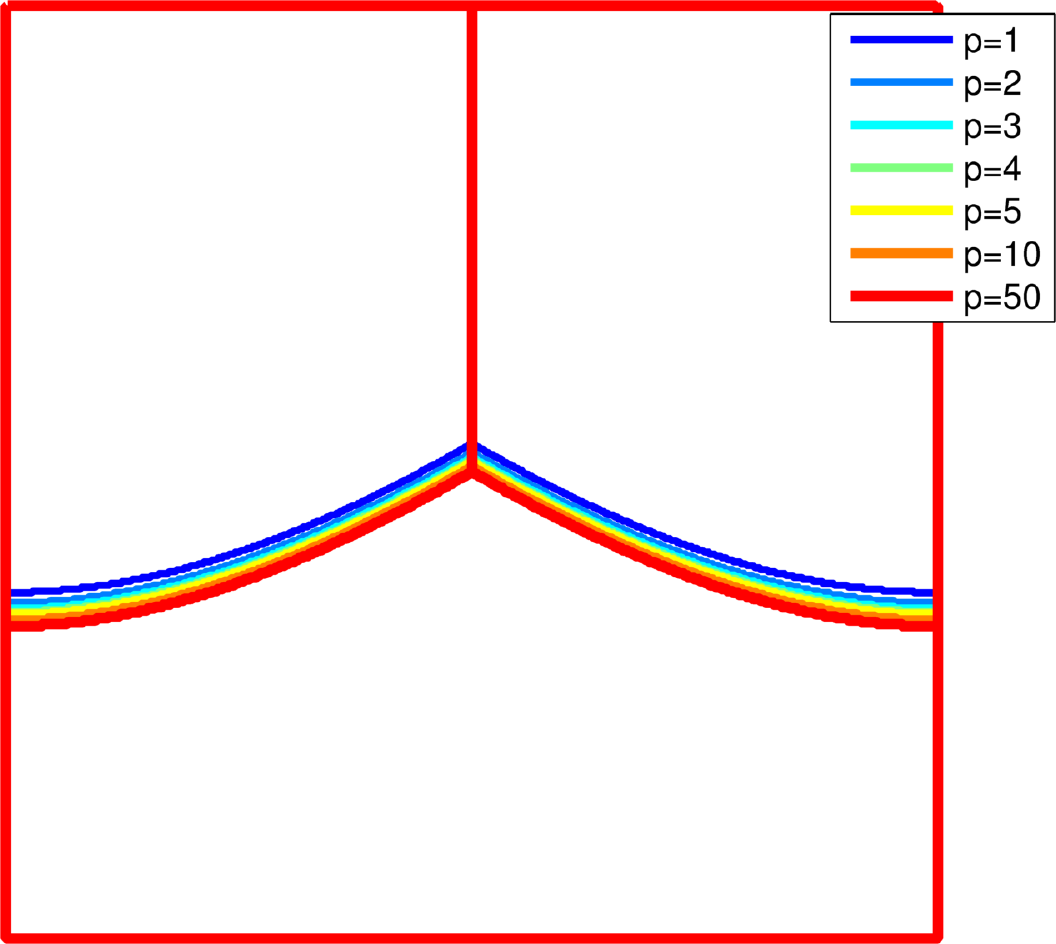}
\includegraphics[width=0.25\textwidth]{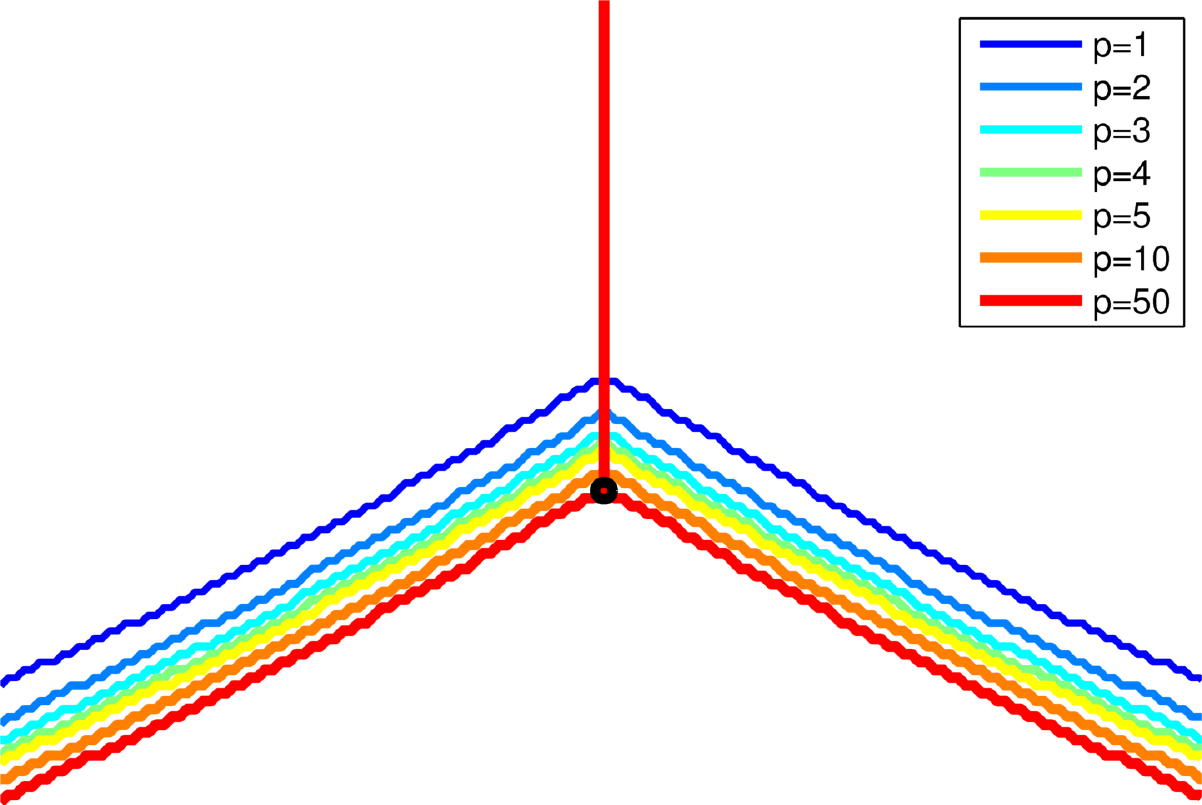}}
\caption{$p$-minimal $3$-partitions of the square {\it vs.} $p$.}
\label{evol.square3}
\end{figure}

Another interesting case is $k=5$. Here we were also able to use a Dirichlet-Neumann approach in order to present an equipartition which is a candidate for minimizing the maximal first eigenvalue. As seen in the previous section the $L^2$ norm criterion does show that the same partition cannot also be optimal for the sum. We observe in Figure \ref{evol.square5} that the energies and the numerical $p$-minimal $5$-partitions evolve when $p$ grows. 
\begin{figure}[h!t]
\begin{center}
\subfigure[$\Bl\Lambda_{5,p}(\cD^{5,p})$, $\Rd\Lambda_{5,\infty}(\cD^{5,p})$ and $\Mg\Lambda^{DN}_{5}(\square)$ vs. $p$\label{fig.sq5nrj}]{\qquad\includegraphics[width=0.3\textwidth]{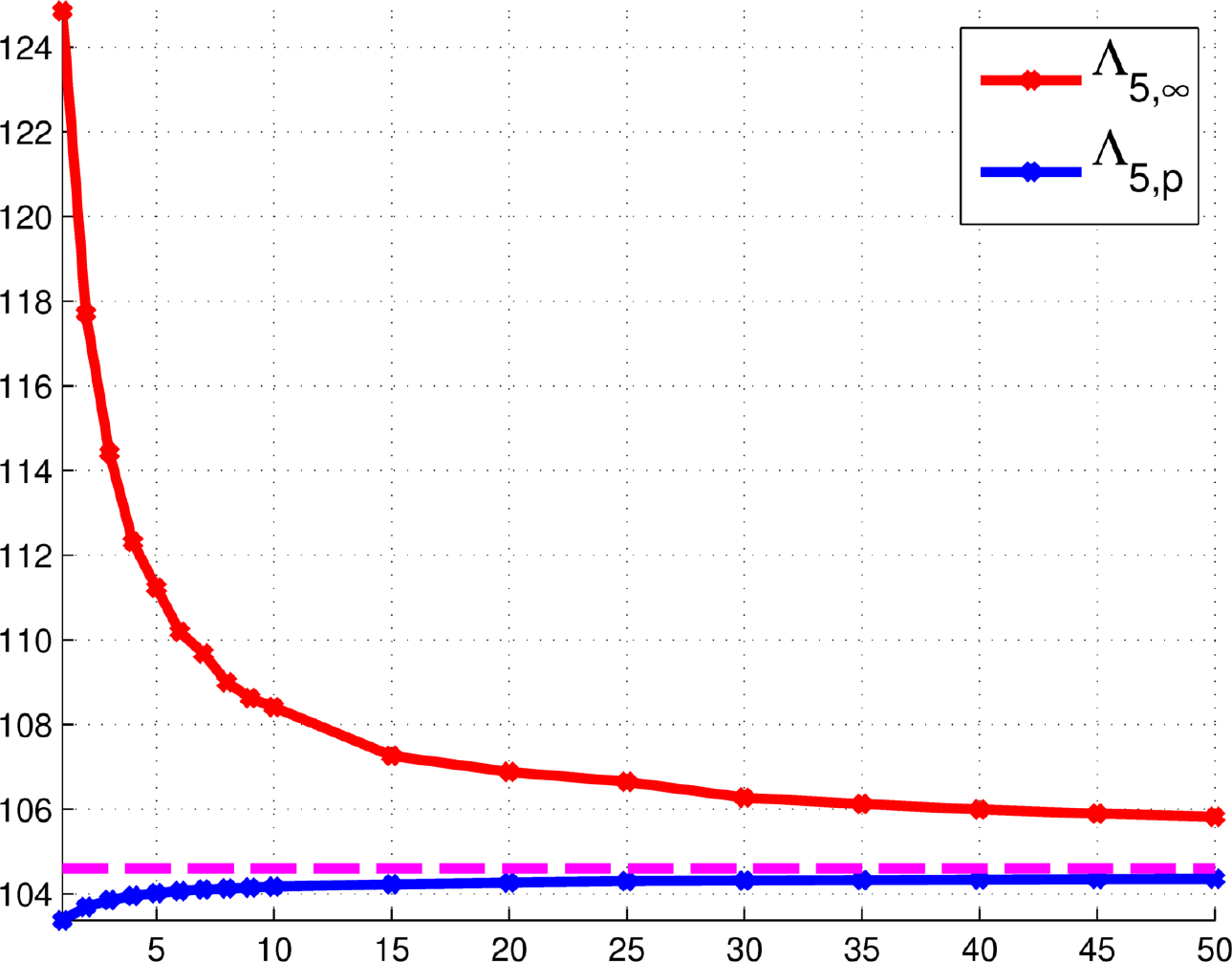}\qquad\qquad}
\subfigure[$\cD^{5,p}$ vs. $p$\label{fig.sq5part}]{\qquad\includegraphics[width=0.26\textwidth]{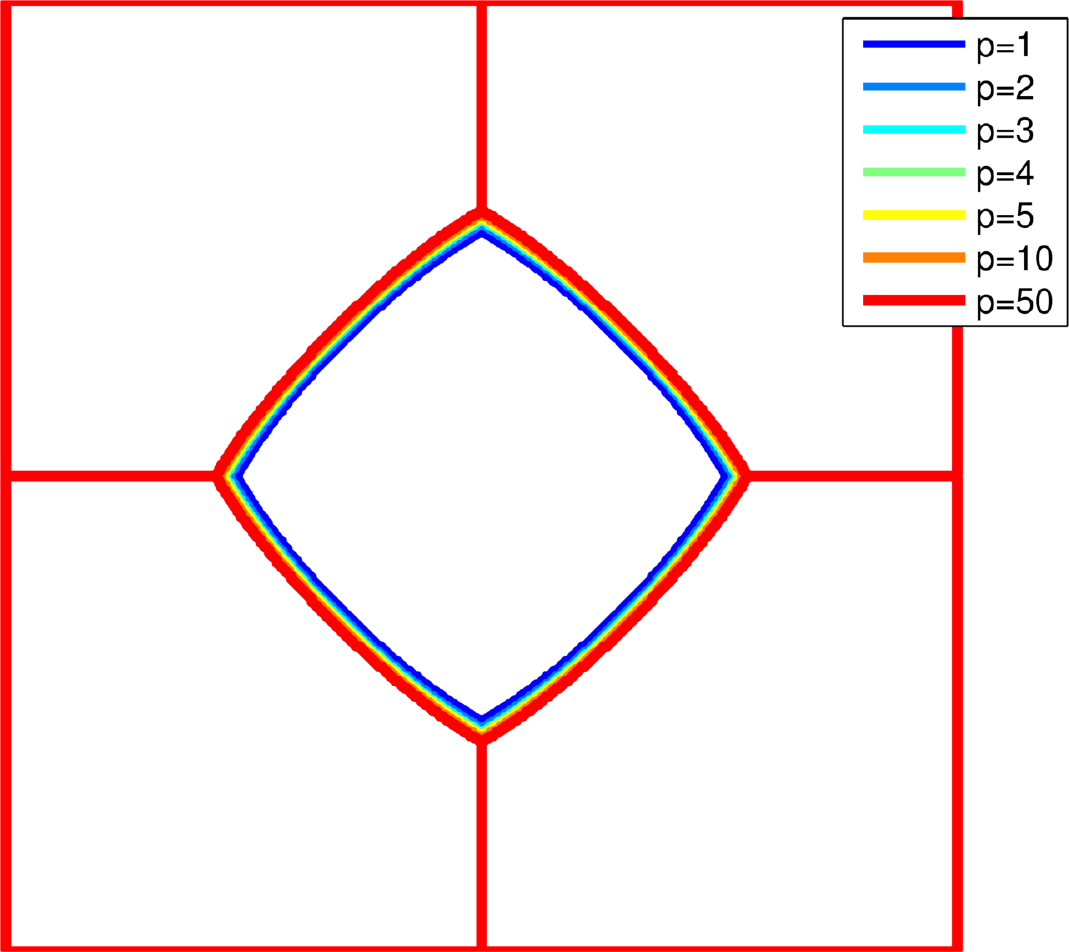}\qquad}
\caption{$p$-minimal $5$-partitions of the square {\it vs.} $p$.}
\label{evol.square5}
\end{center}
\end{figure}

When $k \in \{6,8,9,10\}$ we observe in Figure \ref{evol.square6-10} similar behaviors in the evolution of the energies and the numerical $p$-minimal $k$-partitions.
\begin{figure}[h!t]
\begin{center}
\includegraphics[width=0.2\textwidth]{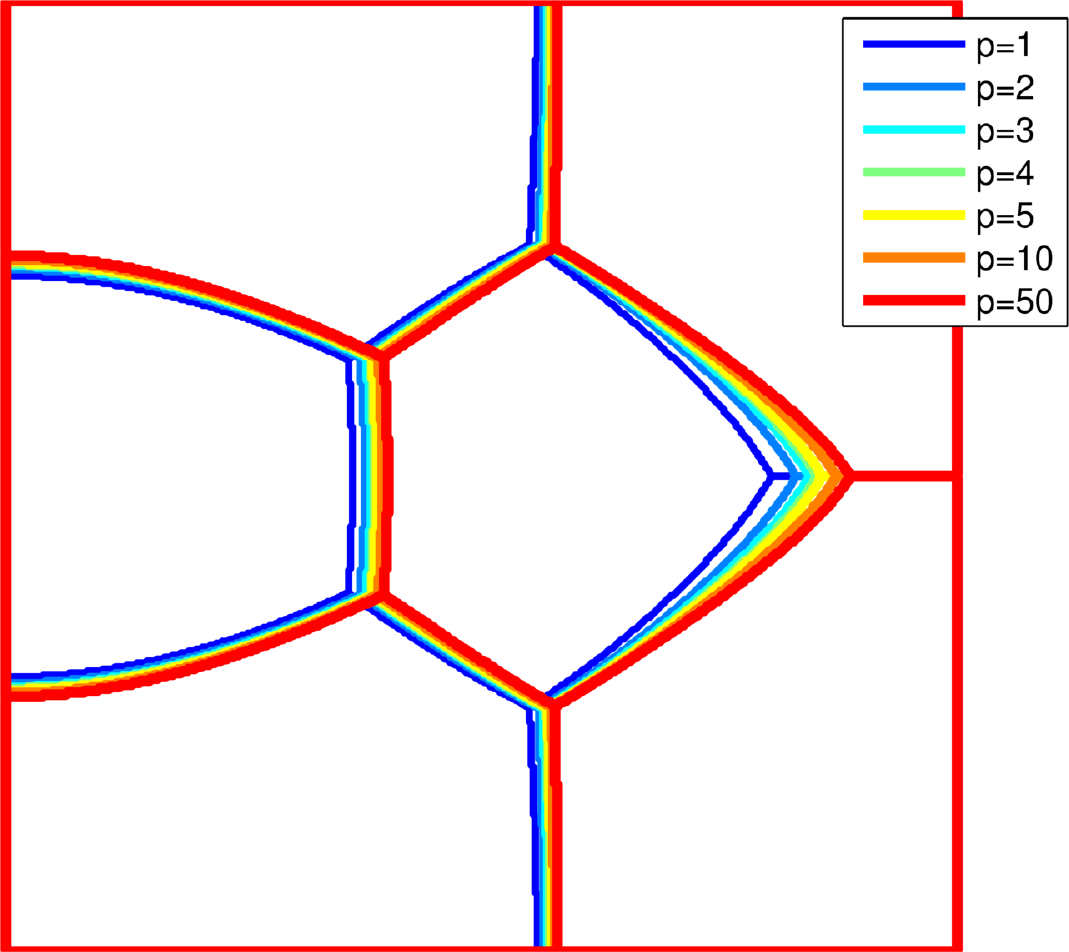}\qquad
\includegraphics[width=0.2\textwidth]{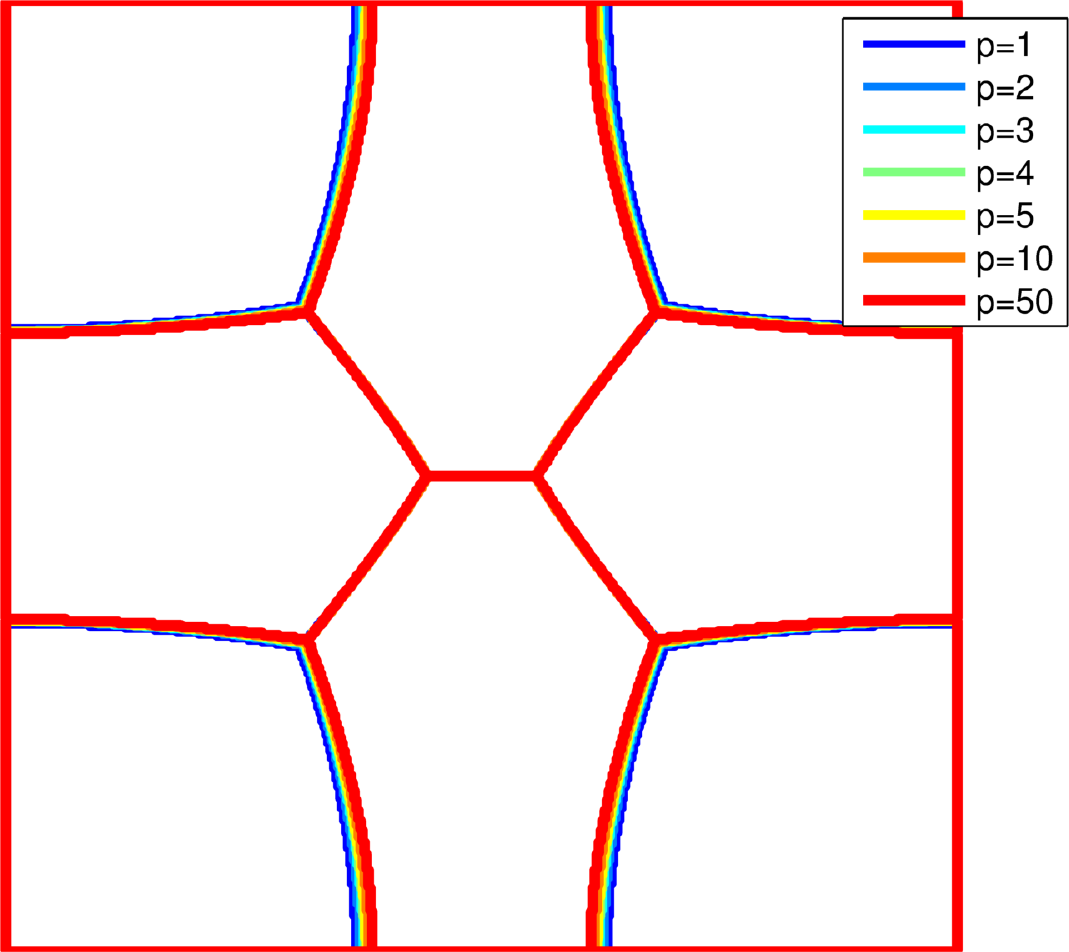}\qquad
\includegraphics[width=0.2\textwidth]{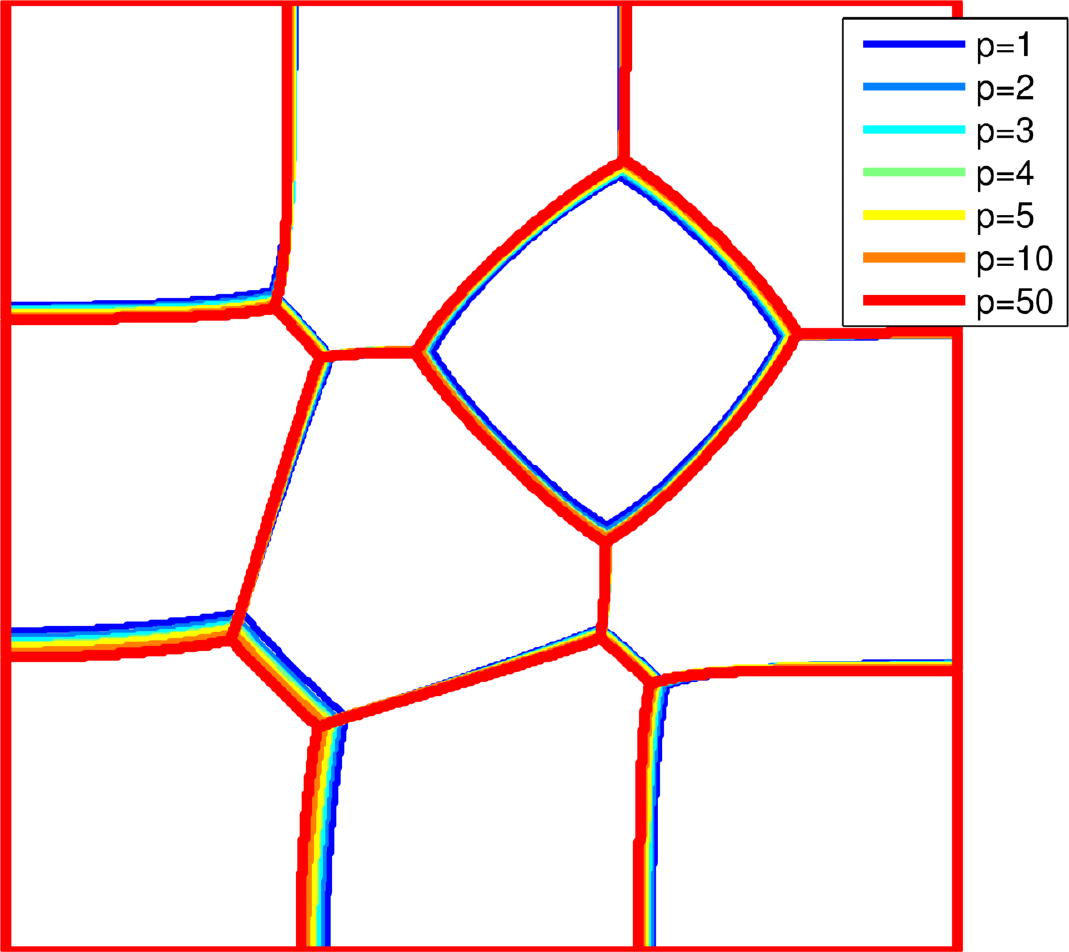}
\caption{$p$-minimal $k$-partitions of the square {\it vs.} $p$, for $k\in\{6,8,10\}$.}
\label{evol.square6-10}
\end{center}
\end{figure}

We mention that for $k=9$ the largest eigenvalue of the partition we obtain for $p=50$ is not smaller than the partition into $9$ equal squares. On the other hand, we know that the partition into $9$ equal squares is not optimal for $p=\infty$ since it is a nodal partition which is not Courant sharp (see Proposition \ref{courant_sharp} and \cite{MR3445517} for more details). 
For reference, the maximal eigenvalue obtained for $p=50$ is $179.21$ (and $178.08$ with the penalized method) and the first eigenvalue of a square of side $1/3$ is $L_{9}(\square)=177.65$. The evolution of the energies and the numerical $9$-partition can be viewed in Figure~\ref{evol.square9}. {\Gn The analysis of the splitting of critical points of odd order in \cite{BonLen16} lets us thinking that a partition with two close triple points can be a rather good candidates.}
\begin{figure}[h!t]
\begin{center}
\subfigure[$\Bl\Lambda_{9,p}(\cD^{9,p})$, $\Rd\Lambda_{9,\infty}(\cD^{9,p})$ and $\Mg L_{9}(\square)$ vs. $p$\label{fig.sq9nrj}]{\qquad\includegraphics[width=0.3\textwidth]{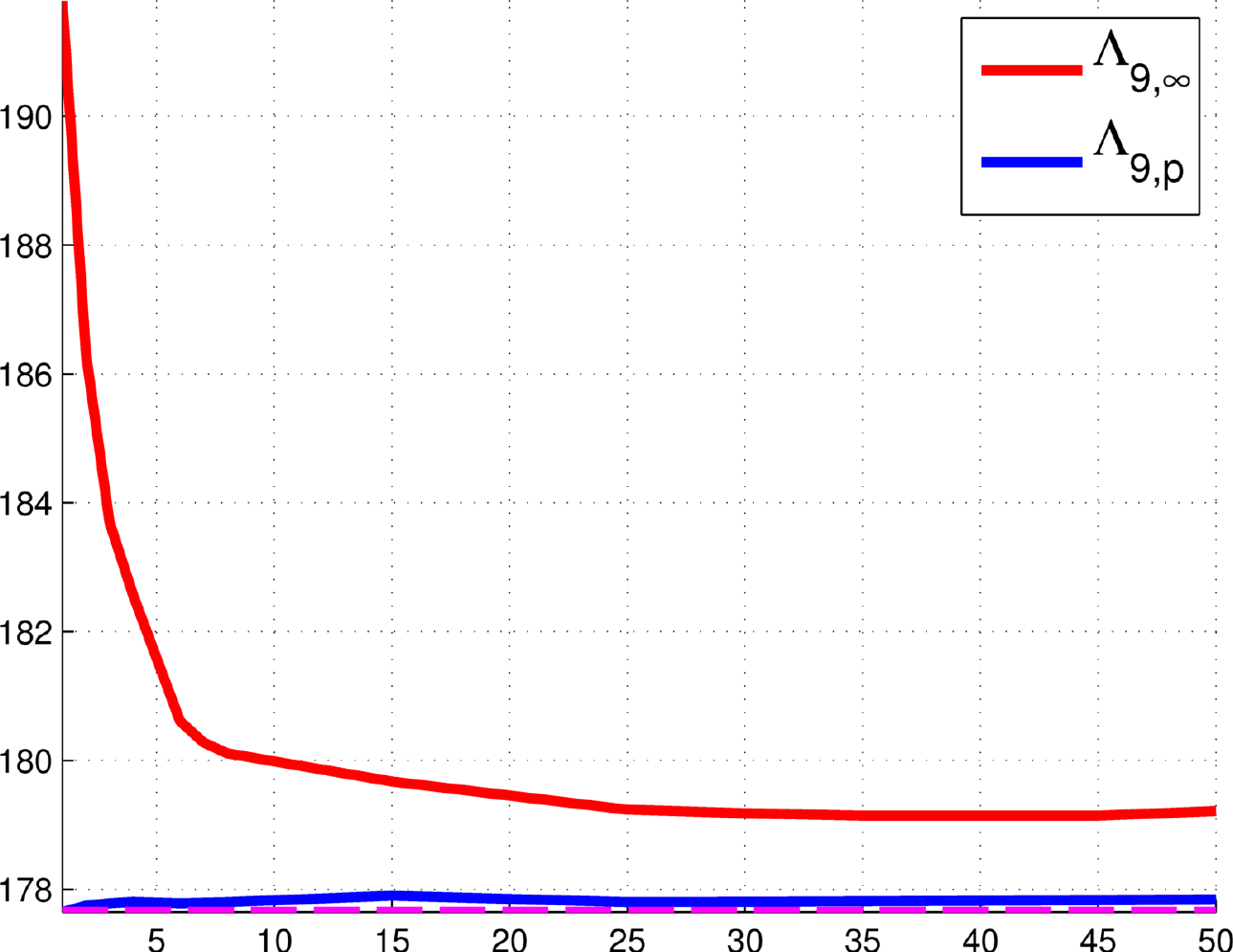}\qquad}
\subfigure[$\cD^{9,p}$ vs. $p$]{\qquad\includegraphics[width=0.32\textwidth]{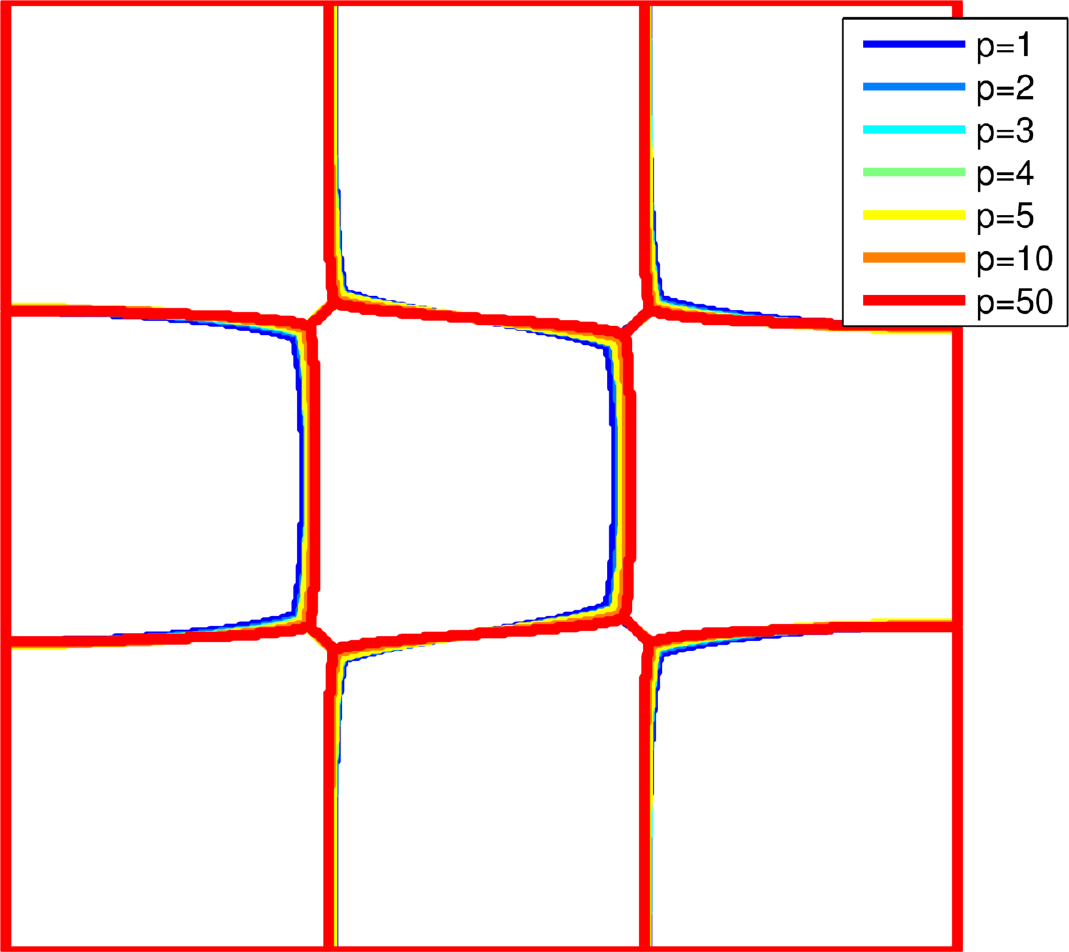}\qquad}
\caption{$p$-minimal $9$-partitions of the square {\it vs.} $p$.}
\label{evol.square9}
\end{center}
\end{figure}

Something different happens for $k=7$ where we have two configurations which have close energies at optimum. We represent the partitions of the two configurations in Figure \ref{compar_sq7} along with a comparison of the maximal eigenvalues and $p$-norms. We can see that while the first configuration has a lower maximal eigenvalue for large $p$, the first configuration always has a lower $p$-norm. We note that when we use the penalization method we find the first configuration which is consistent with the results obtained with the $p$-norm. We remark that these configurations we obtain are similar to the ones presented in \cite{cybhol}. Still, the small differences we observe for the maximal eigenvalues and the $p$-norms may be due to our limited numerical precision. In order to conclude which of these partition is better than the other we would need to use some more refined methods which do not use relaxations.
\begin{figure}[h!t]
\centering 
\subfigure[$\Lambda_{k,p/\infty}(\cD^{5,p}_{j})$ vs. $p$ for $ j={\Bl 1},\ \Mg 2$ ]{\qquad\includegraphics[width=0.28\textwidth]{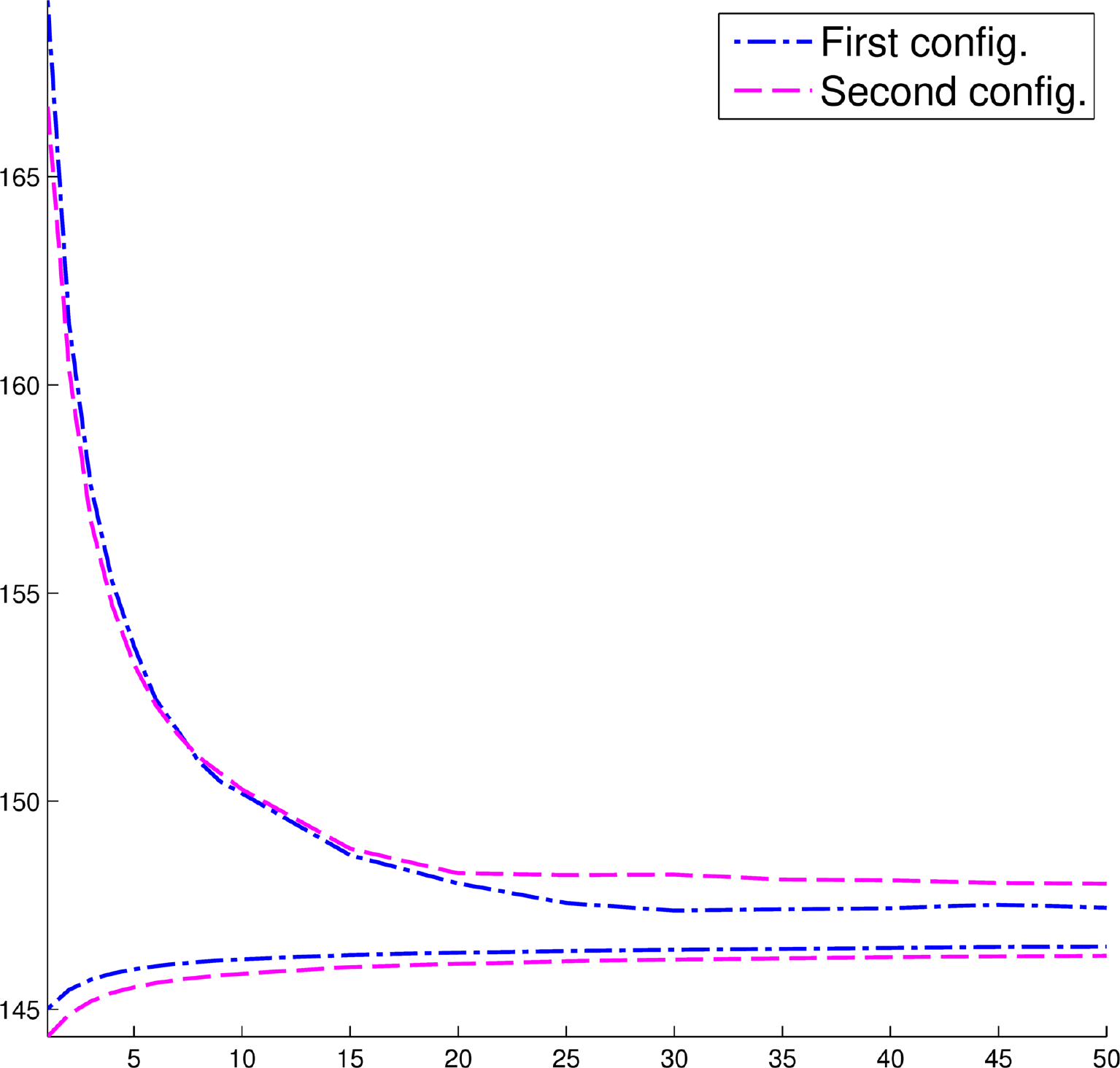}\qquad}
\subfigure[$\cD^{5,p}_{1}(\square)$ vs. $p$\label{compar_sq7b}]{\includegraphics[width=0.28\textwidth]{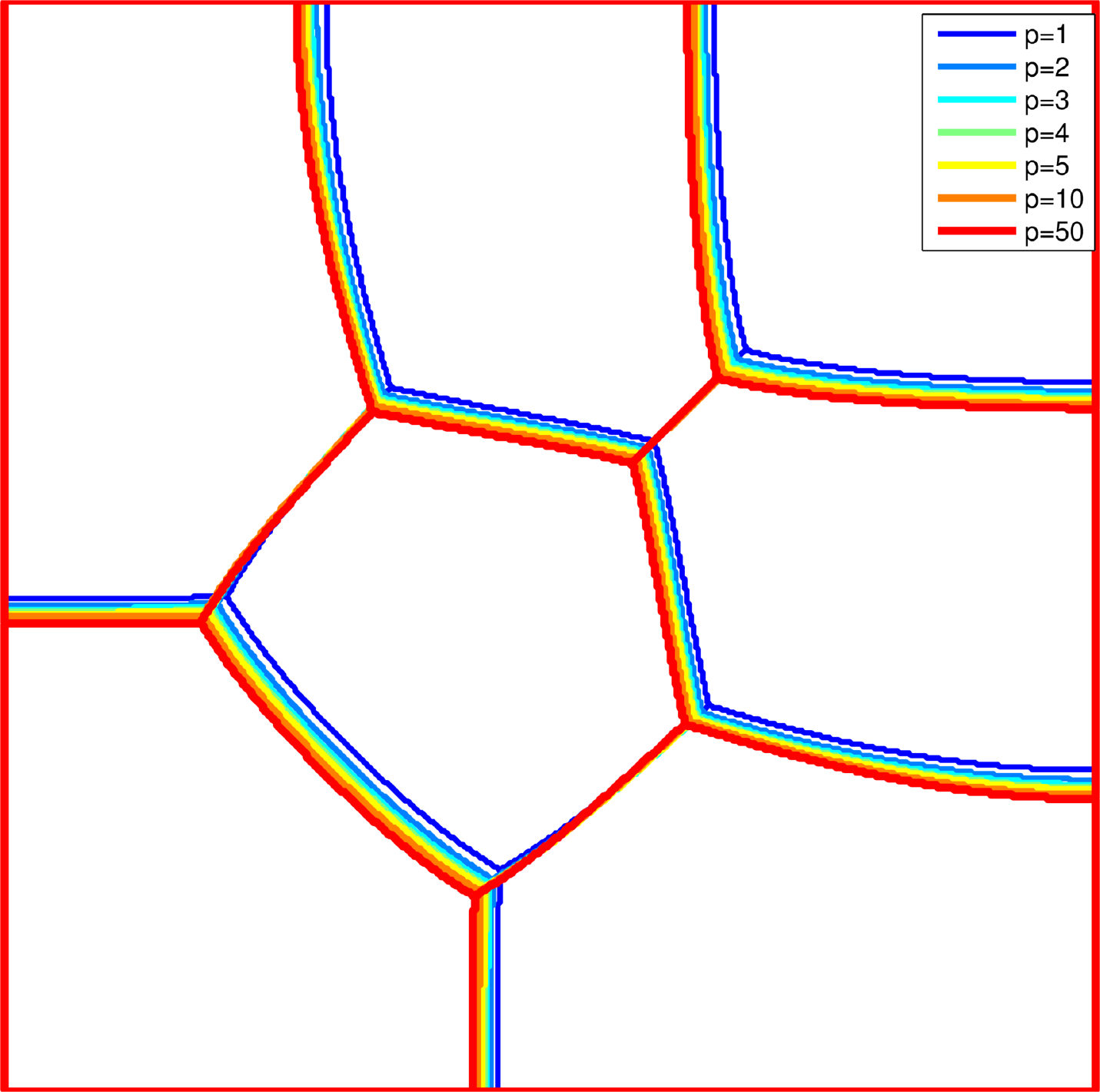}}\quad
\subfigure[$\cD^{7,p}_{2}(\square)$ vs. $p$\label{compar_sq7c}]{\includegraphics[width=0.28\textwidth]{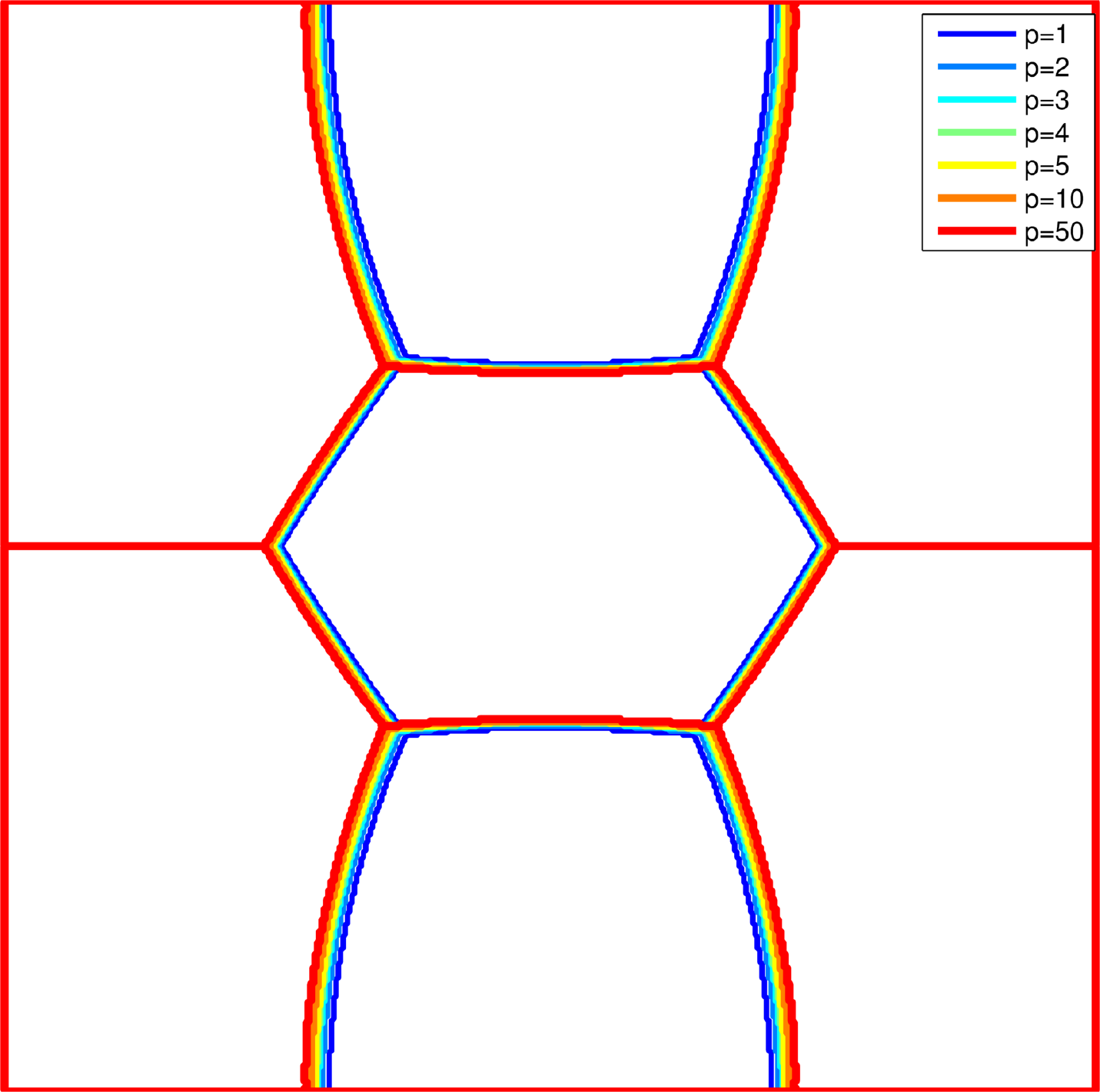}}
\caption{Comparison of two candidates for $7$-partition of the square.}
\label{compar_sq7}
\end{figure}

\subsection{The disk.} 
In this case for $k \in \{2,3,4,5\}$ we obtain numerically that $\mathfrak{L}_{k,p}(\Circle)$ is minimized by $k$ equal sectors starting from $p=1$. In such cases, where we obtain an equipartition when optimizing the sum, the optimal partition is the same for all $p$ and the $p$-norm does not vary with $p$. The partitions can be visualized in Figure \ref{fig.disk}.

In cases $k \in \{6,8,9\}$ we obtain for every $p$ partitions consisting of a rounded regular polygon with $k-1$ sides surrounded by $k-1$ equal subsets of a sector of angle $2\pi/(k-1)$. In these cases we may see clearly how the optimal partition evolves with $p$. For $k \in \{6,8,9\}$ we have seen in the end of the previous section that there seem to be different optimal partitions for the sum and for the max. The evolution of the partitions is represented in Figure \ref{evol.disk6-9}. For $k=10$ the best candidate is obtained with the iterative method. The evolution of the $p$-norm of eigenvalues and of the maximal eigenvalue with respect to $p$ is presented in Figure \ref{disk.evol}. We may see that the candidate found for the sum is not optimal for the max since the maximal eigenvalue strictly decreases with respect to $p$.
\begin{figure}[h!t]
\centering 
\subfigure[$k=6$]{\begin{tabular}{c}\includegraphics[width = 0.2\textwidth]{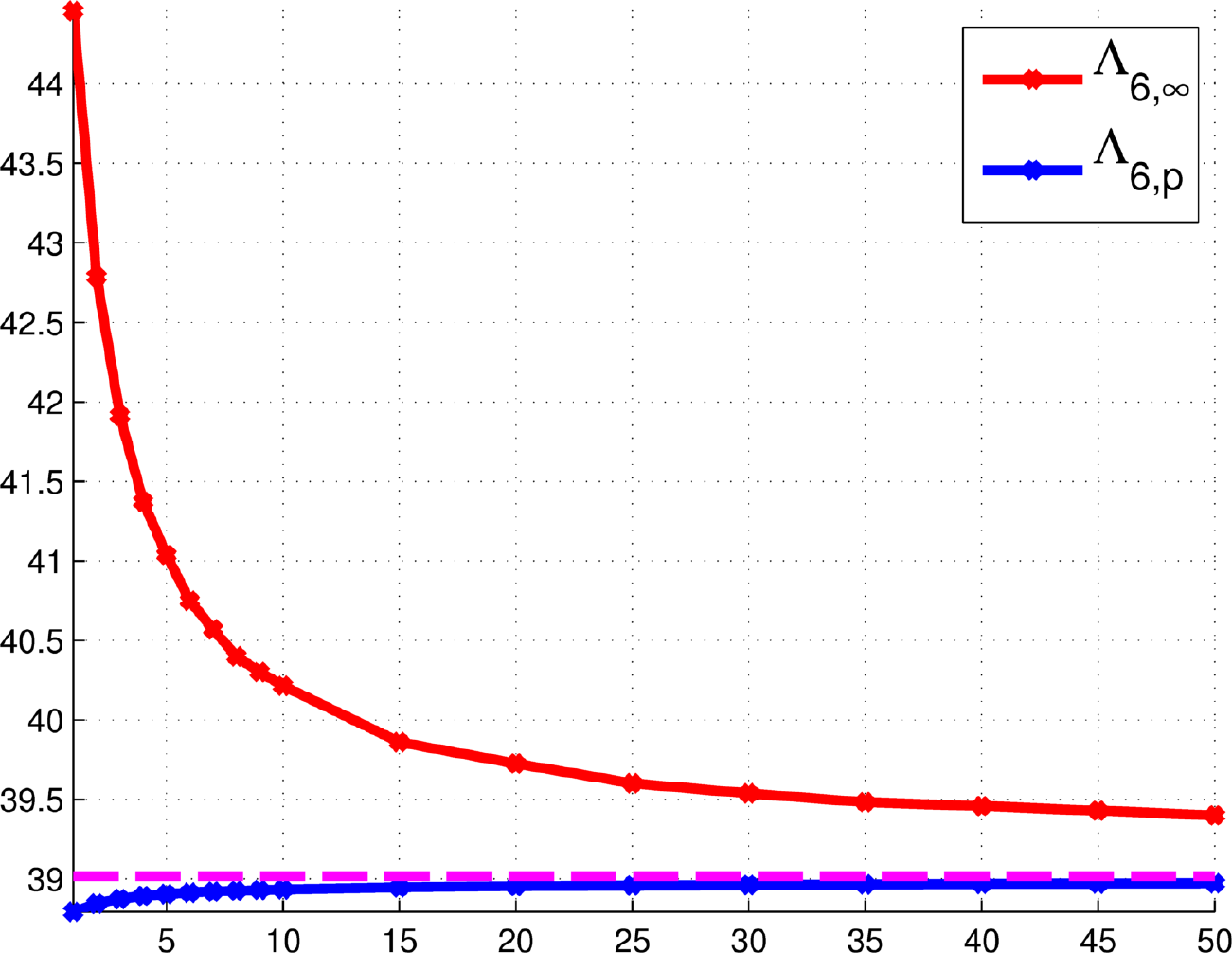}\\
\includegraphics[width=0.2\textwidth]{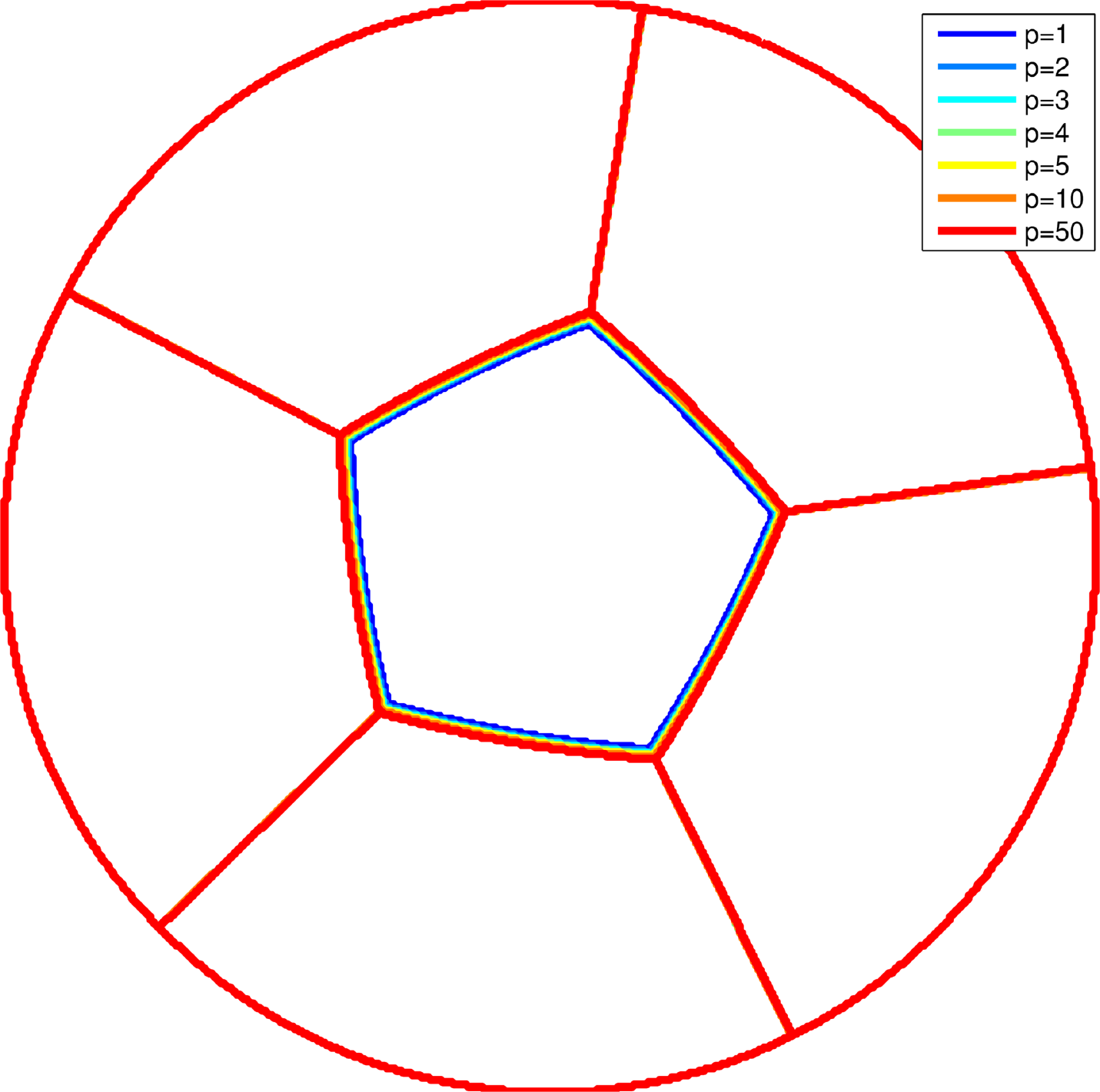}\end{tabular}}
\subfigure[$k=7$\label{evol.disk7}]{\begin{tabular}{c}\includegraphics[width = 0.16\textwidth]{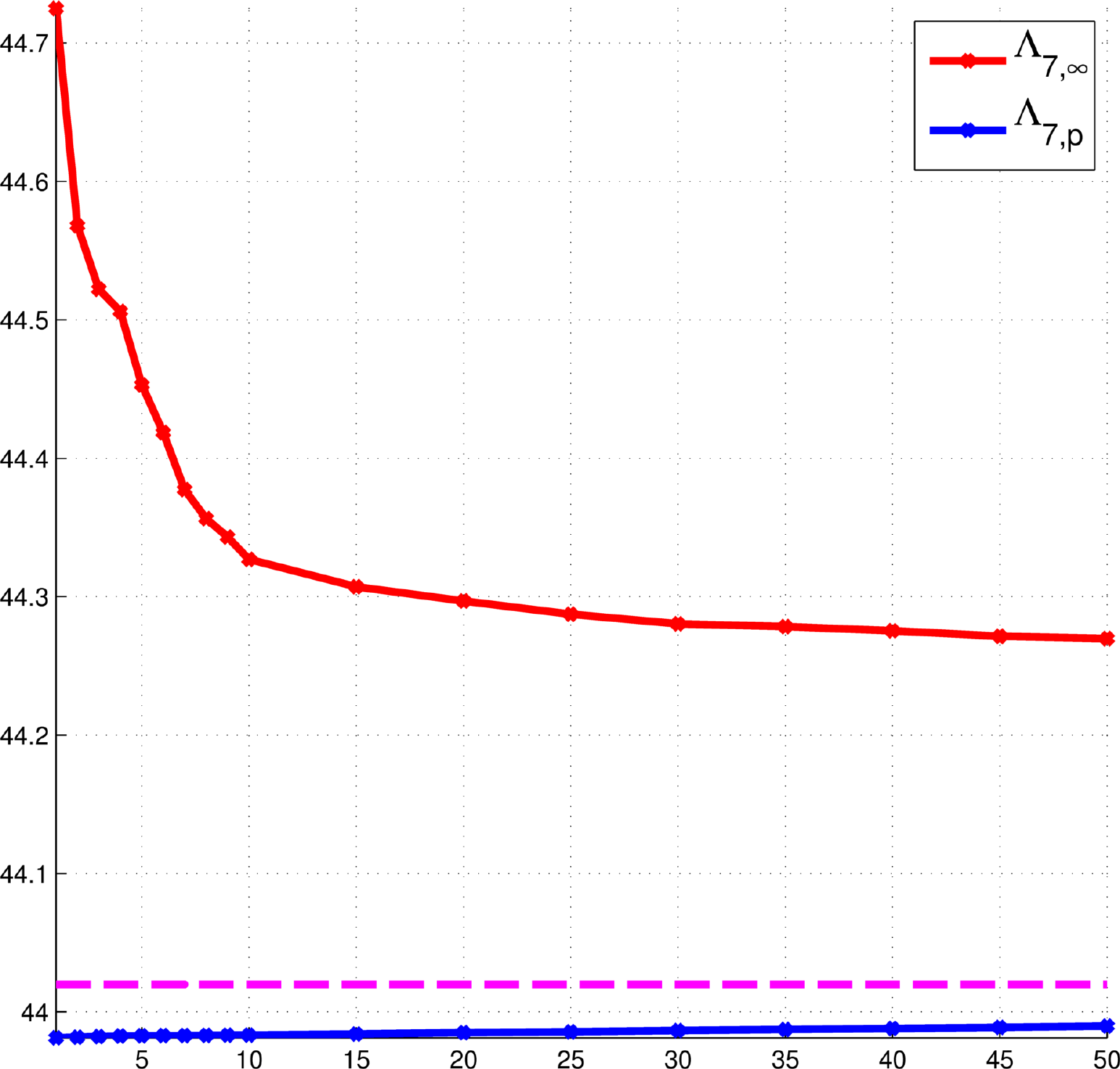}\\
\includegraphics[width=0.2\textwidth]{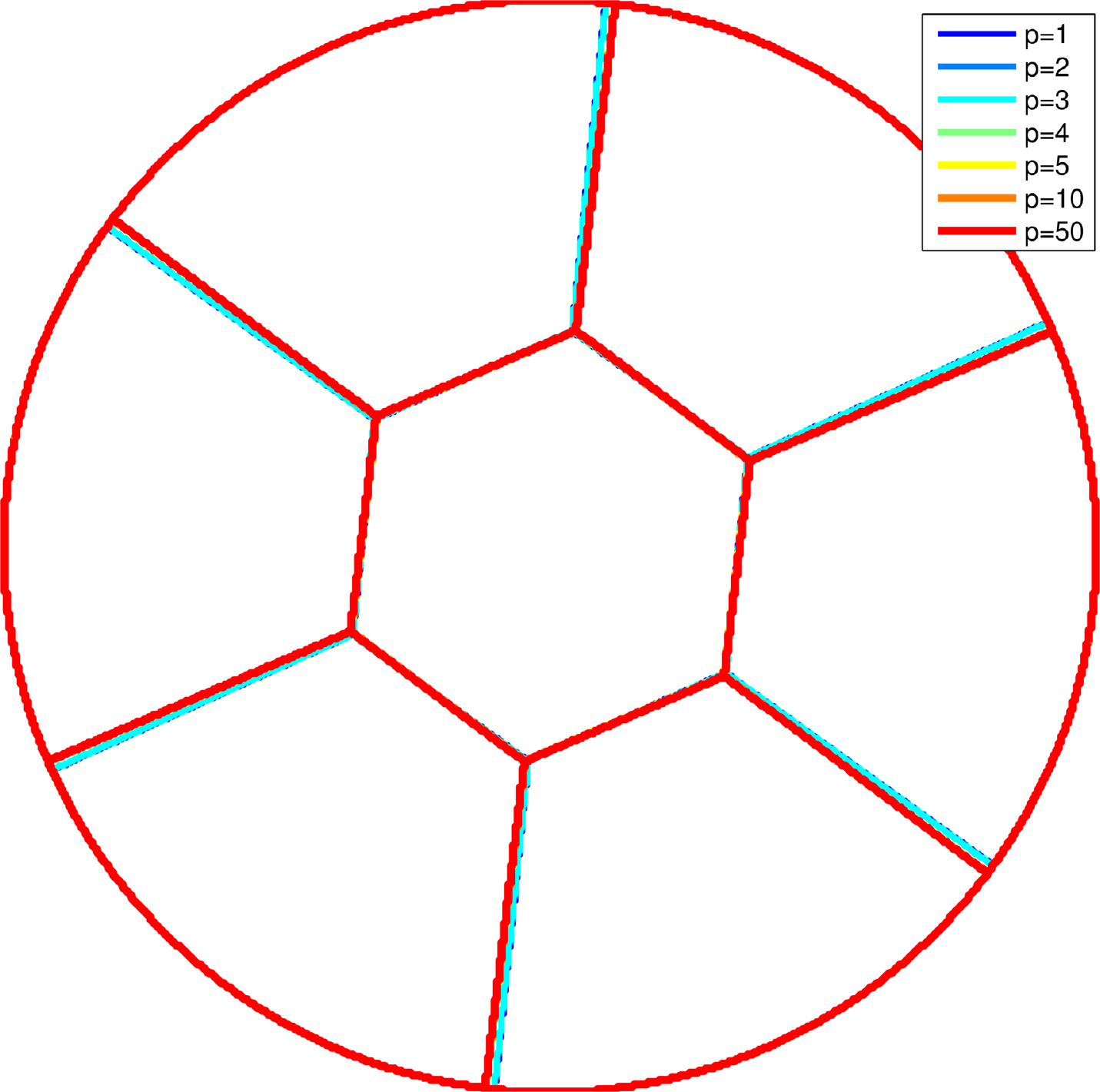}\end{tabular}}
\subfigure[$k=8$]{\begin{tabular}{c}\includegraphics[width = 0.2\textwidth]{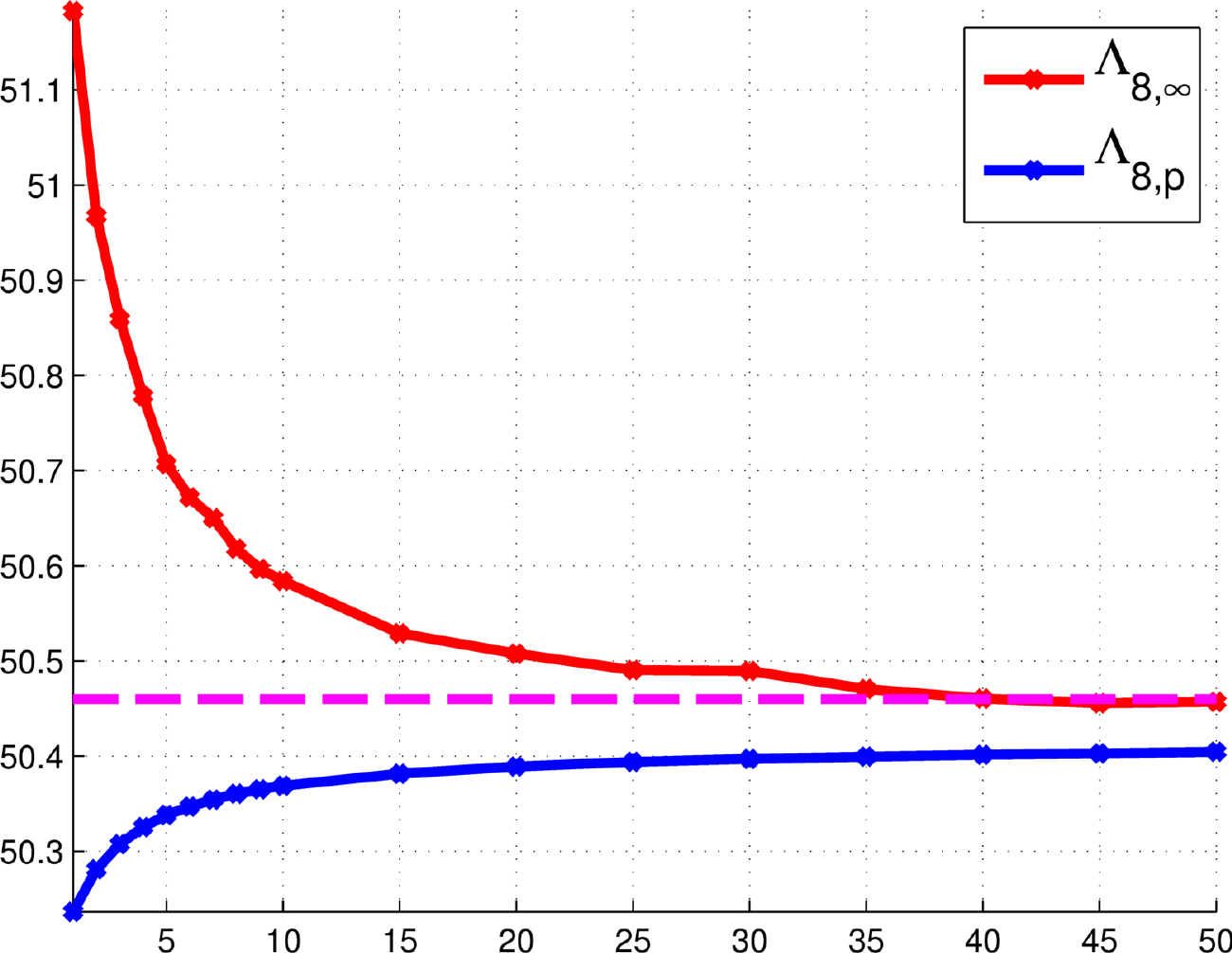}\\
\includegraphics[width=0.2\textwidth]{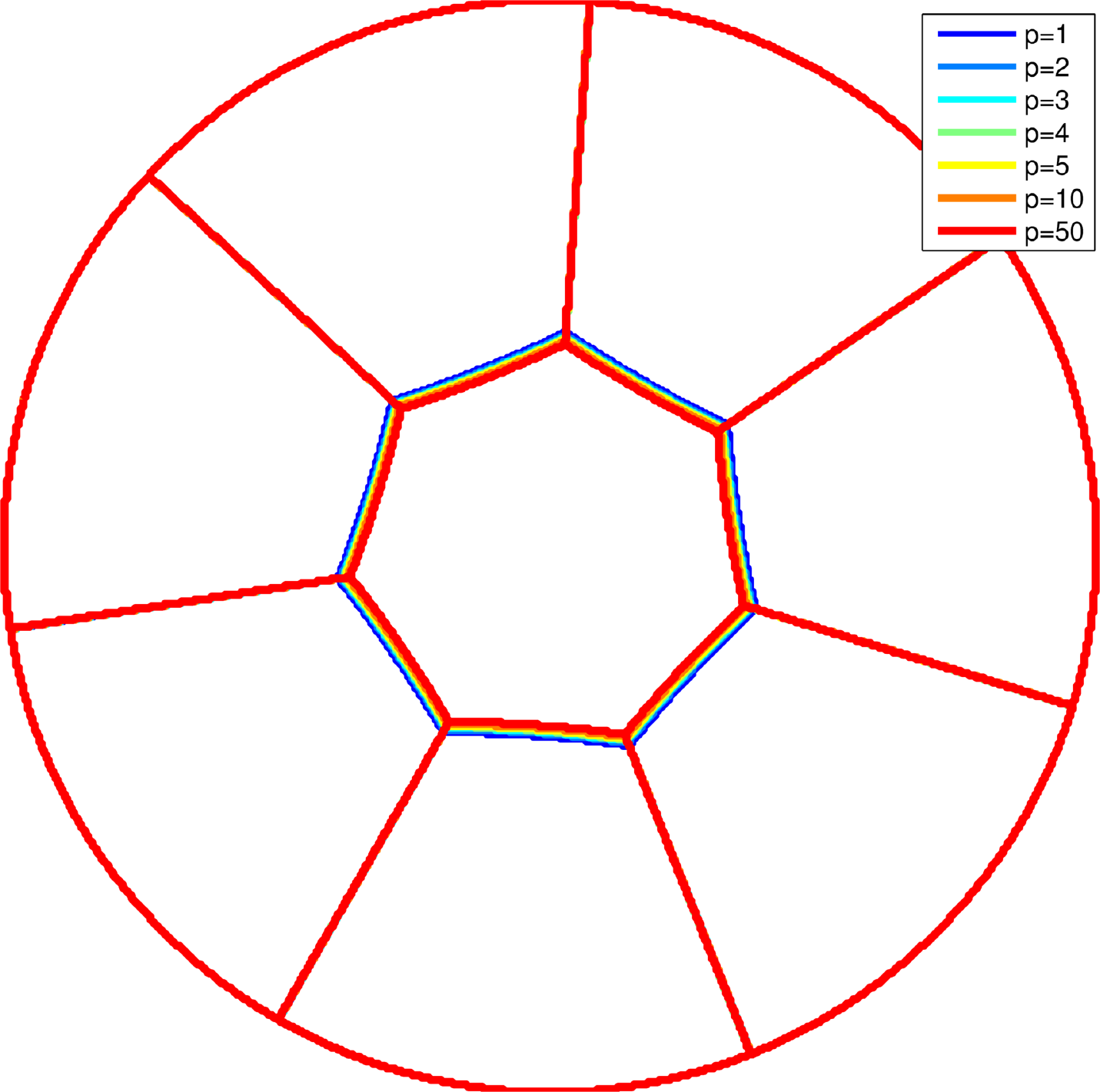}\end{tabular}}
\subfigure[$k=9$]{\begin{tabular}{c}\includegraphics[width = 0.2\textwidth]{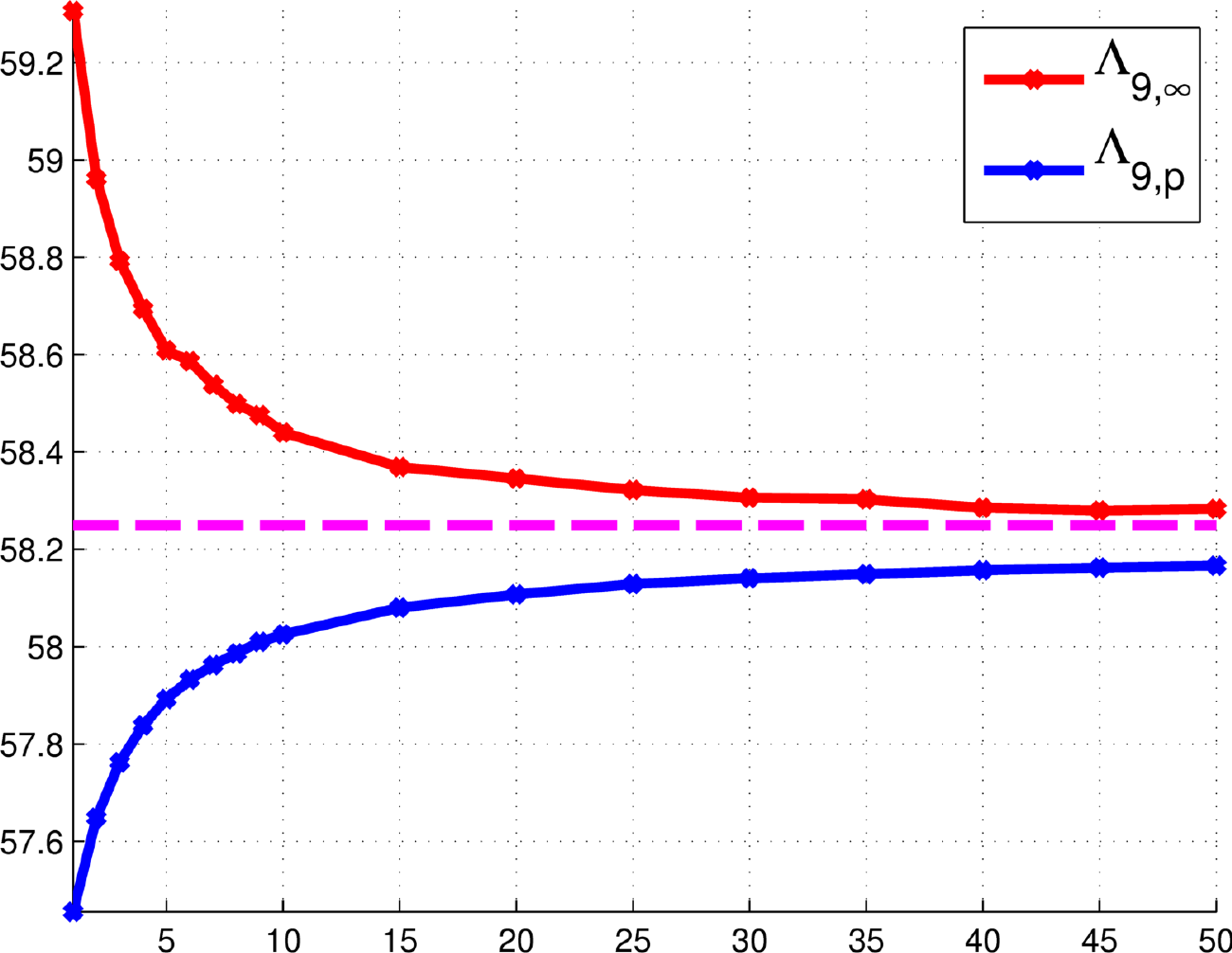}\\
\includegraphics[width=0.2\textwidth]{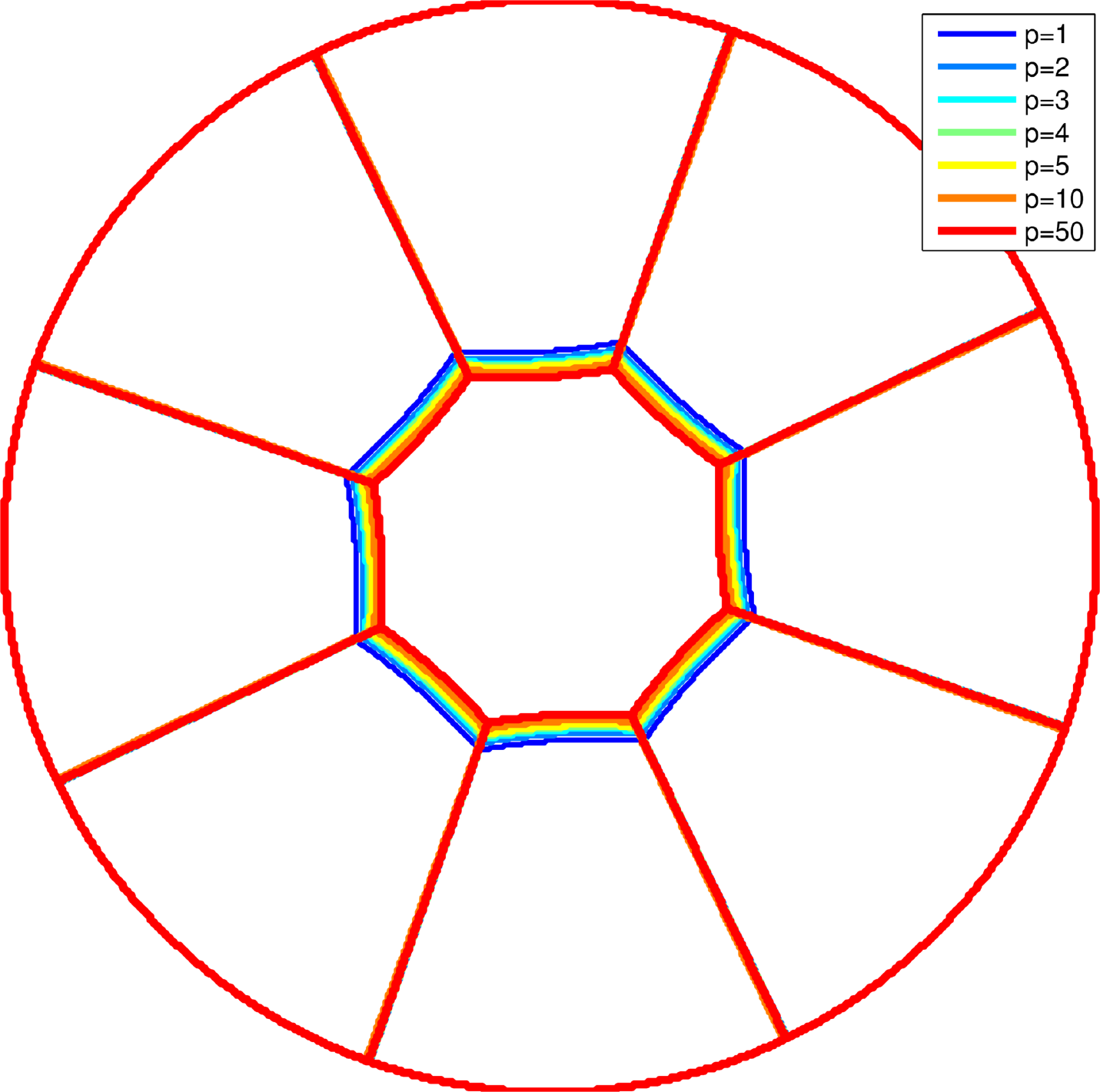}\end{tabular}}
\caption{$p$-minimal $k$-partitions of the disk {\it vs.} $p$, for $k\in\{6,7,8,9\}$.}
\label{evol.disk6-9}
\end{figure}

\begin{figure}[h!t]
\centering 
\includegraphics[width = 0.3\textwidth]{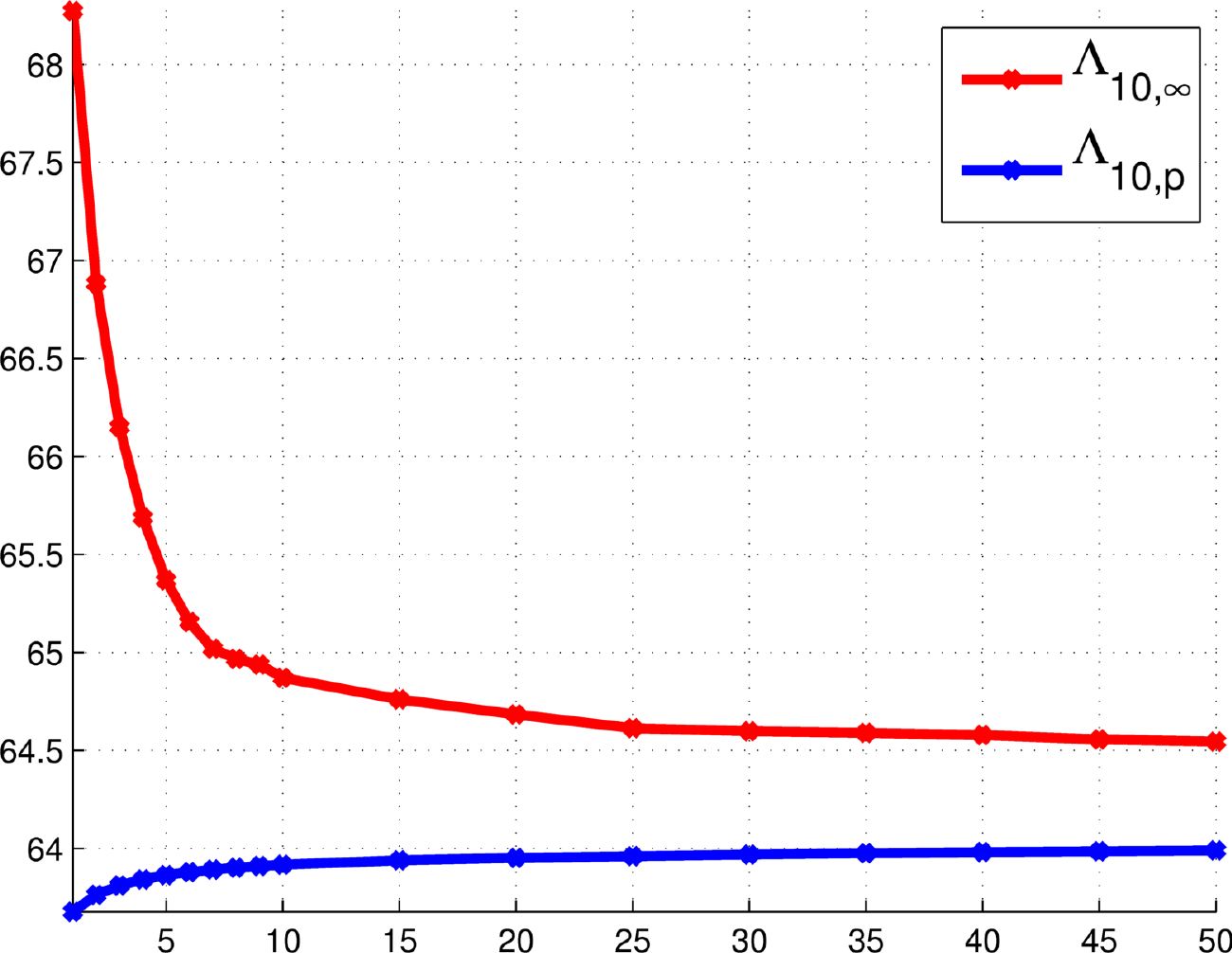}\qquad
\includegraphics[width = 0.27\textwidth]{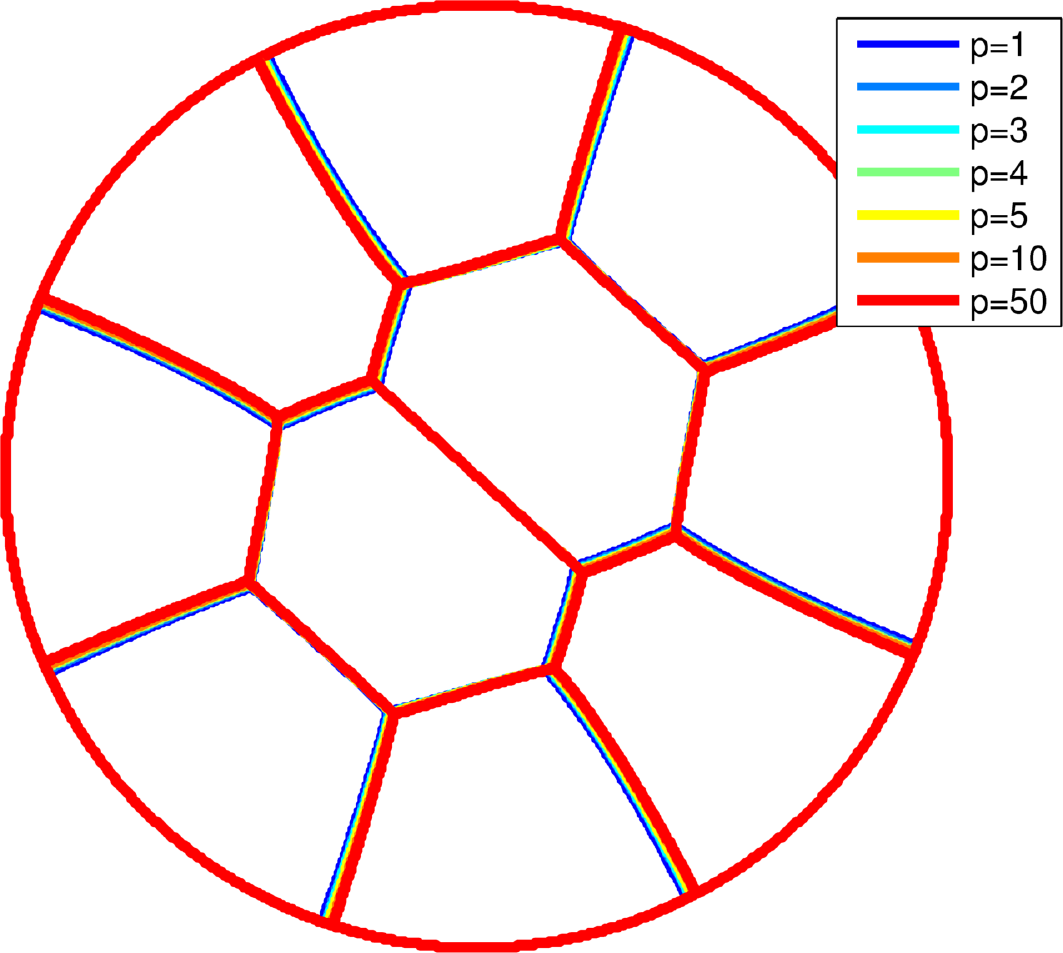}
\caption{$p$-minimal $10$-partitions of the disk {\it vs.} $p$.}
\label{disk.evol}
\end{figure}

For $k=7$ the optimal partition seems to be made out of a regular hexagon and $6$ equal sector portions. The partitions and the evolution of the energies is depicted in Figure \ref{evol.disk7}. Even if the evolution of the energies and the partitions is not as evident as in the other cases we see that the maximal eigenvalue decreases with $p$ and this seems to indicate, like in the analysis performed at the end of the previous section, that the partitions minimizing the sum and the max are not the same.

\section{Conclusion}
\label{section.conclusion}

{\Gn
We constructed three numerical methods in order to analyze the behavior of the minimal spectral partitions as $p$ varies from $1$ to $\infty$. We apply these algorithms to the study of three particular geometries, the square, the equilateral triangle and the disk. We underline, however, that methods described here could be applied on other general geometries.

Our computations allowed us to observe several conjectures, regrouped below.
\begin{conjecture}\label{conj}
\begin{enumerate}[leftmargin=*]
\item When $\Omega$ is a disk and $k \in \{2,3,4,5\}$ the optimal energy $\fL_{k,p}(\Omega)$ is constant with respect to $p$ and the optimal partition consists of $k$ angular sectors of opening $2\pi/k$ (see Figure~\ref{fig.disk}).
\item When $\Omega$ is a square and $k \in \{2,4\}$ the optimal energy $\fL_{k,p}(\Omega)$ is constant with respect to $p$. {\Gn There is a family of minimal $2$-partition for the square and among them the partition} given by two equal rectangles or two equal right-isosceles triangles (see Figure \ref{fig4.tri2}). For $k=4$ the minimal partition is composed of 4 squares (see Figure \ref{fig4.tri4}). 
\item\label{trieq} When $\Omega$ is an equilateral triangle and $k$ is a triangular number, that is to say of the form $k=n(n+1)/2$ with $n\geq2$, we observe again that $\fL_{k,p}(\Omega)$ is constant with respect to $p$. In this case, the minimal $k$-partition consists of 3 equal quadrilaterals, $3(n-2)$ pentagons and $(n-2)(n-3)/2$ regular hexagons (see Figures~\ref{fig.candtrik3} and \ref{evol.equi.6-10}). 
\end{enumerate}
\end{conjecture}

{\Gn We notice that the third point of this conjecture is in accord with the honeycomb conjecture, since for $n$ large we obtain that all cells inside the equilateral triangle are regular hexagons. Several computations for $k=15,21,28,36$ show a similar behavior (see Figure \ref{triangular}).

Concerning the evolution of the minimal $p$-norm energy as $p$ grows, we have seen different behaviors according to $\Omega$, $k$ or $p$. 
It seems that either the energy $\fL_{k,p}(\Omega)$ is constant with $p$ and there exists a $k$-partition which is  $p$-minimal for any $p\geq1$ (see the cases recalled above), or the energy $\fL_{k,p}(\Omega)$ is strictly increasing with $p$. { Recalling the notation introduced in} \eqref{eq.p}, this writes}\begin{conjecture}\label{conj2}
\begin{enumerate}[leftmargin=*]
\item $p_{\infty}(\Omega,k)\in\{1,\infty\}$ and either the energy $\fL_{k,p}(\Omega)$ is constant with respect to $p$ or it is strictly increasing with respect to $p$.
\item Numerical simulations suggest that $p_{\infty}(\Omega,k)=1$ if 
\begin{itemize}
\item $\Omega$ is a disk and $k\in\{2,3,4,5\}$,
\item $\Omega$ is a square and $k\in\{2,4\}$,
\item $\Omega$ is an equilateral triangle and $k=n(n+1)/2$ with $n\geq1$.
\end{itemize}
\end{enumerate}
\end{conjecture}

}

\subsection*{Acknowledgments}
This work was partially supported by the ANR (Agence Nationale de la Recherche), project OPTIFORM, n$^\circ$ANR-12-BS01-0007-02 { and by the Project ``Analysis and simulation of optimal shapes - application to lifesciences'' of the Paris City Hall.} 
The authors thank Michael Floater for suggesting us to look more carefully at $k$-partitions of the equilateral triangle when $k$ is a triangular number. { The authors wish to thank the reviewers for valuable comments which helped improve the quality of the paper.}

\bibliography{BiblioBB}
\bibliographystyle{abbrv}

\end{document}